\documentclass[12pt]{amsart}
\usepackage{amsmath,amsthm,amsfonts,amssymb,amscd,flafter,pinlabel}
\usepackage[mathscr]{eucal}
\usepackage{graphics}
 \usepackage[all,knot,arc]{xy}

\usepackage[usenames,dvipsnames]{color}
\usepackage{subfigure}
\usepackage{graphicx}

\usepackage{mathtools}

\usepackage{pgf}
\usepackage{tikz}
\usetikzlibrary{arrows,automata}
\usepackage[latin1]{inputenc}

\usepackage[colorlinks=true,linkcolor=blue,citecolor=blue,urlcolor=blue, linktocpage=true]{hyperref}

\makeatletter
\newcommand{\@dotsep}{4.5}
\renewcommand*{\dotfill}{
  \leavevmode\leaders
  \hbox{$\m@th \mkern \@dotsep mu\hbox{.}\mkern \@dotsep mu$}
  \hfill\kern\z@
}
\def\@tocline#1#2#3#4#5#6#7{\relax
  \ifnum #1>\c@tocdepth % then omit
  \else
    \par \addpenalty\@secpenalty\addvspace{#2}%
    \begingroup \hyphenpenalty\@M
    \@ifempty{#4}{%
      \@tempdima\csname r@tocindent\number#1\endcsname\relax
    }{%
      \@tempdima#4\relax
    }%
    \parindent\z@ \leftskip#3\relax \advance\leftskip\@tempdima\relax
    \rightskip\@pnumwidth plus4em \parfillskip-\@pnumwidth
    #5\leavevmode\hskip-\@tempdima #6\nobreak\relax
    \dotfill\hbox to\@pnumwidth{\@tocpagenum{#7}}\par
    \nobreak
    \endgroup
  \fi}
\def\l@figure{\@tocline{0}{3pt plus2pt}{0pt}{2.8pc}{}}
\makeatother

\newtheorem{theorem}{Theorem}[section]
\newtheorem{corollary}[theorem]{Corollary}

\newtheorem{question}[theorem]{Question}
\newtheorem{proposition}[theorem]{Proposition}

\newtheorem{lemma}[theorem]{Lemma}

\newtheorem{definition}[theorem]{Definition}

\newtheorem{exm}[theorem]{Example}

\newtheorem {note}[theorem]{Note}

\newtheorem{remark}[theorem]{Remark}

\newtheorem{convention}[theorem]{Convention}

\newcommand{\R}{\ensuremath{\mathbb{R}}}
\newcommand{\Z}{\ensuremath{\mathbb{Z}}}

\newcommand\alphas{\mbox{\boldmath$\alpha$}}
\newcommand\betas{\mbox{\boldmath$\beta$}}

\newcommand\F{\mathbb F}

\def\x{\mathbf{x}}

\def\y{\mathbf{y}}

\def\X{\mathbb{X}}
\def\O{\mathbb{O}}
\def\gen{\mathfrak{S}}

\def\k{\mathbf{k}}

\def\remark{{\bf {\bigskip}{\noindent}Remark. }}

\def\bar{\overline}

\newcommand{\bdy}{\partial}

\newcommand{\az}{\mathcal{A}(\zz)}
\newcommand{\cala}{\mathcal{A}}
\newcommand{\ainf}{\mathcal{A}_\infty}
\newcommand{\gr}{\textrm{gr}}
\newcommand{\gens}{\mathfrak S}

\newcommand{\sss}{{\bf s}}
\newcommand{\ttt}{{\bf t}}

\newcommand{\xx}{{\bf x}}
\newcommand{\yy}{{\bf y}}

\DeclareMathOperator{\Int}{Int}
\DeclareMathOperator{\Inv}{Inv}
\DeclareMathOperator{\inv}{inv}
\DeclareMathOperator{\ind}{ind}
\DeclareMathOperator{\id}{id}

\newcommand{\HH}{\mathcal{H}}

\newcommand{\zz}{\mathcal Z}

\newcommand{\D}{\mathcal{D}}
\newcommand{\I}{\mathcal{I}}

\newcommand{\T}{\mathcal{T}}

\newcommand{\C}{\mathbb C}

\newcommand{\OO}{\mathbb O}
\newcommand{\XX}{\mathbb X}

\newcommand{\xxx}{\mathbf{x}}
\newcommand{\yyy}{\mathbf{y}}
\newcommand{\zzz}{\mathbf{z}}
\newcommand{\www}{\mathbf{w}}

\newcommand{\balpha}{\boldsymbol\alpha}
\newcommand{\bbeta}{\boldsymbol\beta}
\newcommand{\bgamma}{\boldsymbol\gamma}

\newcommand{\ct}{\mathit{CT}}

\newcommand{\cfd}{\mathit{CFD}}

\newcommand{\cfta}{\mathit{CFAT}}
\newcommand{\cftd}{\mathit{CFTD}}
\newcommand{\cftda}{\mathit{CFDTA}}
\newcommand{\cfk}{\mathit{CFK}}
\newcommand{\cft}{\mathit{CFT}}
\newcommand{\hfk}{\mathit{HFK}}

\newcommand{\cfdhat}{\widehat{\cfd}}

\newcommand{\cfkhat}{\widehat{\cfk}}
\newcommand{\hfkhat}{\widehat{\hfk}}
\newcommand{\cftahat}{\widehat{\cfta}}
\newcommand{\cftdhat}{\widehat{\cftd}}
\newcommand{\cfkt}{\widetilde{\cfk}}
\newcommand{\cftt}{\widetilde{\cft}}
\newcommand{\cftdahat}{\widehat{\cftda}}
\newcommand{\cftat}{\widetilde{\cfta}}
\newcommand{\cftdt}{\widetilde{\cftd}}
\newcommand{\cftdat}{\widetilde{\cftda}}
\newcommand{\cfkm}{\cfk^-}
\newcommand{\hfkm}{\hfk^-}

\newcommand{\ctm}{\ct^-}

\newcommand{\brho}{\boldsymbol\rho}

\def\P{\mathcal{P}}
\def\Pp{\boldsymbol{\P}}

\def\I{\mathcal{I}}
\def\a{\mathbf{a}}
\def\e{\boldsymbol{\epsilon}}
\def\ah{\mathbf{a}_{\frac12}}
\def\b{\mathbf{b}}
\def\bh{\mathbf{b}_{\frac12}}

\def\n{\mathbf{n}}

\def\Inv{\mathrm{Inv}}
\def\stphi{(S,T,\phi)}
\def\H{\mathcal{H}}
\def\E{\mathcal{E}}
\def\Aa{\mathcal{A}}

\def\TO{T_\O}
\def\SO{S_\O}
\def\SX{S_\X}
\def\TX{T_\X}
\def\CDA^-{\mathit{CDA}^-}
\def\CATA{\mathit{CATA}}
\def\CDTD{\mathit{CDTD}}
\def\AA{\ensuremath{\mathit{AA}}}
\def\DD{\ensuremath{\mathit{DD}}}
\def\DA{\ensuremath{\mathit{DA}}}
\def\AD{\ensuremath{\mathit{AD}}}

\begin{document}

\title[{Combinatorial Tangle Floer Homology}]{Combinatorial Tangle Floer Homology}

\author{Ina Petkova}
\address {Department of Mathematics, Columbia University\\ New York, NY 10027}
\email {ina@math.columbia.edu}
\urladdr{\href{http://math.columbia.edu/~ina}{http://math.columbia.edu/\~{}ina}}
\author[Vera V\'ertesi]{Vera V\'ertesi}
\address{Institut de Recherche Math\'ematique Avanc\'ee \\Universit\'e de Strasbourg}
\email{vertesi@math.unistra.fr}

\maketitle

\begin{abstract}
In this paper we extend the idea of bordered Floer homology to knots and links in $S^3$: Using a specific Heegaard diagram, we construct gluable combinatorial invariants of tangles in $S^3$, $D^3$ and $I\times S^2$. The special case of $S^3$ gives back  a stabilized version of knot Floer homology.
\end{abstract}

%%%%%%%%%%%%%%%%%%%%%%%%%%%%%%%%%%%%%%%%%%%%%%%%%%%%%%%

%%%%%%%%%%%%%%%%%%%%%%%%%%%%%%%%%%%%%%%%%%%%%%%%%%%%%%%
%\begin{abstract}
%%%%%%%%%%%%%%%%%%%%%%%%%%%%%%%%%%%%%%%%%%%%%%%%%%%%%%%

%%%%%%%%%%%%%%%%%%%%%%%%%%%%%%%%%%%%%%%%%%%%%%%%%%%%%%%	
%\end{abstract}
%%%%%%%%%%%%%%%%%%%%%%%%%%%%%%%%%%%%%%%%%%%%%%%%%%%%%%%

%%%%%%%%%%%%%%%%%%%%%%%%%%%%%%%%%%%%%%%%%%%%%%%%%%%%%%%

\tableofcontents

% Workaround: suppress putting list of figures in main TOC
\let\contentsname\listfigurename
\listoffigures

%%%%%%%%%%%%%%%%%%%%%%%%%%%%%%%%%%%%%%%%%%%%%%%%%%%%%%%

%%%%%%%%%%%%%%%%%%%%%%%%%%%%%%%%%%%%%%%%%%%%%%%%%%%%%%%
% !TEX root = ../tanglefloer.tex
%%%%%%%%%%%%%%%%%%%%%%%%%%%%%%%%%%%%%%%%%%%%%%%%%%%%%%%

\section{Introduction} % (fold)
\label{sec:introduction}
%%%%%%%%%%%%%%%%%%%%%%%%%%%%%%%%%%%%%%%%%%%%%%%%%%%%%%%

%%%%%%%%%%%%%%%%%%%%%%%%%%%%%%%%%%%%%%%%%%%%%%%%%%%%%%%
% section introduction 
%%%%%%%%%%%%%%%%%%%%%%%%%%%%%%%%%%%%%%%%%%%%%%%%%%%%%%%

Knot Floer homology is a categorification of the Alexander polynomial, defined by Ozsv\'ath and Szab\'o \cite{hfk}, and independently by Rasmussen \cite{jrth}, in the early 2000s. To a knot or a link one associates a filtered graded  chain complex over the field of two elements  $\F_2$ or over a polynomial ring $\F_2[U_1, \ldots, U_n]$. The filtered chain homotopy type of this complex is a powerful invariant of the knot. For example, it detects genus \cite{hfkg}, fiberedness \cite{pgf, ynf}, and gives a bound on the four-ball genus \cite{tau}. The definition of knot Floer homology is based on finding a Heegaard diagram presentation for the knot and defining a chain complex by counting  certain pseudo-holomorphic curves in a symmetric product of the Heegaard surface. Suitable choices of Heegaard diagrams (for example, grid diagrams as in \cite{mos,most}, or nice diagrams as in \cite{sarkarwang}) lead to combinatorial descriptions of knot Floer homology. However, in its nature  knot Floer homology is a  ``global" invariant -- one needs a picture of the entire knot to define it -- and local modifications are only partially understood, see for example \cite{hfk, oszskein, mskein}.

Around the same time that knot Floer homology came to life, Khovanov introduced another knot invariant, a categorification of the Jones polynomial now known as Khovanov homology \cite{kh1}.  Khovanov's construction is somewhat simpler in nature, as one builds a chain complex generated by the different resolutions of the knot. Khovanov homology has an extension to tangles \cite{kh3}, thus local modifications can be understood on a categorical level. 

In this paper, we extend knot Floer homology by defining a combinatorial Heegaard Floer type invariant for tangles. Note that a similar extension exists for Heegaard Floer homology, which is an invariant of closed $3$-manifolds, generalizing it to manifolds with boundary \cite{bfh2}; this extension is called bordered Floer homology.

\subsection{Tangle Floer invariants}

\begin{figure}[h]
\centering
     \includegraphics[scale=1]{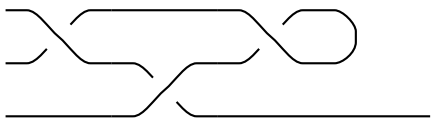} 
      \vskip .2 cm
       \caption[A projection of a tangle in $S^2\times I$.]{\textbf{A projection of a tangle in $S^2\times I$.}}\label{fig:tangle}
\end{figure}

A tangle (see Figure \ref{fig:tangle} and Subsection \ref{ssec:tangles} for  precise definitions) is a properly embedded 1--manifold in $D^3$ or $I\times S^2$. Inspired by \cite{LOT}, we define: 
\begin{itemize}
\item[-] a differential graded algebra $\Aa(\P)$ for any finite set of signed points $\P$ on the equator of $S^2$;
\item[-] a right type $\mathit{A}$ module $\widehat{\mathit{CFTA}}(\T)$ over $\Aa(\partial\T)$ for any tangle $\T$ in $D^3$;
\item[-] a left type $\mathit{D}$ module $\widehat{\mathit{CFDT}}(\T)$ over $\Aa(-\partial\T)$ for any tangle $\T$ in $D^3$;
\item[-] a left-right $\Aa(-\partial^0\T)$-$\Aa(\partial^1\T)$ type $\mathit{DA}$ bimodule $\widehat{\mathit{CFDTA}}(\T)$ for any tangle $\T$ in $I\times S^2$.
\end{itemize}
The above (bi)modules are topological invariants of the tangle. (See Theorems \ref{thm:ainv}, \ref{thm:dinv} and \ref{thm:dainv} for the precise statements.)
\begin{theorem}\label{thm:simpleinv}
For a tangle $\T$ in $D^3$ the type $\mathit{A}$ equivalence class of the module $\widehat{\mathit{CFTA}}(\T)$ is a topological invariant of $\T$, and the  type $\mathit{D}$ equivalence class of  the module $\widehat{\mathit{CFDT}}(\T)$ is a topological invariant of $\T$. For a tangle $\T$ in $S^2\times I$ the  type $\mathit{DA}$ equivalence class of the bimodule $\widehat{\mathit{CFDTA}}(\T)$ is a topological invariant of $\T$.
\end{theorem}
Furthermore, the invariants behave well under compositions of tangles. (See Theorem  \ref{thm:cft_pairing} and Corollary \ref{cor:hat_pairing} for the precise statement.)  \footnote{In each of the equivalences in Theorems \ref{thm:simplepairing} and \ref{thm:simpleknot}, the left hand side should also be tensored with $V^{\otimes(|\T_1|+|\T_2| - |\T_1\circ \T_2|)}$, where $V = \F_2\oplus \F_2$ has one summand in bigrading $(0,0)$ and the other summand in bigrading $(-1, -1)$. This is discussed in the full statements of the theorems, and omitted here for simplicity.}
\begin{theorem}\label{thm:simplepairing}
Suppose that $\T_1$ and $\T_2$ are tangles in $S^2\times I$ such that $\partial^1\T_1=-\partial^0\T_2$. Then up to type $\mathit{DA}$ equivalence
\[ \widehat{\mathit{CFDTA}}(\T_1\circ\T_2)\simeq \widehat{\mathit{CFDTA}}(\T_1)\widetilde{\otimes} \widehat{\mathit{CFDTA}}(\T_2).\]
\end{theorem}
Thus, the above definitions give a functor from the category of oriented tangles $\mathcal{OTAN}$ to the category of bigraded type $\DA$ bimodules up to type $\DA$ equivalence. In other words, our invariant behaves like  a  $(0+1)$-dimensional TQFT.\footnote{Note that it is not a proper TQFT as the target is not the category of vector spaces, and the functor does not respect the monoidal structure of the categories. In fact there is no obvious monoidal structure on the category of type $\DA$ structures.}

Note that there are analogs of Theorem \ref{thm:simplepairing}  if one  of the tangles  is in $D^3$. When $\T_1$ and $\T_2$ are both in $D^3$, then their composition $\T_1\circ\T_2$ is a knot (or a link), and we recover knot Floer homology:
\begin{theorem}\label{thm:simpleknot}
Suppose that $\T_1$ and $\T_2$ are tangles in $D^3$ with $\partial\T_1=-\partial\T_2$, and let $K=\T_1\circ\T_2$ be their composition. Then 
\[\widehat{\mathit{CFK}}(K)\otimes W\simeq \widehat{\mathit{CFTA}}(\T_1)\widetilde{\otimes}  \widehat{\mathit{CFDT}}(\T_2)\]
where $W=\F_2\oplus \F_2$ with Maslov and Alexander  bigradings $(M,A)=(0,0)$ and $(-1,0)$. 
\end{theorem}

The combinatorial description of the invariants depends on the use of a certain Heegaard diagram associated to the tangle (See Figure \ref{fig:HD}.) This diagram is ``nice'' in the sense of Sarkar-Wang \cite{sarkarwang}. 
\begin{figure}[h]
 \centering
       \includegraphics[scale=.6]{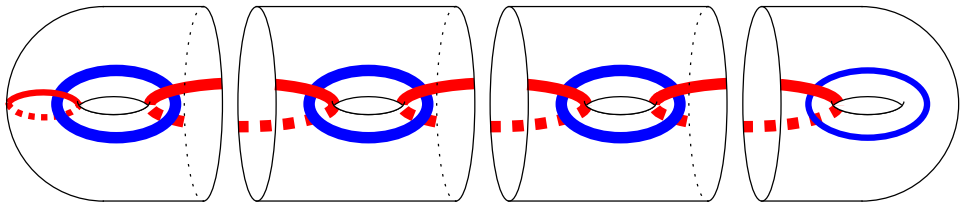} 
       \vskip .2 cm
       \caption[A Heegaard diagram associated to a tangle.]{\textbf{A Heegaard diagram associated to a tangle.} The thick lines denote parallel $\alpha$ and $\beta$ curves. The number of twice punctured tori in the middle depends on how complicated the tangle is. This figure shows the Heegaard diagram for a closed link. Diagrams for tangles can be obtained by deleting one or both of the once punctured tori from the sides.}\label{fig:HD}
\end{figure}
The use of this diagram enables a purely combinatorial description of the generators, as partial matchings of a bipartite graph associated to the tangle. (See Figure \ref{fig:HD} for an example.)
\begin{figure}[h]
\centering
    \includegraphics[scale=1]{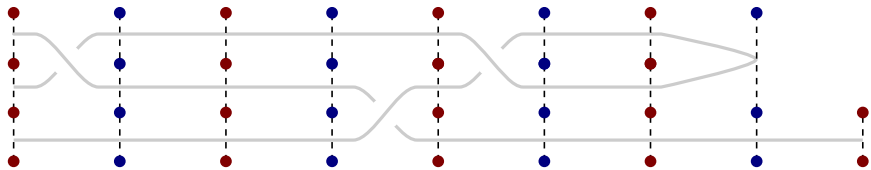} 
      \vskip .2 cm
       \caption[The bipartite graph associated to the tangle of Figure \ref{fig:tangle}.]{\textbf{The bipartite graph associated to the tangle of Figure \ref{fig:tangle}.} The edges (not drawn) are between the consecutive vertex-sets.}\label{fig:bipartite}
\end{figure}

In this paper, we develop two versions of the invariants: one over $\F_2$, which we call a \emph{tilde} version, and an enhanced, \emph{minus}, version over $\F_2[U_1, \ldots, U_n]$. As  Theorem \ref{thm:simpleknot} depends only on a Heegaard diagram description, it holds for both versions. However, we currently only have proofs for the ``tilde'' versions of Theorems \ref{thm:simpleinv} and \ref{thm:simplepairing}. This is due to the fact that our proofs rely on analytic techniques. In Subsection \ref{subsec:sampleproof} we give evidence for the existence of completely combinatorial proofs of Theorems \ref{thm:simpleinv} and \ref{thm:simplepairing} in the ``minus'' version.

We also develop an ungraded ``tilde" version of tangle Floer homology for tangles in arbitrary manifolds with boundary $S^2$ or $S^2\coprod S^2$. Versions of the above theorems hold in this more general case too, see Theorems \ref{thm:dinv}, \ref{thm:ainv}, \ref{thm:dainv}, \ref{thm:cft_pairing}, and Corollary \ref{cor:hat_pairing}.

This TQFT-like description of knot Floer homology allows one to localize questions in Heegaard Floer homology. For instance, in a subsequent note we show that there is a skein exact sequence for tangles. The theory has the potential to help understand the change of knot Floer homology under more complicated local modifications such as mutations, or, for example, help understand the rank of the knot Floer homology of periodic knots.

We hope that our construction may provide a new bridge between Khovanov homology and knot Floer homology. Rasmussen \cite{Ras1} conjectures a spectral sequence connecting the two. It is possible that a relationship between the two theories can be found for simple tangles, and used to prove the conjecture.

The Jones polynomial can be defined in the Reshetikhin-Turaev way, using the vector representation of the quantum algebra $U_q(\mathfrak{sl}_2)$, and since Khovanov's seminal work on categorifying the Jones polynomial, a program for categorification of quantum groups has begun. 
Similarly to the Jones polynomial construction, one can see the Alexander polynomial as a quantum invariant coming from the vector representation $V$ of $U_q(\mathfrak{gl}(1|1))$, see \cite{s1, v1}. However, the categorification $\hfkhat$ of the Alexander polynomial has not yet been understood on a representation theory level. 
In a future paper we show that the decategorification of tangle Floer homology is a tensor power of the vector representation of $U_q(\mathfrak{gl}(1|1))$. We believe that we can build on the structures from this paper to obtain a full categorification of the tensor powers of the vector representation of $U_q(\mathfrak{gl}(1|1))$.

\subsection{Further remarks} Knot Floer homology is defined by counting holomorphic curves in a symmetric product of a Heegaard surface, and for different versions, the projection of those curves to the Heegaard surface is allowed or not allowed to cross two special sets of basepoints $\XX$ and $\OO$.
We develop a theory for tangles that counts curves which cross only $\OO$. While it is hard to define  invariants that count curves which  cross both $\XX$ and $\OO$, it is straightforward to modify the definitions to count curves that cross $\XX$ or $\OO$, but not both. Further, the invariants defined in this paper can be extended over  $\Z$. 

The structures defined in Section \ref{sec:combinatorial} are completely combinatorial, and an algorithm could be programmed to compute the invariants for simple tangles and obtain the knot Floer homology of some new knots. Knots with periodic behavior and knots with low bridge number relative to their grid number are especially suitable.

\subsection{Organization}
After a brief introduction of the relevant algebraic structures in Section \ref{sec:preliminaries}, we turn to defining the invariants from a diagrammatic viewpoint in Section \ref{sec:combinatorial}. In Section \ref{sec:borderedgrid}, we describe the same invariants 
 using a class of diagrams called bordered grid diagrams, as this approach is more suited for some of the proofs and provides a bridge between Section \ref{sec:borderedgrid} and Sections \ref{sec:circles}-\ref{sec:hd_pairing}. Finally, the definitions of the tangle invariants are given in Section \ref{sec:tangleinv}, and their relation to knot Floer homology is proved in Section \ref{sec:apps}. 

Sections  \ref{sec:circles}-\ref{sec:hd_pairing} are devoted to proving invariance by building up a complete homolomorphic theory for tangles in 3--manifolds. The geometric structures (marked spheres) associated to the algebras are introduced in 
Section \ref{sec:circles}, then Section \ref{sec:hdiagrams} describes the various Heegaard diagrams corresponding to tangles in 3--manifolds. The moduli spaces corresponding to these Heeegaard diagrams are defined in Section \ref{sec:moduli}. Then the definitions of the general invariants are given in Section \ref{sec:hd_modules}. The gradings from Section \ref{ssec:grading}  are extended to the  general setting in Section \ref{ssec:morse_gradings}. Section \ref{sec:hd_pairing}  contains the full statements and proofs of Theorems \ref{thm:simplepairing} and \ref{thm:simpleknot}.

\subsection*{Acknowledgments}
We would like to thank  Ko Honda, Mikhail Khovanov, Antony Licata, Robert Lipshitz, Peter Ozsv\'ath, Zolt\'an Szab\'o, and Yin Tian for helpful discussions.  We also thank  Vladimir Fock, Paolo Ghiggini, Eli Grigsby, Ciprian Manolescu, and Andr\'as Stipsicz  for useful comments. 
 IP was partially supported by the AMS-Simons Travel Grant.
VV was supported by ERC Geodycon, OTKA grant number NK81203 and NSF grant number 1104690. 
This collaboration began in the summer of 2013 during the CRM conference ``Low-dimensional topology after Floer". We thank the organizers and CRM for their hospitality.
We also thank the referee for many helpful suggestions.

%%%%%%%%%%%%%%%%%%%%%%%%%%%%%%%%%%%%%%%%%%%%%%%%%%%%%%%

%%%%%%%%%%%%%%%%%%%%%%%%%%%%%%%%%%%%%%%%%%%%%%%%%%%%%%%
% !TEX root = ../tanglefloer.tex
%%%%%%%%%%%%%%%%%%%%%%%%%%%%%%%%%%%%%%%%%%%%%%%%%%%%%%%
\section{Preliminaries} % (fold)
\label{sec:preliminaries}
%%%%%%%%%%%%%%%%%%%%%%%%%%%%%%%%%%%%%%%%%%%%%%%%%%%%%%%

%%%%%%%%%%%%%%%% End Main Theorem Leg %%%%%%%%%%%%%%%%%%%%%

% subsection acknowledgements (end)
%%%%%%%%%%%%%%%%%%%%%%%%%%%%%%%%%%%%%%%%%%%%%%%%%%%%%%%
% section introduction (end)
%%%%%%%%%%%%%%%%%%%%%%%%%%%%%%%%%%%%%%%%%%%%%%%%%%%%%%%

\subsection{Modules, bimodules, and tensor products}

In this paper, we work with the same types of algebraic structures used in bordered Floer homology; cf. \cite{bfh2, bimod}. Below we recall the main definitions. For more detail, see \cite[Section 2]{bimod}.

Let $A$ be a unital differential graded algebra with differential $d$ and multiplication $\mu$ over a base ring ${\bf k}$. In this paper, ${\bf k}$ will always be a direct sum of copies of $\F_2 = \Z/2\Z$. For the algebras we define in the later sections, the base ring for all modules and tensor products is the ring of idempotents.

A \emph{(right) $\cala_\infty$-module over $A$}, or a \emph{type $A$ structure over $A$} is a graded module $M$ over ${\bf k}$, equipped with maps 
$$m_i: M\otimes A[1]^{\otimes (i-1)}\to M[1],$$ 
for $i \geq 1$,  satisfying the compatibility conditions
\begin{align*}
0&= \sum_{i+j = n+1}m_i(m_j(x, a_1, \cdots, a_{j-1}), \cdots , a_{n-1})\\
&+ \sum_{i=1}^{n-1} m_n(x, a_1,\cdots, a_{i-1}, d(a_i),\cdots, a_{n-1})\\
&+ \sum_{i=1}^{n-2} m_{n-1}(x, a_1,\cdots, a_{i-1}, (\mu(a_i,a_{i+1})),\cdots, a_{n-1}).
\end{align*}
A type $A$ structure is \emph{strictly unital} if $m_2(x, 1) = x$ and $m_i(x, a_1,\cdots, a_{i-1})=0$ if $i>2$ and some $a_j \in {\bf k}$. We assume all type $A$ structures to be strictly unital. 

We say that $M$ is \emph{bounded} if  $m_i=0$ for all sufficiently large $i$.

A \emph{(left) type $D$ structure}  over $A$ is a graded ${\bf k}$-module $N$, equipped with a homogeneous map
$$\delta:N\to (A\otimes N)[1]$$
satisfying the compatibility condition
$$(d\otimes \id_N)\circ \delta + (\mu\otimes \id_N)\circ (\id_A\otimes \delta)\circ \delta= 0.$$
We can define maps
$$\delta_k:N\to (A^{\otimes k}\otimes N)[k]$$
inductively by
$$\delta_k = \left\{  \begin{array}{ll}
\id_N & \textrm{ for } k=0\\
(\id_A\otimes \delta_{k-1})\circ \delta & \textrm{ for } k\geq 1 
\end{array}\right.$$

 A type $D$ structure is  \emph{bounded} if for any $x \in N$, $\delta_i(x) = 0$ for all sufficiently large $i$.
 
 One can similarly define left type $A$ structures and right type $D$ structures. 

If $M$ is a right $\cala_\infty$-module over $A$ and $N$ is a left type $D$ structure, and at least one of them is bounded, we can define the \emph{box tensor product} $M\boxtimes N$ to be the  vector space $M\otimes N$ with differential 
$$\bdy: M\otimes N \to (M\otimes N)[1]$$
defined by 
$$\bdy = \sum_{k=1}^{\infty}(m_k\otimes \id_N)\circ(\id_M\otimes \delta_{k-1}).$$
The boundedness condition guarantees that the above sum is finite. In that case $\bdy^2= 0$ and $M\boxtimes N$ is a graded chain complex. In general (boundedness is not required), one can think of a type $D$ structure as a left $\ainf$-module, and take an $\ainf$ tensor product $\widetilde\otimes$, see \cite[Section 2.2]{bfh2}.

Given  unital differential graded algebras $A$ and $B$ over  ${\bf k}$ and  ${\bf j}$ with differential and multiplication $d_A$, $d_B$, $\mu_A$, and $\mu_B$,  respectively, four types of bimodules can be defined in a similar way: type $\DD$, $\AA$, $\DA$, and $\AD$. See \cite[Section 2.2.4]{bimod}.

An \emph{$\ainf$-bimodule} or \emph{type $\AA$ bimodule over $A$ and $B$} is a graded  $({\bf k}, {\bf j})$-bimodule $M$, together with degree $0$  maps
$$m_{i,1,j}:A[1]^{\otimes i}\otimes M\otimes B[1]^{\otimes j}\to M[1]$$
subject to compatibility conditions analogous to those for $A$ structures, see \cite[Equation 2.2.38]{bimod}.

We assume all $\AA$ bimodules to be \emph{strictly unital}, i.e. $m_{1,1,0}(1,x) = x = m_{0,1,1}(x,1)$ and $m_{i,1,j}(a_1, \ldots, a_i, x, b_1, \ldots, b_j) = 0$ if $i+j>1$ and some $a_i$ or $b_j$ lies in ${\bf k}$ or ${\bf j}$.

A \emph{type $\DA$ bimodule over $A$ and $B$} is a graded  $({\bf k}, {\bf j})$-bimodule $M$, together with degree $0$, $({\bf k}, {\bf j})$-linear maps
$$\delta^1_{1+j}: M\otimes B[1]^{\otimes j}\to A\otimes M[1],$$
satisfying a compatibility condition combining those for $A$ and $D$ structures, see \cite[Definition 2.2.42]{bimod}. 
 
 A \emph{type $\AD$ structure} can be defined similarly, with the roles of $A$ and $B$ interchanged.

A \emph{type $\DD$ structure over $A$ and $B$} is a type $D$ structure over $A\otimes_{\F_2}B^{\mathrm{op}}$. In other words, it is a graded $({\bf k}, {\bf j})$-bimodule $M$ and a degree $0$ map $\delta^1:A\otimes M\otimes B[1],$ again with an appropriate compatibility condition.

Note that when $A$ or $B$ is the trivial algebra $\{1\}$, we get a left or a right $A$ or $D$ structure over the other algebra. 

There are notions of boundedness for bimodules similar to those for one-sided modules. There are various tensor products for the various compatible pairs of bimodules. We assume that one of the factors is bounded, and briefly lay out the general description. For details, see \cite[Section 2.3.2]{bimod}.

Let $M$ and $N$ be two structures such that $M$ is module or bimodule with a right  type $A$ action by an algebra $\Aa$, and $N$ is a left type $D$ structure over $\Aa$, or a type $\DA$ or type $\DD$ structure over $\Aa$ on the left and some algebra on the right, with $M$ right bounded or $N$ left bounded. As a chain complex, define 
$$M\boxtimes N = \mathcal F(M)\boxtimes \mathcal F(N),$$
where  $\mathcal F(M)$ forgets the left action on $M$, i.e. turns $M$ into a right type $A$ structure over $\Aa$, and  $\mathcal F(N)$ forgets the right action on $N$, i.e. turns $N$ into a left type $D$ structure over $\Aa$. Endow $M\boxtimes N$ with the bimodule structure maps arising from the left action on $M$ and the right action on $N$. Note that this also makes sense when $M$ is a right type $A$ structure, or $N$ is a left type $D$ structure.

 In general (boundedness is not required), one can think of $N$ as a structure with a left $\Aa$ action, by considering $\Aa\boxtimes N$ (where $\Aa$ is viewed as a bimodule over itself), and take an $\ainf$ tensor product $M\widetilde\otimes N := M\widetilde\otimes (\Aa\boxtimes N)$. Whenever they are both defined, the two tensor products yield equivalent structures, see \cite[Proposition 2.3.18]{bimod}. 

For definitions of morphisms of type $A$, $D$, $\AA$, $\AD$, $\DA$, and $\DD$ structures, and for definitions of the respective types of homotopy equivalences, see \cite[Section 2]{bimod}.

\subsection{Tangles}\label{ssec:tangles}

In this paper we only consider tangles in 3--manifolds with boundary  $S^2$ or $S^2\coprod S^2$, or in closed 3--manifolds.

\begin{definition}\label{def:sphere}
An \emph{$n$-marked sphere} $\mathcal S$ is a sphere $S^2$ with $n$ oriented points $t_1,\dots,t_n$  on its equator $S^1\subset S^2$ numbered respecting the orientation of $S^1$. 
\end{definition}

\begin{definition}\label{tangle_def}
A \emph{marked $2n$-tangle}  $\T$ in an oriented $3$--manifold $Y$ with $\bdy Y \cong S^2$ is a properly embedded $1$--manifold $T$ with $(\bdy Y, \bdy T)$ identified with a $2n$-marked sphere $\mathcal S$.

A \emph{marked $(m,n)$-tangle} $\T$ in an oriented $3$--manifold $Y$ with two boundary components  $\partial^0 Y\cong S^2$ and $\partial^1
Y\cong S^2$ is a properly embedded 1--manifold $T$ with $(\partial^0
Y,\partial^0 Y\cap\partial T)$ and $(\partial^1 Y,\partial^1
Y\cap\partial T)$ each identified with an $m$-marked sphere and an
$n$-marked sphere.
We denote $\bdy T$ along with the ordering information by $\bdy \T = \bdy^0\T\coprod \bdy^1\T$.

\end{definition}

We denote the number of connected components of a tangle $\T$ by $|\T|$. Note that we allow for a tangle to also  have closed components. 

Given a marked sphere $\mathcal S = (S^2, t_1, \ldots, t_n)$, we denote $(-S^2, -t_1, \ldots, -t_n)$ by $-\mathcal S$. 
 If  $\T_1$ and $\T_2$ are two marked tangles in $3$--manifolds $Y_1$ and $Y_2$, where  a component of $(\bdy Y_1, \bdy T_1)$ is identified with a marked sphere $\mathcal S$ and a component of  $(\bdy Y_2, \bdy T_2)$ is identified with  $-\mathcal S$, we can form the union $\T_1\cup_{\mathcal S} \T_2$ by  identifying $Y_1$ and $Y_2$ along these two boundary components.

For a pair $(Y, \T)$, if a component $\bdy^i Y$ of the boundary of $Y$ is identified with $\mathcal S = (S^2, t_1, \ldots, t_n)$, so that $\bdy^i \T$ is the ordered set of points $(t_1, \ldots, t_n)$, we use $-\bdy^i\T$ to denote $(-t_1, \ldots, -t_n)$. So we can glue two tangles $\T_1$ and $\T_2$ along boundary components $\bdy^i\T_1$ and $\bdy^j\T_2$ exactly when $\bdy^i\T_1 = -\bdy^j\T_2$.

 In most of this paper, we only consider tangles in product spaces, where the identification of the boundary with a marked sphere is implied, and the ordering in $\bdy \T$ encodes all the information.

Tangles in subsets of $S^3=\R^3\cup\{*\}$ for example in $D^3$, $I\times S^2$ or $S^3$ itself can be given by their projection to $(-\infty,c]\times \R$ or $[d,\infty)\times \R$, $[c,d]\times \R$
or $\R^2$. We can always arrange a projection to be smooth and to have no triple points, and to have only transverse intersections. 

\begin{definition}
A tangle $\T$ is \emph{elementary} if it contains at most one double point or vertical tangency (a tangency of the form $\{f\}\times \R$). 
\end{definition}
Thus an elementary tangle can consist of straight strands (as on the first picture of Figure \ref{fig:elemtangle}),  can have one crossing (as on the second pictures of Figures \ref{fig:elemtangle} and \ref{fig:coelemtangle}), can be a cap (as on the third picture of Figure \ref{fig:elemtangle}),
or can be a cup (as on the last picture of Figure \ref{fig:elemtangle}).
The above examples are tangles in $[c,d]\times \R$. There is no elementary tangle projection in 
$\R^2$,   an elementary tangle projection in $(-\infty,c]\times \R$ is a single cap, and  an elementary tangle projection in $[d,\infty)\times \R$ is a single cup.

The following two propositions are well known to tangle theorists, and we do not rely on them in the paper, so we only include outlines of their proofs. 

\begin{proposition}
Any tangle projection is the concatenation of elementary tangles.
\end{proposition} 
\begin{proof} If necessary, one can isotope each tangency and/or double point slightly to the left or right, so that no two have the same horizontal coordinate.
\end{proof} 

Further:
\begin{proposition} The concatenations of two sequences of elementary tangles represent isotopic tangles if and only if they are related by a finite sequence of the  moves depicted in Figure \ref{fig:relations}.\end{proposition}
\begin{figure}[h]
 \centering
   \labellist
       \pinlabel $T$ at 709 192
       \pinlabel $T$ at 789 192
       \pinlabel $T$ at 861 192
       \pinlabel $T$ at 744 120
       \pinlabel $T$ at 825 120
       \pinlabel $S$ at 704 68
       \pinlabel $S$ at 865 68
             \endlabellist
       \includegraphics[scale=.48]{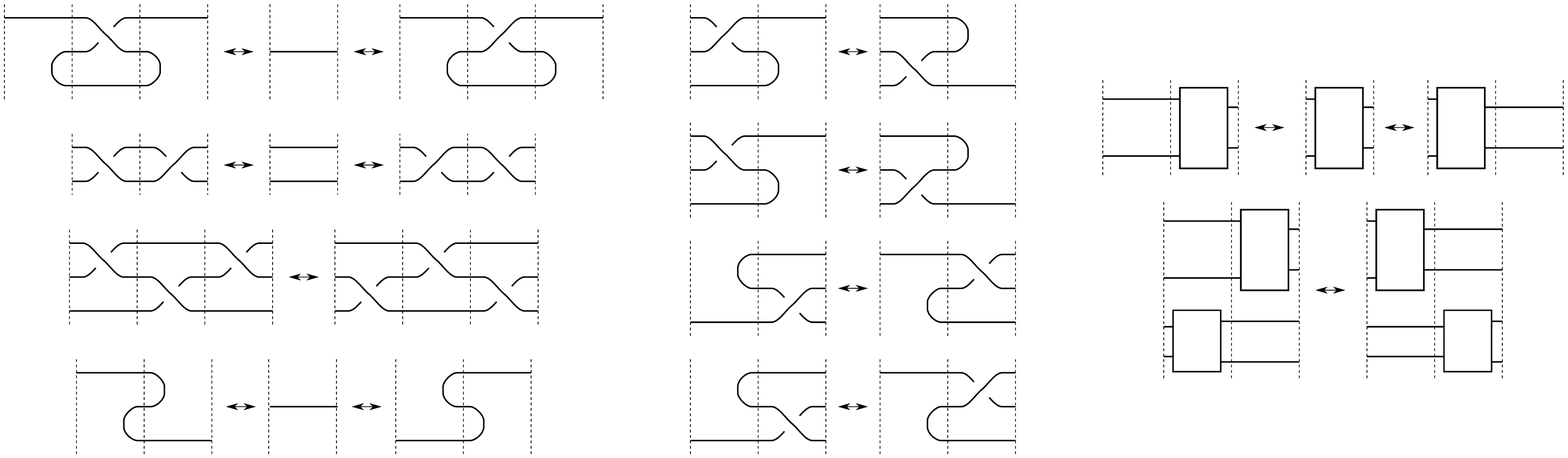} 
       \vskip .2 cm
       \caption[Relations of elementary tangles.]{\textbf{Relations of elementary tangles.} In all diagrams there may be additional horizontal straight strands running above and/or below what is shown. Left column (top to bottom): Reidemeister I move, Reidemeister II move, Reidemeister III move, ``zig-zag'' move. Middle column: crossing-cap/cup slide moves. Right column (top to bottom): introducing straight strands to either side of a tangle or removing them, and sliding two vertically-stacked tangles past each other.}\label{fig:relations}
\end{figure}

\begin{proof}
The three Reidemeister moves are the standard moves that change the combinatorics of the diagram. 

Using elementary Morse theory one can see that the other four types of moves are exactly the moves needed to move between two isotopic diagrams with the same combinatorics. Look at the height function obtained by projecting the tangle to the $x$-coordinate. 
The zig-zag move corresponds to canceling an index 0 critical point with an index 1 critical point or introducing a pair of such critical points.
 The crossing-cup slide moves are isotopies that do not change the Morse function, but slide a strand over or under a critical point. 
 Introducing straight strands simply means taking one extra cut near one of the boundaries of a tangle.
Sliding two vertically-stacked tangles past each other corresponds to moving through a one-parameter family of Morse functions that changes the relative heights for the two disjoint tangles. 
\end{proof}

In this paper, we define a (bi)module for each elementary tangle explicitly, and then define a (bi)module for any tangle by decomposing it into elementary pieces and taking the tensor product of the associated (bi)modules. We prove invariance of the decomposition using analytic techniques (the bordered Heegaard diagrams associated to isotopic tangles are related by Heegaard moves). 
We hope to also find a completely combinatorial proof, i.e. we wish to show directly that the moves from Figure \ref{fig:relations} result in homotopy equivalent tensor products. As a first step, in Section \ref{subsec:sampleproof}  we show invariance under the Reidemeister II and III moves.

%%%%%%%%%%%%%%%%%%%%%%%%%%%%%%%%%%%%%%%%%%%%%%%%%%%%%%%

%%%%%%%%%%%%%%%%%%%%%%%%%%%%%%%%%%%%%%%%%%%%%%%%%%%%%%%
% !TEX root = ../tanglefloer.tex
%%%%%%%%%%%%%%%%%%%%%%%%%%%%%%%%%%%%%%%%%%%%%%%%%%%%%%%

\section{Generalized strand modules and algebras} % (fold)
\label{sec:combinatorial}
%%%%%%%%%%%%%%%%%%%%%%%%%%%%%%%%%%%%%%%%%%%%%%%%%%%%%%%
The aim of this paper is to give a 0+1 $\mathrm{TQFT}$-like description of knot Floer homology. The  description is based on a special kind of Heegaard diagram  associated to a knot (or a link) disjoint union an unknot. 

Given a tangle $\T$, by cutting it into elementary tangles like the ones in Figure \ref{fig:elemtangle}, we can put it on a Heegaard diagram like the one depicted on Figure \ref{fig:HD}, where the genus of the diagram is the number of elementary pieces. The parts of the Heegaard diagram corresponding to the elementary pieces are depicted on Figures \ref{fig:borderedgrid} and \ref{fig:borderedgrid2}. Note that the Heegaard diagram is obtained by gluing together a once punctured torus, some twice punctured tori, and another once punctured torus.  In the sequel, we will associate 
an algebra to each cut of the tangle, a left type $A$ module
 and a right type $D$ structure to the once punctured tori, and a type $\mathit{DA}$ bimodule to each of the twice punctured tori.

In this section, we will describe the algebras, modules, and  bimodules from a purely combinatorial viewpoint, with no mention of Heegaard diagrams. In Section \ref{sec:borderedgrid}, we relate these structures to bordered diagrams.

In the sequel, we define generalized strand algebras and modules whose structure depends on the extra information, encoded in a structure we will refer to as  \emph{shadow}. We define the \emph{``minus''}-version of the theory, and the \emph{``tilde''}-version can be obtained by setting the $U_O$'s to $0$. In this section we describe the modules and algebras via strand diagrams, but some of the notions feel more natural in the bordered grid diagram reformulation (see Section \ref{sec:borderedgrid}). The reader who is familiar with the strand algebras of \cite{bfh2} should be able to understand the main idea of the definitions just by looking at the examples and the figures. 

Although in this paper the main theorem is only proved for the ``tilde''-version, we have strong evidence that it holds for the ``minus'' version as well. This  is why we develop both versions, but at first reading one can ignore the $U$-powers (i.e. set $U_O=0$) and work in the ``tilde''-version.
%%%%%%%%%%%%%%%%%%%%%%%%%%%%%%%%%%%%%%%%%%%%%%%%%%%%%%%
% SHADOWS
%%%%%%%%%%%%%%%%%%%%%%%%%%%%%%%%%%%%%%%%%%%%%%%%%%%%%%%
\subsection{Type $\mathit AA$  structures -- Shadows}\label{ssec:shadows}
The objects underlying all structures are shadows:
\begin{definition}\label{def:shadow}
For $n,m\in \mathbb{N}$, fix sets of integers $\a=\{1,\dots,n\}$ and $\b=\{1,\dots, m\}$, and sets of half-integers $\ah=\{1\frac12,\dots,n-\frac12\}$ and $\bh=\{1\frac12,\dots,m-\frac12\}$. Let $(\SX,\TX,\xi)$ and $(\SO,\TO,\omega)$ be triples such that $\SX,\TO\subset \ah$ and $\TX,\SO\subset \bh$,  $|\TX|=|\SX|$ and $|\TO|=|\SO|$, and $\xi\colon \SX\to \TX$ and $\omega\colon \SO\to \TO$ are two bijections.
The quadruple $\P=(m, n, \xi, \omega)$ is called a \emph{shadow}. 
\end{definition}
Note that  $\TX, \SX$ and $\TO, \SO$ are suppressed from the notation. 
See Figure \ref{fig:shadow} for diagrams of shadows associated to elementary tangles (c.f.\  subsection \ref{ssec:graphtangles}). 
\begin{figure}[h]
 \centering
       \includegraphics
       [scale=1.3]
       {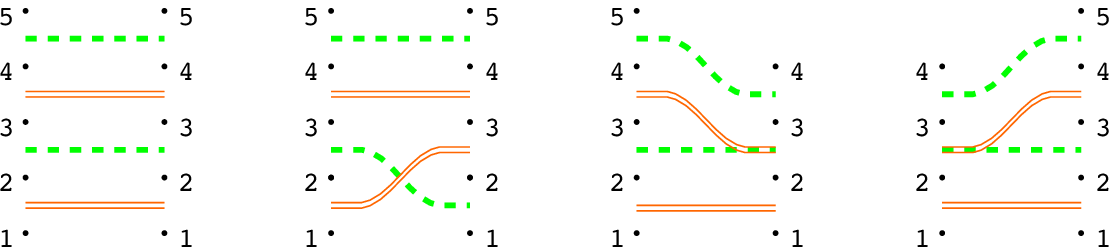} 
       \vskip .2 cm
       \caption[Examples of shadows.]{\textbf{Examples of shadows.} On each diagram $\b$ and $\bh$ are on the left hand side, while $\a$ and $\ah$ are on the right hand side. Double (orange) lines connect $\{1\}\times \{s_X\}$ with $\{0\}\times \{\xi s_X\}$ (for $s_X\in S_\X$) and dashed (green) lines connect $\{0\}\times \{s_O\}$ with $\{1\}\times \{\omega s_O\}$ (for $s_O\in S_\O$).}\label{fig:shadow}
\end{figure}
The information in the subsets $\SX,\TO\subset \ah$ and $\TX,\SO\subset \bh$ can be encoded as follows:
\begin{definition}
The \emph{boundaries} of a shadow $\P$ are defined as 
\begin{align*}
\e^0&=\e^0(\P)=(\epsilon_1^0,\dots,\epsilon_{m-1}^0) \in (2^{\{\pm 1 \}})^{m-1} \quad\textrm{ and} \\
\e^1&=\e^1(\P) = (\epsilon_1^1,\dots,\epsilon_{n-1}^1)\in(2^{\{\pm 1 \}})^{n-1}
\end{align*}
as follows. For a point 
$j+\frac12\in \bh$, the subset $\epsilon_j^0\subset \{\pm 1\}$ contains $-1$ if and only if $j+\frac12 \in \SO$, and $+1\in\epsilon_j^0$ if and only if $j\in \TX$. Similarly, for $j+\frac12 \in \ah$ define $\epsilon_j^1\subset \{\pm 1\}$ by $+1\in \epsilon_j^1$ if and only if $j+\frac12\in \TO$, and $-1\in\epsilon_j^0$ if and only if $j+\frac12 \in \SX$. 
\end{definition}
By reversing the above process, we can recover $\SX,\TO\subset \ah$ and $\TX,\SO\subset \bh$ from $\e^0$ and $\e^1$ by setting 
$\SX=\{j+\frac12\in\ah : -1\in \epsilon_j^1\}$, $\TX=\{j+\frac12\in\bh : +1\in \epsilon_j^0\}$, $\SO=\{j+\frac12 \in\bh: -1\in \epsilon_j^0\}$ and $\TO=\{j+\frac12\in\ah : +1\in \epsilon_j^0\}$.
The following shadows 
 will play an important role in our discussion.
%%%%%%%%%%%%%%%%%%%%%%%%%%%%%%%%%%%%%%%%%%%%%%%%%%%%
%    MAIN EXAMPLES
%%%%%%%%%%%%%%%%%%%%%%%%%%%%%%%%%%%%%%%%%%%%%%%%%%%%

\begin{exm}[Straight lines]\label{exm:straight}
For $\e^0 =(\epsilon_j^0)_{j=1}^k \in\{\pm 1\}^{k}$ let $\e^1=-\e^0$ and define $\SX,\TX, \SO$ and $\TO$ as in the previous paragraph. Consider the shadow ${}_{\e^0}\E_{\e^1}=(k+1,k+1,\id_{\SX},\id_{\SO})$. See the first picture of Figure \ref{fig:shadow} for $k=4$ and $\e^0=(+1,-1,+1,-1)$.
\end{exm}

%%%%%%%%%%%%%%%%%%%%%%%%%%%%%%%%%%%%%%%%%%%%%%%%%%%%
The next three examples correspond to elementary tangles.
\begin{exm}[Crossing]\label{exm:crossing} For $\e^0=(\epsilon_j^0)_{j=1}^{k} \in\{\pm 1\}^{k}$ and $1< i\le k$ define $\e^1=(\epsilon_j^1)_{j=1}^{k}$ where
\[\epsilon_j^1=\left\{
\begin{array}{ll}
-\epsilon_{i-1}^0 & \text{if } j=i\\
-\epsilon_{i}^0 & \text{if } j=i-1\\
-\epsilon_j^0& \text{otherwise.}\
\end{array} 
\right.
\]
Define $\SX,\TX, \SO$ and $\TO$ as before, and for $s_O\in \SO$ define
\[\omega s_O=\left\{
\begin{array}{ll}
i+\frac12 & \text{if } s_O=i-\frac12\\
i-\frac12 & \text{if } s_O=i+\frac12\\
s_O & \text{otherwise.}\
\end{array} 
\right.
\]
For $s_X\in \SX$ define
\[\xi s_X=\left\{
\begin{array}{ll}
i+\frac12  & \text{if } s_X=i-\frac12\\
i-\frac12  & \text{if } s_X=i+\frac12\\
s_X & \text{otherwise.}\
\end{array}
\right.
\]
Consider the shadow ${}_{\e^0}\mathcal{X}_{\e^1}(i)=(k+1,k+1,\xi,\omega)$. See the second picture of Figure \ref{fig:shadow} for $k=4$, $i=2$ and $\e=(+1,-1,+1,-1)$.

\end{exm}
%%%%%%%%%%%%%%%%%%%%%%%%%%%%%%%%%%%%%%%%%%%%%%%%%%%%
\begin{exm}[Cap]\label{exm:cap}
For $\e^0=(\epsilon_i^0)_{i=1}^{k} \in\{\pm1\}^{k}$ and $0\le i\le k$ such that $\epsilon_{i-1}^0\epsilon_{i}^0=-1$ define $\e^1=(\epsilon_i^1)_{i=1}^{k-1} \in\{\pm 1,\{\pm 1\}\}^{k-1}$
by 
\[\epsilon_j^1=\left\{
\begin{array}{ll}
-\epsilon_j^0 & \text{if } j<i\\
-\epsilon_{j-1}^0 & \text{if } j>i\\
\{\pm1\}& \text{if } j=i.\
\end{array} 
\right.
\]
Define $\SX,\TX, \SO$ and $\TO$ as before and for $s_O\in \SO$ define 
\[\omega s_O=\left\{
\begin{array}{ll}
s_O & \text{if } s_O<i\\
s_O-1 & \text{if } s_O>i.
\end{array} 
\right.
\]
For $t_X\in \TX$ define
\[\xi^{-1} t_X=\left\{
\begin{array}{ll}
t_X & \text{if } t_X<i\\
t_X-1 & \text{if } t_X>i.
\end{array} 
\right.
\]
and consider the shadow ${}_{\e^0}\mathcal{D}_{\e^1}(i)=(k+1,k,\xi,\omega)$. See the third picture of Figure \ref{fig:shadow} for $k=4$ $i=3$,  and $\e^0=(+1,-1,+1,-1)$.
\end{exm}
%%%%%%%%%%%%%%%%%%%%%%%%%%%%%%%%%%%%%%%%%%%%%%%%%%%%
\begin{exm}[Cup]\label{exm:cup}
This is the mirror of a cap. For $\e^1=(\epsilon_i^1)_{i=1}^{k} \in\{\pm 1\}^{k}$ and $0\le i\le k$ such that $\epsilon_{i-1}^1\epsilon_{i}^1=-1$ define $\e^0=(\epsilon_i^0)_{i=1}^{k-1} \in\{\pm,\{\pm 1\}\}^{k-1}$
by 
\[\epsilon_j^0=\left\{
\begin{array}{ll}
-\epsilon_j^1 & \text{if } j<i \\
-\epsilon_{j-1}^1 & \text{if } j>i \\
\{\pm1\}& \text{if } j=i.\
\end{array} 
\right.
\]
Define $\SX,\TX, \SO$ and $\TO$ as before and for $t_O\in \TO$ define 
\[\omega^{-1} t_O=\left\{
\begin{array}{ll}
t_O & \text{if } t_O<i\\
t_O-1 & \text{if } t_O>i.
\end{array} 
\right.
\]
For $s_X\in \SX$ define
\[\xi s_X=\left\{
\begin{array}{ll}
s_X & \text{if } s_X<i\\
s_X-1 & \text{if } s_X>i.
\end{array} 
\right.
\]
and consider the shadow ${}_{\e^0}\mathcal{C}_{\e^1}(i)=(k,k+1,\xi,\omega)$. See the fourth picture of Figure \ref{fig:shadow} for $k=4$ $i=3$,  and $\e^1=(-1,+1,-1,+1)$.
\end{exm}
\begin{exm} \label{ex:hole} Given any shadow $\P$, one can introduce a gap at either its left or right  hand side. We discuss the construction for the left hand side. Given $i\in \b$ let $m'=m+1$, $n'=n$, and define $\mathcal{L}_i(\P)=(n',m',\xi',\omega')$ by $(\e^1)'\coloneqq\e^1$ and  
$(\e^0)'=((\epsilon_j^o)')_{j=1}^{m'}$ where
\[(\epsilon_j^0)'=\left\{
\begin{array}{ll}
\epsilon_j^0 & \text{if } j < i\\
\emptyset &\text{if } j=i\\
\epsilon_{j-1}^0& \text{if } j> i.\
\end{array} 
\right.
\]
Define $\SX',\TX', \SO'$ and $\TO'$ as before and for $s_O\in \SO$ define 
\[\omega' s_O=\left\{
\begin{array}{ll}
\omega s_O & \text{if } s_O<i\\
\omega s_O-1 & \text{if } s_O>i
\end{array} 
\right.
\]
For $t_X\in \TX$ define
\[(\xi')^{-1} t_X=\left\{
\begin{array}{ll}
\xi^{-1}t_X & \text{if } t_X<i\\
\xi^{-1}t_X-1 & \text{if } t_X>i
\end{array} 
\right.
\]
Similarly for $i\in \a$ we can introduce a gap on the right hand side to obtain the shadow $\mathcal{R}_i(\P)$.
\end{exm}

%%%%%%%%%%%%%%%%%%%%%%%%%%%%%%%%%%%%%%%%%%%%%%%%%%%%
% TANGLES
%%%%%%%%%%%%%%%%%%%%%%%%%%%%%%%%%%%%%%%%%%%%%%%%%%%%

\subsubsection{Diagrams and tangles associated to shadows.}\label{ssec:graphtangles}
Shadows can be best understood through their diagrams: 
\begin{definition} 
A \emph{diagram} of a shadow $\P$ is a quadruple   $D(\P)=(\{0\}\times\bh,\{1\} \times\ah,x,o)\subset I \times \R$, where $x$ is a set of properly embedded arcs connecting $(1,s_X)$ to $(0,\xi s_X)$ (for $s_X \in \SX$), and $o$ is a set of properly embedded arcs connecting  $(0,s_O)$ to $(1,\omega s_O)$ (for $s_O\in \SO$) such that there are no triple points, and the number of intersection points of all arcs is minimal within the isotopy class fixing the boundaries. 
\end{definition}
Any two diagrams of $\P$ are related by a sequence of Reidemeister III moves (see the first picture of Figure \ref{fig:reidemeister}) and  isotopies relative to the boundaries. We do not distinguish different diagrams of the same shadow and will refer to both the isotopy class (rel. boundary) or a representative of the isotopy class as the diagram of $\P$.

\begin{definition}\label{def:tangle}
To  a shadow $\P$ we can associate a \emph{tangle} $\T(\P)$ as follows. Start from $D(\P)\subset I\times \R$.
If $j+\frac12 \in \SX\cap \TO$ (that is $\epsilon^1_j=\{\pm 1\}$) then there is one arc starting and one arc ending at $(1,j+\frac12)$. Smooth the corner at  $(1,j+\frac12)$ by pushing the union of the two arcs slightly in the interior of $I\times \R$, as shown in Figure \ref{fig:elemtangle}. Do the same at $(0,j+\frac12)$ for $j+\frac12 \in \TX\cap \SO$. This process results in a smooth properly immersed set of arcs. Remove the self-intersection of the union of the above set of arcs by slightly lifting up the interior of arcs with bigger slope.
After this process  we obtain a tangle projection in $I\times\R$ or in $(0,1]\times \R\cong (-\infty,1]\times \R$, $[0,1)\times\R\cong[0,\infty)\times \R$ or $(0,1)\times\R\cong \R^2$, if the resulting projection does not intersect $\{0\}\times \R$ and/or $\{1\}\times \R$. Then the tangle $\T(\P)=\T$ lives in $I\times S^2$, $D^3$ or in $S^3$ with boundaries 
$\partial^0\T= \{0\}\times \{j+\frac12 : \epsilon^0_j=+1\}- \{0\}\times \{j+\frac12 : \epsilon^0_j=-1\}$
and 
$\partial^1\T= \{1\}\times \{j+\frac12 : \epsilon^1_j=+1\}- \{0\}\times \{j+\frac12 : \epsilon^1_j=-1\}.$
\end{definition}

The elementary tangles corresponding to Examples \ref{exm:straight}-\ref{exm:cup} are depicted on Figure \ref{fig:elemtangle}.
\begin{figure}[h]
 \centering
       \includegraphics
       [scale=1.3]
       {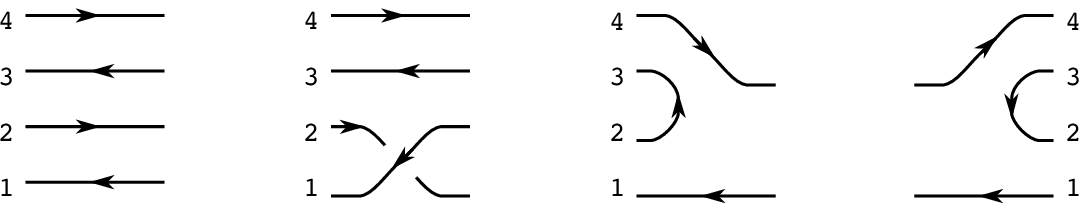} 
       \vskip .2 cm
       \caption[Elementary tangles corresponding to the shadows of Figure \ref{fig:shadow}.]{\textbf{Elementary tangles corresponding to the shadows of Figure \ref{fig:shadow}.}
  }\label{fig:elemtangle}
\end{figure}

%%%%%%%%%%%%%%%%%%%%%%%%%%%%%%%%%%%%%%%%%%%%%%%%%%%%
% GENERATORS
%%%%%%%%%%%%%%%%%%%%%%%%%%%%%%%%%%%%%%%%%%%%%%%%%%%%

\subsubsection{Generators} Now we start describing the type $\mathit{AA}$ structure associated to a shadow $\P$.
The underlying set is generated by the following elements.
\begin{definition}
For a shadow $\P$ let $\gen(\P)$ denote the set of triples $f=\stphi$, where $S\subset \b$, $T\subset \a$, with $|S|=|T|$ and $\phi\colon S\to T$ a bijection. 
\end{definition}
Note that we can also think of generators as partial matchings of the complete bipartite graph on the vertex sets $(\a,\b)$.
For any generator $f=(S,T,\phi)$ we can draw a set of arcs on the diagram of $\P$ by connecting each $(0,s)$ to $(1,\phi s)$ with a monotone properly embed arc. See Figure \ref{fig:generators} for diagrams of the generators. Again, in these diagrams we do not have triple points,  the number of intersection points of all strands is minimal, and we do not distinguish different diagrams of the same generator. Any two diagrams with minimal intersections are related by a sequence of Reidemeister III moves (See the first picture of  Figure \ref{fig:reidemeister}). Note that the generators naturally split into subsets $\gen_i(\P)=\{(S,T,\phi)\colon \vert S\vert =\vert T\vert =i\}$.
Then 
$\gen(\P)=\cup_{i=1}^{\min\{n,m\}} \gen_i(\P)$.

Fix  a variable $U_O$ for each pair $O=(s_O,\omega s_O)\in \SO\times\TO$.  
\begin{definition}
Let $C^-(\P)$ be the module generated by $\gen(\P)$ over $\k=\F_2[U_{O}]_{s_O\in \SO}$.  
\end{definition}
\begin{figure}[h]
 \centering
       \includegraphics
       [scale=1.3]
       {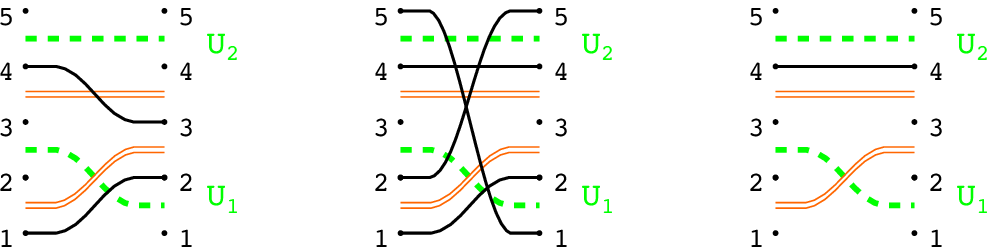} 
       \vskip .2 cm
       \caption[Diagrams of some generators $\stphi\in \gen(\P)$.]{\textbf{Diagrams of some generators $\stphi\in \gen(\P)$.} Solid black lines connect $s$ with $\phi s$.} \label{fig:generators}
\end{figure}

\subsubsection{Inner differential} Note that so far $C^-(\P)$ depends only on $m$ and $n$, but not on the particular structure of $(\SX,\TX,\xi)$ and $(\SO,\TO,\omega)$. The first dependence can be seen in the differential, which is described by resolutions of intersections of the diagram, subject to some relations. (See Figure \ref{fig:reidemeister}.) 
The intersections of the diagram of  a generator $\stphi$ correspond to inversions of the partial permutation $\phi$. 

Let $\phi\colon S\to T$ be a bijection between subsets $S$ and $T$ of two ordered sets $\b$ and $\a$. Define
\[\Inv(\phi)=\{(s_1,s_2)\in  S\times S : s_1<s_2 \text{ and } \phi s_1 >\phi s_2 \}.\]
Given two ordered sets $\b\cup \bh$ and $\a\cup \ah$, and bijections  $\phi\colon S\to T$ and $\omega\colon\SO\to\TO$ for $S\subset \b, T\subset \a, \SO\subset\bh, \TO\subset \ah$, define
\[\Inv(\phi,\omega)=\{(s,s_O)\in S\times \SO : s<s_O \text{ and } \phi s>\omega s_O \text{, or } s>s_O \text{ and } \phi s <\omega s_O\}.\] 
Define the set $\Inv(\phi,\xi^{-1})$ and  for $s_O\in\SO$ the set $\Inv(\phi,\omega\vert_{s_O})$ similarly.
Denote the sizes of these sets by $\inv(\phi), \inv(\phi,\omega), \inv(\phi,\xi^{-1})$ and $\inv(\phi,\omega|_{s_O})$, respectively.

The differential of a generator $\stphi$ can be given by resolving intersections. 
For $\tau=(s_1,s_2) \in \Inv(\phi)$ define the new generator $(S,T,\phi^{\tau})$, where 
$\phi^{\tau}=\phi\circ\tau$ is the \emph{resolution of $\phi$ at $\tau$} (for simplicity here and throughout the paper $\tau$ denotes both the pair $(s_1,s_2)$ and the 2-cycle permutation $(s_1 s_2)$).
A resolution of $\tau=(s_1,s_2)\in \Inv(\phi)$ is \emph{allowed} if $\inv(\phi^{\tau})= \inv(\phi)-1$ (Compare with the second picture of Figure \ref{fig:reidemeister}.) and  $\inv(\phi,\xi^{-1})= \inv(\phi^{\tau},\xi^{-1})$ (Compare with the third picture of Figure \ref{fig:reidemeister}.). The set of inversions with allowed resolutions is denoted by $\Inv_0(\phi)\subset \Inv(\phi)$. 

\begin{figure}[h]
 \centering
       \includegraphics[width=.99\textwidth]{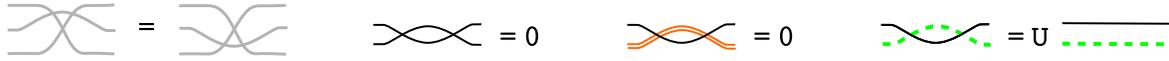} 
       \vskip .2 cm
       \caption[Relations of diagrams.]{\textbf{Relations of diagrams.} In the first relation the strands can correspond to $\phi,\xi$, or $\omega$.} \label{fig:reidemeister}
\end{figure}

Given a pair  $O=(s_O,\omega s_O)$ and a $2$-cycle permutation $\tau$ such that $\phi\circ \tau$ is defined, define
%For an inversion $\tau$ with an allowed resolution and $O=(s_O,\omega s_O)$ let
\[n_O(\tau;\phi)=\frac12\left( \inv(\phi,\omega\vert_{s_O})-\inv(\phi^{\tau},\omega\vert_{s_O})\right).\] 
When $\phi$ is clear from the context we will omit it from the notation and will write $n_O(\tau)$ or $n_O(s_1,s_2)$ for $n_O(\tau;\phi)$.
Note that $n_O(\tau)$ is always an integer. 
The differential is defined on generators by
\[
\partial \stphi=\sum_{\tau \in \Inv_0(\phi)}\left(\prod_{s_O\in \SO }U_O^{n_O(\tau)}\right)(S,T,\phi^{\tau}).
\]
Compare this equation with the last relation of Figure \ref{fig:reidemeister}. Also see Figure \ref{fig:differential} for an example.
\begin{figure}[h]
 \centering
       \includegraphics[scale=1.2]{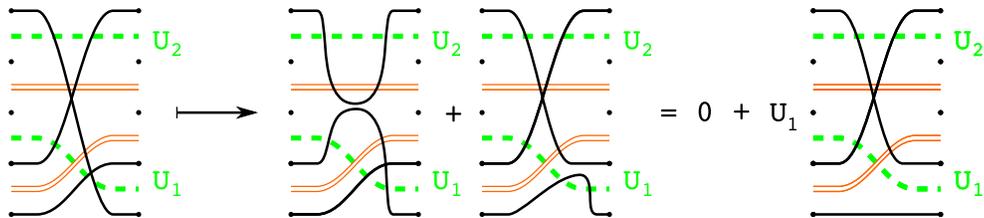} 
       \vskip .2 cm
       \caption[Example of the differential.]{\textbf{Example of the differential.} Note that the second and the third diagrams do not have minimal intersections, thus they do not represent generators. We get the differential by removing the extra intersections using the relations of Figure \ref{fig:reidemeister}.} \label{fig:differential}
\end{figure}
Extend $\partial^-$ linearly to the whole $C^-(\P)$.
\begin{proposition}
$(C^-(\P),\partial)$ is a chain complex.
\end{proposition}
\begin{proof}
The differential first resolves intersection points, and then applies the  relations of Figure 
\ref{fig:reidemeister} to minimize the number of intersection points. When we apply the differential twice, then we can equivalently first resolve two intersection points and then apply the relations  Figure \ref{fig:reidemeister} all at once. This proves that any term of 
\[
\partial^2 \stphi=\sum_{\tau_1\in \Inv_0(\phi)}\sum_{\tau_2\in \Inv_0(\phi^\tau_1)}\prod_{s_O\in \SO }U_O^{n_O(\tau_1;\phi)+n_O(\tau_2;\phi^{\tau_1})}
(S,T,(\phi^{\tau_1})^{\tau_2})
\]
appears twice with exactly the same coefficient and thus cancels.
\end{proof}

%%%%%%%%%%%%%%%%%%%%%%%%%%%%%%%%%%%%%%%%%%%%%%%%%%%%%%%
%TYPE-A
%%%%%%%%%%%%%%%%%%%%%%%%%%%%%%%%%%%%%%%%%%%%%%%%%%%%%%%

\subsubsection{Composition of shadows -- type $A$  maps.}\label{ssec:typeA-strand}
Let $\P_1=(m_1, n_1,\xi_1,\omega_1)$ and $\P_2=(m_2, n_2, \xi_2,\omega_2)$ be two shadows. If $n_1=m_2$, $S_{\X_1}=T_{\X_2}$, $T_{\O_1}=S_{\O_2}$, then we can define the \emph{concatenation} of the shadows as
$\P_1*\P_2=(m, n, \xi, \omega)$ where $m=m_1$, $n=n_2$, $(\SX,\TX,\xi)=(S_{\X_1},T_{\X_2},\xi_1\circ \xi_2)$ and $(\TO,\SO,\omega)=(S_{\O_2}, T_{\O_1},\omega_2\circ\omega_1)$. 

\begin{definition}
We say that $\P_1$ and $\P_2$ as above are \emph{composable} if  the numbers of intersection points add up i.e.
$\inv(\xi)=\inv(\xi_1)+\inv(\xi_2)$, $\inv(\omega)=\inv(\omega_1)+\inv(\omega_2)$ and $\inv(\omega, \xi^{-1})=\inv(\omega_1,\xi_1^{-1})+\inv(\omega_2, \xi_2^{-1})$. In this case $\P_1$ and $\P_2$ have a well-defined \emph{composition} $\P_1\circ\P_2=\P_1*\P_2$. 
\end{definition}

Note that on the diagram composable means that after the concatenation the resulting shadow still has minimal intersection.

\begin{exm} In Figure \ref{fig:shadow} all shadows that can be concatenated are immediately composable. However, the first two pictures of Figure \ref{fig:composeshadow} can be concatenated, but they are not composable. 
\begin{figure}[h]
 \centering
       \includegraphics[scale=1.3]{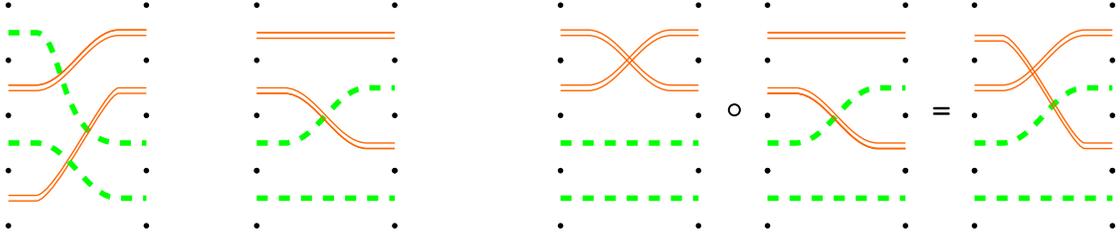} 
       \vskip .2 cm
       \caption[Composition of shadows.]{\textbf{Composition of shadows.} Left: two shadows that are not composable. Right: two composable shadows and their composition.} \label{fig:composeshadow}
\end{figure}
\end{exm}
If $\P_1$ and $\P_2$ are composable, then there is a composition map
\[
\cdot\colon C^-(\P_1)\otimes C^-(\P_2)\to C^-(\P_1\circ\P_2)
\]
defined as follows. Let $f_1=(S_1,T_1,\phi_1)$ and $f_2=(S_2,T_2,\phi_2)$ be generators of $C^-(\P_1)$ and $C^-(\P_2)$ respectively. If $T_1=S_2$, then the concatenation $(S,T,\phi)=(S_1,T_2,\phi_2\circ\phi_1)$ is well-defined. If $\inv(\phi)=\inv(\phi_1)+\inv(\phi_2)$ and $\inv(\phi,\xi^{-1})=\inv(\phi_1,\xi_1^{-1})+\inv(\phi_2,\xi_2^{-1})$, then $f_1\cdot f_2$ is defined by 
\[(S_1,T_1,\phi_1)\cdot(S_2,T_2,\phi_2)=\prod_{s_O\in \TO}U_O^{\frac 1 2 (\inv(\phi_1,\omega_1\vert_{s_O}) +\inv(\phi_2,\omega_2\vert_{\omega_1 s_O}) - \inv(\phi,\omega\vert_{s_O})) }\stphi\]
In all other cases $f_1\cdot f_2$ is defined to be $0$.
See Figure \ref{fig:multiplication} and \ref{fig:gridtypeA} for examples.
\begin{figure}[h]
 \centering
       \includegraphics[scale=1.3]{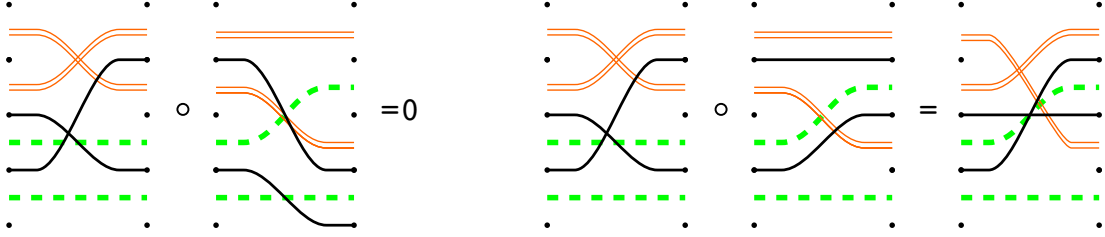} 
       \vskip .2 cm
       \caption[Composition of generators.]{\textbf{Composition of generators.} The first composition is 0 by the third relation of Figure \ref{fig:reidemeister}.}\label{fig:multiplication}
\end{figure}

Note that this composition is consistent with the differential and associative:
\begin{proposition}
\label{prop:dm}
Let $\P_1$ be composable with $\P_2$. Then the following square commutes:
\[\xymatrix{\ar @{} [dr] |{}
C^-(\P_1)\otimes C^-(\P_2) \ar[d]^{\partial\otimes\id+\id\otimes\partial} \ar[r]^\cdot & C^-(\P_1\circ\P_2) \ar[d]^\partial \\
C^-(\P_1)\otimes C^-(\P_2) \ar[r]^\cdot & C^-(\P_1\circ\P_2) }\]
If in addition $\P_2$ is composable with the shadow $\P_3$, then $\P_1\circ\P_2$ is composable with $\P_3$, $\P_1$ is composable with $\P_2\circ\P_3$ and the following square commutes:
\[\xymatrix{\ar @{} [dr] |{}
C^-(\P_1)\otimes C^-(\P_2)\otimes C^-(\P_3) \ar[d]^{\cdot\otimes\id} \ar[r]^{\id \otimes\cdot} & C^-(\P_1)\otimes C^-(P_2\circ\P_3) \ar[d]^{\cdot} \\
C^-(\P_1\circ\P_2)\otimes C^-(\P_3) \ar[r]^\cdot & C^-(\P_1\circ\P_2\circ\P_3) }\]
\end{proposition}
\begin{proof}
This statement again follows from the facts that one can first do all the operations (resolving intersections and concatenating generators) and then reduce the intersection points by the relations of Figure \ref{fig:reidemeister} and that both equations are obvious without the relations.
\end{proof}
\begin{definition} For a shadow $\P$, define the shadows ${\E_R}={\E_R}(\P)$ and ${\E_L}={\E_L}(\P)$ by the quadruples $(m,m,\id_{\TX},\id_{\SO})$ and $(n,n,\id_{\SX},\id_{\TO})$, respectively. 
In general, let $\mathcal{E}$ be the shadow given by the quadruple $(n,n,\id_{\SX}, \id_{\SO})$, where $\SX\subset \bh$, $\SO\subset \ah$ are any subsets.  Then $\E\circ\E=\E$, so we call $\E$ an \emph{idempotent shadow}.
\end{definition}
Note that idempotent shadows are exactly shadows corresponding to straight lines (Example \ref{exm:straight}). 
By Proposition \ref{prop:dm}, the induced multiplication $\cdot\colon C^-(\E)\times C^-(\E)\to C^-(\E)$ upgrades $C^-(\E)$ to a differential algebra: 
\begin{definition} For an idempotent shadow $\E$, let $\Aa(\E)$ be the differential algebra $(C^-(\E),\cdot,\partial)$.
\end{definition}

In Subsection  \ref{ssec:grading} we will define a grading that turns $\Aa(\E)$ into a differential graded algebra. 
Again by Proposition \ref{prop:dm} $(C^-(\P),\partial,\cdot,\cdot)$ is a left-right $\Aa({\E_L})$-$\Aa({\E_R})$ differential module which we can turn into a type $\mathit{AA}$  structure:
\begin{definition} 
With the above notation, let $\CATA^-(\P)$ be the left-right  $\mathit{AA}$  structure $(C^-(\P),\{m_{i,1,j}\})$ over $\Aa({\E_L})$ and $\Aa({\E_R})$, where 
\[m_{i,1,j}\colon \Aa({\E_L})^{\otimes i}\otimes C^-(\P) \otimes \Aa({\E_R})^{\otimes j}\to  C^-(\P)\]
with  $m_{i,1,j}=0$ for $i>1$ or $j>1$, and  nonzero maps given by
\[\begin{array}{ccc}
m_{0,1,0}(f)=\partial f, \qquad & m_{1,1,0}(a_L\otimes f)=a_L\cdot f, \qquad& m_{0,1,1}( f\otimes a_R)= f \cdot a_R
\end{array}\]  
\end{definition}
The gradings of $\CATA^-(\P)$ will only be defined in Subsection \ref{ssec:grading}. Since $\CATA^-(\P)$  comes from a two-sided differential module, we have:
\begin{proposition}
For any shadow $\P$ the structure maps of $\CATA^-(\P)$ satisfy the type $\mathit{AA}$ structure identites. \qed
\end{proposition}
The idempotents of $\Aa(\E)$ are given by $(S,S,\id_S)$ where $S\subset \b$. Let $\I(\Aa(\E))$ denote the set of idempotent elements of $\Aa(\E)$. For a generator $f=\stphi$ define 
\[\begin{array}{ccc}
\iota_L(f)&=&(S,S,\id_S)\in\I(\Aa(\E_L)),\\ 
\iota_R(f)&=&(T,T,\id_T)\in\I(\Aa(\E_R)).
\end{array}\]
These idempotents are defined so that we have  $\iota_L(f)\cdot f \cdot \iota_R(f)=f$.
%%%%%%%%%%%%%%%%%%%%%%%%%%%%%%%%%%%%%%%
%CO-SHADOWS
%%%%%%%%%%%%%%%%%%%%%%%%%%%%%%%%%%%%%%%%%%%%%%%%%%%%%%%

\subsection{Type $\mathit{DD}$ structures -- Mirror-shadows} To define type $D$ structures we need to work with co-chain complexes associated to ``mirrors'' of shadows. For a shadow $\P = (m,n,\xi,\omega)$, define its \emph{mirror} $\P^*$ to be the same quadruple $(m,n,\xi,\omega)$. In the sequel we will always associate ``dual''-structures to $\P^*$, 
that is why we make the distinction in the notation. To a mirror-shadow $\P^*$ we associate the co-chain complex $(C^-(\P^*),\partial^*)=(C^-(\P),\partial)^*$. Thus the elements of $C^-(\P^*)$ are of the form $(S,T,\phi)^*$ and the co-differential $\partial^*$ introduces intersection points:
\[
\partial^* \stphi^*=\sum_{\tau\in \Inv_0^*(\phi)}\prod_{s_O\in \TO }U_O^{-n_O(\tau; \phi)}(S,T,\phi^{\tau})^*
\]
where the elements of $\Inv^*_0(\phi)$ are elements of $\Inv(\phi)^c$ such that $\inv(\phi^{\tau})= \inv(\phi)+1$ and  $\inv(\phi,\xi^{-1})= \inv(\phi^{\tau},\xi^{-1})$.

Let 
$\Aa(\E^L)$ and $\Aa(\E^R)$ be the algebras corresponding to the idempotent shadows 
$\E^L=\E^L(\P^*)=(n,n,\id_{{\SX}^c},\id_{{\TO}^c})$ and 
$\E^R=\E^R(\P^*)=(m,m,\id_{{\TX}^c},\id_{{\SO}^c})$, where ${.}^c$ denotes the complement of subsets in the appropriate set they are contained in (See Definition \ref{def:shadow}).  Then for $f^*=(S,T,\phi)^*$ let 
\[\begin{array}{ccc}
\iota^L(f^*)&=&(T^c,T^c,\id_{T^c})\in\I(\Aa(\E^L));\\ 
\iota^R(f^*)&=&(S^c,S^c,\id_{S^c})\in\I(\Aa(\E^R)).
\end{array}\]
This definition enables us to define a bimodule structure ${}_{\I(\Aa(\E^L))}C^-(\P^*)_{\I(\Aa(\E^R))}$ by extending the following multiplications to $C^-(\P^*)$. For an idempotent generator $\iota\in \I(\Aa(\E^L))$ let
\[
\iota \cdot (S,T,\phi)^*{}=\left\{
\begin{array}{ll}
(S,T,\phi)^* & \text{if } \iota^L(S,T,\phi)^*=\iota;\\
0 & \text{otherwise}.
\end{array} 
\right.
\]
and for $\iota\in \I(\Aa(\E^R))$ let
\[
(S,T,\phi)^*\cdot \iota=\left\{
\begin{array}{ll}
(S,T,\phi)^* & \text{if } \iota^R(S,T,\phi)^*=\iota;\\
0 & \text{otherwise}.
\end{array} 
\right.
\]
\subsubsection{Diagrams and tangles associated to mirror-shadows}

For a mirror-shadow $\P^*$ we use different conventions to associate diagrams and tangles:
\begin{definition}
Let $D^*(\P^*)$ be the mirror of $D(\P)$ with respect to the vertical axis $\{\frac12\}\times \R$. 
\end{definition}
To indicate that we work with mirrors we put a grey background underneath $D^*(\P^*)$.
\begin{definition}
Let $\T^*(\P^*)$ denote the mirror (with respect to the vertical axis) of $\T(\P)$ with the over-crossings changed to under-crossings.
\end{definition}
See Figure \ref{fig:coshadow}  and \ref{fig:coelemtangle} for the elementary examples.
\begin{figure}[h]
 \centering
       \includegraphics
       [scale=1.3]
       {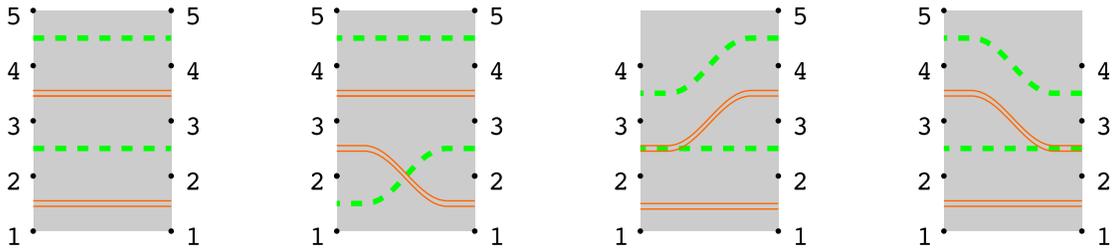} 
       \vskip .2 cm
       \caption[Examples of diagrams of mirror-shadows.]{\textbf{Examples of diagrams of mirror-shadows.} On each figure $\a$ and $\ah$ is on the left, while $\b$ and $\bh$ is on the right hand side. Double (orange) lines connect $\{0\}\times \{s_X\}$ with $\{1\}\times \{\xi s_X\}$ and dashed (green) lines connect $\{1\}\times \{s_O\}$ with $\{0\}\times \{\omega s_O\}$.}\label{fig:coshadow}
\end{figure}

\begin{figure}[h]
 \centering
       \includegraphics
       [scale=1.3]
       {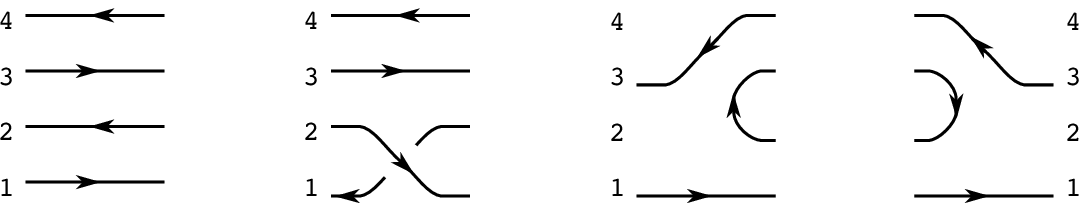} 
       \vskip .2 cm
       \caption[Elemantary tangles corresponding to the mirror-shadows of Figure \ref{fig:coshadow}.]{\textbf{Elemantary tangles corresponding to the mirror-shadows of Figure \ref{fig:coshadow}.}}
       \label{fig:coelemtangle}
\end{figure}

%%%%%%%%%%%%%%%%%%%%%%%%%%%%%%%%%%%%%%%%%%%%%%%%%%%%%%%
%TYPE-D
%%%%%%%%%%%%%%%%%%%%%%%%%%%%%%%%%%%%%%%%%%%%%%%%%%%%%%%

\subsubsection{Wedge product of shadows and mirror-shadows -- type $D$ maps.}\label{subsection:wedge}
The mirror-shadow $\P_1^*$ and shadow $\P_2$ have a \emph{well-defined wedge product} if $m_1=m_2$, $T_{\X_1}=T_{\X_2}^c$, and $S_{\O_1}=S_{\O_2}^c$. This means exactly that $\E^R(\P_1^*)=\E_L(\P_2)$. Denote the ordered pair by $\P_1^*\wedge \P_2$.  Diagrammatically, we indicate a wedge product by placing the corresponding diagrams next to each other. See Figure \ref{fig:concatenation} for an example.
\begin{figure}[h]
 \centering
       \includegraphics[scale=1.3]{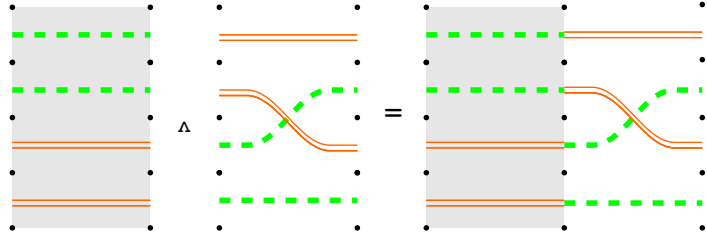} 
       \vskip .2 cm
       \caption[Wedge product of a mirror-shadow and a shadow.]{\textbf{Wedge product of a mirror-shadow and a shadow.}} \label{fig:concatenation}
\end{figure}
Similarly, the shadow $\P_1$ and mirror-shadow $\P_2^*$ have a well-defined wedge product 
if $n_1=n_2$, $S_{\X_1}=S_{\X_2}^c$, and $T_{\O_1}=T_{\O_2}^c$. The pair  is denoted by $\P_1\wedge \P_2^*$.

Let $\I=\I(\Aa(\E^R(\P_1^*)))=\I(\Aa(\E_L(\P_2)))$. 
Define 
\[C^-(\P_1^*\wedge \P_2)=C^-(\P_1^*)\otimes_\I C^-(\P_2)\]
and note that it is a module over $\F_2[U_O]_{s_O\in S_{\O_1}\cup S_{\O_2}}$. For generators $f_1^*=(S_1,T_1,\phi_1)^*\in\gen(\P_1^*)$ and $f_2=(S_2,T_2,\phi_2)\in\gen(\P_2)$ such that $f = f_1^*\otimes f_2$ is non-zero, i.e. such that $S_1 = S_2^c$, define a map
\[\partial_\wedge(f_1^*\otimes f_2)=\partial^*(f_1^*)\otimes f_2+f_1^*\otimes \partial(f_2)+\partial_\textrm{mix}(f_1^*\otimes f_2)\]
where $\partial^*$ and $\partial$ are the differentials on $C^-(\P_1^*)$ and $C^-(\P_2)$, respectively, and
 $\partial_{\textrm{mix}}$ is defined below by looking at pairs of points in $S_1\cup S_2 = \b$. 
%\textcolor{green}{$S_1\cup S_2=\a$, and we don't really look at all points, we don't look at the ones $\partial$ or $\partial^*$ takes care of..}

\begin{itemize}
\item For a pair $(p,q)\in S_1\times S_2$ define $f^{p q}=(f_1^*)^{p q}\otimes f_2^{p q}$, where $(f_1^*)^{p q}=(S_1^{p q},T_1^{p q},\phi_1^{p q})^*$, $f_2^{p q}=(S_2^{p q},T_2^{p q},\phi_2^{p q})$. Here   $S_1^{p q}= S_1\setminus \{p\}\cup\{q\}$, $T_1^{p q}= T_1$ and for $s_1 \in S_1^{p q}$
\[
\phi_1^{p q} s_1 =\left\{
\begin{array}{ll}
\phi_1 p & \text{if } s_1=q ;\\
\phi_1s_1 & \text{otherwise} .
\end{array} 
\right.
\]
Similarly 
$S_2^{p q}= S_2\setminus \{q\}\cup\{p\}$, $T_2^{p q}= T_2$ and for $s_2\in S_2^{p q}$
\[
\phi_2^{p q} s_2 =\left\{
\begin{array}{ll}
\phi_2 q & \text{if } s_2=p ;\\
\phi_2 s_2  & \text{otherwise}.
\end{array} 
\right.
\]
Diagrammatically, $f^{p q}$ is obtained from $f$ by exchanging the $p$ and $q$ endpoints of the two strands ending at $p$ and at $q$.
The pair $(p,q)\in S_1\times S_2$ is \emph{exchangeable} if 
\begin{itemize}
\item[-] $\Inv(\phi_1)\supset \Inv({\phi_1}^{p q})$
\item[-] $\Inv(\phi_2)\subset\Inv(\phi_2^{p q})$
\item[-] $\Inv(\phi_1,\xi_1^{-1})\supset\Inv({\phi_1}^{p q},\xi_1^{-1})$, and
\item[-] $\Inv(\phi_2,\xi_2^{-1})\subset\Inv(\phi_2^{p q},\xi_2^{-1})$. 
\end{itemize}
Diagrammatically this means that while doing the exchange  we cannot pick up crossings with black or orange strands on the $\P_1^*$-side and we cannot lose crossings with black or orange strands on the $\P_2$-side. 
 Given such an exchangeable pair $(p,q)$, for  
$O_1=(s_{O_1},\omega_1 s_{O_1})$ with $s_{O_1}\in S_{\O_1}$ let 
\[n_{O_1}(p q)=\left|\Inv(\phi_1^{p q},\omega_1\vert_{s_{O_1}})\setminus\Inv(\phi_1,\omega_1\vert_{s_{O_1}})\right|,\]
and for $O_2=(s_{O_2},\omega s_{O_2})$ with $s_{O_2}\in S_{\O_2}$ let
\[n_{O_2}(p q)=
\left|\Inv(\phi_2,\omega_2\vert_{s_{O_2}})\setminus\Inv(\phi_2^{p q},\omega_2\vert_{s_{O_2}})\right|.\]

\item For a pair $(p,q)\subset S_1$ with $p<q$ and $(p,q)\in \Inv(\phi_1)$ define $f^{pq} = (f_1^*)^{pq}\otimes f_2$,  where $(f_1^*)^{pq}= (S_1, T_1, \phi_1^{(p,q)})$. The pair $(p,q)\subset S_1$ is \emph{exchangeable} if 
\begin{itemize}
\item[-] each $t\in [p,q]\cap \b$ is in $S_1$ and $\phi_1t\in [\phi_1q, \phi_1p]$, and 
\item[-] each $t\in [p,q]\cap \bh$ is in $T_{\X_1}$ and $\xi_1^{-1}t\in [\phi_1q, \phi_1p]$.
\end{itemize}

Diagrammatically, this means that in $f$ each black or orange strand that ends between $p$ and $q$ is on the $\P_1^*$-side and crosses both black strands ending at $p$ and at $q$. 
Given such an exchangeable pair $(p,q)$,  for  
$O_1=(s_{O_1},\omega_1 s_{O_1})$ with $s_{O_1}\in S_{\O_1}$ let 
\[ n_{O_1}(pq)= \begin{cases} 1 &\mbox{if } s_{O_1}\in[p,q] \mbox{ and } \omega s_{O_1}\notin [\phi_1 q,\phi_1 p] \\
0 & \mbox{otherwise}, \end{cases}
\]
and for $O_2=(s_{O_2},\omega s_{O_2})$ with $s_{O_2}\in S_{\O_2}$ let
\[ n_{O_2}(pq)= \begin{cases} 1 &\mbox{if } s_{O_2}\in[p,q]\\
0 & \mbox{otherwise}. \end{cases}
\]

\item For a pair $(p,q)\subset S_2$ with $p<q$ and $(p,q)\notin \Inv(\phi_2)$  define $f^{pq} = f_1^*\otimes f_2^{pq}$, where $f_2^{pq}\coloneqq (S_2, T_2, \phi_2^{(p,q)})$. The pair $(p,q)\subset S_2$ is \emph{exchangeable} if 
\begin{itemize}
\item[-] each $t\in [p,q]\cap \b$ is in $S_2$ and $\phi_2t\in [\phi_2p, \phi_2p]$, and 
\item[-] each $t\in  [p,q]\cap \bh$ is in $T_{\X_2}$ and  $\xi_2^{-1}t\in [\phi_2p, \phi_2q]$.
\end{itemize}
Diagrammatically this means that in $f$ all black and orange strands that end between $p$ and $q$ are on the $\P_2$-side, and they do not cross either of the two black  strands ending at $p$ and at $q$.  Given such an exchangeable pair $(p,q)$,  for  
$O_1=(s_{O_1},\omega_1 s_{O_1})$ with $s_{O_1}\in S_{\O_1}$ let 
\[ 
n_{O_1}(pq)= \begin{cases} 1 &\mbox{if } s_{O_1}\in[p,q]\\
0 & \mbox{otherwise},\end{cases}
\]
and for $O_2=(s_{O_2},\omega s_{O_2})$ with $s_{O_2}\in S_{\O_2}$ let
\[ 
n_{O_2}(pq)= \begin{cases} 1 &\mbox{if } s_{O_2}\in[p,q] \mbox{ and } \omega s_{O_2}\notin [\phi_2 p,\phi_2 q] \\
0 & \mbox{otherwise}. \end{cases}
\]
\end{itemize}
Denote the set of exchangeable pairs for $f$ by $\mathit{Exch}(f).$

Then 
\[\partial_\textrm{mix}(f)=\sum_{(p,q)\in\mathit{Exch}(f)}\prod_{s_O\in S_{\O_1}\cup S_{\O_2}} U_O^{n_O(p q)} f^{p q}.\] 
See Figure \ref{fig:comultiplication} for an example of the mixed differential.
\begin{figure}[h]
 \centering
       \includegraphics[scale=1]{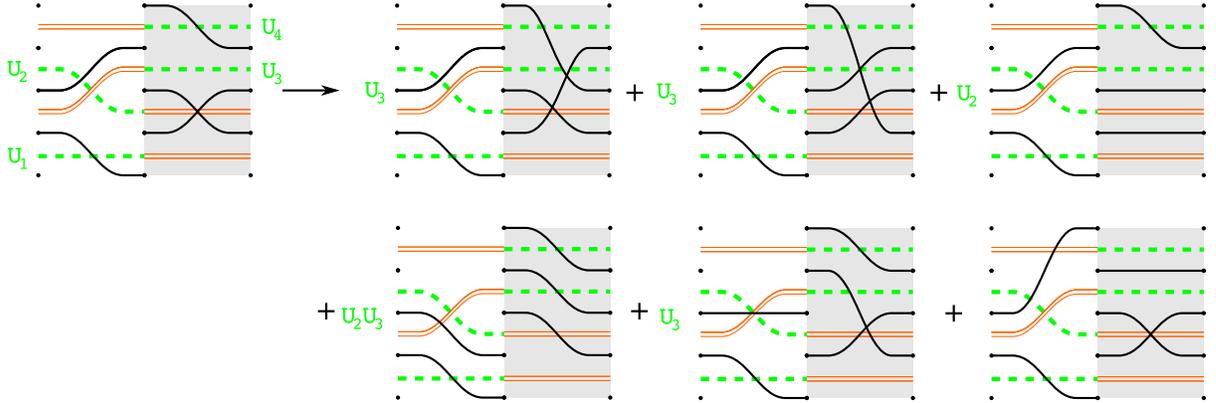} 
       \vskip .2 cm
       \caption[The differential $\partial_\wedge$.]{\textbf{The differential $\partial_\wedge$.} The last four terms on the right hand side correspond to $\partial_{\textrm{mix}}$.} \label{fig:comultiplication}
\end{figure}

Extend $\partial_{\wedge}$ linearly to the whole module $C^-(\P_1^*\wedge \P_2)$. 
\begin{proposition}\label{thm:cocomposition}
$(C^-(\P_1^*\wedge \P_2),\partial_\wedge)$ is a chain complex.
\end{proposition}
The proof of Proposition \ref{thm:cocomposition} is straightforward after the reformulation of the algebra to the language of bordered grid diagrams in Subsection \ref{ssec:glueing} and thus it will be given there.

If  $\P_1$ and $\P_2^*$ have a well-defined wedge product then $\partial_\wedge$ can be defined similarly on $C^-(\P_1\wedge \P_2^*)=C^-(\P_1)\otimes_{\I(\Aa(\E_R(\P_1)))} C^-(\P_2^*)$ by 
\[\partial_\wedge(f_1\otimes f_2^*)=\partial_1(f_1)\otimes f_2^*+f_1\otimes \partial_2^*(f_2^*)+\partial_\textrm{mix}(f_1\otimes f_2^*),\]
where the mixed differential $\partial_{\textrm{mix}}$ is defined by following the same shadow and mirror-shadow rules as earlier. Specifically, we look at pairs of black strands, and exchange their endpoints in $T_1\cup T_2$ if the following conditions are met:
\begin{itemize} 
\item If one endpoint is in $T_1$ and the other in $T_2$, then while doing the exchange we cannot pick up crossings with black or orange strands on the $\P_2^*$-side and we cannot lose crossings with black or orange strands on the $\P_1$-side. If we pick up crossings with green strands on the $\P_2^*$-side or lose crossings with green strands on the $\P_1$-side, we record it with $U_O$-variables.
\item If both endpoints are in $T_1$, then each black or orange strand that ends between the two points must be on the $\P_1$-side and cannot cross either of the given two black strands. A green strand that ends between the two points but is either on the $\P_2^*$-side or crosses one of the two black strands is recorded with a $U_O$-variable.  
\item If both endpoints are in $T_2$, then each black or orange strand that ends between the two points must be on the $\P_2^*$-side, and crosses both of the given two black strands. A green strand that ends between the two points but either doesn't cross both black strands or is on the $\P_1$-side is recorded with a $U_O$-variable.
\end{itemize}

Then we have 
\begin{proposition}\label{thm:cocomposition2}
$(C^-(\P_1\wedge \P_2^*),\partial_\wedge)$ is a chain complex. 
\end{proposition}
The proof of Proposition \ref{thm:cocomposition2} will be given in Subsection \ref{ssec:glueing} as well. 

These propositions allow us to define left and right type $D$ maps on generators $f^*=(S,T,\phi)^*$ by
\[
\begin{array}{ccc}
\delta^R &\colon& C^-(\P^*)\to  C^-(\P^*)\otimes \Aa(\E^R)\\
(S,T,\phi)^*&\mapsto &\partial_\wedge((S,T,\phi)^*\otimes \iota^R(S,T,\phi)^*).
\end{array}
\]
and 
\[
\begin{array}{ccc}
\delta^L &\colon& C^-(\P^*)\to \Aa(\E^L)\otimes C^-(\P^*)\\
(S,T,\phi)^*&\mapsto &\partial_\wedge(\iota^L(S,T,\phi)^*\otimes(S,T,\phi)^*).
\end{array}
\]
The maps $\delta^L$ and $\delta^R$ extend to the whole module $C^-(\P^*)$ and by merging them we can define a type $\mathit{DD}$  structure:
\begin{definition}
With the above notation let $\CDTD^-(\P^*)$ be the left-right  type $\DD$ structure $(C^-(\P^*),\delta^1)$ over $\Aa({\E^L})$ and $\Aa({\E^R})$, where 
\[\delta^1\colon  C^-(\P^*)\to \Aa({\E^L})\otimes C^-(\P^*) \otimes \Aa({\E^R})\]
is defined via
\[\delta^1(f^*)=\iota^L(f^*)\otimes \partial^*(f^*)\otimes \iota^R(f^*)+\iota^L(f^*)\otimes \partial_{\textrm{mix}}(f^*\otimes \iota^R(f^*))+\partial_{\textrm{mix}}(\iota^L(f^*)\otimes f^*)\otimes \iota^R(f^*).\]
\end{definition}
The type $\mathit{DD}$ structure identities hold as a consequence of Propositions \ref{thm:cocomposition} and \ref{thm:cocomposition2}:

\begin{proposition}\label{prop:typeDrelns}
Let $\P^*$ be a mirror shadow. Then
\begin{enumerate}
\item\label{item:left} as defined above $(C^-(\P^*),\delta^L)$ is a left type $D$ structure over $\Aa({\E^L})$;
\item\label{item:right} as defined above $(C^-(\P^*),\delta^R)$ is a right type $D$ structure over $\Aa({\E^R})$;
\item\label{item:both} $\CDTD^-(\P^*)$ is a left-right type $\DD$ structure  over $\Aa({\E^L})$ and $\Aa({\E^R})$.
\end{enumerate}
\end{proposition}
\begin{proof}
As the proofs of all parts of the proposition are similar, we only prove item (\ref{item:left}). Recall that the left type $D$ identity that we need to show is
\[(m_2\otimes \id)\circ(\id_\Aa\otimes \delta^L)\circ \delta^L+(\partial_\Aa\otimes \id)\circ\delta^L=0.\] 
Let $f^*$ be a generator of $C^-(\P^*)$ and let $\iota=\iota^L(f^*)$. Using $\partial\iota=0$, we can rewrite the first term on the left hand side as
\[(\partial_\Aa\otimes \id)\circ\delta^L(f^*)=(\partial_\Aa\otimes \id)\circ \partial_\textrm{mix}(\iota\otimes f^*),\]
and using also that $(\partial^*)^2=0$, we can rewrite the second term  on the left hand side as
\[(m_2\otimes \id)\circ(\id_\Aa\otimes \delta^L)\circ \delta^L(f^*)=\partial_\textrm{mix}(\iota\otimes \partial^* f^*)+(\id_\Aa\otimes \partial^*)\circ \partial_{\textrm{mix}}(\iota\otimes f^*)+\partial_\textrm{mix}^2(\iota\otimes f^*).\]
The resulting four terms are exactly the nonzero summands of  $\partial_\wedge^2(\iota\otimes f^*)$, which, since $\partial_\wedge$ is a chain map, vanishes. This finishes the proof of item (\ref{item:left}). 
\end{proof}

This concept can be extended to multiple wedge products as follows. Let $\Pp=(\P_1^\circ,\dots, \P_p^\circ)$ be an alternating sequence of shadows and mirror-shadows with well-defined consecutive wedge products. (Here and throughout the paper $\P^\circ$ indicates $\P$ or $\P^*$.) Then we can define a differential on 
\[C^-(\Pp)=C^-(\P_1^\circ)^\circ\otimes \cdots \otimes C^-(\P_p^\circ)^\circ\]
by defining it on $\mathbf{f}=f_1^\circ\otimes\cdots\otimes f_p^\circ$ as
\[\partial_\wedge{\mathbf{f}}=\sum_{j=1}^p f_1^\circ\otimes \cdots \otimes \partial^\circ(f_j^\circ)\otimes\cdots \otimes f_p^\circ+\sum_{j=1}^{p-1} f_1^\circ\otimes \cdots \otimes \partial_\textrm{mix}(f_j^\circ\otimes f_{j+1}^\circ)\otimes\cdots \otimes f_p^\circ.\]
Observe that depending on whether $\Pp$ starts (ends) with a shadow or mirror-shadow $C^-(\Pp)$ is equipped with a type $\mathit{AA}$, $\mathit{AA}$, $\mathit{DA}$ or $\mathit{DD}$ structure. Denote these structures by $\mathit{CATA}^-(\Pp)$, $\mathit{CATD}^-(\Pp)$, $\mathit{CDTA}^-(\Pp)$ or $\mathit{CDTD}^-(\Pp)$. Or sometimes -- as the type is anyways specified by the sequence $\Pp$ -- we will refer to any of the above structures as $\mathit{CT}^-(\Pp)$. 

\subsubsection{Tangles associated to wedge products}
Let $\Pp=(\P_1^\circ,\dots, \P_p^\circ)$ be an alternating sequence of shadows and mirror-shadows with well-defined consecutive wedge products. Having a well-defined wedge product exactly means that the associated diagrams $\D(\P^\circ_j)$ and thus the associated tangles $\T(\P^\circ_j)$ match up. Thus let $\D(\Pp)$ and $\T(\Pp)$ be their concatenations.

%%%%%%%%%%%%%%%%%%%%%%%%%%%%%%%%%%%%%%%%%%%%%%%%%%%%%%%
%SELF-GLUEING
%%%%%%%%%%%%%%%%%%%%%%%%%%%%%%%%%%%%%%%%%%%%%%%%%%%%%%%

\subsection{One-sided modules}\label{ssec:selfgluing-strand} 
When a shadow or a mirror-shadow corresponds to a tangle with $\partial^0=\emptyset$ or $\partial^1=\emptyset$, then the left or right map can be contracted to a differential giving a one-sided right or left module. Thus, in this subsection we would like to ``close up'' one side of the bimodule and incorporate one of the type $A$  (or type $D$) maps as a new component of the differential. (Note that this ``closing up'' is easier to follow in the related Subsection \ref{ssec:selfgluing-grid}).  Below we will describe in detail the closing up  of the left type $D$ map on 
a type $\mathit{DD}$  bimodule associated to a mirror-shadow. This way we obtain a right type $D$ structure.

Suppose that for a mirror-shadow $\P^*$ we have $\ah=\SX=\TO$. Then we can define a new component of the differential ${}_D\partial$ that will correspond to resolving some crossings (remember that originally the type $D$ map corresponds to introducing crossings) so that  $\partial^*+{}_D\partial$ is a differential (i.e. has square $0$) when restricted to $\gen_n(\P^*)$ (where $\gen_n(\P^*)$ consists of the generators $(S,T,\phi)^*$ with $|S|=|T|=n$).

Consider a generator  $f^*=(S,T,\phi)^*\in\gen_n(\P^*)$. Suppose that for $s_1<s_2$ the pair $(s_1,s_2)$ is in $\Inv(\phi)$, i.e. $\phi(s_1)>\phi(s_2)$.
We say that the exchange $(s_1,s_2)$ is \emph{allowable} if for any $t\in [\phi(s_2),\phi(s_1)]$ we have $\phi^{-1}(t)\in[s_1,s_2]$ and similarly for any $s_X\in [\phi(s_2),\phi(s_1)]$ we have $\xi(s_X)\in[s_1,s_2]$.
Denote the set of such allowable pairs by ${}_D\mathrm{Exch}(\phi)\subset S\times S$. See Figure \ref{fig:onesided} for an example.
\begin{figure}[h]
 \centering
       \includegraphics[scale=1.2]{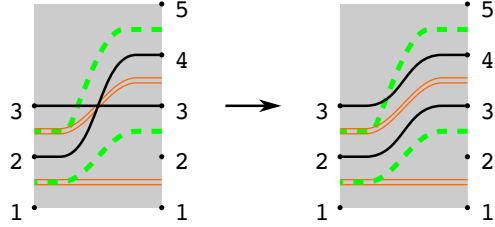} 
       \vskip .2 cm
       \caption[The differential ${}_D\partial$.]{\textbf{The differential ${}_D\partial$.}} \label{fig:onesided}
\end{figure}
For $O=(s_O,\omega s_O)$ define
\[ {}_D n_O(s_1,s_2)= \begin{cases} 1 &\mbox{if } \omega s_O\in [\phi(s_2),\phi(s_1)] \mbox{ and } s_O\notin [s_1,s_2]\\
0 & \mbox{otherwise}. \end{cases}
\]
Then define
\[{}_D\partial f^*=\sum_{(s_1,s_2)\in {}_D\mathrm{Exch}(\phi)} U_O^{{}_Dn_O(s_1,s_2)}(f^{(s_1,s_2)})^*\]
The map $\partial^*+{}_D\partial$ can be extended to the module $C_n^-(\P^*)$ generated by $\gen_n(\P^*)$ over $\k$. Although $({}_D\partial)^2\neq 0$ we have
\begin{lemma}\label{lem:selfglueingchain}
$(C_n^-(\P^*),\partial^*+{}_D\partial)$ is a chain complex.
\end{lemma}
The proof of Lemma \ref{lem:selfglueingchain} will be given using the grid diagram reformulation of $\partial^*+{}_D\partial$ as the differential of an annular bordered grid diagram in Subsection \ref{ssec:selfgluing-grid}. 
\begin{definition}
With the above notation let $\mathit{CTD}^-(\P^*)$ be the right  type $D$ structure $(C^-_n(\P^*),\delta^1)$ over $\Aa({\E^R})$, where 
\[{\delta}^1\colon C_n^-(\P^*)^*\to  C_n^-(\P^*)^*\otimes \Aa(\E^R)\] is given by
\[f^*\mapsto \delta^R(f^*)+{}_D\partial f^*\otimes \iota^R(f^*).\]
\end{definition}
Aside from the gradings that will be defined later, Lemma \ref{lem:selfglueingchain} shows that  $\mathit{CTD}^-(\P^*)$ is indeed a right type $D$ structure. 

The contraction of the right type $D$ map $\partial_D$ can be defined similarly for mirror-shadows 
with $\TX=\SO=\bh$ by exchanging pairs 
$(s_1,s_2)\in\Inv(\phi)$ such that any $s\in[s_1,s_2]$ has $\phi(s)\in[\phi(s_2),\phi(s_1)]$ and any $t_X\in[s_1,s_2]$ has $\xi^{-1}(t_X)\in[\phi(s_2),\phi(s_1)]$. In this way we obtain a left  type $D$ structure over $\Aa({\E^L})$ on $C^-_m(\P^*)$: $\mathit{CDT}^-(\P^*)$.
In this paper we do not need to contract the type $A$  actions, but the definitions go similarly with the only difference that ${}_A\partial$ and $\partial_A$ introduce crossings.

\begin{convention}\label{conv:contract}
Whenever the leftmost and/or rightmost shadow or mirror-shadow in a given well defined wedge product $\Pp$ is contractible, we will assume that the corresponding differential $\partial$ or $\partial^*$ has been replaced with the appropriate map ${}_D\partial$,  $\partial_D$, ${}_A\partial$, or $\partial_A$ in the definition of $\partial_{\wedge}$, to produce a one-sided module $\mathit{CTD}^-(\Pp)$, $\mathit{CDT}^-(\Pp)$ or $\mathit{CAT}^-(\Pp)$, or $\mathit{CTA}^-(\Pp)$, or a chain complex $\mathit{CT}^-(\Pp)$. In these cases again we may use the notation $\mathit{CT}^-(\Pp)$ to refer to any of these structures, as the type is specified by the sequence $\Pp$.
\end{convention}

%%%%%%%%%%%%%%%%%%%%%%%%%%%%%%%%%%%%%%%%%%%%%%%%%%%%%%%
%GRADING
%%%%%%%%%%%%%%%%%%%%%%%%%%%%%%%%%%%%%%%%%%%%%%%%%%%%%%%

\subsection{Gradings}\label{ssec:grading}
Unlike for other bordered theories, one can define surprisingly simple absolute gradings on the structures here.
For a shadow $\P$, we define the \emph{Maslov} and \emph{Alexander} gradings of a generator $f=\stphi$ of the module as
\[\begin{array}{lcl}
M(f) &=& \inv(\phi)-\inv(\phi,\omega) + \inv(\omega)\\
2A(f) &=& \inv(\phi, \xi^{-1}) - \inv(\phi,\omega) + \inv(\omega) - \inv(\xi^{-1})-|\TX|.
\end{array} \]
For $O=(s_O,\omega s_O)$ define
\[\begin{array}{lcl}
M(U_O f)&=&M(f)-2\\
A(U_Of)&=&A(f)-1.
\end{array} \] 
This defines a grading on $C^-(\P)$ and consequently on $\CATA^-(\P)$.

For a mirror-shadow $\P^{\ast}$ the gradings on $f^*=(S{},T{},\phi{})^*$are defined as
\[\begin{array}{lcl}
M(f^*)&=&    - \inv(\phi) + \inv(\phi,\omega) - \inv(\omega) - |\SO| \\
 2A(f^*)&=& - \inv(\phi, \xi^{-1}) + \inv(\phi,\omega) - \inv(\omega) + \inv(\xi^{-1})-|\SO|, 
\end{array} \]
and again
\[\begin{array}{lcl}
M(U_O f^*)&=&M(f^*)-2\\
A(U_Of^*)&=&A(f^*)-1.
\end{array} \]
This defines a grading on $C^-(\P^*)$ and consequently on $\CDTD^-(\P^*)$.
For an alternating sequence of shadows and mirror-shadows $\Pp=(\P_1^\circ, \ldots, \P_p^\circ)$ with  well-defined consecutive wedge product define the gradings on $\mathbf{f}=f_1^\circ\otimes\cdots\otimes f_p^\circ$ as the sums
\[\begin{array}{lcl}
M(\mathbf{f})&=&\sum_{j=1}^p M(f_j^\circ);\\
A(\mathbf{f})&=&\sum_{j=1}^p A(f_j^\circ).
\end{array} \]
All the differentials, multiplications and wedge products  behave well with the gradings. 
\begin{theorem}\label{thm:typeAA}  For a shadow $\P$, horizontal shadow $\E$, and composable shadows $\P_1$ and  $\P_2$:
\begin{enumerate}
\item $(C^-(\P),\partial)$ is a graded chain complex with grading $M$. Moreover $\partial$ preserves $A$;
\item The multiplication $\cdot\colon C^-(\P_1)\otimes C^-(\P_2)\to C^-(\P_1\circ\P_2)$ is a degree $(0,0)$ map;
\item $\Aa(\E)$ is a differential graded algebra with grading $M$. Moreover $A$ is preserved by both the multiplication and the differential;
\item\label{thm:typeAA-item4} $\CATA^-(\P)$ is a left-right  differential graded bimodule over $\Aa(\E_L)$ and ${\Aa(\E_R)}$ (in particular a type $\mathit{AA}$ structure) with grading $M$. Moreover $A$ is preserved both by the multiplication and the differential.
\end{enumerate}
\end{theorem}
\begin{theorem}\label{thm:typeDD}
For a mirror-shadow $\P^*$: 
\begin{enumerate}
\item\label{thm:typeDD-item1} $(C^-(\P^*),\partial^*)$ is a graded chain complex with grading $M$. Moreover $\partial^*$ preserves $A$;
\item\label{thm:typeDD-item3}  $\CDTD^-(\P^*)$ is a left-right type $\mathit{DD}$  structure over $\Aa(\E^L)$ and $\Aa(\E^R)$ with grading $M$. Moreover $\delta^1$ preserves $A$.
\end{enumerate}
\end{theorem}
For tangles in $I\times S^2$ we have:
\begin{theorem}\label{thm:typeAD-DA} Suppose that $\Pp=(\P_1^\circ,\dots,\P_p^\circ)$ is an alternating sequence of shadows and mirror-shadows with well-defined consecutive wedge product. If in addition $\P_1^\circ$ does not have contractible left hand side and $\P_p^\circ$ does not have contractible right hand side. Then
\begin{enumerate} 
\item\label{them:itemAA}  if $\P_1$ and $\P_p$ are both shadows then
$\mathit{CATA}^-(\Pp)$ is a left-right  type $\mathit{AA}$  structure over  $\Aa(\E_L(\P_1))$ and $\Aa(\E_R(\P_p))$ with grading $M$. Moreover $A$ is preserved by all multiplications $m_{0,1,0}, m_{1,1,0}$ and $m_{0,1,1}$;
\item\label{thm:itemAD} if $\P_1$ is a shadow and $\P_p^*$ is a mirror-shadow then
$\mathit{CATD}^-(\Pp)$ is a left-right   type $\mathit{AD}$  structure over $\Aa(\E_L(\P_1))$ and $\Aa(\E^R(\P_p^*))$ with grading $M$. Moreover $A$ is preserved by the maps $\delta_1^1$ and $\delta_2^1$;
\item\label{thm:itemDA} if $\P_1^*$ is a mirror-shadow and $\P_p$ is a shadow then
$\mathit{CDTA}^-(\Pp)$ is a left-right   type $\mathit{DA}$  structure over $\Aa(\E^L(\P_1^*))$ and $\Aa(\E_R(\P_p))$ with grading $M$. Moreover $A$ is preserved by the maps $\delta_1^1$ and $\delta_2^1$;
\item\label{thm:itemDD}  if $\P_1^*$ and $\P_p^*$ are both mirror-shadows then
$\mathit{CDTD}^-(\Pp)$ is a left-right  type $\mathit{DD}$ structure over  $\Aa(\E^L(\P_1^*))$ and $\Aa(\E^R(\P_p^*))$  with grading $M$. Moreover $A$ is preserved by the map $\delta^1$.
\end{enumerate}
\end{theorem}
For tangles in $D^3$ and $S^3$:
\begin{theorem}\label{thm:typeA-typeD}
Suppose that $\Pp=(\P_1^\circ,\dots,\P_p^\circ)$ is an alternating sequence of shadows and mirror-shadows with well-defined consecutive wedge product. Then
\begin{enumerate}
\item if $\P_1^\circ$ is left-contractible, and $\P_p$ is a non-right contractible shadow then $\mathit{CTA}^-(\Pp)$ is a right  type $\mathit{A}$ structure over $\Aa(\E_R(\P_p))$ with grading $M$. Moreover $A$ is preserved by all multiplications $m_{0}$ and $m_{1}$;
\item if $\P_1^\circ$ is left-contractible, and $\P_p^*$ is a non-right contractible mirror-shadow then $\mathit{CTD}^-(\Pp)$ is a right type $\mathit{D}$ structure over  $\Aa(\E^R(\P_p^*))$ with grading $M$. Moreover $A$ is preserved by the map $\delta^1$;
\item if $\P_p^\circ$ is right-contractible, and $\P_1$ is a non-left contractible shadow then $\mathit{CAT}^-(\Pp)$ is a left  type $\mathit{A}$ structure over $\Aa(\E_L(\P_1))$ with grading $M$. Moreover $A$ is preserved by all multiplications $m_{0}$ and $m_{1}$;
\item if $\P_p^\circ$  is right-contractible, and $\P_1^*$ is a non-left contractible mirror-shadow then $\mathit{CDT}^-(\Pp)$ is a left  type $\mathit{D}$ structure over $\Aa(\E^L(\P_1^*))$ with grading $M$. Moreover $A$ is preserved by the map $\delta^1$;
\item if $\P_1^\circ$ is left-contractible and $\P_p^\circ$ is right-contractible, then $\mathit{CT}^-(\Pp)$ is a graded chain complex over $\k$ with grading $M$. Moreover $\partial$ preserves $A$.
\end{enumerate}

\end{theorem}

\begin{proof}[Proof of Theorems \ref{thm:typeAA}, \ref{thm:typeDD}, \ref{thm:typeAD-DA} and \ref{thm:typeA-typeD}] 
Theorem \ref{thm:typeAA} and (\ref{thm:typeDD-item1}) of Theorem \ref{thm:typeDD} are 
consequences of Propositions \ref{prop:dm}, \ref{thm:cocomposition} and \ref{thm:cocomposition2} and the definition of the grading.
Item (\ref{thm:typeDD-item3}) of Theorem \ref{thm:typeDD} is a consequence of  Theorem \ref{thm:typeAD-DA}, and the ungraded version of each item  of Theorems \ref{thm:typeAD-DA} and \ref{thm:typeA-typeD} follows from Propositions \ref{thm:cocomposition} and \ref{thm:cocomposition2}. Thus, what is left to check is that $\partial_\wedge$ is a degree $(-1,0)$ map. 
To keep notation simple, we will give a proof in the case of $(C^-(\P_1^*\wedge \P_2),\partial_\wedge)$. Other cases follow the same way. Given a generator $f=f_1^*\otimes f_2=(S_1,T_1,\phi_1)^*\otimes(S_2,T_2,\phi_2)$, then 
\[\partial_\wedge(f_1^*\otimes f_2)=\partial_1^*(f_1^*)\otimes f_2+f_1^*\otimes \partial_2(f_2{})+\partial_\textrm{mix}(f_1^*\otimes f_2).\]
For the first two terms the statement follows from Theorem \ref{thm:typeAA} and (\ref{thm:typeDD-item3}) of Theorem \ref{thm:typeDD}. Next note that
\[\begin{array}{lcl}
M(f)&=&-\inv(\phi_1)+\inv(\phi_2)+\inv(\phi_1,\omega_1)-\inv(\phi_2,\omega_2)-\inv(\omega_1)+\inv(\omega_2)-|S_{\O_1}|\\
2A(f)&=&-\inv(\phi_1,\xi_1^{-1})+\inv(\phi_2,\xi_2)+\inv(\phi_1,\omega_1)-\inv(\phi_2,\omega_2)-\\
&&-\inv(\omega_1)+\inv(\omega_2)+\inv(\xi_1^{-1})-\inv(\xi_2^{-1})-|S_{\O_1}|-|T_{\X_2}|\\
\end{array}
\]
For an exchangeable pair $(p,q)\in S_1\times S_2$ we can write up the same two equations by changing $\phi_1$ and $\phi_2$ to $\phi_1^{p q}$ and $\phi_2^{p  q}$ respectively.

Since $S_{\O_1}{}\sqcup S_{\O_2}=\{1,\dots, m_1{} \}$ and since the intersection points only change for strands that end or start between $p$ and $q$ we have

\[
\begin{array}{lcl}
|p-q|&=&|\Inv(\phi_1^{p  q},\omega_1)\setminus \Inv(\phi_1,\omega_1)|+|\Inv(\phi_1,\omega_1)\setminus\Inv(\phi_1^{p  q},\omega_1)|+\\
&&+|\Inv(\phi_2,\omega_2)\setminus \Inv(\phi_2^{p  q},\omega_2)|+|\Inv(\phi_2^{p  q},\omega_2)\setminus\Inv(\phi_2,\omega_2)|\\
&=&|\Inv(\phi_1^{p  q},\omega_1)|- |\Inv(\phi_1,\omega_1)|-2|\Inv(\phi_1^{p  q},\omega_1)\setminus \Inv(\phi_1,\omega_1{})|+\\
&&+|\Inv(\phi_2,\omega_2)|- |\Inv(\phi_2^{p  q},\omega_2)|-2|\Inv(\phi_2,\omega_2)\setminus \Inv(\phi_2^{p  q},\omega_2)|\\
&=&-2\sum_{s_O} n_O(p  q)+|\Inv(\phi_1{}^{p  q},\omega_1)|- |\Inv(\phi_1{},\omega_1{})|\\
&&+|\Inv(\phi_2,\omega_2)|- |\Inv(\phi_2^{p  q},\omega_2)|
\end{array}
\]

Since the pair $(p,q)$ is exchangeable, we have $\Inv(\phi_1)\subset \Inv(\phi_1^{p q})$, so for the inversions of $\phi_1{}$ and $\phi_2$ the analog of the above formula simplifies to 
\[\inv(\phi_1{}^{p q})-\inv(\phi_1)+\inv(\phi_2)-\inv(\phi_2^{p q})=|p-q|-1\]
Similarly we get 
\[\inv(\phi_1{}^{p q},\xi{}_1^{-1})-\inv(\phi_1{},\xi{}_1^{-1})+\inv(\phi_2,\xi_2^{-1})-\inv(\phi_2^{p q},\xi_2^{-1})=|p-q|\]
which gives 
\[M(f)-M(\prod_{s_O\in S_{\O_1}\cup S_{\O_2}}U^{n_O(p q)}f^{p q})=1\]
and  
\[A(f)-A(\prod_{s_O\in S_{\O_1}\cup S_{\O_2}}U^{n_O(p q)}f^{p q})=0.\]  
Similar counting arguments work for exchangeable pairs $(p,q)$ with $(p,q)\subset S_1$ or $(p,q)\subset S_2$.
\end{proof}

\subsection{Pairing generalized strand modules}
Taking a wedge product of a shadow and a mirror-shadow corresponds to  taking the box tensor product of their algebraic structures:
\begin{theorem}\label{thm:wedgevsbox} Let $\P_1$ and $\P_2$ be shadows. Then 
\begin{enumerate}
\item\label{thm:wedgevsbox-itemDA} if the mirror-shadow $\P_1^*$ and shadow $\P_2$ have well-defined wedge products then the left-right  type $\mathit{DA}$ structures $\mathit{CDTA}^-(\P_1^*\wedge\P_2)$  and $\CDTD^-(\P_1^*)\boxtimes\CATA^-(\P_2)$ over $\Aa(\E^L(\P_1^*))$ and $\Aa(\E_R(\P_2))$ are isomorphic as type $\mathit{DA}$ structures.

\item\label{thm:wedgevsbox-itemAD} if the shadow $\P_1$ and mirror-shadow $\P_2^*$ have well-defined wedge products then the left-right  type $\mathit{AD}$ structures $\mathit{CATD}^-(\P_1\wedge\P_2^*)$ and $\CATA^-(\P_1)\boxtimes\CDTD^-(\P_2^*)$ over $\Aa(\E_L(\P_1))$ and $\Aa(\E^R(\P_2^*))$  are isomorphic as type $\mathit{AD}$ structures.
\end{enumerate}
\end{theorem}

\begin{proof}
This follows directly from the definition of $\delta^L$,  $\delta^R$, and $\partial_{\textrm{mix}}$. 
\end{proof}

Similar theorems hold for multiple wedge products of shadows and mirror-shadows.

\subsection{Relations between the $U$-actions}
Let  $\Pp=(\P_1^\circ,\dots,\P_p^\circ)$ be an alternating sequence of shadows and mirror-shadows with well-defined consecutive wedge products. For $s_O\in S_{\O_i}$ and $s_O'\in S_{\O_{i'}}$ let $O=(s_O,\omega_{i}s_O)$ and $O'=(s_O',\omega_{i'}s_O')$. 
\begin{definition}
The pairs $O$ and $O'$ are \emph{connected by a path of length $k$} if there is a sequence of elements $s_O=s_0,s_1,\dots,s_k=s_O'$ such that $s_l\in S_{\O_{j_l}}$  and $s_{l+1}= \xi_{j_l'}\omega_{j_{l}}s_{l}$. Here, depending on whether $\P^\circ_{j_l}$ is a shadow or a mirror shadow $\omega_{j_l}s_l$ is in $S_{\X_{j_l}}\coprod S_{\X_{j_l-1}}$ or $S_{\X_{j_l}}\coprod S_{\X_{j_l+1}}$, thus $j_l'$ equals $j_l-1$, $j_l$ or $j_l+1$.
\end{definition}
An example of a path is pictured on Figure \ref{fig:path}.
\begin{figure}[h]
 \centering
       \includegraphics[scale=1.2]{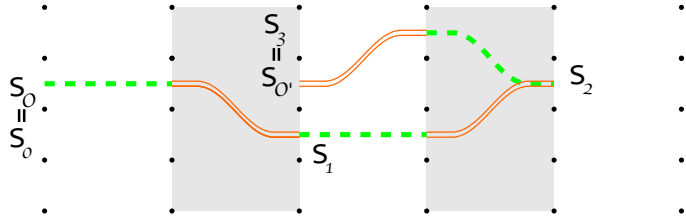} 
       \vskip .2 cm
       \caption[A path of length three.]{\textbf{A path of length three.}}\label{fig:path}
\end{figure}
\begin{lemma}\label{lem:Uaction}
Suppose that $O$ and $O'$ are connected by a path. Then the actions of $U_O$ and $U_{O'}$ on $\mathit{CT}^-(\Pp)$ are equivalent.
\end{lemma} 
Here and throughout the paper ``equivalent'' means equivalence for the appropriate structures. Thus, it means type $\mathit{AA}$   equivalence for $\mathit{CATA}^-(\Pp)$,  type $\mathit{DA}$  equivalence for $\mathit{CDTA}^-(\Pp)$,  type $\mathit{AD}$  equivalence for 
$\mathit{CATD}^-(\Pp)$, and  type $\mathit{DD}$ equivalence for $\mathit{CDTD}^-(\Pp)$.

The proof of Lemma \ref{lem:Uaction} will be given in the next section after introducing bordered grid diagrams.

%%%%%%%%%%%%%%%%%%%%%%%%%%%%%%%%%%%%%%%%%%%%%%%%%%%%%%%
%BORDERED GRID
%%%%%%%%%%%%%%%%%%%%%%%%%%%%%%%%%%%%%%%%%%%%%%%%%%%%%%%

\section{Bordered grid diagrams}\label{sec:borderedgrid}
In what follows we introduce bordered grid diagrams and structures corresponding to bordered grid diagrams. As it will turn out, all of these notions are reformulations of notions from Section \ref{sec:combinatorial}.

Bordered grid diagrams are a relative version of the grid diagrams used in combinatorial knot Floer homology \cite{mos,most}. Many of the definitions below are parallel to the ones in  \cite{mos,most}. 

\begin{definition} 
A \emph{bordered grid diagram} $G\subset [c_1,c_2]\times[d_1,d_2]$ is given by a quadruple $(\alphas,\betas,\X,\O)$ where  $\alphas=\{\alpha_a\}_{a\in \a}$ is a set of horizontal arcs  indexed by $\a=(d_1,d_2)\cap\Z$ with $\alpha_a=[c_1,c_2]\times\{a\}$, and $\betas=\{\beta_b\}_{b\in\b}$ is a set of vertical  arcs indexed by $\b= (c_1,c_2)\cap\Z$ with $\beta_b=\{b\}\times[d_1,d_2]$.
 The markings $\X$ and $\O$ are subsets of $ [c_1,c_2]\times[d_1,d_2]\cap (\Z+\frac12)\times (\Z+\frac12)$ with the property that for each horizontal and vertical line $| [c_1,c_2]\times \{j+\frac12\}\cap \X|\le 1$, $| [c_1,c_2]\times\{j+\frac12\}\cap \O|\le 1$, $|\{j+\frac12\}\times [d_1,d_2]\cap \X|\le 1$, and $|\{j+\frac12\}\times[d_1,d_2]\cap \O|\le 1$. 

By identifying the edges $[c_1,c_2]\times\{d_1\}$ and $[c_1,c_2]\times\{d_2\}$ we get an \emph{annular bordered grid diagram} $G_{\b}=(\alphas,\widetilde{\betas},\X,\O)$, where $\widetilde{\betas}$ now consists of closed curves $\widetilde{\beta}_b=\{b\}\times[d_1,d_2]/\sim$. 
Similarly, by identifying the edges $\{c_1\}\times[d_1,d_2]$ and $\{c_2\}\times[d_1,d_2]$ we get another annular bordered grid diagram $G_{\a}=(\widetilde{\alphas},\betas,\X,\O)$. 
\end {definition}
A bordered grid diagram is an example of a multi-pointed bordered Heegaard diagram for that tangle; for the general definition of such diagrams, we refer to Section \ref{sec:hdiagrams} below. 
In the sequel we will consider modules associated to bordered grid diagrams, annular bordered grid diagrams, and plumbings of annular bordered grid diagrams. Since all of these diagrams are ``nice'' in the sense of Definition \ref{def:nice}, the structure maps have a combinatorial description. 

\subsection{Generators}\label{subsec:gen}
For each $O\in \O$ fix a variable $U_O$, and let $C^-(G)$ be the free module generated over $\k=\F_2[U_O]_{O\in\O}$ by tuples of intersection points $\x\subset \alphas\cap\betas$ with the property that $|\alpha_a\cap \x|\le 1$ and $|\beta_b\cap\x|\le 1$. The set of generators is denoted by $\gen(G)$.
Note that the generators naturally split into subsets $\gen_i(G)=\{\x \colon \vert\x\vert =i\}$. Then 
$\gen(G)=\cup_{i=1}^{\min\{c_2-c_1,d_2-d_1\}} \gen_i(G)$.

\subsection{Inner differential}\label{subsec:innerdiff}
The differential can be defined by counting rectangles entirely contained in the open rectangle $(c_1,c_2)\times (d_1,d_2)$ and with boundaries on $\alphas\cup\betas$. For $c_1<b_1<b_2<c_2$ and $d_1<a_1<a_2<d_2$,  $R=[b_1,b_2]\times [a_1,a_2]$ \emph{is a rectangle from $\x$ to $\y$} if $\x\cap R=\{(b_1,a_1),(b_2,a_2)\}$,  $\y\cap R=\{(b_1,a_2),(b_2,a_1)\}$ and $\x\setminus R=\y\setminus R$. The rectangle $R$ is empty if $\X\cap R=\emptyset$. The set of empty rectangles from $\x$ to $\y$ is denoted by ${\Re}_0(\x,\y)$. The differential on $\x\in\gen(G)$ is defined by 
\[\partial \x=\sum_{\y\in\gen(G)}\sum_{R\in {\Re}_0(\x,\y)}\prod_{O\in \O} U_O^{|R\cap O|}\y\]
Figure \ref{fig:griddifferential} gives an example of the inner differential. Extend $\partial$ for $C^-(G)$ linearly. 
By the usual arguments for grid diagrams (that every domain representing a term in $\partial^2$ has an alternate decomposition) we have:
\begin{proposition}
$(C^-(G),\partial)$ is a chain complex. \qed
\end{proposition}

\subsection{Type $\mathit{AA}$ structures -- bordered grid diagrams associated to shadows}
All the structures from Section \ref{sec:combinatorial} have equivalent formulations via  bordered grid diagrams, which will be discussed in this and the following sections.
To a shadow $\P$ given by the quadruple $(m, n, \xi,\omega)$ 
we associate the following bordered grid diagram $G(\P)$.
\begin{definition} Let $G=G(\P)=(\alphas,\betas,\X,\O)\subset [-m-1,0]\times[0,n+1]\subset \R^2$ as follows. For $a\in\a$ let $\alpha_a=  [-m-1,0]\times \{a\}$ and for $b\in \b$ let $\beta_b=\{-b\}\times [0,n+1]$ then let $\alphas=\{\alpha_a\}_{a\in\a}$ and $\betas=\{\beta_b\}_{b\in\b}$, also let $\X=\{(-\xi s_X,s_X)\}_{s_X\in \SX}$ and $\O=\{O=(-s_O, \omega s_O)\}_{s_O\in \SO} $. 
\end{definition}
On Figure \ref{fig:borderedgrid} we depict the bordered grid diagrams corresponding to the shadows of Figure \ref{fig:shadow}.
\begin{figure}[h]
 \centering
       \includegraphics[scale=0.9]{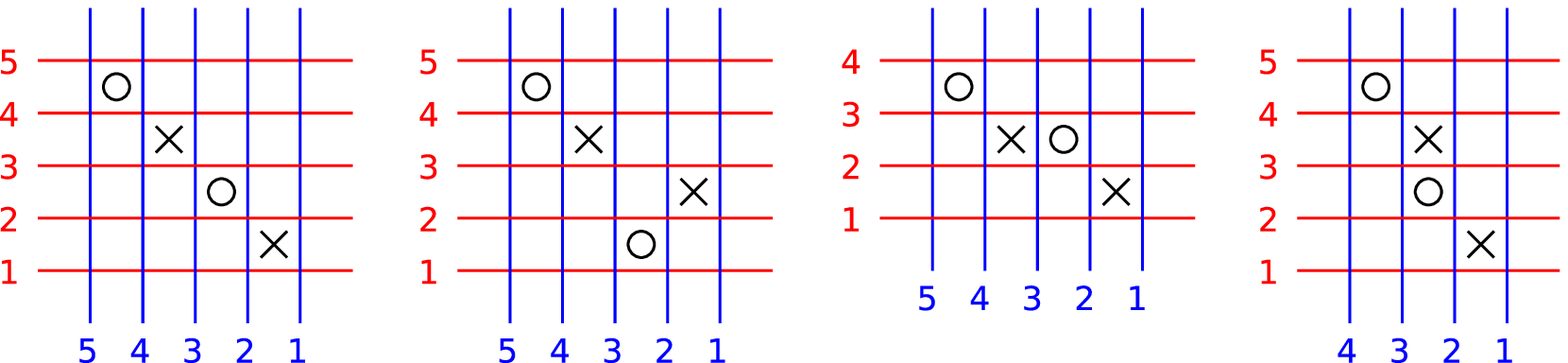} 
       \vskip .2 cm
       \caption[Bordered grid diagrams corresponding to the shadows of Figure \ref{fig:shadow}.]{\textbf{Bordered grid diagrams corresponding to the shadows of Figure \ref{fig:shadow}.}}\label{fig:borderedgrid}
\end{figure}

An equivalent way to associate a bordered grid diagram $G'(\P)$ to the shadow $\P$ is to take the $180^\circ$ rotation of $G(\P)$. Thus $G'(\P)=(\alphas',\betas',\X',\O')$ lies in the opposite quadrant $[0,m+1]\times[-n-1,0]$ with $\alphas'=\{\alpha'_a\}_{a\in\a}$ where  $\alpha_a'= [0,m+1]\times \{a'\}$, $\betas'=\{\beta_b'\}_{b\in\b}$ where $\beta_b'=\{-b\}\times [-n-1,0]$, $\X=\{(\xi s_X,-s_X)\}_{s_X\in \SX}$, and $\O=\{O=(s_O,-\omega s_O)\}_{s_O\in \SO}$. All that follows could be reformulated to $G'(\P)$ by doing a $180^\circ$ rotation to give isomorphic chain-complexes and type $\mathit{AA}$ structures to those for $G(\P)$.
\subsubsection{Tangles associated to $G(\P)$}
Let us complete $G(\P)$ with some extra basepoints 
\[\begin{array}{lcl}
\X_\partial&=&\{(-s,0) : s\in \SO\setminus \TX\}\cup \{(0,s) : s\in \TO\setminus \SX\} \text{   and}\\ 
\O_\partial&=&\{(-s,0) : s\in \TX\setminus \SO\}\cup \{(0,s) : s\in \SX\setminus \TO\}.\\
\end{array}
\]
Then define the associated tangle $\T(G)$ just like one would for a closed grid diagram: connect the points $\X\cup\X_\partial$ to $\O\cup\O_\partial$ horizontally and $\O\cup\O_\partial$ to $\X\cup\X_\partial$ vertically so that vertical strands cross over horizontal strands. Then, after smoothing, $\T(G)$ is a tangle projection in $[-m-1,0]\times[0,n+1]$ with boundary $\partial^0=(\X_\partial-\O_\partial)\cap [-m-1,0]\times\{0\}$  and $\partial^1=(\X_\partial-\O_\partial)\cap \{1\}\times [0,n+1]$. See Figure \ref{fig:knotborderedgrid} for some examples.
\begin{figure}[h]
 \centering
       \includegraphics[scale=0.9]{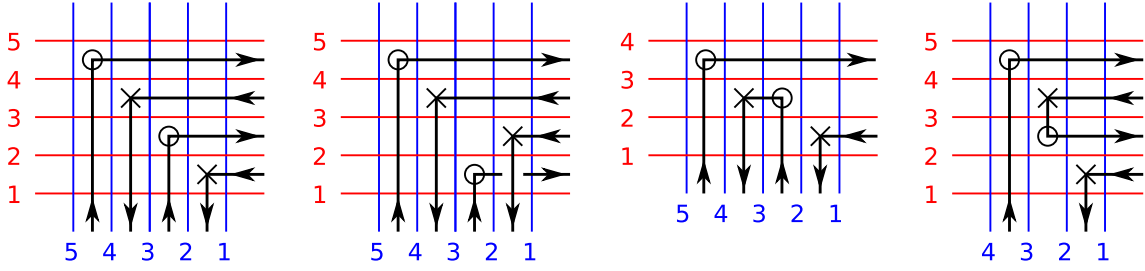} 
       \vskip .2 cm
       \caption[The tangles associated to the bordered grid diagrams of Figure \ref{fig:borderedgrid}.]{\textbf{The tangles associated to the bordered grid diagrams of Figure \ref{fig:borderedgrid}.}}
\label{fig:knotborderedgrid}
\end{figure}
Note that this tangle can be easily identified (by, for example, using polar coordinates and mapping $(r,\vartheta) \in [-m-1,0]\times [0,n+1]$ to $(2\frac{\pi-\vartheta}{\pi},r)\in I\times \R$) with a tangle in $I\times \R$, which we will call  $\T(G)$ as well. 
\begin{proposition}
Let $\P$ be a shadow. Then for $G=G(\P)$ the tangles $\T(\P)$ and $\T(G)$ are isotopic relative to the boundary.
\end{proposition}
\begin{proof}
Let $\T(G)\subset I\times \R$ be the tangle (projection) associated to $G=G(\P)$. If $p\in\T(G)$ has a vertical tangency, then depending on whether $\T(G)$ near $p$ is to the right (or left) from this tangency,
it is coming from an $X=(-\xi s_X,s_X)$ and an $O=(-s_O,\omega s_O)$ in the same horizontal (or vertical) line of the grid, thus $s_X=\omega s_O$ (or $\xi s_X=s_O$). If for example $s_X=\omega s_O$, then there is no more $X$'s or $O$'s in the same horizontal line of the grid, thus the point with the vertical tangency can be isotoped to $(0,s_O)\in I\times \R$ without altering or crossing other parts of the tangle. Do this with every point with vertical tangency and notice that the resulting tangle is $\T(\P)$.
\end{proof}

\subsubsection{Generators}
Recall that  $C^-(G)$ is the free module generated over $\k$ by the tuples of intersection points $\x=(\alpha_{\phi s}\cap\beta_{s})_{s\in S}$, where $S\subset \b$, and $\phi\colon S\to\a$ is an injection with image $T=\phi(S)$. There is a one to one correspondence between $\gen(\P)$ and $\gen(G)$ given by  associating $\x=(\alpha_{\phi s}\cap\beta_{s})_{s\in S}\in \gen(G)$ to $\stphi\in\gen(\P)$.

\subsubsection{Inner differential}
The differential of Subsection \ref{subsec:innerdiff} translates to the following. For $s_1<s_2$ and $t_1<t_2$, and  $\x=(\alpha_{\phi s}\cap\beta_{s})_{s\in S}$ and $\y=(\alpha_{\phi^{(s_1,s_2)} s}\cap\beta_{s})_{s\in S}$, where 
$s_1,s_2\in S$ and $\phi\colon S\to T$ satisfies $\phi s_1=t_2$, $\phi s_2=t_1$, $R=[-s_2,-s_1]\times [t_1,t_2]$ is a rectangle from $\x$ to $\y$. Note that then automatically $(s_1,s_2)\in \Inv(\phi)$.
\begin{figure}[h]
 \centering
       \includegraphics[scale=0.9]{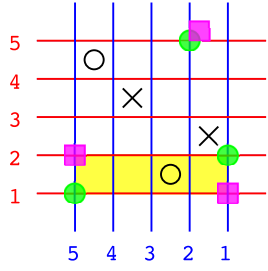} 
       \vskip .2 cm
       \caption[The inner differential for bordered grid diagrams.]{\textbf{The inner differential for bordered grid diagrams.} The generator $\x$ denoted by (light green) dots corresponds to the first strand diagram of Figure \ref{fig:differential}. The only empty rectangle (denoted by yellow) starting from $\x$ connects it to the generator $\y$ denoted by a (pink) square. The latter generator corresponds to the last strand diagram on Figure \ref{fig:differential}. The rectangle passes through the $\O$ marking $O_1$. Thus $\partial \x=U_1\y$. }\label{fig:griddifferential}
\end{figure}

Thus with the above definition of the inner differential:
\begin{proposition}\label{prop:isom} The chain complexes $(C^-(G),\partial)$ and $(C^-(\P),\partial)$ are isomorphic.

Moreover, if  $R$ is a rectangle from  $\x=(\alpha_{\phi s}\cap\beta_s)_{s\in S}$ to $\y=(\alpha_{\phi^{(s_1,s_2)} s}\cap\beta_s)_{s\in S}$, then
\begin{enumerate}
\item  $A(S,T,\phi)-A(S,T,\phi^{(s_1,s_2)})=\vert R\cap \X\vert-\vert R\cap \O\vert$;
\item if $R\in {\Re}_0(\x,\y)$ then  $M(S,T,\phi)-M(S,T,\phi^{(s_1,s_2)})=1-2\vert R\cap \O\vert$. 
\end{enumerate}
\end{proposition}
\begin{proof} \belowdisplayskip=-12pt
If $(s_1,s_2)\in \Inv(\phi)$ then $R=[-s_2,-s_1]\times [\phi s_2, \phi s_1]$ defines a rectangle in $[-m+1,0]\times [0,n+1]$. The statement follows from the following three equations:
\[\begin{array}{lcl}
|R\cap \x |&=& \left\vert \{(-s,\phi s): s\in S,  -s_2<-s<-s_1\text{ and }\phi s_2< \phi s < \phi s_1\}\right\vert \\
&=&\left\vert\Inv \phi\setminus \Inv \phi^{(s_1,s_2)}\cup\{(s_1,s_2)\}\right\vert\\
|R\cap \X|&=&\left\vert\{(-\xi s_X,s_X):s_X\in \SX, -s_2<-\xi s_X<-s_1\text{ and }\phi s_2 < s_X < \phi s_1\}\right\vert\\
&=&\left\vert\Inv (\phi,\xi^{-1})\setminus \Inv (\phi^{(s_1,s_2)},\xi^{-1})\right\vert\\
|R\cap \O|&=&\left\vert\{(-s_O,\omega s_O) : s_O\in \SO, -s_2< -s_O< -s_1\text{ and }\phi s_2 < \omega s_O < \phi s_1\}\right\vert\\
&=&\left\vert\Inv (\phi,\omega\vert_{s_O})\setminus \Inv (\phi^{(s_1,s_2)}\omega\vert_{s_O})\right\vert
\end{array}\] 
\end{proof}

\subsubsection{Type A structures} The left and right algebra actions by $\Aa({\E_L})$ and $\Aa({\E_R})$ are defined by counting sets of partial rectangles as follows. First, we will describe the right action. The left action, as it will be spelled out later, is similar. For the action of $\Aa({\E_R})$ we consider sets of partial rectangles that intersect the left and right boundaries $\{-m-1,0\}\times(0,n+1)$. 
We consider the following two types of  partial rectangles depending on whether the rectangle intersects the left or the right boundary edge:

\begin{itemize}
\item $H= [-s_1,0]\times [t_1,t_2]$, where $t_1<t_2$, or

\item $H=[-m-1,-s_2]\times [t_1,t_2]$, where $t_1<t_2$,
\end{itemize}
where $s_i\in \b$ and $t_i\in \a$.

Now fix $S\subset \b$ and generators $\x=(\alpha_{\phi s}\cap\beta_{s})_{s\in S}$ and $\y=(\alpha_{\phi' s}\cap\beta_{s})_{s\in S}$. Let $r=(\phi(S),\phi'(S),\phi'\circ \phi^{-1})\in\Aa(\E_R)$. Suppose that ${\bf H} = \{H_1, \ldots, H_l\}$
 is a set of partial rectangles of the above two types. We say that $\bf H$  \emph{connects $\x$ and $r$ to  $\y$} if for the rectangles in $\bf H$,  all bottom-left and top-right corners that are in the interior of $G$ are distinct points and form the set $\x\setminus (\x\cap \y)$, and  all bottom-right and top-left corners that are in the interior of $G$ are distinct points and form the set $\y\setminus (\x\cap \y)$. We say that $\bf H$ is \emph{allowed} if for each $H_i\in \bf H$ we have  $H_i\cap\X = \emptyset$ and $H_i\cap (\x\cap \y) = \emptyset$, no partial rectangle in $\bf H$ is completely contained in another rectangle in $\bf H$, and no two partial rectangles touching opposite boundary edges have overlapping interiors. See Figure \ref{fig:forbidden_rect}.
  Note that when $\bf H$ consists of only one partial rectangle $H$, this is equivalent to the condition $\Int H\cap\X=\Int H \cap\x=\emptyset$.
 
 \begin{figure}[h]
 \centering
       \includegraphics[scale=0.8]{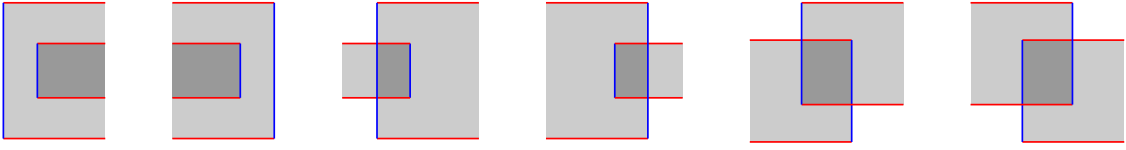} 
       \vskip .2 cm
       \caption[Forbidden pairs of partial rectangles.]{\textbf{Forbidden pairs of partial rectangles.} A set of rectangles $\bf H$ is allowed if no partial rectangle in it contains no points in $\X$ or $\x\cap\y$, and no two partial rectangles in it are in relative configuration as depicted here.}\label{fig:forbidden_rect}
\end{figure}

Note that for a fixed generator $\x$ and algebra generator $r$, there is at most one $\y$ and at most one $\bf H$ as above. Thus, we can define the action of $r$ on $\x$ as follows. If there is no set of empty partial rectangles from $\x$ and $r$ to any $\y$, then $\x\cdot r=0$. Otherwise, let $\bf H$ and $\y$ be the unique objects such that  $\bf H$ is an allowed set of partial rectangles connecting $\x$ and $r$ to $\y$. Then
\[\x\cdot r=\prod_{s_O\in \TO} U_O^{|O\cap {\bf H}|}\y.\]
where  $O\cap {\bf H}=\cup (O\cap H_i)$.

See Figure \ref{fig:gridtypeA} for examples of the type $A$ multiplication. 

\begin{figure}[h]
 \centering
       \includegraphics[scale=0.9]{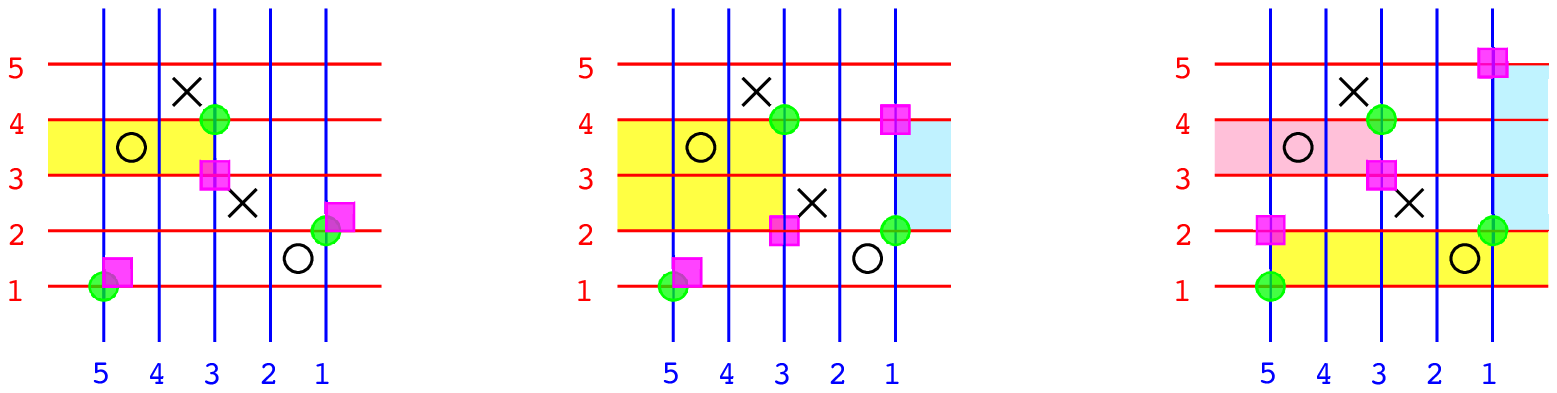} 
        \vskip .5 cm
       \includegraphics[scale=0.8]{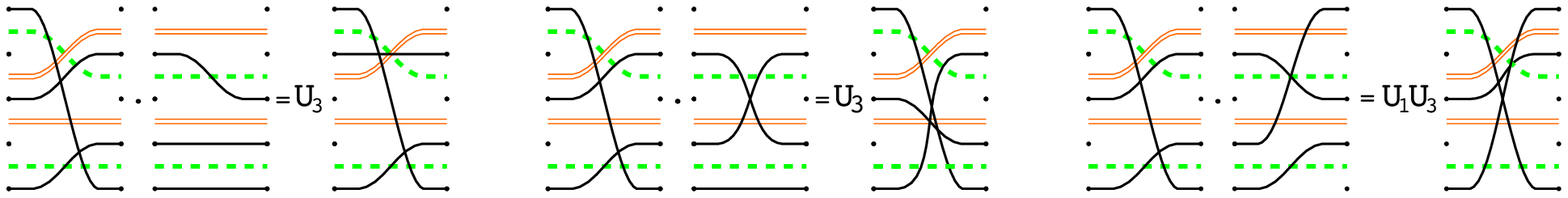} 
       \vskip .2 cm
       \caption[Examples of the right type $A$  action.]{\textbf{Examples of the right type $A$  action.} Top: examples of allowed sets of partial rectangles for the right action, starting at the generator formed by the green dots. Bottom: the corresponding right multiplications, viewed as concatenations of strand diagrams.}\label{fig:gridtypeA}
\end{figure}

The left action can be similarly defined using partial rectangles touching the top or bottom parts of the boundary $(-m-1,0)\times\{0,n+1\}$ or by rotating the rectangles by $90^\circ$. See Figure 
\ref{fig:leftgridtypeA}. 
\begin{figure}[h ]\centering
       \includegraphics[scale=0.9]{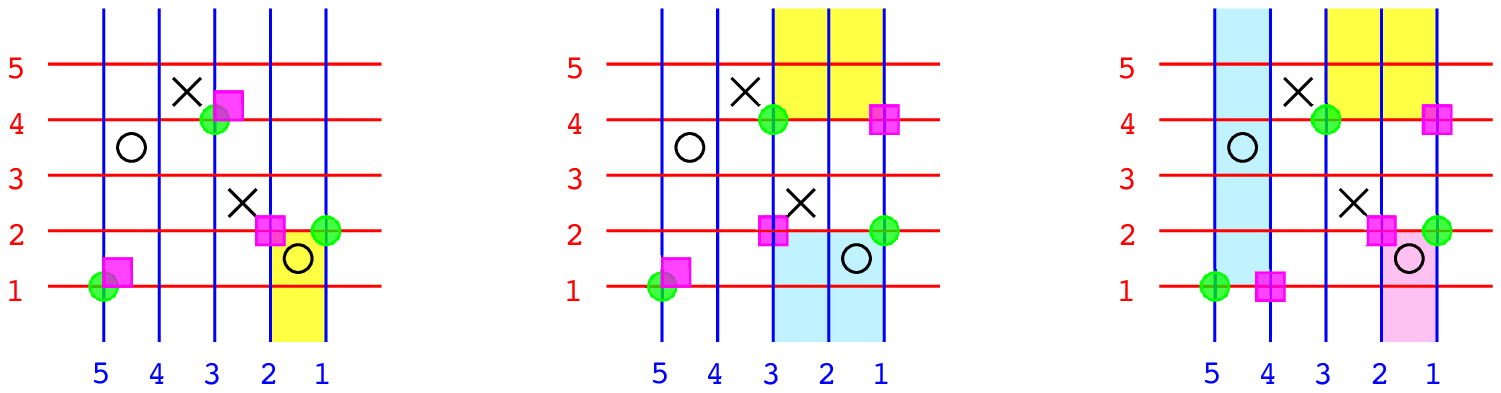} 
        \vskip .5 cm
       \includegraphics[scale=0.8]{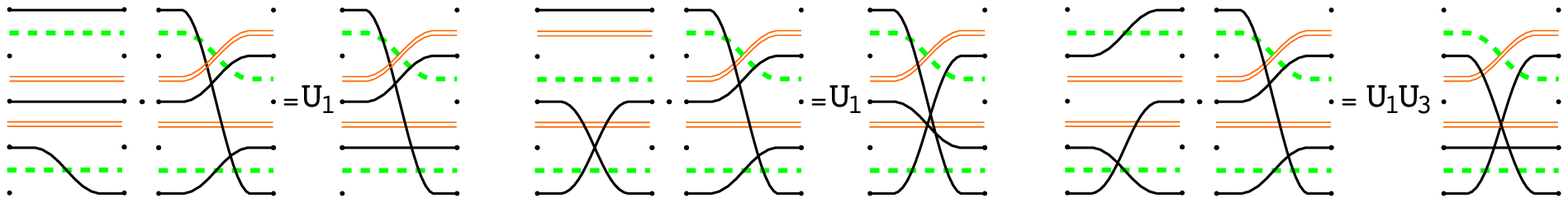} 
       \vskip .2 cm
       \caption[Examples of the left type $A$ action.]{\textbf{Examples of the left type $A$ action.} Top: examples of allowed sets of partial rectangles for the left action, starting at the generator formed by the green dots. Bottom: the corresponding left multiplications, viewed as concatenations of strand diagrams.}\label{fig:leftgridtypeA}
\end{figure}
\begin{definition}
With the above notation, let $\CATA^-(G)$ be the left-right  type $\mathit{AA}$ bimodule $(C^-(\P),\{m_{i,1,j}\})$ over $\Aa({\E_L})$ and $\Aa({\E_R})$, where 
\[m_{i,1,j}\colon \Aa({\E_L})^{\otimes i}\otimes C^-(\P) \otimes \Aa({\E_R})^{\otimes j}\to  C^-(\P)\]
with  $m_{i,1,j}=0$ for $i>1$ or $j>1$, and the nonzero maps are given by
\[\begin{array}{ccc}
m_{0,1,0}(f)=\partial f, \qquad & m_{1,1,0}(a_L\otimes f)=a_L\cdot f, \qquad& m_{0,1,1}( f\otimes a_R)= f \cdot a_R
\end{array}\]  
\end{definition}
It is not immediate to see that the above definition indeed gives a type $\mathit{AA}$ bimodule, but the next proposition says that it is isomorphic to $\CATA^-(\P)$ which by Theorem \ref{thm:typeAA}  is a type $\mathit{AA}$ structure. 
\begin{proposition}\label{prop:gridtypea1}
Let $\P$ be a shadow and let $G=G(\P)$. Then
the one to one correspondence between the generators gives rise to an isomorphism of the  structures $\CATA^-(\P)$ and $\CATA^-(G(\P))$. 
\end{proposition}
\begin{proof} Observe that $\bf H$ connects $\x$ and $r$ to $\y$ exactly when  the strand diagrams corresponding to $\x$ and $r$ can be concatenated. The result of the concatenation is the strand diagram corresponding to $\y$ when $\bf H$ is allowed, and zero otherwise. Indeed, the obstructions to $\bf H$ being allowed correspond to the Reidemeister II relations involving black and orange strands. Similarly, the count  $O\cap \bf H$ corresponds to the count $n_O$.
\end{proof}
\subsection{Type $\mathit{DD}$ structures -- bordered grid diagrams associated to mirror-shadows}
The bordered grid diagram $G^*(\P^*)$ associated to the mirror-shadow $\P^*$ is the mirror of $G(\P)$ with respect to a vertical axis. 
\begin{definition}
$G^*=G^*(\P^*)=(\alphas,\betas,\X,\O)\subset [0,m+1]\times[0,n+1]\subset \R^2$ as follows. For $a\in\a$ let $\alpha_a=  [0,m+1]\times \{a\}$ and for $b\in \b$ let $\beta_b=\{b\}\times [0,n+1]$ then let $\alphas=\{\alpha_a\}_{a\in\a}$ and $\betas=\{\beta_b\}_{b\in\b}$. Also let $\X=\{(\xi s_X,s_X)\}_{s_X\in \SX}$ and $\O=\{O=(s_O, \omega s_O)\}_{s_O\in \SO}$. 
\end{definition}
Figure \ref{fig:borderedgrid2} shows the bordered grid diagrams corresponding to the mirror-shadows of Figure \ref{fig:coshadow}.

\begin{figure}[h]
 \centering
       \includegraphics[scale=0.9]{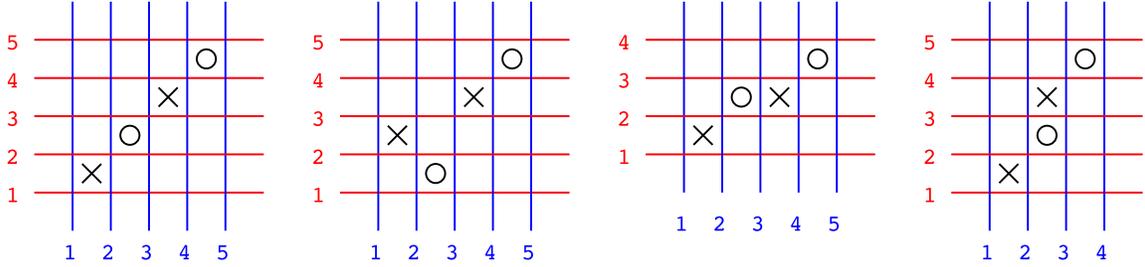} 
       \vskip .2 cm
       \caption[Bordered grid diagrams corresponding to the shadows of Figure \ref{fig:coshadow}.]{\textbf{Bordered grid diagrams corresponding to the mirror-shadows of Figure \ref{fig:coshadow}.}}\label{fig:borderedgrid2}
\end{figure}

By mirroring $G(\P)$ with respect to the horizontal axis instead, we get a bordered grid diagram $(G^*)'(\P^*)$
equivalent to $G^*(\P^*)$.

As in the case for $G(\P)$, the generators $\gen(G^*)$ are tuples of intersection points, and similarly there is a one-to-one correspondence between $\gen(G^*)$ and $\gen(\P^*)$ identifying $(S,T,\phi)^*$ with the set of intersection points $\x=(\alpha_{\phi s}\cap \beta_s)_{s\in S}$.   The differential $\partial$ is again given by counting empty rectangles.
\begin{proposition}
The chain complexes $(C^-(\P^*),\partial^*)$ and $(C^-(G^*),\partial)$ are isomorphic.
Moreover if  $R$ is a rectangle from   
$\x=(\alpha_{\phi^{(s_1,s_2)} s}\cap\beta_{s})_{s\in S}$ to $\y=(\alpha_{\phi s}\cap\beta_s)_{s\in S}$ then
\begin{enumerate}\label{prop:isomD}
\item  $A(S,T,\phi)-A(S,T,\phi^{(s_1,s_2)})=\vert R\cap \X\vert-\vert R\cap \O\vert$;
\item If $R\in {\Re}_0(\x,\y)$ then  $M(S,T,\phi)-M(S,T,(\phi)^{(s_1,s_2)})=1-2\vert R\cap \O\vert$. 
\end{enumerate}
\end{proposition}
\begin{proof}
Essentially the same as the proof of Proposition \ref{prop:isom}.
\end{proof}

Associate the tangle $\T^*(G^*)$ that is the mirror of $\T(G)$, again with respect to the vertical axis.

\subsubsection{Type $D$ maps}\label{ssec:Dgrid}
Define a bimodule structure ${}_{\I(\Aa(\E^L))}C^-(G^*)_{\I(\Aa(\E^R))}$ using the one-to-one correspondence between $\gen(G^*)$ and $\gen(\P^*)$. In other words, if the correspondence maps  $\x\in \gen(G^*)$ to  $f^*\in \gen(\P^*)$, then define $\iota\cdot \x\cdot \iota'= \iota\cdot f^*\cdot\iota'$. For such a pair $\x$ and $f^*$, define  $\iota^L(\x)= \iota^L(f^*)$ and $\iota^R(\x)= \iota^R(f^*)$.
Similar to the type $A$ maps, we define left and right type $D$ maps 
\[\delta^L \colon C^-(G^*)\to\Aa(\E^L)\otimes C^-(G^*),\qquad\delta^R\colon C^-(G^*)\to C^-(G^*)\otimes\Aa(\E^R)\] also by counting partial rectangles. In the following we describe the left type $D$ map $\delta^L$ in detail.

Let $\x=(\alpha_{\phi s}\cap\beta_{s})_{s\in S}$ be a generator.  We define a map $\partial^L$ by counting  partial rectangles that intersect the left and/or right boundaries $\{0,m+1\}\times [0,n+1]$. We distinguish four types of partial rectangles as follows:
\begin{itemize}
\item 
 $H=[0,s_1]\times [t_1,t_2]$, where $s_1\in S$, $t_1<t_2$ and $t_2 = \phi s_1, t_1\notin\phi(S)$.  
 Let $T_1 = \phi(S)^c$, $T_2 = \phi(S)^c\setminus\{t_1\}\cup \{t_2\}$, and define $\rho:T_1\to T_2$ by $\rho t_1=t_2$ and $\rho |_{T_1\setminus\{t_1\}} = \id_{T_1\setminus\{t_1\}}$. Let  $r = (T_1, T_2, \rho)\in \Aa(\E^L)$. Let  $\y$ be the set of intersection points $\x\setminus\{(s, t_2)\}\cup \{(s, t_1)\}$.

\item 
$H=[s_2, m+1]\times [t_1,t_2]$, where $s_2\in S$, $t_1<t_2$ and $t_1 = \phi s_2, t_2\notin\phi(S)$. Let $T_2 = \phi(S)^c$, $T_1 = \phi(S)^c\setminus\{t_2\}\cup \{t_1\}$, and define $\rho:T_2\to T_1$ by $\rho t_2=t_1$ and $\rho |_{T_2\setminus\{t_2\}} = \id_{T_2\setminus\{t_2\}}$. Let  $r = (T_2, T_1, \rho)\in \Aa(\E^L)$. Let $\y = \x\setminus\{(s, t_1)\}\cup \{(s, t_2)\}$.

\item $H=[0,m+1]\times [t_1,t_2]$, where $t_1, t_2\notin \phi(S)$ and $t_1<t_2$. Define $\rho:\phi(S)^c\to \phi(S)^c$ by $(t_1t_2)\circ\id_{\phi(S)^c}$ and let $r=(\phi(S)^c, \phi(S)^c, \rho)\in \Aa(\E^L)$. Let $\y = \x$.
\end{itemize}
In any of the above three cases we say that  the partial rectangle  $H$ \emph{connects $\x$ and $r$ to $\y$}. 
$H$ is \emph{empty} if $H\cap\X=H\cap\x=\emptyset$. For $O = (s_O, t_O)\in\O$ set $n_{t_O}(H)=|O\cap H|$.
\begin{itemize}
\item $H= ([0,s_1]\cup [s_2,m+1])\times [t_1,t_2]$, where $s_1<s_2$, $t_1<t_2$ and $t_1=\phi s_2, t_2 = \phi s_1$.  Then $H$ \emph{connects $\x$ and $r=(S^c,S^c,\id_{S^c})$ to $\y=(\alpha_{((t_1 t_2)\circ\phi) s}\cap\beta_{s})_{s\in S}$}. 
\end{itemize}
In this last case there is an extra condition on $H$ being empty: we require that for the projection $\pi_2\colon (s,t)\mapsto t$ the images $\pi_2(\X\cap [s_1,s_2]\times[t_1,t_2])$ and $\pi_2(\x\cap [s_1,s_2]\times[t_1,t_2])$ are precisely $[t_1,t_2]\cap\ah$ and $[t_1,t_2]\cap\a$. For $O\in\O$, let $n_{t_O}(H)=|O\cap H|$, and for $t_O^c\in ([t_1,t_2]\cap \ah)\setminus \TO$, let $n_{t^c_O}(H)=1$.

Given $\x$, $\y$, and $r$, let $\H_0(\x,\y, r)$ denote the set of empty partial rectangles connecting $\x$ and $r$ to $\y$ (note that that set is either empty, or consists of one partial rectangle). Define  
\[\partial^L \x=\sum_{\substack{\y\in \gen(G^*)\\ r\in \gen(\E^L)}}\sum_{H\in \H_0(\x,\y, r)}r\otimes\prod_{t_O\in{ \ah{}}} U_O^{n_{t_O}(H)}\y.\]
See Figure \ref{fig:leftgridtypeD} for an example of $\partial^L$.

\begin{figure}[h]
 \centering
     \hspace{2.2cm}  \includegraphics[scale=0.72]{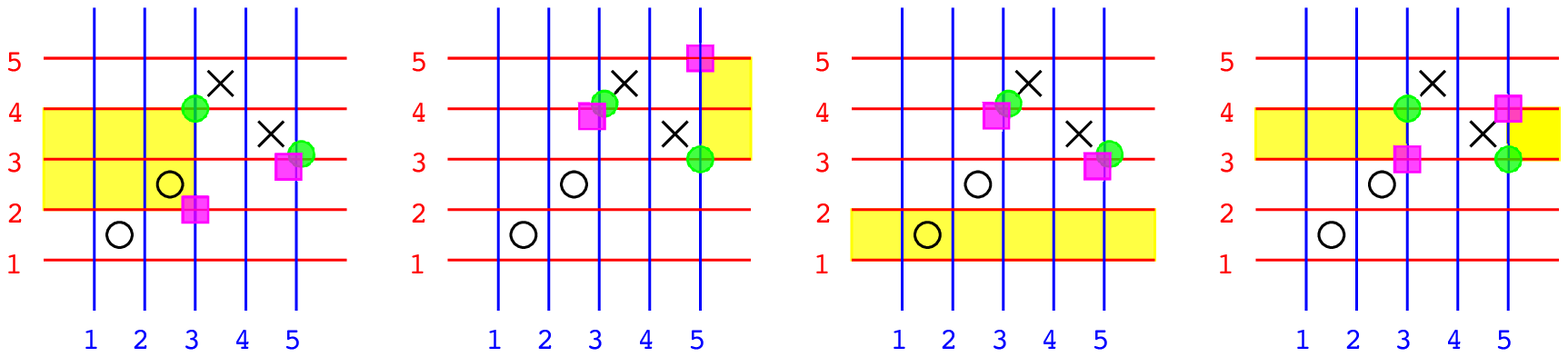} 
        \vskip .5 cm
       \includegraphics[scale=0.81]{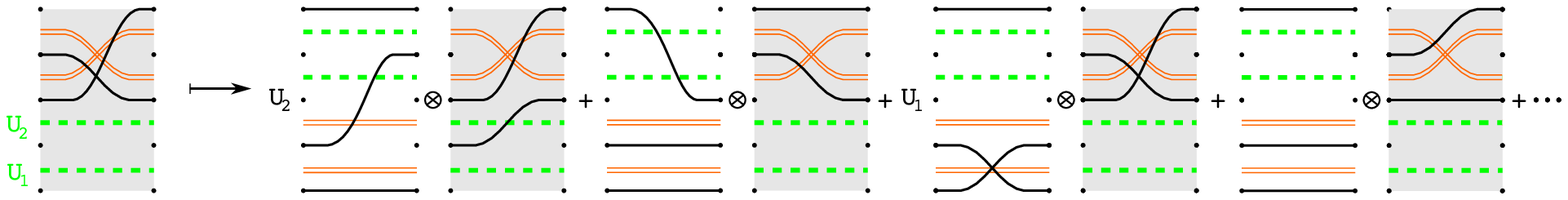} 
       \vskip .2 cm
       \caption[The four types of rectangles corresponding to  the map $\partial^L$.]{\textbf{The four types of rectangles corresponding to the map $\partial^L$.} Top: examples of the four types of rectangles for $\partial^L$ applied to the generator formed by the green dots. Bottom: the respective terms of $\delta^L$ applied to the strand diagram corresponding to the green dots.}\label{fig:leftgridtypeD}
\end{figure}

Then the left type $D$ map is defined on generators by 
\[\delta^L\x =\iota^L(\x)\otimes \partial \x + \partial^L \x.\]
In other words, $\delta^L$ is defined by counting empty rectangles in the interior of the grid, as well as empty rectangles that touch the left and/or right boundary of the grid.

The right type $D$ map $\delta^R$ can be defined in a similar way as the sum $\delta^R = \partial\otimes \iota^R + \partial^R$ using  a map $\partial^R$ that counts partial rectangles that intersect the top and bottom boundary of $[0,n+1]\times[0,m+1]$. 

The left and the right type $D$ maps can be merged together to define a type $\DD$ map by counting all empty rectangles. 
\begin{definition}
For $G^*=G^*(\P^*)$ define $\CDTD^-(G^*)$ be the left-right type  $\mathit{DD}$ structure $(C^-(G^*),\delta^1\})$ over $\Aa({\E^L})$ and $\Aa({\E^R})$, where 
\[\delta^1\colon  C^-(G^*)\to \Aa({\E^L})\otimes C^-(G^*) \otimes \Aa({\E^R})\]
is defined via
\[\delta^1(\x)=\iota^L(\x)\otimes\partial^R(\x)+\iota^L(\x)\otimes \partial (\x)\otimes \iota^R(\x)+ \partial^L(\x)\otimes\iota^R(\x).\]
\end{definition}
\begin{proposition}\label{prop:gridtyped1}
For $G^*=G^*(\P^*)$ the one to one correspondence between generators gives rise to an isomorphism between $\CDTD^-(G^*)$ and $\CDTD^-(\P^*)$.
\end{proposition}
While Proposition \ref{prop:gridtyped1} and the fact that $\CDTD^-(G^*)$ satisfies the type $\mathit{DD}$ identities could be proven directly, we will chose a longer way. First we understand how to glue bordered grid diagrams. Then, as it is explained later,  both statements are consequences of Proposition \ref{thm:chain} and \ref{thm:chain2}.

\subsection{Gluing bordered grid diagrams}\label{ssec:glueing}
Suppose that $G_1=G(\P_1)= (\alphas^1,\betas^1,\X_1,\O_1)$ and $G_2^*=G^*(\P_2^*)= (\alphas^2,\betas^2,\X_2,\O_2)$, where $\P_1$ and $\P_2^*$ have well-defined wedge product. This means that $n_1=n_2$, so $G=G_1\cup G_2^*/\sim \subset [-m_1-1,m_2+1]\times [0,n_1+1]/\sim$ is a bordered grid diagram where the edges $\{-m_1-1\}\times[0,n_1+1]$ and $\{m_2+1\}\times[0,n_2+1]$ are identified. Here $\betas=\betas^1\cup \betas^2$, and the $\alphas$-arcs are glued to form the new circles $\widetilde\alpha_a=[-m_1-1,m_2+1]\times \{a\} / \sim$. Similarly, $\X=\X_1\cup \X_2$ and $\O=\O_1\cup \O_2$. Note that since $\P_1$ and $\P_2^*$ have a well-defined wedge product every annulus between the alpha circles $\widetilde\alpha_a$ and $\widetilde\alpha_{a+1}$ contains exactly one element of $\X$ and one element of $\O$.

Informally, we glued $G_2^*$ to the right of $G_1$ and identifed the left and right edges of the resulting rectangle to obtain an annulus. Alternatively, one can shift coordinates in $\R^2$ and view the annulus by placing $G_2^*$ to the left of $G_1$ and then identifying the left and right edges of the resulting rectangle to obtain an annulus. Abstractly, the annulus is simply the result of identifying each ``$\alphas$-boundary edge" of one grid with an $\alphas$-boundary edge of the other grid, so that the labels on the $\alphas$-curves match up, and the gluing respects the orientation on the two surfaces of the grids.

We define $C^-(G)$ to be the free module generated over $\F_2[U_O]_{O\in \O}$ by tuples of intersection points $\x\subset \widetilde\alphas\cap \betas$ such that there is one point on each $\widetilde\alphas$-circle, and at most one point on each $\betas$-arc. Observe that the generating set is precisely 
\[\gen(G) = \{\x = (\x_1,\x_2)\in \gen(G_1)\times \gen(G_2^*)| |\x_1\cap \alpha^1_a|=1 \textrm{ if and only if } |\x_2\cap \alpha^2_a|=0\}.\]
Define a map $\partial$ on $\gen(G)$ by counting empty rectangles in the interior of $G$ (note that rectangles may cross the newly identified edges), and extend linearly to all of $C^-(G)$. By standard grid diagram arguments, $\partial$ is a differential. See Figure \ref{fig:glueinggrid} for an example of the identification where $G_1$ is drawn to the right.
\begin{figure}[h]
 \centering
       \includegraphics[scale=0.9]{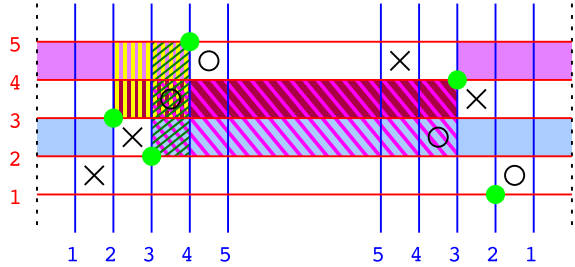} 
       \vskip .2 cm
       \caption[The differential on the annular grid diagram associated to the example of Figure \ref{fig:comultiplication}.]{\textbf{The differential on the annular grid diagram associated to the example of Figure \ref{fig:comultiplication}.} The dashed lines on the right and left hand side are glued together.
       The green dots corresponds to the strand diagram on the left hand side of Figure \ref{fig:comultiplication}, and the six rectangles to the nonzero terms in the differential of that diagram.}
       \label{fig:glueinggrid}
\end{figure}

Now there is a one to one correspondence between generators of $\P_1\wedge \P_2^*$ and $\gen(G)$ given by mapping $(S_1,T_1,\phi_1)\otimes (S_2,T_2,\phi_2)^*$ to $(\x_1,\x_2)$, where $\x_1=(\alpha^1_{\phi_1 s}\cap\beta^1_s)_{s\in S_1}$ and $\x_2=(\alpha^2_{\phi_2 s}\cap\beta^2_{s})_{s\in S_2}$. 
We show below that under this correspondence the differential $\partial$ on $C^-(G)$ agrees with $\partial_\wedge$ on $C^-(\P_1\wedge \P_2^*)$. 
In particular, it follows that $(C^-(\P_1\wedge \P_2^*),\partial_\wedge)$ is a chain complex as it is stated in Theorem \ref{thm:cocomposition2}. 
\begin{proposition}\label{thm:chain}
The structures $(C^-(\P{}_1\wedge \P_2^*),\partial_\wedge)$ and $(C^-(G),\partial)$ are isomorphic. 
\end{proposition}
\begin{proof}
Let $(\x_1,\x_2)$ be the generator of $(C^-(G),\partial)$ corresponding to the element $f=f_1\otimes f_2^*=(S_1,T_1,\phi_1)\otimes (S_2,T_2,\phi_2)^*$ in $(C^-(\P_1\wedge \P^*_2),\partial_\wedge)$. Recall that the differential $\partial_\wedge$ of $f_1\otimes f_2^*$ is given by the formula
\[\partial_\wedge(f_1\otimes f_2^*)=\partial(f_1)\otimes f_2^*+f_1\otimes \partial^*(f_2^*)+\partial_\textrm{mix}(f_1\otimes f_2^*),\]
while the differential of $(\x_1,\x_2)$ in $(C^-(G),\partial)$ is given by counting rectangles. Suppose that the rectangle $R$ contributes to the differential $\partial$. Then depending on the position of $R$ the result corresponds to different components of the differential $\partial_\wedge$ as follows:
\begin{itemize}
\item If $R$ is entirely contained in $G_1$, then $R$ corresponds to a term of $\partial(f_1)\otimes f_2^*$;
\item If $R$ is entirely contained in $G_2^*$, then $R$ corresponds to a term of $f_1\otimes \partial^*(f_2^*)$;
\item If $R$  intersects both $G_1$ and $G_2^*$, each in a  connected component, then $R$ intersects exactly one of the vertical lines $\{0\}\times [0,n_1+1]$ or $\{-m_1-1\}\times [0,n_1+1]\sim \{m_2+1\}\times [0,n_1+1]$.  In the first case  $R\cap\{0\}\times [0,n_1+1] = \{0\}\times [p,q]$ for some $p<q$, and in the second case  $R\cap\{m_1-1\}\times [0,n_1+1] = \{m_1-1\}\times [q,p]$ for some $q<p$.  Then $(p,q)\in S_1\times S_2$ is an exchangeable pair, and $R$ corresponds to a term of $\partial_\textrm{mix}$;
\item If $R$ intersects both $G_1$ and $G_2^*$ and $R\cap G_1$ has one component while $R\cap G_2^*$ has two components, then let $R\cap\{0\}\times [0,n_1+1] = \{0\}\times [p,q]$ for some $p<q$. The pair $(p,q)\subset S_2$ is exchangeable  and $R$ corresponds to a term of $\partial_\textrm{mix}$;
\item Similarly if $R$ intersects both $G_1$ and $G_2^*$ and $R\cap G_1$ has two components while $R\cap G_2^*$ has one component, then $R\cap\{0\}\times [0,n_1+1] = \{0\}\times [p,q]$ for some $p<q$. The pair $(p,q)\subset S_1$ is exchangeable  and $R$ corresponds to a term of $\partial_\textrm{mix}$.
\end{itemize}
Conversely, any term of $\partial_\wedge (f_1\otimes f_2^*)$ appears in the above list, thus the statement is proved. 
\end{proof}

Note that the writeup of the above proof uses coordinates for the case  when $G_1$ is viewed sitting to the left of $G_2^*$.

Similarly, if $G_1^*=G^*(\P_1^*)$ and $G_2'=G_2'(\P_2)$, then we can glue $(G_1^*)'$ to $G_2$ along the $x$-axis, i.e. place $G_2$ above $G_1^*$, and identify the resulting horizontal boundaries. Alternatively, we can view the annulus by placing $G_2$ below $G_1^*$ and then identifying the horizontal edges of the resulting rectangle. Abstractly, the annulus is the result of identifying $\betas$-boundary edges.  For the resulting annular grid diagram, we define a chain complex $(C^-(G), \partial)$, where again generators over $\F_2[U_O]_{O\in \O_1\cup\O_2}$ are tuples of intersection points with exactly one point on each $\widetilde\betas$-circle and at most one point  on each $\alphas$-arc, and the differential counts empty rectangles. Once again we have:
\begin{proposition}\label{thm:chain2}
The structures $(C^-(\P^*_1\wedge \P_2),\partial_\wedge)$ and $(C^-(G),\partial)$ are isomorphic.
\end{proposition}
\begin{proof}
The proof is analogous to that of Proposition \ref{thm:chain}.
\end{proof}
As an immediate consequence we have:
\begin{proof}[Proof of Proposition \ref{thm:cocomposition} and  \ref{thm:cocomposition2}]
Both statements follow from Propositions \ref{thm:chain} and \ref{thm:chain2} for $C^-(\E^L(\P^*)\wedge \P^*)$ and  $C^-(\P^*\wedge \E^R(\P^*))$, along with the fact that $\partial$ is a differential for the corresponding grid diagrams.
\end{proof}

In general, suppose we have an alternating sequence of shadows and mirror-shadows $\Pp=(\P_1^\circ,\dots,\P_p^\circ)$ with well-defined consecutive wedge products. We can glue the grid diagrams $G^\circ (\P_1^\circ),\ldots, G^\circ (\P_p^\circ)$ by alternating the gluing along  horizontal or vertical edges to obtain the nice bordered Heegaard diagram $G$ on plumbings of annuli. 
We can associate a tangle to $G$, which is simply the concatenation of $\T^\circ(G^\circ(\P_1^\circ)),$ $\dots,$ $\T^\circ(G^\circ(\P_p^\circ))$. See, for example,  Figure \ref{fig:reid2}.

Let $C^-(G)$ be the free module over $\F_2[U_O]_{O\in \O_1\cup \cdots\cup \O_p}$ generated by tuples of intersection points, one point on each $\widetilde\alphas$-circle, at most one  on each $\alphas$-arc, one on each $\widetilde\betas$-circle, and at most one on each $\betas$-arc, and let $\partial$ be the differential on $C^-(G)$ defined by counting empty rectangles. Then
\begin{proposition}\label{thm:chain3}
The structures $(C^-(\Pp),\partial_\wedge)$ and $(C^-(G),\partial)$ are isomorphic. 
\end{proposition}
\begin{proof}
The proof is analogous to that of Proposition \ref{thm:chain} (here, any empty rectangle is either fully contained in one grid, or intersects two consecutive grids).
\end{proof}
 When the gluing maps between adjacent grids are clear from the context, we will use the otherwise ambiguous notation $G^\circ (\P_1^\circ)\cup \cdots \cup G^\circ (\P_p^\circ)$ for $G$. We will also sometimes write $\x_1\cup \ldots\cup \x_p$ for $(\x_1, \ldots, \x_p)$.

We are now ready to prove Proposition \ref{prop:gridtyped1}.
\begin{proof}[Proof of Proposition \ref{prop:gridtyped1}]
By definition, the maps $\delta^L$,  $\delta^R$, and $\delta^1$ on a generator $f^*$ of $\CDTD^-(\P^*)$ correspond to the map $\partial_{\wedge}$ on the generators $\iota^L(f^*)\otimes f^*$, $f^*\otimes\iota^R(f^*)$, and  $\iota^L(f^*)\otimes f^*\otimes\iota^R(f^*)$of $\E^L(\P^*)\wedge \P^*$, $\P^*\wedge \E^R(\P^*)$, and $\E^L(\P^*)\wedge \P^*\wedge \E^R(\P^*)$, respectively.

It is also not hard to see that the  maps $\delta^L$,  $\delta^R$, and $\delta^1$ on a generator $\x$ of $\CDTD^-(G^*)$ correspond to the map $\partial$ on the generators $\iota^L(\x)\cup \x$, $\x\cup \iota^R(\x)$, and 
 $\iota^L(\x)\cup \x\cup \iota^R(\x)$ of the grid diagrams $G(\E^L(\P^*))\cup G^*$, $G^*\cup G(\E^R(\P^*))$, and  $G(\E^L(\P^*))\cup G^*\cup G(\E^R(\P^*))$, respectively. We outline the correspondence for $\delta^L$ here. The other cases are analogous. An empty rectangle starting at $\x$ that stays in $G^*$ contributes to $\partial(\x)$, hence to $\iota^L(\x)\otimes \partial (\x)$,  as well as to $\partial (\iota^L(\x)\cup \partial (\x))$.
 An empty partial rectangle starting at $\x$ in $G^*$ of the form $[0,t_1]\times [s_1,s_2]$, $[t_2, m+1]\times [s_1,s_2]$, $[0,m+1]\times [s_1,s_2]$, or $([0,t_1]\cup [t_2,m+1])\times [s_1,s_2]$ contributes to $\partial^L(\x)$ and corresponds to the empty rectangle  $[-s_1,t_1]\times [s_1,s_2]$,  $([-n-1,s_2]\cup [t_2, m+1])\times [s_1,s_2]$, $([-n-1, -s_2]\cup [-s_1,m+1])\times [s_1,s_2]$, or $([-n-1, t_1]\cup[t_2, m+1])\times [s_1, s_2]$, respectively,  in $G(\E^L(\P^*))\cup G^*$ which contributes to $\partial(\iota^L(\x)\cup \x)$.

By Propositions \ref{thm:chain},  \ref{thm:chain2}, and \ref{thm:chain3},  the correspondence between generators of $\E^L(\P^*)\wedge \P^*$ and $G(\E^L(\P^*))\cup G^*$, $\P^*\wedge \E^R(\P^*)$ and $G^*\cup G(\E^R(\P^*))$, and $\E^L(\P^*)\wedge \P^*\wedge \E^R(\P^*)$ and $G(\E^L(\P^*))\cup G^*\cup G(\E^R(\P^*))$, respectively,  carries the map $\partial_{\wedge}$ to the map $\partial$. Therefore, the structures $(C^-(G^*), \delta^L)$ and $(C^-(\P^*), \delta^L)$,  $(C^-(G^*), \delta^R)$ and $(C^-(\P^*), \delta^R)$, and $(C^-(G^*),  \delta^1)$ and $(C^-(\P^*), \delta^1)$ are pairwise isomorphic. In particular, $\CDTD^-(G^*)$ and $\CDTD^-(\P^*)$ are isomorphic. Further, by Proposition \ref{prop:typeDrelns}, $(C^-(G^*), \delta^L)$ is a left type $D$ structure, $(C^-(G^*), \delta^R)$ is a right type $D$ structure, and $\CDTD^-(G^*)$ is a left-right type $\DD$ structure.
\end{proof}

The above proof sums up to the following observation. For a mirror-shadow $\P^*$, the maps $\delta^L$,  $\delta^R$, and $\delta^1$ on a generator $f^*$ correspond to gluing $G^*(\P^*)$ to $G(\E^R(\P^*))$ along the $\betas$-curves and/or to $G(\E^L(\P^*))$ along the $\alphas$-curves, and then taking the inner differential of the generator of the resulting diagram corresponding to $\iota^L(f^*)\otimes f^*$, $f^*\otimes\iota^R(f^*)$, or  $\iota^L(f^*)\otimes f^*\otimes\iota^R(f^*)$, respectively.

If $G$ is the bordered Heegaard diagram corresponding to an alternating sequence of shadows and mirror-shadows $\Pp=(\P_1^\circ,\dots,\P_p^\circ)$ with well-defined consecutive wedge products, then  $C^-(G)$ has a left type $A$ or $D$ map depending on whether $\P_1^\circ$ is shadow or a mirror-shadow, defined by counting partial rectangles in $G^\circ (\P_1^\circ)$ as usual, and similarly it has a right type $A$ or $D$ map depending on whether $\P_p^\circ$ is a shadow or a mirror shadow. Denote the resulting structures by $\mathit{CATA}^-(G)$, $\mathit{CDTA}^-(G)$, $\mathit{CATD}^-(G)$ or $\mathit{CDTD}^-(G)$, or simply by $\mathit{CT}^-(G)$.

\subsection{Self-gluing of bordered grid diagrams}\label{ssec:selfgluing-grid}
In this subsection we discuss annular bordered grid diagrams corresponding to one-sided modules. Let $G^*=G^*(\P^*) = (\alphas,\betas,\X,\O)$ correspond to a mirror-shadow $\P^*$ with $\ah=\SX=\TO$. This means that each row of $G^*$ contains both an $X$ and an $O$, thus the 
annular bordered grid diagram $G^*_\a = (\widetilde{\alphas},\betas,\X,\O)$ will have an $X$ and an $O$ in each of its annuli.
See Figure \ref{fig:selfgluing}.
\begin{figure}[h]
 \centering
       \includegraphics[scale=0.9]{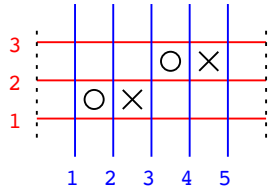} 
       \vskip .2 cm
       \caption[Self-gluing of a bordered grid diagram.]{\textbf{Self-gluing of a bordered grid diagram.} The dashed lines are identified.}
       \label{fig:selfgluing}
\end{figure}
Take the subset $\gen_n(G^*_\a)$ of generators that occupy each $\widetilde{\alpha}$-circle. Then the map $\partial$ that also counts the rectangles which cross the line $\{0\}\times [0,m+1]\sim \{n+1\}\times[0,m+1]$ endows $C_n^-(G^*_\a)$  with a chain complex structure, and under the usual identification of $\gen_n(\P^*)$ with $\gen_n(G^*_\a)$ we have
\begin{proposition}\label{prop:selfgrid}
$(C_n^-(G^*_\a),\partial)$ is a  chain complex isomorphic to $(C_n^-(\P^*),\partial^*+{}_D\partial)$.
\end{proposition}
\begin{proof}
The proof is similar to the proof of Proposition \ref{prop:isom}. The terms in ${}_D\partial$ correspond to those empty rectangles that cross the gluing, as follows. For the generator $f=(S,T,\phi)$ corresponding to the intersection point $\x=(\alpha_s,\beta_{\phi s})_{s\in S}$, the pair $(s_1,s_2)$ is allowable exactly when the glued up rectangle $R=([0,s_1]\cup [s_2,n+1])\times [\phi(s_2),\phi(s_1)]$ is empty (i.e. $\x\cap R=\X\cap R=\emptyset$). Then $R$ connects $\x$ to $\y=(\alpha_s,\beta_{\phi^{(s_1,s_2)} s})_{s\in S}$ and $n_O$ measures the multiplicity of $O$ in $R$.\end{proof}

Lemma \ref{lem:selfglueingchain} now  follows from Proposition \ref{prop:selfgrid}.

As in Section \ref{ssec:Dgrid}, we can define a right type $D$ map on $C_n^-(G^*_\a)$ by $\delta^1 \x =\partial \x \otimes \iota^R(\x) + \partial^R \x.$, to obtain a right type $D$ structure  $\mathit{CTD}^-(G^*_\a)$ which,  by arguments analogous to those for Proposition \ref{thm:chain},  is isomorphic to $\mathit{CTD}^-(\P^*)$.  We can similarly define structures $\mathit{CDT}^-(G^*_\b)$, $\mathit{CAT}^-(G_\b)$ and $\mathit{CTA}^-(G_\a)$ isomorphic to $\mathit{CDT}^-(\P)$, $\mathit{CAT}^-(\P)$ and $\mathit{CTA}^-(\P)$. 

\begin{convention}\label{conv:contractgrid}
Similar to Convention \ref{conv:contract}, if $G_1^\circ\cup \cdots \cup G_p^\circ$ corresponds to an alternating sequence of shadows and mirror-shadows $\Pp=(\P_1^\circ,\dots,\P_p^\circ)$, and
 $G_1^\circ$ and/or $G_p^\circ$ can be self-glued, we will always self-glue it, to produce a nice diagram $G$ whose invariant is a  one-sided module or a chain complex that agrees with $\mathit{CT}^-(\Pp)$.
\end{convention}

\subsection{Pairing for plumbings of bordered grid diagrams.}
Gluing bordered grid diagrams corresponds to taking a box tensor product of their algebraic invariants:

\begin{theorem}
Given an alternating sequence of shadows and mirror-shadows $\Pp=(\P_1^\circ,\dots,\P_p^\circ)$ with well-defined consecutive wedge products, denote $G^\circ(\P_1^\circ)\cup \cdots \cup G^\circ(\P_i^\circ)$ and $G^\circ(\P_{i+1}^\circ)\cup \cdots \cup G^\circ(\P_p^\circ)$ by $G_i$ and $G_i'$, respectively. 
The obvious identification of generators gives an  isomorphism
\[\mathit{CT}^-(G_i\cup G_i') \cong \mathit{CT}^-(G_i)\boxtimes \mathit{CT}^-(G_i').\]
\end{theorem}
\begin{proof}This follows from the equivalences proven earlier in this section, along with Theorem \ref{thm:wedgevsbox}. Alternatively, one can notice that by definition of the type $D$ and type $A$ actions for bordered grid diagrams, pairing them via $\boxtimes$ corresponds to matching partial rectangles for the type $D$ maps with  sets of partial rectangles for the type $A$ maps along the boundary. The possible pairings correspond to empty rectangles in the union of the two diagrams that cross the gluing. 
\end{proof}

\subsection{Relations between the $U$-actions}
Let  $\Pp=(\P_1^\circ,\dots,\P_p^\circ)$ be an alternating sequence of shadows and mirror-shadows with well-defined consecutive wedge products. Let  $G$ be the nice bordered Heegaard diagram obtained by gluing $G^\circ (\P_1^\circ), \dots ,G^\circ (\P_p^\circ)$ as before. 

The pairs $O=(s_O,\omega_ is_O)$ and $O'=(s_O',\omega_{i'} s_O')$ are connected by a path  exactly when $O$ and $O'$ lie on the same component of the tangle $\T(G)$ associated to $\Pp$, or in other words if there is a sequence of $O=O_1,X_1,O_2,X_2, \dots,X_{k-1},O_k=O'$ such that $O_j$ and $X_j$ are in the same row, and $X_{j-1}$ and $O_j$ are in the same column (note that we also require that none of the $X$s are in the first or last parts $G^\circ(\P_1^\circ)$ or $G^\circ(\P_p^\circ)$). Now we are ready to prove Lemma \ref{lem:Uaction}:
\begin{proof}[Proof of Lemma \ref{lem:Uaction}]
First let us assume that $\mathit{CT}^-(\Pp)$ is a type $\mathit{AA}$ structure. Then we need to prove that there is a type $\mathit{AA}$ map $\H$ such that $(U+U')\id_{\mathit{CT}^-(\Pp)}=\partial \H$. It is enough to prove this statement in the case when $O$ and $O'$ are of distance 1 (the general case then can be obtained by adding up the homotopies for all $j$). This means that there is a point $X$ which is in the  row of $O$ and in the column of $O'$. By definition $X$ is not in $G^\circ(\P_1^\circ)$ or $G^\circ(\P_p^\circ)$, thus the horizontal and vertical rows containing it are both closed up to annuli. This means that the map $\H_X$ that counts rectangles that cross $X$ once consists of the single map $\mathit{CT}^-(\Pp)\to \mathit{CT}^-(\Pp)$ with no nontrivial components of the type $\Aa(\E_L(\P_1))^{\otimes l}\otimes \mathit{CT}^-(\Pp)\otimes\Aa(\E_R(\P_p))^{\otimes r}\to \mathit{CT}^-(\Pp)$ for $l,r>0$. And as in \cite{mos} the map $\H_X$ satisfies  $(U+U')\id_{\mathit{CT}^-(\Pp)}=\partial \H_X$.

The argument goes exactly the same way for the other types of structures, with the observation that if $\Pp$ starts or ends with a mirror-shadow, then we can complete it by adding $\E^R(\P_1^*)$ and/or $\E^L(\P_p^*)$ and denote the obtained sequence of shadows and mirror-shadows by $\Pp'$. Then chain homotopy in $(C^-(\Pp'),\partial)$ gives type $\mathit{DD}$ (or $\mathit{DA}$, or $\mathit{AD}$) equivalence of $\mathit{CT}^-(\Pp)$.
\end{proof}

%%%%%%%%%%%%%%%%%%%%%%%%%%%%%%%%%%%%%%%%%%%%%%%%%%%%%%%
%%%%%%%%%%%%%%%%%%%%%%%%%%%%%%%%%%%%%%%%%%%%%%%%%%%%%
% BIMODULES FOR TANGLES
%%%%%%%%%%%%%%%%%%%%%%%%%%%%%%%%%%%%%%%%%%%%%%%%%%%%%%%
%%%%%%%%%%%%%%%%%%%%%%%%%%%%%%%%%%%%%%%%%%%%%%%%%%%%%%%

\section{Modules associated to tangles}\label{sec:tangleinv}

In this section we will associate a left type $D$ structure or a right type $A$  structure to a tangle in $D^3$, a type $\mathit{DA}$ structure  to a tangle in $I\times S^2$, and a bigraded  chain complex to a knot (or link) in $S^3$. The main idea is to cut $\T$ into elementary pieces $\T=\T_1\circ \dots \circ \T_p$, associate a type $A$  structure to $\T_1$ if it is in $D^3$, a type $D$ structure to $\T_p$ if it is in $D^3$, and type 
$\mathit{DA}$ structures to all the other $\T_j$'s, and then take their box-tensor product. 
The structures associated to elementary pieces are the structures defined earlier for wedge products of appropriate shadows and mirror-shadows. The hard part -- of course -- is to prove independence of the cut. Although we believe that there is a completely combinatorial proof of the independence, in this paper we will only provide a proof that uses holomorphic curve techniques, see Section \ref{sec:hd_modules}. As a consequence of that, we can only prove independence for the ``tilde''-version of the theory. 
%%%%%%%%%%%%%%%%%%%%%%%%%%%%%%%%%%%%%%%%%%%%%%%%%%%%%%%
% ALGEBRAS TO ENDS
%%%%%%%%%%%%%%%%%%%%%%%%%%%%%%%%%%%%%%%%%%%%%%%%%%%%%%%

\subsection{Algebras associated to $\partial \T$}
For a sequence of oriented points with signs $\e=(\epsilon_1,\dots,\epsilon_k)$, let $n=k+1$, and remember that the sequence $\e=\e^1$
 corresponds to two complementary subsets $\SX=\{j+\frac12 : \epsilon_j=-1\}$ and $\TO=\{j+\frac12 : \epsilon_j=+1\}$ of the set $\{1\frac12,\dots,n-\frac12\}$. Set $\e^0=-\e^1$. This determines $\TX(=\SX)$ and $\TO(=\SO)$ in a similar vein. Take the idempotent shadow ${}_{\e^0}\mathcal{E}_{\e^1}=(n,n,\id_{\SX},\id_{\SO})$ of Example \ref{exm:straight}. This defines the algebra $\Aa_{\e}=\Aa({}_{\e^0}\mathcal{E}_{\e^1})$.

Given a tangle $\T$ with left boundary $\partial^0\T$ and right boundary $\partial^1\T$ (any of these sets can be empty if the tangle is closed from that side), let $\e^0=\e(\partial^0\T)$ and $\e^1=\e(\partial^1\T)$ be the sequences of signs of $\partial^0 \T$ and $\partial^1\T$, respectively. Let $\Aa(\partial^0 \T)=\Aa_{-\e^0}$ and $\Aa(\partial^1 \T)=\Aa_{\e^1}$. The minus sign in the second definition is there so that if we cut $\T=\T_1\circ\T_2$, then $\e^1(\partial^1\T_1)=-\e^0(\partial^0\T_2)$ thus $\Aa(\partial^1 \T_1)=\Aa(-\partial^0 \T_2)$.

%%%%%%%%%%%%%%%%%%%%%%%%%%%%%%%%%%%%%%%%%%%%%%%%%%%%%%%
% INVARIANTS OF TANGLES
%%%%%%%%%%%%%%%%%%%%%%%%%%%%%%%%%%%%%%%%%%%%%%%%%%%%%%%
\subsection{Invariants associated to a tangle}
Given a sequence of shadows and mirror-shadows $\Pp=(\P_1^\circ,\dots,\P_p^\circ)$ with well-defined consecutive wedge products, each $\P_j^\circ$ has a tangle $\T_j=\T^\circ(\P_j^\circ)$ associated to it. Note that if $\P_j$ is a shadow then at all crossings the strand with the bigger slope goes over the strand with the smaller slope, while if $\P_j^*$ is a mirror-shadow then at all crossings the strand with the smaller slope goes over the strand with the bigger slope. Since $\P_j^\circ$ and $\P_{j+1}^\circ$ have well-defined wedge product, thus $\P_j^\circ$ is not left-contractible, and $\P_{j+1}^\circ$ is not right-contractible, so $\partial^1\T_j\neq \emptyset$ and $\partial^0\T_{j+1}\neq \emptyset$ for $1\le j\le p-1$. If $\P_1^\circ$ is left-contractible then $\partial^0\T_1=\emptyset$ and if  $\P_p^\circ$ is left-contractible then $\partial^1\T_p=\emptyset$. This means that the composition-tangle $\T(\Pp)=\T_1\circ\cdots\circ\T_p$ can be  in $S^3$, $D^3$ or in $S^2\times I$. Moreover any tangle $\T$ can be constructed in the above way. 

\begin{lemma}\label{lem:decomp}
Let $\T$ be a tangle in $S^3$, $D^3$ or in $S^2\times I$. Then there is a sequence of shadows $\Pp=(\P_1^\circ,\dots,\P_p^\circ)$ such that $\T$ is isotopic to $\T(\Pp)$ (relative to the boundary), and
\begin{itemize}
\item[-] if $\partial^0\T=\emptyset$ then $\P_1^*$ is a mirror-shadow;
\item[-] if $\partial^1\T=\emptyset$ then $\P_p^*$ is a mirror-shadow;
\item[-] if $\partial^0\T\neq\emptyset$ then $\P_1^*$ is a mirror-shadow and 
$\e^0(\P_1^*)=\e^0(\T)$;
\item[-]  if $\partial^1\T\neq\emptyset$ then $\P_p$ is a shadow and  
 $\e^1(\P_p)=\e^1(\T)$.
\end{itemize}
\end{lemma}
The first two assumptions are in the statement for cosmetic reasons (to match with the assumptions of Sections \ref{sec:circles}-\ref{sec:hd_pairing}), while,
as we will see later, the last two assumptions ensure that the associated invariant has the correct type and is defined over the correct algebras.
\begin{proof}
The statement is clearly true for elementary tangles $\T$. Indeed, depending on the type of crossing in $\T$, or whether $\T$ is a cap or a cup we can always bisect $\T$ into two pieces $\T_-\circ \T_+$ such that one of $\T_-$ or $\T_+$ consists of straight strands (possibly with a gap) and the other one is isotopic to $\T$, and 
at the (possible) crossing of $\T_-$ (or $\T_+$) the strand with the smaller slope goes over (under) the strand with bigger slope, or $\T_-$ (or $\T_+$) is a cup (or a cap). Let $\P^*_-$ and $\P_+$ be the mirror shadow and shadow corresponding to  $\T_-$ and $\T_+$ (i.e. $\T_-=\T^*(\P^*_-)$ and $\T_+=\T(\P_+)$). Note that in this case the condition $\epsilon^0(\P^*_-)=\epsilon^0(\T)$ is equivalent to $\P_-^*$ not having a gap on its left side. Similarly the condition $\epsilon^1(\P_+)=\epsilon^1(\T)$ to $\P_+$ not having a gap on its right side.

In the general case, put $\T$ in a not obviously split position. This means that when cutting it up into elementary tangles $\T=\T_1\circ\cdots\circ\T_p$, every cut intersects the tangle. Then, by the previous paragraph, each $\T_i$ is isotopic to
$\T^*((\P_i)_-^*)\circ\T((\P_i)_+)$. Thus if $\partial^1\T\neq \emptyset$ then the decomposition
$\T= \T^*((\P_1)_-^*)\circ\T((\P_1)_+)\circ\cdots\circ\T^*((\P_p)_-^*)\circ\T((\P_p)_+)$ works. Otherwise $\T_p$ is a single cap, thus it can be written as $\T_p=\T^*(\P_p^*)$, where $\P_p^*$ does not have a gap on its right. This means that the decompositon $\T= \T^*((\P_1)_-^*)\circ\T((\P_1)_+)\circ\cdots\circ\T^*((\P_{p-1})_-^*)\circ\T((\P_{p-1})_+)\circ\T^*(\P_{p}^*)$ satisfies all criterions of the lemma.
\end{proof}

Note that by construction, if $\partial^0\T=\emptyset$, then $\T_1^-$ is left-contractible, and if $\partial^1\T=\emptyset$, then $\T_p^-$ is right-contractible. 

\begin{definition} Let $\T$ be a tangle given by a sequence of shadows $\Pp=(\P_1^\circ,\dots,\P_p^\circ)$  as in Lemma \ref{lem:decomp}. 

If $\partial^0\T=\emptyset$ and $\partial^1\T=\emptyset$, then define the chain complex by
\[\mathit{CT}^-(\Pp)=\mathit{CTD}^-(\P_1^*)\boxtimes\cdots\boxtimes \mathit{CDT}^-(\P_p^*).\]

If $\partial^0\T=\emptyset$ and $\partial^1\T\neq\emptyset$, then define the right  type $\mathit{A}$ structure over $\Aa(\partial^1\T)$ by
\[\mathit{CTA}^-(\Pp)=\mathit{CTD}^-(\P_1^*)\boxtimes\cdots\boxtimes \mathit{CATA}^-(\P_p).\]

If $\partial^0\T\neq\emptyset$ and $\partial^1\T=\emptyset$, then define  the left  type $\mathit{D}$ structure over $\Aa(\partial^0\T)$ by
\[\mathit{CTD}^-(\Pp)=\mathit{CDTD}^-(\P_1^*)\boxtimes\cdots\boxtimes \mathit{CDT}^-(\P_p^*).\]

If $\partial^0\T\neq\emptyset$ and $\partial^1\T\neq\emptyset$, then define the left-right  type $\mathit{DA}$ structure over $\Aa(\partial^0\T)$ and $\Aa(\partial^1\T)$  by
\[\mathit{CDTA}^-(\Pp)=\mathit{CDTD}^-(\P_1^*)\boxtimes\cdots\boxtimes \mathit{CATA}^-(\P_p).\]
\end{definition}
Whenever the sequence $\Pp$ is clear from the context, we simplify the notation of the above bimodules to $CT^-(\T)$.
In this paper we will not prove that $\mathit{CT}^-(\T)$ as defined above is an invariant of $\T$. We will only prove it for the weaker version $\widetilde{\mathit{CT}}(\T)$.  From now on, we restrict ourselves to the  ``tilde''-theory by 
setting all $U_O$ to 0 . A consequence of Theorems \ref{thm:cft_pairing} and \ref{thm:invariance} is:
\begin{theorem}\label{thm:welldef} Suppose that $\Pp=(\P_1^\circ,\dots,\P_p^\circ)$ and $\boldsymbol{\mathcal{Q}}=(\mathcal{Q}_1^\circ,\dots,\mathcal{Q}_q^\circ)$ give  tangles (in the sense of Lemma \ref{lem:decomp}) isotopic to $\T$. Then
for some integers $k(\P)$ and $k(\mathcal{Q})$, the (bi)modules $\widetilde{\mathit{CT}}(\P_1^\circ)\boxtimes\cdots\boxtimes \widetilde{\mathit{CT}}(\P_p^\circ)\boxtimes V^{\otimes k(\mathcal{Q})}$
and $\widetilde{\mathit{CT}}(\mathcal{Q}_1^\circ)\boxtimes\cdots\boxtimes \widetilde{\mathit{CT}}(\mathcal{Q}_{q}^\circ)\otimes V^{\otimes k(\mathcal{P})}$
are equivalent. Here $V=\F_2\oplus \F_2$, where one of the $\F_2$ components has bigrading $(M,A)=(-1,-1)$ and the other one has bigrading $(M,A)=(0,0)$.  
\end{theorem}
The integers $k(\P)$ and $k(\mathcal{Q})$  in the above theorem can be computed explicitly. 
For a shadow $\P$ (or mirror-shadow $\P^*$), define $k(\P)=|S_\X|$ (or $k(\P^*)=|S_\X|$). For a sequence of shadows and mirror-shadows $\Pp=(\P_1^\circ,\cdots,\P_p^\circ)$ with a well-defined wedge product, define
$k(\Pp)=\sum_{j=1}^p k(\P_j)$. 

The $\DA$ bimodule for the trivial tangle is equivalent to the identity bimodule, or more precisely:
\begin{theorem}
If $\Pp=(\E_1^*,\E_2)$ is a sequence of an idempotent mirror-shadow and shadow for a tangle $\T$ consisting of $m$ straight strands, then
$$\widetilde{\mathit{CATA}}(\E_2)\boxtimes \widetilde{\mathit{CDTA}}(\Pp)  \simeq \widetilde{\mathit{CATA}}(\E_2) \otimes V^{\otimes m}.$$
\end{theorem}
\begin{proof}
The proof follows from the results in Sections  \ref{sec:circles}-\ref{sec:hd_pairing}, but we outline it here nevertheless. 
One can represent the sequence $(\E_2, \E_1^*, \E_1)$ by a plumbing of bordered grid diagrams. One can perform Heegaard moves to this plumbing to obtain the bordered grid diagram for $\E_2$. Every index zero/three destabilization results in an extra $V$ factor. Observe that $\widetilde{\mathit{CATA}}(\E_2)$ is just the tilde version of the algebra $\Aa(\E_2)$.
\end{proof}

% !TEX root = ../tanglefloer.tex
%%%%%%%%%%%%%%%%%%%%%%%%%%%%%%%%%%%%%%%%%%%%%%%%%%%%%%%

%%%%%%%%%%%%%%%%%%%%%%%%%%%%%%%%%%%%%%%%%%%%%%%%%%%%%%%
%NOTATIONS BEGINNING
%%%%%%%%%%%%%%%%%%%%%%%%%%%%%%%%%%%%%%%%%%%%%%%%%%%%%%%
\subsection{Sample invariance proofs}\label{subsec:sampleproof}
Although the proof of Theorem \ref{thm:welldef} is proved entirely in Section \ref{sec:hd_modules},  to give evidence that the theory can be  defined combinatorially we give sample proofs for statements from Theorem \ref{thm:welldef}. Most of the arguments rely on the generalisation of the commutation move for grid diagrams. 

\subsubsection{Generalized commutation}

In all the (bordered) Heegaard diagrams we have been working with all regions (connected components of $\Sigma\setminus(\alphas\cup\betas)$) are rectangles, and each annulus between two neighbouring $\alphas$ circles  or $\betas$ circles contains exactly one $X$ and one $O$. In the following this will be our assumption on the Heegaard diagrams, and we will call these diagrams \emph{rectangular}. Note that for rectangular diagrams the connected components of $\Sigma\setminus\alphas$ (or $\Sigma\setminus \betas$) are annuli or punctured spheres with at most two boundary components intersecting $\alphas$ (or $\betas$) and the rest of the boundary components are subsets of $\partial\Sigma$. Thus rectangular diagrams are always constructed as a plumbing of annuli.

\begin{figure}[h]
 \centering
       \includegraphics
       [scale=0.8]
       {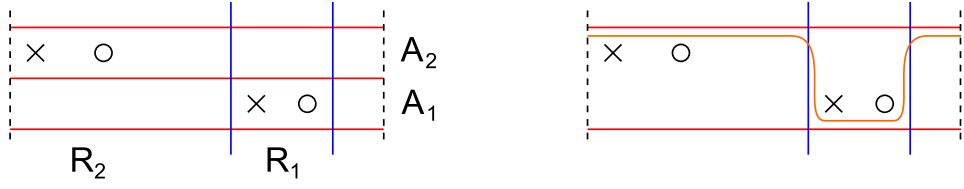} 
       \vskip .2 cm
       \caption[Generalized commutation.]{\textbf{Generalized commutation.} The left and right hand side of each diagram are identified.  
  }\label{fig:gencomm}
\end{figure}

So let $\H=(\Sigma,\alphas=\alphas^c\cup\alphas^a,\betas=\betas^c\cup\betas^a,\X,\O)$ be a rectangular Heegaard diagram such that every annulus contains an $X$. Then in the usual way we can define a chain complex with underlying module $C^-(\H)$ generated over $\k=\F[U_O]_{O\in\O}$ by intersection points $\x\in\gen(\H)$ with one intersection point on each circle $\alphas^c$ and each circle $\betas^c$ and at most one intersection point on each arc $\alphas^a$ and each arc $\betas^a$. The differential is defined by counting empty rectangles: a rectangle from a generator $\x$ to a generator $\y$ is an embedded rectangle $R\subset \Sigma$ with boundary $\partial R \subset \alphas\cup\betas$ such that $\x\cap R$ is the two corners of $R$ where $(T\alphas,T\betas)$ form a positive basis of $T\Sigma$ and  $\y\cap R$ is the two corners of $R$ where $(T\alphas,T\betas)$ form a negative basis of $T\Sigma$ (here the orientation on the tangent vectors comes from the orientation on $\bdy R$). A rectangle $R$ is called empty if $\mathit{int}(R)\cap(\x\cup\y)=\emptyset$ and $R\cap\X=\emptyset$. Denote the set of empty rectangles from $\x$ to $\y$ by $\mathcal{R}_0(\x,\y)$. Then define
\[\partial \x=\sum_{y\in\gen(\H)}\sum_{R\in\mathcal{R}_0(\x,\y)}\prod_{O\in\O}U^{|R\cap O|}\y.\]
This can be extended to whole $C^-(\H)$ and using the usual arguments we conclude:
\begin{lemma}
$(C^-(\H),\partial)$ is a chain complex. \qed
\end{lemma}
Take three consecutive alpha circles $\alpha_1$, $\alpha_2$ and $\alpha_3$, i.e. so that $\alpha_1$ and $\alpha_2$ bound the annulus $A_1$ and $\alpha_2$ and $\alpha_3$ bound the annulus $A_2$. All connected components of $\betas\cap (A_1\cup A_2)$ are intervals. Suppose that two of these intervals corresponding to different $\betas$-curves subdivide $A_1\cup A_2$ into two rectangles $R_1$ and $R_2$ such that $(\X\cup\O)\cap A_1\subset R_1$ and $(\X\cup\O)\cap A_2\subset R_2$. Then we can define a new Heegaard diagram $\mathcal{H}'$  by changing $\alpha_2$ to $\alpha_2^\prime$, where $\alpha_2^\prime$ is the smoothing  of $(\alpha_3\setminus \partial R_1) \cup (\partial R_1\setminus \alpha_3)$ isotoped in the complement of $\X\cup\O$ so that it is disjoint from $\alphas\setminus \{\alpha_2\}$, transverse to all $\beta$-curves and intersects them only once.  See Figure \ref{fig:gencomm}. 
Then 
\begin{lemma}[Generalized commutation]
The chain complexes $(C^-(\H),\partial)$ and $(C^-(\H'),\partial')$ are chain homotopy equivalent.
\end{lemma}
\begin{proof} The proof is literally  the same as  in the closed case  (see Section 3.1.\ of  \cite{most}): the chain maps count pentagons, while the homotopy counts hexagons of the triple Heegaard diagram.
\end{proof}
For sequences of shadows and mirror-shadows, the proof goes the same way: 
\begin{lemma}\label{lem:DDequiv} Let  $\Pp=(\P_1^\circ,\P_2^\circ,\dots \P_p^\circ)$ and $\boldsymbol{\mathcal{Q}}'=(\mathcal{Q}_1^\circ,\mathcal{Q}_2^\circ,\dots \mathcal{Q}_p^\circ)$ be sequences of shadows and mirror-shadows with well-defined wedge products. Assume that the corresponding grid diagrams $G(\boldsymbol{\P})$ and $G(\boldsymbol{\mathcal{Q}})$
are related to each other by  generalized commutation. Then the associated structures 
 $\mathit{CT}^-(\boldsymbol{\P})$ and $\mathit{CT}^-(\boldsymbol{\mathcal{Q}})$ are equivalent. \qed
\end{lemma}
Using Lemma \ref{lem:DDequiv}, we can prove the following:
\begin{proposition}
Let $\Pp=\{\P_1^\circ,\dots,\P_p^\circ\}$ and $\boldsymbol{\mathcal{Q}}=\{\mathcal{Q}_1^\circ,\dots,\mathcal{Q}_p^\circ\}$ be sequences with corresponding tangles (in the sense of Lemma \ref{lem:decomp})  $\T(\Pp)$ and $\T(\boldsymbol{\mathcal{Q}})$, respectively. Suppose that $\T(\Pp)$ and $\T(\boldsymbol{\mathcal{Q}})$ are related to each other by Reidemeister II and Reidemeister III moves. Then the (bi)modules $\mathit{CT}^-(\T(\Pp))$
and $\mathit{CT}^-(\T(\boldsymbol{\mathcal{Q}}))$
are equivalent.
\end{proposition}
\begin{proof}
As it is shown on Figure \ref{fig:reid2}, a Reidemeister II move is simply a general commutation on the associated grid diagram.
\begin{figure}[h]
 \centering
       \includegraphics[scale=0.8]{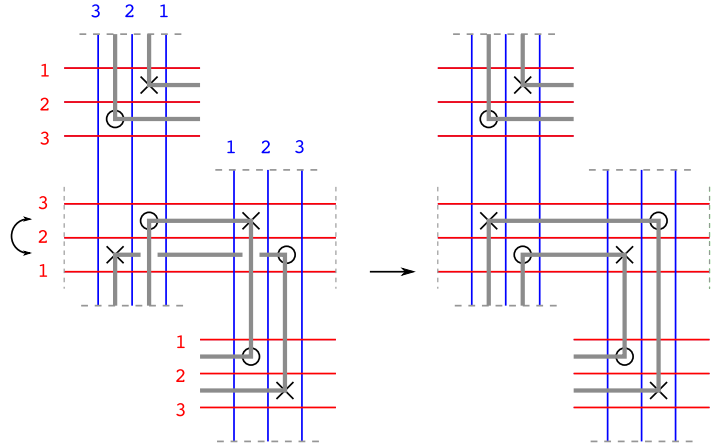} 
       \vskip .2 cm
       \caption[Diagram for simplifying a Reidemeister II move.]{\textbf{Diagram for simplifying a Reidemeister II move.} The first picture corresponds to two canceling crossings, the arrow corresponds to a generalized commutation, and the second picture corresponds to straight strands. This image can have more straight strands that are not affected by the moves.}
       \label{fig:reid2}
\end{figure}
A Reidemeister III move can be achieved with a sequence of commutation moves, see figure \ref{fig:reidem3}. 
\begin{figure}[p]
 \centering
       \includegraphics[scale=0.5]{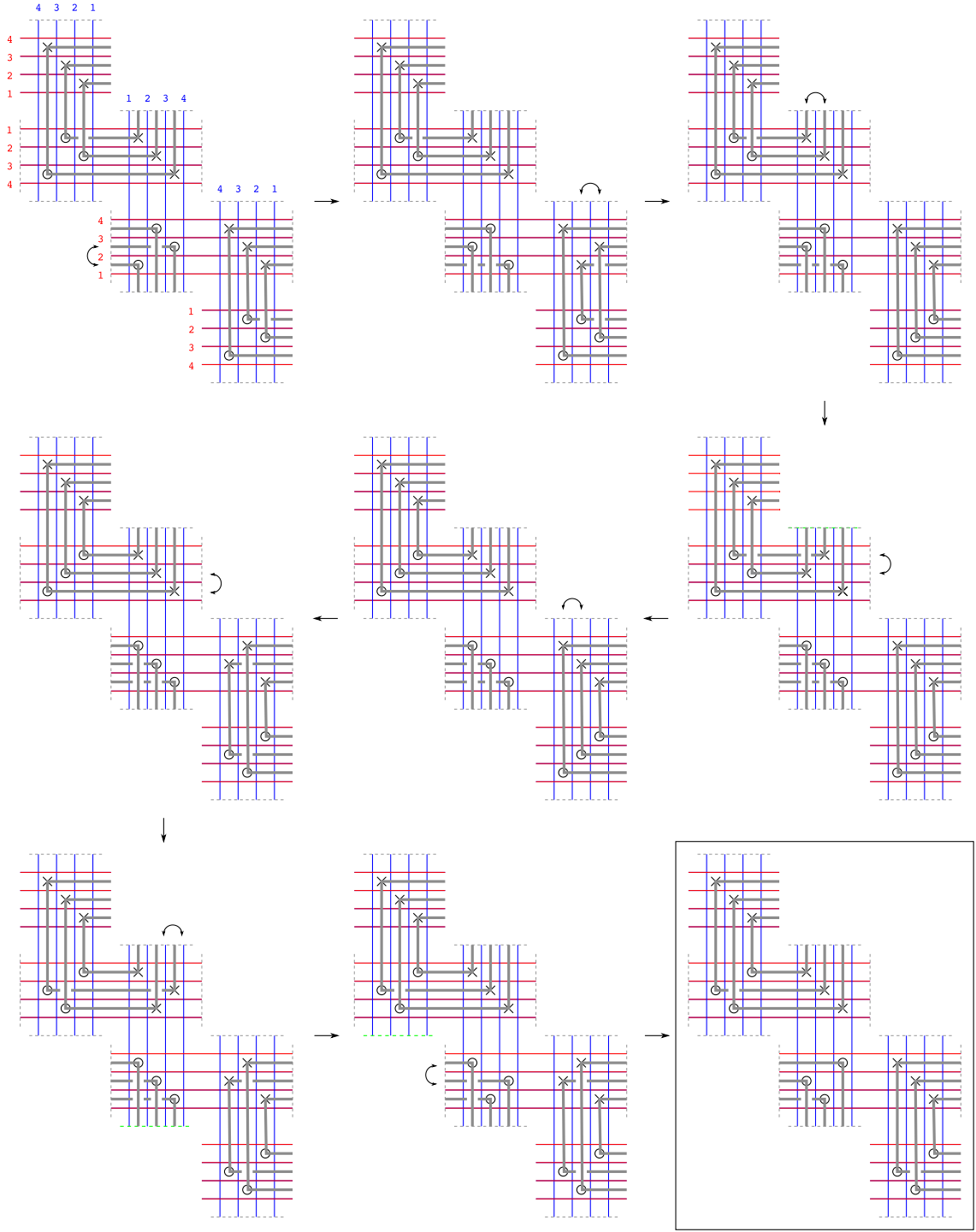} 
       \vskip .2 cm
       \caption[Commutation moves corresponding to a Reidemeister III move.]{\textbf{Commutation moves corresponding to a Reidemeister III move.} Again, this image can have more straight strands that are not affected by the moves.}
       \label{fig:reidem3}
\end{figure}
\end{proof}

%%%%%%%%%%%%%%%%%%%%%%%%%%%%%%%%%%%%%%%%%%%%%%%%%%%%%%%
% section combinatorial (end)
%%%%%%%%%%%%%%%%%%%%%%%%%%%%%%%%%%%%%%%%%%%%%%%%%%%%%%%

%%%%%%%%%%%%%%%%%%%%%%%%%%%%%%%%%%%%%%%%%%%%%%%%%%%%%%%

%%%%%%%%%%%%%%%%%%%%%%%%%%%%%%%%%%%%%%%%%%%%%%%%%%%%%%%
% !TEX root = ../tanglefloer.tex
%%%%%%%%%%%%%%%%%%%%%%%%%%%%%%%%%%%%%%%%%%%%%%%%%%%%%%%

\section{Relation to knot Floer homology} % (fold)
\label{sec:apps}
%%%%%%%%%%%%%%%%%%%%%%%%%%%%%%%%%%%%%%%%%%%%%%%%%%%%%%%

This  section provides the connection between $\mathit{CT}^-$ and $\cfkm$. 

Let $\P_1^\circ, \ldots, \P_n^\circ$ be a sequence of shadows and mirror-shadows as in Lemma \ref{lem:decomp} such that the associated tangle $L= \T^\circ(\P_1^\circ) \circ \ldots \circ\T^\circ(\P_n^\circ)$ is a closed link. After self-gluing the first and last grid in  $G^\circ (\P_1^\circ)\cup \cdots \cup G^\circ (\P_n^\circ)$, we obtain a diagram that is a plumbing of annuli and has one boundary component. Close off the boundary by gluing on a disk with one $X$ and one $O$ in it. The resulting closed Heegaard diagram $\HH$ represents  the link $L\cup U$, where  $U$ is an unknot unlinked from $L$. 

\begin{theorem}\label{recover_hfk}
We have a graded homotopy equivalence
\[\mathit{CT}^-(\P_1^{\circ})\boxtimes\cdots \boxtimes  \mathit{CT}^-(\P_n^{\circ})\simeq g\cfkm(\HH)\]
that maps a homogeneous generator in Maslov grading $m$ and Alexander grading $a$ to a homogeneous generator in Maslov grading $m+\frac{|L|}{2}$ and Alexander grading $a+\frac{|L|}{2}$.
\end{theorem}
Before we prove Theorem \ref{recover_hfk}, we review the basic construction for knot Floer homology, see also  \cite{hfk, jrth, most, oszlink}.

Let $\HH_L = (\Sigma, \balpha, \bbeta, \OO, \XX)$ be a Heegaard diagram for a knot or a link $L$ with $l$ components, where $\OO$ and $\XX$ are sets of $k\geq l$ basepoints. Let $\mathfrak S$ be the set of generators of $\HH_L$. 
The \emph{knot Floer complex} $\cfkm(\HH_L)$ is generated over $\F_2[U_1, \ldots, U_k]$ by $\mathfrak S$, with differential

$$\bdy^-(\xx) = \sum_{\yy\in \mathfrak S}\sum_{\substack{B\in \tilde\pi_2(\xxx, \yyy)\\  \ind B = 1}}\#\mathcal M^B(\xxx,\yyy) \prod_{O_i\in \OO}(U_i^{n_{O_i}(B)})\cdot \yy,$$
where $\tilde\pi_2(\xxx, \yyy)$ is the set of homology classes from $\xx$ to $\yy$ which may cross both $\OO$ and $\XX$. The complex has a differential grading called the \emph{Maslov} grading. As a relative grading, it is defined by
\begin{align*}
M'(\xx) - M'(\yy) &= \ind B - 2n_{\OO}(B)\\
M'(U_i\xx) &= M'(\xx) - 2
\end{align*}
for any $\xxx, \yyy\in \mathfrak S$, and $B\in  \tilde\pi_2(\xxx, \yyy)$.
The complex also comes endowed with an  \emph{Alexander} filtration, defined by
\begin{align*}
A'(\xx) - A'(\yy) &= n_{\XX}(B) - n_{\OO}(B)\\
A'(U_i\xx) &= A'(\xx) - 1, 
\end{align*}
and normalized so that 
\begin{equation}\label{normalized_a}
\#\{\xx\in \mathfrak S| A'(\xx) = a\} = \#\{\xx \in \mathfrak S| A'(\xx) = -a\}\textrm{ mod } 2
\end{equation}

The \emph{associated graded object} $g\cfkm(\HH_L)$ is also generated over $\F_2[U_1, \ldots, U_k]$ by $\mathfrak S$, and its differential is given by 
$$\bdy^-(\xx) = \sum_{\yy\in \mathfrak S}\sum_{\substack{B\in \tilde\pi_2(\xxx, \yyy)\\  \ind B = 1\\n_{\XX}(B) = 0}}\#\mathcal M^B(\xxx,\yyy) \prod_{O_i\in \OO}(U_i^{n_{O_i}(B)})\cdot \yy.$$
The Alexander filtration descends to a grading on $g\cfkm(\HH_L)$. 
The bigraded  homology 
$$\hfkm(L)\coloneqq H_{\ast}(g\cfkm(\HH_L))$$
is an invariant of $L$.

The Maslov grading is normalized so that after setting each $U_i$ to zero we get
$$H_{\ast}(\cfkm(\HH_L)/(U_i=0))\cong H_{\ast+k-1-\frac{l-1}{2}}(T^{k-1}),$$ 
where $\ast$ denotes the grading $M'$, and  we ignore the Alexander filtration on $\cfkm(\HH_L)$.

One can also set each $U_i=0$ to obtain the filtered chain complex over $\F_2$
$$\cfkhat(\HH_L)\coloneqq \cfkm(\HH_L)/(U_i=0).$$
The associated graded object to $\cfkhat(\HH_L)$ is $g\cfkhat(\HH_L)$, with differential 
$$\hat\bdy(\xx) = \sum_{\yy\in \mathfrak S}\sum_{\substack{B\in \tilde\pi_2(\xxx, \yyy)\\  \ind B = 1\\n_{\XX}(B) = 0 = n_{\OO}(B)}}\#\mathcal M^B(\xxx,\yyy)\cdot \yy.$$ 
We denote its homology, which is an invariant of $L$, by $\hfkhat(L)\coloneqq H_{\ast}(g\cfkhat(\HH_L))$.

There is another grading, which we refer to as the \emph{$\XX$-normalized grading}, defined by 
\begin{align*}
N'(\xx) - N'(\yy) &= \ind B - 2n_{\XX}(B)\\
N'(U_i\xx) &= N'(\xx), 
\end{align*}
and normalized so that 
$$H_{\ast}(g\cfkm(L)/(U_i=1))\cong H_{\ast+k-1-\frac{l-1}{2}}(T^{k-1}),$$ 
where $\ast$ denotes the grading $N'$.

It turns out that
\begin{equation}\label{mna}
N' = M' - 2A' - (k-l),
\end{equation}
so instead of using Equation \ref{normalized_a} to normalize the Alexander grading, we can use Equation \ref{mna}.

Next, we put the grading from Section \ref{ssec:grading} in the context of grid diagrams.

Let $\P$ be a shadow,  let $G = G(\P)$ be the corresponding grid, and $G^{\ast}$ be the grid corresponding to $\P^{\ast}$. We define a few special generators below. 

Let $f_{\OO}$ be  the generator of $G$ formed by picking the top-right corner of each $O$, see Figure \ref{fig:gr_module}, and let $f_{\OO}'$ be  the generator formed by picking the bottom-left corner of each $O$. 
Similarly, let $f^{\ast}_{\OO}$ be  the generator of $G^{\ast}$ formed by picking the bottom-left corner of each  $O$, together with the top-right corner of the grid $G^{\ast}$, see Figure \ref{fig:gr_comodule}, and let $f'^{\ast}_{\OO}$ be  the generator formed by picking the top-right corner of each  $O$, together with the bottom-left corner of the grid $G^{\ast}$. 

 \begin{figure}[h]
 \centering
       \includegraphics[scale=.9]{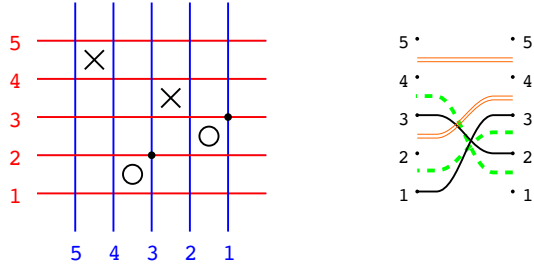} 
       \vskip .2 cm
       \caption[The generator $f_{\OO}$.]{\textbf{The generator $f_{\OO}$.} Left:  the generator $f_{\OO}$ on a grid diagram $G$. Right: the corresponding generator on the shadow for $G$.}\label{fig:gr_module}
\end{figure}
 
\begin{figure}[h]
 \centering
       \includegraphics[scale=.9]{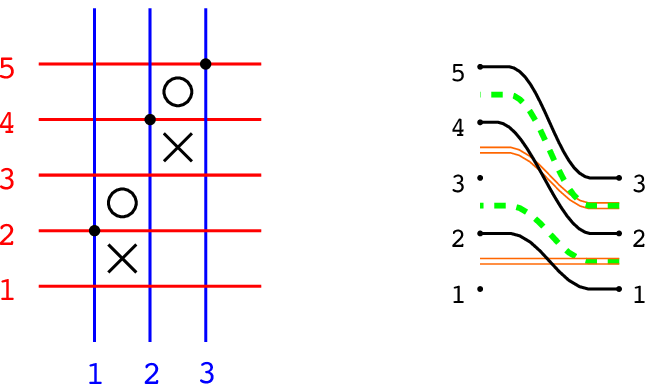} 
       \vskip .2 cm
       \caption[The generator $f_{\OO}^{\ast}$.]{\textbf{The generator $f_{\OO}^{\ast}$.} Left:  the generator $f_{\OO}^{\ast}$ on a grid diagram $G^{\ast}$. Right: the corresponding generator on the mirror-shadow for $G^{\ast}$.}\label{fig:gr_comodule}
\end{figure}

 Let $f_{\XX}$ and $f_{\XX}'$  be  the generators of $G$ formed by picking the top-right (respectively bottom-left) corner of each $X$.  Similarly, let $f^{\ast}_{\XX}$ and   $f'^{\ast}_{\XX}$ be  the generators of  $G^{\ast}$ formed by picking the bottom-left (respectively top-right) corner of each $X$, and the top-right (respectively bottom-left) corner of the grid.

\begin{lemma}\label{abs_m_lemma}
For the generators defined above, we have
\begin{align*}
 M(f_{\OO})  &= M(f_{\OO}') = M(f^{\ast}_{\OO})  = M(f'^{\ast}_{\OO})= -|\OO|\\
  M(f_{\XX})  &= M(f'_{\XX}) = \inv(\xi^{-1}) - \inv(\xi^{-1}, \omega) + \inv(\omega) \\
  M(f^{\ast}_{\XX}) &= M(f^{\ast}_{\XX})=  -\inv(\xi^{-1}) + \inv(\xi^{-1}, \omega) - \inv(\omega) - |\OO| \\
       A(f_{\XX}) &= \frac{M(f_{\XX})}{2} = A(f'_{\XX}) \\
   A(f^{\ast}_{\XX})  &= \frac{M(f^{\ast}_{\XX})}{2} = A(f'^{\ast}_{\XX}) . 
 \end{align*}
\end{lemma}

\begin{proof}
Write out $f_{\OO} = \stphi$. Let $t = |\OO|$, let $g_1, \ldots, g_t$ be the dashed (green) strands in the graphical representation for the shadow $\P$, and let $f_1, \ldots, f_t$ be the strands for $f$, where $f_i$ is the strand that starts immediately below, and ends immediately above $g_i$.

 Recall that $\inv(\phi)$ counts intersections between pairs  in $\{f_1, \ldots,  f_t\}$, $\inv(\omega)$ counts intersections between pairs  in $\{g_1, \ldots,  g_t\}$, and $\inv(\phi,\omega)$ counts the total number of intersections between  a strand in $\{f_1, \ldots,  f_t\}$  and a strand in $\{g_1, \ldots,  g_t\}$.
 
Observe that $\inv(\phi) = \inv(\omega)$, since each $f_i$ is just a perturbation of $g_i$.
Also,   $f_i$ intersects $g_j$ exactly when $i\neq j$ and  $g_i$ intersects $g_j$, or $i=j$ , so $\inv(\phi,\omega) = 2 \inv(\omega) + |\SO|$.
Thus, 
$$M(f_{\OO}) =   \inv(\phi) - \inv(\phi,\omega) + \inv(\omega) = \inv(\omega) -2 \inv(\omega) - |\SO|+ \inv(\omega) = -|\SO|.$$

Similarly, write out $f^{\ast}_{\OO} = (S{},T{},\phi{})$. Again let $t = |\OO|$, let $g_1, \ldots, g_t$ be the dashed (green) strands in the graphical representation for the shadow $\P^{\ast}$, and let $f_1, \ldots, f_{t+1}$ be the strands for $f$, where $f_i$ is the strand that starts and ends immediately below $g_i$, for $1\leq i\leq t$, and $f_{t+1}$ connects the highest point to the left to the highest point to the right. Clearly $\inv(\phi) = \inv(\omega)$ and $ \inv(\phi,\omega) = 2\inv(\omega)$, since this time for a fixed $i$, $f_i$ and $g_i$ do not intersect, so
$$M(f^*_{\OO}) = - \inv(\phi)+\inv(\phi,\omega) - \inv(\omega) -|\SO| = -|\SO|.$$ 

The proof for $f'_{\OO}$ and $f'^{\ast}_{\OO}$ is analogous.

Now write $f_{\XX} = \stphi$. With notation as above, it is clear that each $f_i$  is a perturbation of the corresponding double (orange) strand for $\XX$. Reasoning as above, we see that 
$$M(f_{\XX}) =  \inv(\phi)-\inv(\phi,\omega) + \inv(\omega) = \inv(\xi^{-1}) - \inv(\xi^{-1}, \omega)+ \inv(\omega).$$
Next, 
\begin{align*}
A(f_{\XX}) &= \frac12\left(\inv(\phi, \xi^{-1}) - \inv(\phi,\omega) + \inv(\omega) - \inv(\xi^{-1})-|\TX|\right)\\
&=\frac12\left(2\inv(\xi^{-1}) + |\TX| - \inv(\xi^{-1},\omega) + \inv(\omega) - \inv(\xi^{-1})-|\TX|\right)\\
&= \frac12\left(\inv(\xi^{-1})  -  \inv(\xi^{-1},\omega)+ \inv(\omega) \right) \\
&= \frac{M(f_{\XX})}{2}
\end{align*}

The proof for $f^{\ast}_{\XX}$, $f'_{\XX}$, and $f'^{\ast}_{\XX}$  is analogous. 
\end{proof}

We are now ready to prove Theorem  \ref{recover_hfk}. 

 \begin{proof}[Proof of Theorem \ref{recover_hfk}]
 Each shadow $\P_i^{\circ}$ has a corresponding grid diagram $G_i^{\circ}$. Both for grids and for shadows, we abbreviate the notation for the (bi)modules $\mathit{CTA}^-, \mathit{CDTD}^-$, etc. by $\ctm$. For shadows and the corresponding grids we consider the type $A$ or type $\AA$ structures, and for mirror-shadows and the corresponding grids we consider the type $D$ or type $\DD$ structures.  By Propositions \ref{prop:gridtypea1} and \ref{prop:gridtyped1}, the modules $\mathit{CT}^-(\P_i^{\circ})$ and $\ctm(G_i^{\circ})$ are isomorphic. The type $A$ or $\AA$ structures $\ctm(G_i^{\circ})$ are defined by counting empty rectangles and certain sets of half-rectangles that do not intersect $\XX$, whereas the type $D$ or $\DD$ structures are defined by counting empty rectangles and (individual) half-rectangles that do not intersect $\XX$.
 So the differential on $\ctm(G_1^{\circ})\boxtimes\cdots \boxtimes  \ctm(G_n^{\circ})$ counts empty rectangles in the diagram $G_1^{\circ}\cup\ldots\cup G_n^{\circ}$ that do not intersect $\XX$, hence  $\ctm(G_1^{\circ})\boxtimes\cdots \boxtimes  \ctm(G_n^{\circ})$ is isomorphic to the complex $g\cfkm$ associated to the closure of the nice diagram $G_1^{\circ}\cup\ldots\cup G_n^{\circ}$, with and $X$ and an $O$ added in the new region, which represents $L\cup U$.  It remains to check that this last isomorphism preserves the Maslov and Alexander gradings.

Let $\HH$ be the Heegaard diagram obtained by closing up the plumbing of annuli $G_1^{\circ}\cup \ldots\cup G_n^{\circ}$. We argue that the absolute Maslov grading on $\HH$ (obtained by adding the gradings on each $G_i^{\circ}$) is correct. Let $k_i$ be the number of $O$s in each grid $G_i^{\circ}$, and let $k = \sum_{i=1}^n k_i$.    Let $\xxx_{\OO} = f_{\OO_1}^{\ast}\boxtimes f_{\OO_2}\boxtimes f'^{\ast}_{\OO_3}\boxtimes f'_{\OO_4} \boxtimes f_{\OO_5}^{\ast}\boxtimes \cdots\boxtimes f^{\circ}_{\OO_n}$ (the decoration $\circ$ depends on $n \textrm{ mod } 4$, as specified according to the first four factors).  By Lemma \ref{abs_m_lemma}, 
$M(\xxx_{\OO}) = M(f_{\OO_1}^{\ast})+M(f_{\OO_2})+\cdots + M(f^{\circ}_{\OO_n}) = -|\OO_1|-\cdots - |\OO_n| = -k$.

Form a set of $\gamma$ circles $\bgamma$ by performing handleslides (which are allowed to cross $\XX$ but not $\OO$) of $k_i$ of the $\beta$ circles and a perturbation of one $\beta$ circle for each $G_i$, as in Figure  \ref{fig:gr_triangle}.
We look at the holomorphic triangle map (see \cite{osz5, osz14}) associated to $(\Sigma, \balpha, \bbeta, \bgamma, \OO)$. Observe that $(\Sigma, \bbeta, \bgamma, \OO)$ is a diagram for $(S^1\times S^2)^{\#k}$, and let $\Theta$ be the top-dimensional generator (on the diagram this is the set of intersection points at which the small bigons start).
Let $\yyy$ be the generator of $(\Sigma, \balpha, \bgamma, \OO)$ nearest to $\xxx_{\OO}$. There is a holomorphic triangle that maps $\xxx_{\OO}\otimes \Theta$ to $\yyy$, see  Figure \ref{fig:gr_triangle}.

\begin{figure}[h]
 \centering
       \includegraphics[scale=1.25]{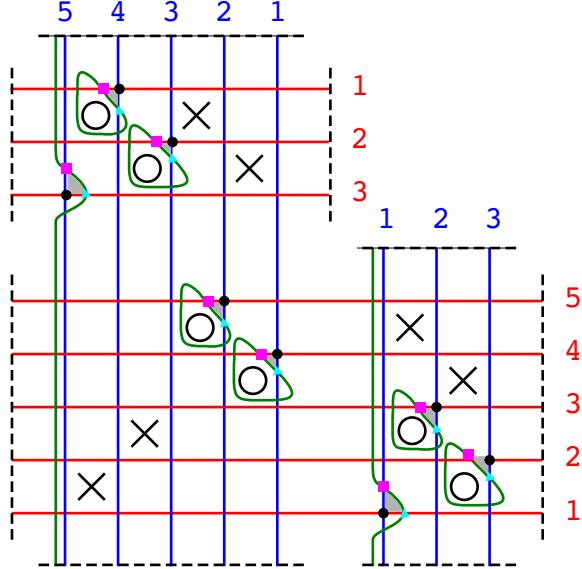} 
       \vskip .2 cm
       \caption[The union $\HH$ of three grid diagrams.]{\textbf{The union $\HH$ of three grid diagrams:} $G_1^*$ (top), $G_2$ (bottom left), and $G_3^*$ (bottom right). The black dots form the generator $\xxx_{\OO}$, the purple squares form $\yyy$, and the cyan triangles form $\Theta$.}\label{fig:gr_triangle}
\end{figure}

Observe that $(\Sigma, \balpha, \bgamma, \OO)$
is a diagram for $S^3$ with $2^k$ generators, for which the differential vanishes (each small bigon ending at an intersection point in $\yyy$ is cancelled by the  corresponding horizontal annulus with the small region region containing an $O$ removed). By looking at the small bigons, one sees that $\yyy$ is the bottom-most generator of $(\Sigma, \balpha, \bgamma, \OO)$, so its Maslov grading is $-k+\frac{l}{2}$, where $l=|L| = |L\cup U|-1$.  Since $\xxx_{\OO}$, $\Theta$, and $\yyy$ are connected by a Maslov index zero triangle, the Maslov grading of $\xxx_{\OO}$ should be $M'(\xxx_{\OO}) =-k+\frac{l}{2}$ too.

Next, we argue that the Alexander grading on $\HH$ is correct. For that purpose, let $\xxx_{\XX} = f_{\XX_1}^{\ast}\boxtimes f_{\XX_2}\boxtimes f'^{\ast}_{\XX_3}\boxtimes f'_{\XX_4}\boxtimes\cdots\boxtimes f^{\circ}_{\XX}$. A priori, 
$A'(\xxx_{\XX}) = A(\xxx_{\XX})+s= A(f^{\ast}_{\XX_1})+ \cdots+ A(f^{\circ}_{\XX_n})+s$, where  $s$ is a constant. We show the shift $s$ is zero.
By Lemma \ref{abs_m_lemma}, 
$$A(\xxx_{\XX})  = A(f^{\ast}_{\XX_1})+ A(f_{\XX_2})+\cdots+A(f^{\circ}_{\XX_n}) = \frac{M(f^{\ast}_{\XX_1})}{2}+ \frac{M(f_{\XX_2})}{2}+\cdots+\frac{M(f^{\circ}_{\XX_n})}{2} = \frac{M(\xxx_{\XX})}{2},$$
and we just showed that $M \equiv M'- \frac{l}{2}$, so $A(\xxx_{\XX}) = (M'(\xxx_{\XX})- \frac{l}{2})/2$.
On the other hand, using the holomorphic triangles argument above, we see that the $\XX$-normalized  grading of $\xxx_{\XX}$ is $N'(\xxx_{\XX}) = -k+ \frac{l}{2}$. The closed diagram has one additional $X$ and one $O$ in the outside region that we closed off, for a total of $k+1$ basepoints of each type,   so by Equation \ref{mna}, 
$$A'(\xxx_{\XX}) = \frac 1 2 \left(M'(\xxx_{\XX}) - N'(\xxx_{\XX})- ((k+1)-(l+1))\right) =\frac{M'(\xxx_{\XX})+ \frac{l}{2}}{2}, $$
so
\begin{equation*}A(\xxx_{\XX}) = A'(\xxx_{\XX}) - \frac{l}{2}.\qedhere
\end{equation*}

\end{proof}

%%%%%%%%%%%%%%%%%%%%%%%%%%%%%%%%%%%%%%%%%%%%%%%%%%%%%%%
% section apps (end)
%%%%%%%%%%%%%%%%%%%%%%%%%%%%%%%%%%%%%%%%%%%%%%%%%%%%%%%

%%%%%%%%%%%%%%%%%%%%%%%%%%%%%%%%%%%%%%%%%%%%%%%%%%%%%%%

%%%%%%%%%%%%%%%%%%%%%%%%%%%%%%%%%%%%%%%%%%%%%%%%%%%%%%%
% !TEX root = ../tanglefloer.tex
%%%%%%%%%%%%%%%%%%%%%%%%%%%%%%%%%%%%%%%%%%%%%%%%%%%%%%%
\section{Matched circles and their algebras} % (fold)
\label{sec:circles}

Just as closed $3$-manifolds and knots or links in closed $3$-manifolds can be represented by Heegaard diagrams, and bordered $3$-manifolds can be represented by bordered Heegaard diagrams, tangles in $3$-manifolds with boundary can be represented by suitable Heegaard diagrams, which we will call bordered Heegaard diagrams for tangles. 

We define two types of (multipointed) bordered Heegaard diagrams for tangles in $3$-manifolds with  one boundary component. The reason we need two slightly different diagrams is so the result after gluing is a valid closed Heegaard diagram for a link, with the same number of $\alpha$-curves as $\beta$-curves, and with the correct number of basepoints (this should become apparent once the reader goes through the relevant definitions and examples). We also define Heegaard  diagrams for tangles in $3$-manifolds with two  boundary components. We restrict our work to the case where all boundary components are spheres.

\subsection{Matched circles}

An $n$-marked sphere  $\mathcal S = (S^2, t_1, \ldots, t_n)$  has a compatible handle decomposition as follows:
\begin{itemize}
\item[-] Start with $n+2$ two-dimensional $0$-handles $h_0^0, \ldots, h_{n+1}^0$, where the core of $h_i^0$ is $t_i$ for $1\leq i\leq n$.
\item[-] Attach $1$-handles $h_1^1, \ldots, h_{n+1}^1$, so that $h_i^1$ is attached to $h_{i-1}^0$ and $h_i^0$.
\item[-] Attach a $2$-handle to the resulting boundary to obtain $S^2$. 
\end{itemize}

As a first step towards building  Heegaard diagrams for tangles, we represent marked  spheres by matched circles.  First we define matched circles even more generally.
\begin{definition}
A \emph{marked matched circle $\zz$} is a sextuple $(Z, {\bf a}, \mu, \XX, \OO, \zzz)$ of: 
\begin{itemize}
\item[-] an oriented circle $Z$
\item[-] $2n+2$ points ${\bf a} = \{a_1, \ldots, a_{2n+2}\}$ on $Z$ labeled with order induced by the orientation on $Z$
\item[-] a matching $\mu:{\bf a}\to [n+1]$  (where $[n+1]\coloneqq \{1, \ldots, n+1\}$) so that surgery on $Z$ along the matched pairs in ${\bf a}$ yields $n+2$ circles.
\item[-]  Two sets of points, $\XX = \{X_1, \ldots, X_k\}$ and $\OO = \{O_1, \ldots, O_l\}$, and a pair of points $\zzz=\{z^-, z^+\}$ in $Z\setminus {\bf a}$, so that there is exactly one point in each circle obtained after surgery on the matched pairs in ${\bf a}$, and so that one of the points in $\zzz$ is on the interval $(a_{2n+2}, a_1)$.
\end{itemize}
\end{definition}
See, for example, Figure \ref{fig:mmc}.

\begin{figure}[h]
 \centering
  \labellist
       \pinlabel $\textcolor{red}{a_1}$ at 90 650
       \pinlabel $\textcolor{red}{a_{10}}$ at 127 596
       \pinlabel $\textcolor{red}{a_2}$ at 88 705
       \pinlabel $\rotatebox{55}{\textcolor{red}{\ldots}}$ at 262 630
       \pinlabel $X_1$ at 220 765
       \pinlabel $X_2$ at 264 725
       \pinlabel $z^-$ at 100 620       
       \pinlabel $O_1$ at 280 675
       \pinlabel $O_2$ at 85 675   
       \pinlabel $z^+$ at 220 590 
        \pinlabel $\rotatebox{54}{$\to$}$ at 114 721
             \endlabellist
       \includegraphics[scale=.55]{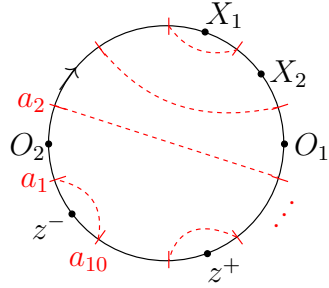} 
       \vskip .2 cm
       \caption[A marked matched circle.]{\textbf{A marked matched circle.} Here $n = 4$. The matching on ${\bf a}$ is illustrated schematically with dotted lines.}\label{fig:mmc}
\end{figure}

Given a marked matched circle $\zz = (Z, {\bf a}, \mu, \XX, \OO, \zzz)$, the marked matched circle $\zz^{\ast}$ is given by $(Z', {\bf a}', \mu', \XX', \OO', \zzz')$, where there is an orientation-reversing homeomorphism $f:Z\to Z'$ such that 
\begin{itemize}
\item[-] $f({\bf a}) =  {\bf a}'$ and $\mu = \mu'\circ f$
\item[-] $f(z^+) = (z')^-$ and $f(z^-) = (z')^+$
\item[-] $f(\XX) = \OO'$ and $f(\OO) = \XX'$
\end{itemize}
 In other words, $\zz^{\ast}$ is obtained from $\zz$ by taking the mirror, swapping $\XX$ and $\OO$ and swapping $z^+$ and $z^-$.  We will soon study Heegaard diagrams whose boundaries are marked matched circles, and gluing two diagrams along boundary components $\zz_1$ and $\zz_2$ will be allowed exactly when $\zz_1 = \zz_2^{\ast}$.

A marked sphere $\mathcal S = (S^2, t_1, \ldots, t_n)$ is represented by the following marked matched circle. 

\begin{definition}\label{s_def}
The \emph{marked matched circle $\zz(\mathcal S)$ associated to $\mathcal S$} is given by the sextuple $(Z, {\bf a}, \mu, \XX, \OO, \zzz)$ with  ${\bf a} =\{a_1, \ldots, a_{2n+2}\}$ and matching $\mu(a_i) = i =  \mu(a_{2n+3 - i})$ for $1\leq i\leq 2n+1$. The set $\XX$ consists of one point in each interval $(a_i, a_{i+1})$ on the circle $Z$, whenever $t_i$ has positive orientation, and the set $\OO$ consists of one point in each interval $(a_i, a_{i+1})$ on the circle $Z$, whenever $t_i$ has negative orientation, for $1\leq i \leq n$. The point $z^-$ is in the interval between $a_{2n+2}$ and $a_1$, and $z^+$ is in the interval $(a_{n+1}, a_{n+2})$. 
\end{definition}

See, for example, Figure \ref{fig:cmmc}.

\begin{figure}[h]
 \centering
  \labellist
       \pinlabel $\textcolor{red}{a_1}$ at 120 -10
       \pinlabel $\textcolor{red}{a_2}$ at 165 30
       \pinlabel $\textcolor{red}{a_3}$ at 185 85
       \pinlabel $\textcolor{red}{a_4}$ at 170 135
       \pinlabel $\textcolor{red}{a_5}$ at 120 172
       \pinlabel $\textcolor{red}{a_{10}}$ at 50 -10
       \pinlabel $\rotatebox{55}{\textcolor{red}{\ldots}}$ at 10 140
       \pinlabel $O_1$ at 150 10
       \pinlabel $O_2$ at 177 55
       \pinlabel $z^-$ at 89 -10     
       \pinlabel $X_1$ at 180 110
       \pinlabel $X_2$ at 150 155   
       \pinlabel $z^+$ at 89 180 
        \pinlabel $\rotatebox{246}{$\to$}$ at 9 115
        \pinlabel $\textcolor{red}{a_1}$ at 445 -10
        \pinlabel $\textcolor{red}{a_2}$ at 490 30
       \pinlabel $\textcolor{red}{a_3}$ at 510 85
       \pinlabel $\textcolor{red}{a_4}$ at 495 135
       \pinlabel $\textcolor{red}{a_5}$ at 445 172
       \pinlabel $\textcolor{red}{a_{10}}$ at 375 -10
       \pinlabel $\rotatebox{55}{\textcolor{red}{\ldots}}$ at 335 140
       \pinlabel $X_1$ at 475 10
       \pinlabel $X_2$ at 502 55
       \pinlabel $z^+$ at 414 -10     
       \pinlabel $O_1$ at 505 110
       \pinlabel $O_2$ at 475 155   
       \pinlabel $z^-$ at 414 180 
        \pinlabel $\rotatebox{68}{$\to$}$ at 335 113
             \endlabellist
       \includegraphics[scale=.55]{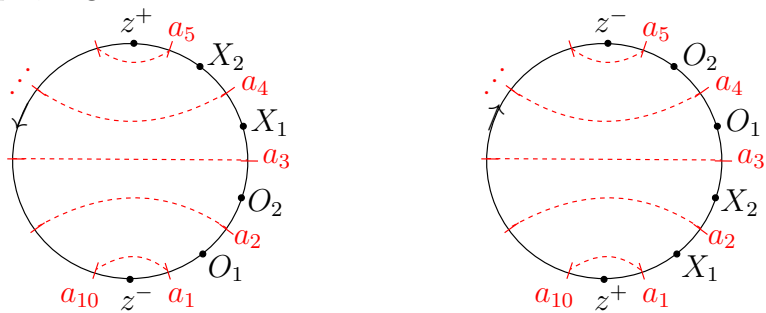} 
       \vskip .2 cm
       \caption[Examples of marked matched circles.]{\textbf{Examples of marked matched circles.} Left: the marked matched circle $\zz(\mathcal S)$ associated to $\mathcal S = (S^2, -, -, +, +)$. Right: the marked matched circle $\zz(\mathcal S)^{\ast}$.}\label{fig:cmmc}
\end{figure}

We can recover the sphere $\mathcal S$ from $\zz(\mathcal S)$  in the following way. We take a disk with boundary $Z$, attach $2$-dimensional $1$-handles along the matched pairs in ${\bf a}$, and fill the resulting $2n+2$ boundary components with $2$-handles. We take $\{t_1, \ldots, t_{n}\}$ to be the cores of the $2$-handles that do not intersect $(a_{2n+2},a_1)$  and $(a_{n+1}, a_{n+2})$, and we  orient $t_i$ positively if the attaching circle for the corresponding $2$-handle contains an $\XX$ marking, and negatively if the attaching circle contains an $\OO$ marking. This is the dual handle decomposition to the one described at the beginning of this section.

\subsection{The algebra associated to a marked matched circle}\label{sec:reeb_alg}

Given a marked matched circle, we define an algebra similar to the algebras from  \cite{bfh2} and \cite{bs}.  For marked matched circles associated to marked spheres, these algebras are precisely the ones from Section \ref{ssec:typeA-strand}. The reason we give another description is that the  interpretation in this section fits better with the geometric setup in the forthcoming sections.
Below, we use the same  notation as \cite[Chapter 3]{bfh2} for our analogous structures, and caution the reader to remember that our matched circles are different from the ones in \cite{bfh2}.

\begin{definition}\label{alg}
The \emph{strands algebra} $\cala(n,k, t)$ is a free $\F_2$-module generated by partial permutations $a = (S, T, \phi)$, where $S$ and $T$ are $k$-element subsets of the set $[2n+2]\coloneqq \{1, \ldots, 2n+2\}$ and $\phi:S\to T$ is a non-decreasing bijection such that  $\phi(i)\leq t$ if and only if $i\leq t$.
Let $\Inv(\phi)$ be the set of inversions of $\phi$,  i.e. the number of pairs $i,j\in S$ with $i<j$ and $\phi(j)<\phi(i)$, and $\inv(\phi) =\# \Inv(\phi)$. 
Multiplication on $\cala(n,k, t)$ is given by

\begin{displaymath}
(S, T, \phi)\cdot(U, V, \psi) = \left\{ \begin{array}{ll}
(S, V, \psi\circ\phi) & \textrm{if $T=U$, $\inv(\phi)+ \inv(\psi) = \inv(\psi\circ\phi)$}\\
0 & \textrm{otherwise.}
\end{array} \right.
\end{displaymath}

For an inversion $c = (i,j)$ of $\phi$, define $\phi_c$ by $\phi_c(i) = \phi(j)$, $\phi_c(j) = \phi(i)$, and $\phi_c(l) = \phi(l)$ for $l\neq i,j$. The differential on $\cala(n,k, t)$ is given by 
$$\bdy(S, T, \phi) = \sum_{\substack{c\in \Inv(\phi)\\ \inv(\phi_{c}) = \inv\phi - 1}} (S, T, \phi_c).$$
\end{definition}

Compare with  \cite[Section 3.1.1]{bfh2}.
We can represent a generator $(S, T, \phi)$ by a strands diagram of horizontal and upward-veering strands. Compare with  \cite[Section 3.1.2]{bfh2}.
 In this notation, the product becomes concatenation, where double crossings are set to zero. The differential corresponds to resolving crossings, subject to the same double crossing rule.

The ring of idempotents $\mathcal I(n,k, t)\subset \cala(n,k, t)$ is generated by all elements of the form $I(S) \coloneqq (S, S, \textrm{id}_S)$ where $S$ is a $k$-element subset of $[2n+2]$.

Fix a marked matched circle $\zz=(Z, {\bf a}, \mu, \XX, \OO, \zzz)$  with $|{\bf a}| = 2n+2$. Recall that one of the points in $\zzz$ is on the interval $(a_{2n+2}, a_1)$, and let $t$ be the number for which the other point in $\zzz$ is on the interval $(a_t, a_{t+1})$.

If we forget the matching on the circle for a moment, we can view $\cala(n, t) = \bigoplus_{i}\cala(n, i, t)$ as the algebra generated by certain sets of Reeb chords in $(Z\setminus \zzz, {\bf a})$: We can view a set $\brho$ of Reeb chords, no two of which share initial or final endpoints, as a strands diagram of upward-veering strands. For such a set $\brho$, we define the \emph{strands algebra element associated to $\brho$} to be the sum of all ways of consistently adding horizontal strands to the diagram for $\brho$, and we denote this element by $a_0(\brho)\in \cala(n,t)$. The basis over $\F_2$ from Definition \ref{alg} is in this terminology the non-zero elements of the form $I(S)a_0(\brho)$, where $S\subset \bf a$.

For a subset ${\bf s}$ of $[n+1]$, a \emph{section} of ${\bf s}$ is a set  $S\subset \mu^{-1}({\bf s})$, such that $\mu$ maps $S$ bijectively to ${\bf s}$. To each ${\bf s}\subset [n+1]$ we associate an idempotent in $\cala (n,t)$ given by
$$I({\bf s}) = \sum_{S \textrm{ is a section of } {\bf s}} I(S).$$
Let $\mathcal I(\zz)$ be the subalgebra generated by all $I({\bf s})$, and let ${\bf I} = \sum_{\bf s}I({\bf s})$.
\begin{definition}
The \emph{algebra $\cala(\zz)$} is the subalgebra of $\cala(n,t)$ generated (as an algebra) by $\mathcal I(\zz)$ and by all $a({\brho}) \coloneqq{\bf I}a_0(\brho)\bf{ I}$. We refer to $a(\brho)$ as the \emph{algebra element associated to $\brho$}.
\end{definition}

Note that this definition, which is what we use for the ``tilde" version of our invariants, does not take into account the $\XX$ and $\OO$ labels on $\zz$. 

The non-zero elements $I({\bf s})a(\brho)$ form a basis for $\Aa(\zz)$ over $\F_2$. Note that for a non-zero generator $I(\sss)a(\brho)$, there is a unique primitive idempotent $I(\ttt)$ such that $I(\sss)a(\brho) = I(\sss)a(\brho)I(\ttt)$. We can represent a generator $I({\bf s})a(\brho)$ by a strands diagram by adding dashed horizontal strands to the strands diagram for $\brho$, one for each horizontal strand that appears in the expansion of $I({\bf s})a(\brho)$ as a sum of elements of $\Aa(n,t)$.

As a special case,  let $\zz(\mathcal S)=(Z, {\bf a}, \mu, \XX, \OO, \zzz)$ be a marked matched circle for a marked sphere $\mathcal S$, with $|{\bf a}| = 2n+2$. Recall the definition of a shadow (Definition \ref{def:shadow}), and let $\E$ be the idempotent shadow corresponding to the interval of $\zz(\mathcal S)$ containing $a_1, \ldots, a_{n+1}$, i.e.  $(n+1, n+1, \id_{\SX},\id_{\SO})$ where 
$\SO = \{s+\frac{1}{2}| \textrm{ there is an $X$ between $a_s$ and $a_{s+1}$}\}$
and 
  $\SX = \{1\frac 1 2, \ldots, n\frac 1 2\}\setminus \SO$.
   Recall the definition of the algebra $\Aa(\E)$ from Section  \ref{ssec:typeA-strand}. Let $\widehat{\Aa}(\E)\coloneqq \Aa(\E)/(U_i=0)$ be the algebra obtained from $\Aa(\E)$  after setting all $U_i$ to zero. 

\begin{proposition}
For $\E$ and $\zz(\mathcal S)$ as above, the algebras $\widehat\Aa(\E)$ and $\Aa(\zz(\mathcal S))$ are isomorphic.
\end{proposition}

\begin{proof}
As long as we do not need to keep track of the bigrading, we can think of  $\widehat\Aa(\E)$ simply  as the algebra $\Aa(\widehat \E)$ for the shadow $\widehat\E=(n+1, n+1, \id_{\SX},\id_{\SO})$  where $\SX = \{1\frac 1 2, \ldots, n\frac 1 2\}$ and $\SO = \emptyset$.

We first outline the correspondence of generators. Suppose $\stphi$ is a generator for $\widehat\Aa(\E)$. The corresponding element $I({\bf s})a(\brho)\in \Aa(\zz(\mathcal S))$ has starting idempotent ${\bf s} = S$ and the following set of Reeb chords $\brho$: the Reeb chord from $i$ to $\phi(i)$ if $\phi(i)>i$, and the Reeb chord from $2n+3-i$ to $2n+3 - \phi(i)$  if $\phi(i)<i$. 

\begin{figure}[h]
 \centering
       \includegraphics[scale=.9]{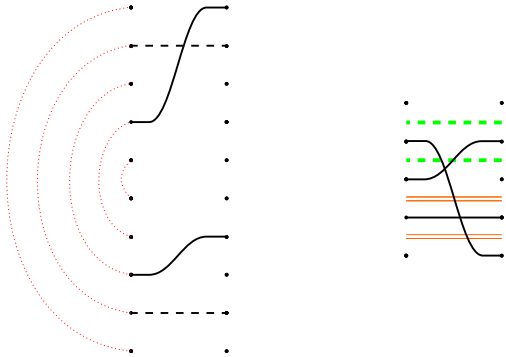} 
       \vskip .2 cm
       \caption[Example of corresponding generators of $\az$ and $\Aa(\E)$.]{\textbf{Example of corresponding generators of $\az$ and $\Aa(\E)$.} Left: a generator of $\az$, where $\zz$ is the circle in Figure \ref{fig:cmmc}. Right: the corresponding generator of $\Aa(\E)$, for the idempotent shadow $\E$ associated to $\zz$.}\label{fig:alg_reeb}
\end{figure}

Note that since there is a double (orange) line at every half-integer height in the diagram of $\widehat\E$,  the concatenation of  two strand diagrams is automatically zero whenever an upward-veering and a downward-veering strand are concatenated. 
Thus,  the concatenation of two strand diagrams in $\widehat\Aa(\E)$ is nonzero exactly when it is nonzero for the corresponding generators in $\Aa(\zz(\mathcal S))$.

The differential of $\widehat\Aa(\E)$ is obtained by summing over all the ways of resolving a crossing, where resulting double crossings are set to zero. Again having a double line at every half-integer height means that resolving crossings between an upward-veering strand and a downward-veering strand is no longer allowed. The allowed resolutions are only those of crossings between two upward-veering strands, two downward-veering strands, an upward-veering and a horizontal strand, or a downward-veering and a horizontal strand.  The first two kinds correspond to resolving a crossing between two Reeb chords in the lower half or upper half of a strand diagram, respectively, and the other two kinds correspond to resolving a crossing between a Reeb chord in the lower half, respectively upper half, of a strand diagram and  a horizontal strand in a section of $\sss$.
\end{proof}

\section{Heegaard diagrams} 
\label{sec:hdiagrams}

We represent tangles by a type of Heegaard diagrams, which we call \emph{multipointed bordered Heegaard diagrams for tangles}, or just tangle Heegaard diagrams. In a sense, our work in this section is a variation of the bordered Heegaard diagrams from \cite{bfh2} and \cite{bimod}, and many of the statements we make and their proofs are analogous to the ones in \cite{bfh2} and \cite{bimod}. We have tried to provide detailed references, and we also encourage the reader to compare our subsections with the corresponding ones in \cite[Chapter 4]{bfh2} and \cite[Chapter 5]{bimod}.

\subsection{$3$-manifolds with one boundary component} \label{ssec:hd1}

\begin{definition}
A \emph{type $1$ multipointed bordered Heegaard diagram for a tangle}, or simply a type $1$ tangle Heegaard diagram, is a sextuple $\HH = (\Sigma, \balpha, \bbeta, \XX, \OO, \zzz)$ where
\begin{itemize}
\item[-] $\Sigma$ is a compact surface of genus $g$ with one boundary component
\item[-] $\balpha = \{\alpha_1^a,\ldots, \alpha_{2n+1}^a, \alpha_1^c, \ldots, \alpha_{t}^c\}$ is a set of pairwise disjoint, embedded curves: $2n+1$  arcs, each  with boundary on $\bdy \Sigma$, and $t$ closed curves in the interior of $\Sigma$
\item[-] $\bbeta$ is a set of $t+n$ pairwise disjoint curves embedded in the interior of $\Sigma$
\item[-] $\XX$ and $\OO$ are two $(t+2n-g)$-tuples of points in  $\Sigma\setminus(\balpha\cup\bbeta)$
\item[-] $\zzz = \{z^-, z^+\}$ is a set of two oppositely oriented points on $\bdy\Sigma \setminus \balpha$ 

\end{itemize}
subject to the conditions
\begin{itemize}
\item[-] $\bbeta$ span a $g$-dimensional subspace of $H_1(\Sigma; \Z)$
\item[-] $\{\alpha_1^c, \ldots, \alpha_t^c\}$ span a $g$-dimensional subspace of  $H_1(\Sigma; \Z)$, and along with the arcs, $\balpha$ span a $g+1$-dimensional subspace of $H_1(\Sigma, \bdy \Sigma; \Z)$
\item[-] $\{\alpha_1^a,\ldots, \alpha_{2n+1}^a\}$ induce a concentric matching on $\bdy\Sigma$. Specifically, they are labeled so that we can order the points on $\bdy\balpha$ according to the orientation of $\bdy \Sigma$ as  $a_1, \ldots, a_{4n+2}$ so that $\bdy\alpha^a_i = \{a_i, a_{4n+3-i}\}$.
\item[-] $z^-$ lies in the interior of the segment with boundary $a_{4n+2}$ and $a_1$ of $\bdy\Sigma\setminus \balpha$, and $z^+$ lies on the segment with boundary  $a_{2n+1}$ and $a_{2n+2}$.
\item[-] Each of the $t-g$ components of $\Sigma\setminus \balpha$ that do not meet $\bdy\Sigma$
 contains one $X\in\XX$ and one $O\in\OO$, and each of the $2n$ components of $\Sigma\setminus \balpha$ that contain two segments of $\bdy\Sigma\setminus \balpha$ contains either an $X$ in the interior and an $O$ on the segment of $\bdy\Sigma\setminus \balpha$ with the lower indexed endpoints, or an $O$ in the interior and an $X$ on the segment of $\bdy\Sigma\setminus \balpha$ with the lower indexed endpoints. 
\item[-] Each of the $t+n-g$ components of $\Sigma\setminus \bbeta$ that do not meet $\bdy\Sigma$ contains exactly one $X$ and one $O$. The unique component of $\Sigma\setminus\bbeta$ that meets  $\bdy\Sigma$ contains $n$ $X$s and $n$ $O$s on $\bdy \Sigma$.
\end{itemize}
\end{definition}

Figure \ref{fig:hd1} is an example of a type $1$  Heegaard diagram for a tangle.

\begin{figure}[h]
 \centering
  \labellist
       \pinlabel $\textcolor{red}{a_1}$ at 266 48
       \pinlabel $\textcolor{red}{a_2}$ at 262 74
       \pinlabel $\textcolor{red}{a_3}$ at 260 100
       \pinlabel $\textcolor{red}{a_4}$ at 330 130
       \pinlabel $\textcolor{red}{a_5}$ at 336 106
       \pinlabel $\textcolor{red}{a_6}$ at 338 80  
       \pinlabel ${z^-}$ at 310 6
       \pinlabel ${z^+}$ at 310 184
     \endlabellist
       \includegraphics[scale=.6]{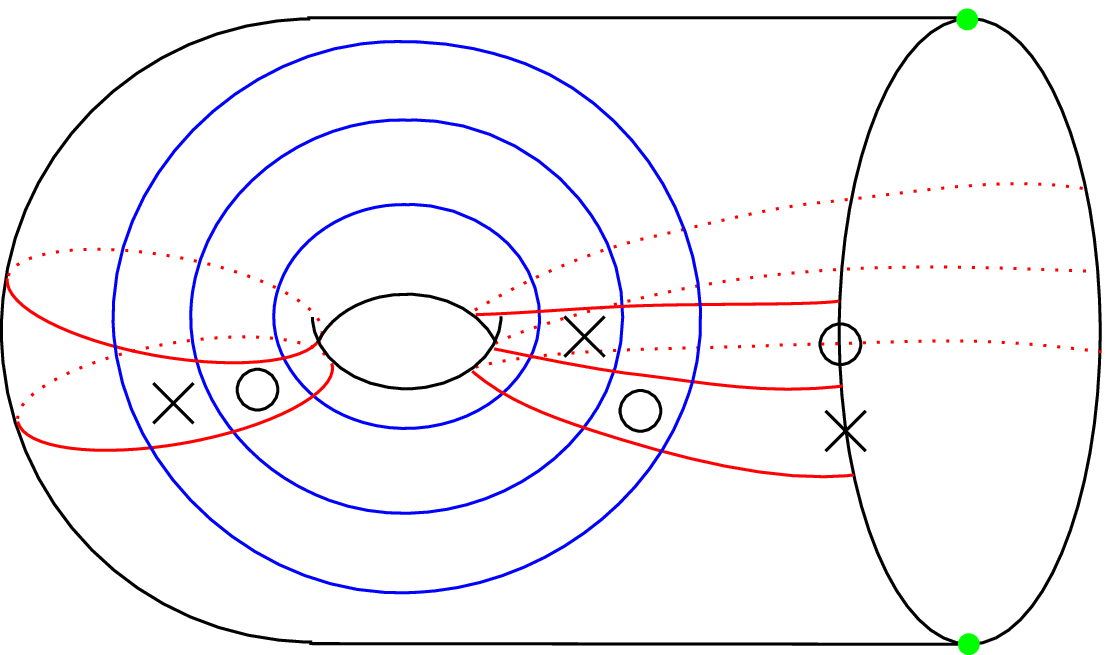} 
       \vskip .2 cm
       \caption[A type $1$ tangle Heegaard diagram.]{\textbf{A type $1$ tangle Heegaard diagram.}}\label{fig:hd1}
\end{figure}

A type $1$ tangle Heegaard diagram gives rise to a pair $(Y, \T)$, where $Y$ is a $3$-manifold with $\bdy Y\cong S^2$, and $\T$ is  marked $2n$-tangle in $Y$. We outline the topological construction below. 

Let $\mathcal S$ be the marked sphere associated to $(Y, \T)$. Note that  $\bdy \HH\cong \zz(\mathcal S)$, so we begin by building $\mathcal S$ from $\zz(\mathcal S)$. Next, let $[-\epsilon, 0]\times Z$ be a collar neighborhood of $\bdy \Sigma$, so that $\{0\}\times Z$ is identified with   $\bdy \Sigma$. Choose a neighborhood $Z\times [1,2]$ of $Z$ in $\mathcal S$, so that $Z\times \{2\}$ is in the interior of the $0$-handle from the decomposition described right after Definition \ref{s_def}. Glue $\Sigma\times [1,2]$ to $[-\epsilon, 0]\times \mathcal S$ so that the respective submanifolds $([-\epsilon, 0]\times Z)\times [1,2]$ and  $[-\epsilon, 0]\times (Z\times [1,2])$ are identified. Call the resulting 3-manifold $Y_0$.

Now attach a $3$-dimensional $2$-handle to each $\beta_i\times \{2\} \subset \bdy Y_0$ and to each $\alpha_i^c\times \{1\}\subset \bdy Y_0$ to obtain a manifold $Y_1$.
Next, join each $\alpha_i^a\times \{1\}$ to the core of the corresponding handle in $\{-\epsilon\} \times \mathcal S$ along their boundary to form a circle, and attach a $2$-handle to each such circle.   The resulting manifold, call it $Y_2$,
has the following boundary components:
\begin{itemize}
\item[-] $t+n-g$ spheres which meet $\Sigma\times \{2\}$ but do not meet $\{-\epsilon\} \times  \mathcal S$
\item[-] a sphere which meets both $\Sigma\times \{2\}$ and $\{-\epsilon\} \times  \mathcal S$
\item[-]  $t-g$ spheres which meet $\Sigma\times \{1\}$ but do not meet $\{-\epsilon\} \times  \mathcal S$
\item[-]  $2n$ spheres which meet both $\Sigma\times \{1\}$ and $\{-\epsilon\} \times  \mathcal S$ but do not meet  $\{-\epsilon\}\times \zzz \subset \{-\epsilon\} \times  \mathcal S$
\item[-] a sphere which meets both $\Sigma\times \{1\}$ and $(-\epsilon,z^-)\in \{-\epsilon\} \times  \mathcal S$, and a sphere which meets both $\Sigma\times \{1\}$ and $(-\epsilon,z^+)\in \{-\epsilon\} \times  \mathcal S$
\item[-]  The sphere $\{0\}\times  \mathcal S\subset  [-\epsilon, 0]\times  \mathcal S$
\end{itemize}
Glue $3$-balls to all but the last sphere. Call the result $Y$. 

Last, we construct a tangle $\T\subset Y$. Draw arcs from the $X$s to the $O$s in $(\Sigma\setminus \bbeta)\times \{3/2\}$, and push the interiors of the arcs into $(\Sigma\setminus \bbeta)\times (3/2,2]$.
Draw arcs from $O$s to $X$s in $(\Sigma\setminus \balpha)\times \{3/2\}$.
 The union of all arcs is an oriented, marked $2n$-tangle, where the marking, i.e. the ordering on $\bdy \T\subset \bdy Y$ comes from the order in which those $X$s and $O$s  that are on $\bdy \Sigma$ appear along $(a_1, a_{2n})\subset Z\times \{3/2\}\subset \mathcal S$. Observe that drawing an arc from $z^-$ to $z^+$ in  $(\Sigma\setminus \bbeta)\times \{3/2\}$  produces a $1$-component tangle which is unlinked from $\T$, and, together with an arc in the $3$-handle that was glued to  the sphere which meets both $\Sigma\times \{2\}$ and $\{-\epsilon\} \times  \mathcal S$, it bounds a disk away from $\T$ that lies entirely in that $3$-handle.
See, for example, Figure \ref{fig:hd1_top}. 
\begin{figure}[h]
 \centering
  \labellist
       \pinlabel $Y$ at 180 360
       \pinlabel $\bdy Y\cong S^2$ at 280 360
        \pinlabel $\T$ at 160 143
     \endlabellist
       \includegraphics[scale=.6]{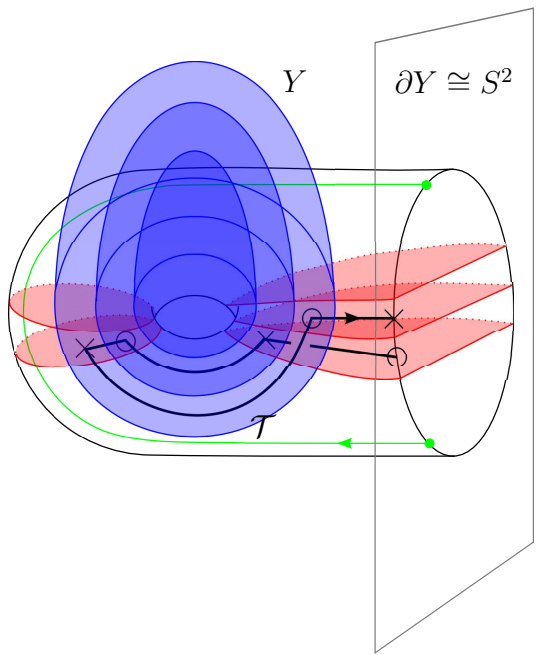} 
       \vskip .2 cm
       \caption[Building a tangle $(Y, \T)$ from a Heegaard diagram.]{\textbf{Building a tangle $(Y, \T)$ from a Heegaard diagram.}}\label{fig:hd1_top}
\end{figure}

\begin{definition}\label{morse_compat_s}
Given a marked sphere $\mathcal S = (S^2, t_1, \ldots, t_{n})$, we say that a Morse function $f$ on $S^2$ (with an implicit choice of a Riemannian metric $g$) is \emph{compatible with $\mathcal S$} if 
\begin{enumerate}
\item $t_1, \ldots, t_{n}$ are index $0$ critical points of $f$
\item $f$ has $n+2$ index $0$ critical points in total, $t_0, t_1, \ldots, t_{n}, t_{n+1}$ 
\item $f$ has $n+1$ index $1$ critical points $p_1, \ldots, p_{n+1}$, with $p_i$ flowing down to $t_{i-1}$ and $t_i$
\item $f$ has a unique index $2$ critical point
\end{enumerate}
\end{definition}

\begin{definition}\label{morse_compat}
Given a tangle $(Y, \T)$, we say that a self-indexing Morse function $f$ on $Y$ (with an implicit choice of a Riemannian metric $g$) is \emph{compatible with $(Y, \T)$} if
\begin{enumerate}
\item $\bdy Y$ is totally geodesic, $\nabla f$ is parallel to $\bdy Y$, $f|_{\bdy Y}$ is a Morse function compatible with $\mathcal S$, and $f|_T$ is a Morse function, where $T\subset Y$ is the underlying $1$ manifold for the marked tangle $\T$.
\item The index $1$ critical points for $\bdy Y$ are also index $1$ critical points for $Y$.
\item The  index $0$ critical points for $T$, along with the two additional index $0$ critical points for $\bdy Y$, are precisely the index $0$ critical points for $Y$.
\item The  index $1$ critical points for $T$, along with the index $2$ critical point for $\bdy Y$, are precisely the index $3$ critical points for $Y$.
\end{enumerate}
\end{definition}

\begin{proposition}
Every pair $(Y, \T)$ has a type $1$ Heegaard diagram.
\end{proposition}
\begin{proof}
We describe a compatible Morse function. Choose a Morse function $f'$ and metric $g'$ on $T$ which takes value $0$ on $\bdy T$ and is self-indexing except that it takes value $3$ on the index $1$ critical points. Extend to a pair $(f'', g'')$ on $T\cup \bdy Y$, so that $f''$ is also self-indexing on $\bdy Y$, except that it takes value $3$ on index $2$ critical points of $\bdy Y$, and is compatible with $\mathcal S$. Extend $f''$ and $g''$ to $f$ and $g$ on a neighborhood of $T\cup \bdy Y$ satisfying the conditions of Definition \ref{morse_compat}, and extend $f$ and $g$ arbitrarily to a Morse function and  metric on the rest of $Y$.  

Since $Y$ is connected, the graph formed by flows between the index $0$ and index $1$ critical points is connected. In fact, since the flows from the index $1$ critical points on $\bdy Y$  remain on $\bdy Y$, it follows that every index $0$ critical point of  $Y' \coloneqq Y \setminus \nu(T\cup \bdy Y)$ is connected by an edge in this graph to an index $1$ critical point of $Y'$, so we modify $f$ in the interior of $Y'$ to cancel every index $0$ critical point of $Y'$ with an index $1$ critical point of $Y'$. Similarly, we eliminate all index $3$ critical points of $Y'$.  

Finally, given these $f$ and $g$,  we construct  a type $1$ tangle Heegaard diagram. Start with Heegaard surface $\Sigma  = f^{-1}(3/2)$, oriented as the boundary of $f^{-1}([0,3/2])$. Let $\balpha$ be the set of points on $\Sigma$ that flow down to the index $1$ critical points, label the arcs $\balpha^a$ and their endpoints compatibly with $\mathcal S$, and let $\bbeta$ be the set of points on $\Sigma$ which flow up to the index $2$ critical points. Mark the positive intersections of $T\cap \Sigma$ with $O$s, and the negative intersections with $X$s. Also place an $X$ in each region $(a_i, a_{i+1})$ of $\bdy\Sigma\setminus\balpha\cup (a_1, a_{2n+1})$ if the points in that region flow down to a positive endpoint $t_i$ of the tangle $T$, and an $O$ if those points flow to a negative endpoint $t_i$ of $T$. Finally, place a point labeled $z^-$ in $(a_{4n+2}, a_1)$, and a point $z^+$ in $(a_{2n+1}, a_{2n+2})$.
\end{proof}

The Morse theory construction implies the following proposition.

\begin{proposition}
Any two type $1$ tangle Heegaard diagrams for a given tangle $(Y, \T)$ are related by a sequence of \emph{Heegaard moves}:
\begin{itemize}
\item[-] isotopies of the $\alpha$-curves and $\beta$-curves, not crossing $\bdy \Sigma\cup \XX\cup \OO$
\item[-] handle slides of $\alpha$-curves over $\alpha$-circles and $\beta$-circles over $\beta$-circles 
\item[-] index one/two stabilizations (and their inverses, destabilizations) in the interior of $\Sigma$: forming the connected sum with a torus with one $\alpha$-circle and one $\beta$-circle meeting transversely in a single point
\item[-] index zero/three stabilizations (and their inverses, destabilizations) in the interior of $\Sigma$: replacing a neighborhood of an $X$ with  one $\alpha$-circle and one $\beta$-circle, isotopic to each other and intersecting in two points, and adding an $O$ in the middle of the three new regions, and an $X$ in each of the new side regions, or replacing a neighborhood of an $O$ with such $\alpha$ and $\beta$ curves, along with an $X$ in the middle new region, and an $O$ in each side region (see Figure \ref{fig:stab}). 
\end{itemize}
\end{proposition}

\begin{proof}
The proof follows from the Morse calculus used in the proofs of \cite[Proposition 3.3]{oszlink} and \cite[Proposition 4.10]{bfh2}.
\end{proof}

\begin{figure}[h]
 \centering
  \labellist
       \pinlabel $\to$ at 60 50
       \pinlabel $\to$ at 400 50
     \endlabellist
       \includegraphics[scale=.6]{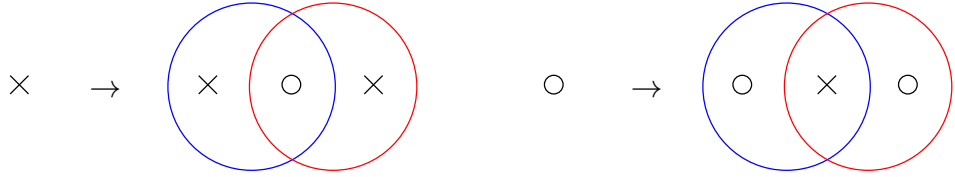} 
       \vskip .2 cm
       \caption[Index zero/three stabilization.]{\textbf{Index zero/three stabilization.}}\label{fig:stab}
\end{figure}

We also define type $2$ tangle Heegaard diagrams. The definition is slightly different from that of type $1$ diagrams, so that when one glues a type $1$ and a type $2$ diagram that agree along the boundary, the resulting closed diagram is a valid Heegaard diagram for a link. 
\begin{definition}
A \emph{type $2$ multipointed bordered Heegaard diagram for a tangle} is a sextuple $\HH = (\Sigma, \balpha, \bbeta, \XX, \OO, \zzz)$ where
\begin{itemize}
\item[-] $\Sigma$ is a compact surface of genus $g$ with one boundary component
\item[-] $\balpha = \{\alpha_1^a,\ldots, \alpha_{2n+1}^a, \alpha_1^c, \ldots, \alpha_{t}^c\}$ is a set of pairwise disjoint, embedded curves: $2n+1$  arcs, each  with boundary on $\bdy \Sigma$, and $t$ closed curves in the interior of $\Sigma$
\item[-] $\bbeta$ is a set of $t+n+1$ pairwise disjoint curves embedded in the interior of $\Sigma$
\item[-] $\XX$ and $\OO$ are two $(t+2n-g+1)$-tuples of points in  $\Sigma\setminus(\balpha\cup\bbeta)$
\item[-] $\zzz$ is an oriented arc in $\Sigma\setminus (\balpha\cup \bbeta)$ with boundary on  $\bdy\Sigma \setminus \balpha$ 

\end{itemize}
subject to the conditions
\begin{itemize}
\item[-] $\bbeta$ span a $g$-dimensional subspace of $H_1(\Sigma; \Z)$
\item[-] $\{\alpha_1^c, \ldots, \alpha_t^c\}$ span a $(g-1)$-dimensional subspace of  $H_1(\Sigma; \Z)$, and along with the arcs, $\balpha$ span a $g$-dimensional subspace of $H_1(\Sigma, \bdy \Sigma; \Z)$
\item[-] $\{\alpha_1^a,\ldots, \alpha_{2n+1}^a\}$ induce a concentric matching on $\bdy\Sigma$, and they are labeled so that we can order the points on $\bdy\balpha$ according to the orientation of $-\bdy \Sigma$ as  $a_1, \ldots, a_{4n+2}$ so that $\bdy\alpha^a_i = \{a_i, a_{4n+3-i}\}$.
\item[-] $\zzz^+\coloneqq\bdy^+(\zzz)$ lies in the interior of the segment with boundary $a_{4n+2}$ and $a_1$ of $\bdy\Sigma\setminus \balpha$, and $\zzz^- \coloneqq\bdy^-(\zzz)$ lies on the segment with boundary  $a_{2n+1}$ and $a_{2n+2}$.
\item[-] Each of the $t-g+1$ components of $\Sigma\setminus \balpha$ that do not meet $\bdy\Sigma$ contains one $X\in \XX$ and one $O\in\OO$, and each of the $2n$ components of $\Sigma\setminus \balpha$ that meet $\bdy\Sigma$ but do not meet $\zzz$ contains either an $X$ in the interior and an $O$ on the segment of $\bdy\Sigma\setminus \balpha$ with the lower indexed endpoints, or an $O$ in the interior and an $X$ on the segment of $\bdy\Sigma\setminus \balpha$ with the lower indexed endpoints. 
\item[-] Each of the $t+n-g+1$ components of $\Sigma\setminus \bbeta$ that do not meet $\bdy\Sigma$ contains exactly one $X$ and one $O$. The unique component of $\Sigma\setminus\bbeta$ that meets  $\bdy\Sigma$ contains $n$ $X$s and $n$ $O$s on $\bdy \Sigma$.
\end{itemize}
\end{definition} 

Figure \ref{fig:hd2} is an example of a type $2$ tangle Heegaard diagram.

\begin{figure}[h]
 \centering
  \labellist
       \pinlabel $\textcolor{red}{a_1}$ at 58 48
       \pinlabel $\textcolor{red}{a_2}$ at 60 74
       \pinlabel $\textcolor{red}{a_3}$ at 60 100
       \pinlabel $\textcolor{red}{a_4}$ at -10 130
       \pinlabel $\textcolor{red}{a_5}$ at -12 106
       \pinlabel $\textcolor{red}{a_6}$ at -13 80  
       \pinlabel $\textcolor{green}{\zzz^+}$ at 42 14
       \pinlabel $\textcolor{green}{\zzz^-}$ at 38 170
     \endlabellist
       \includegraphics[scale=.6]{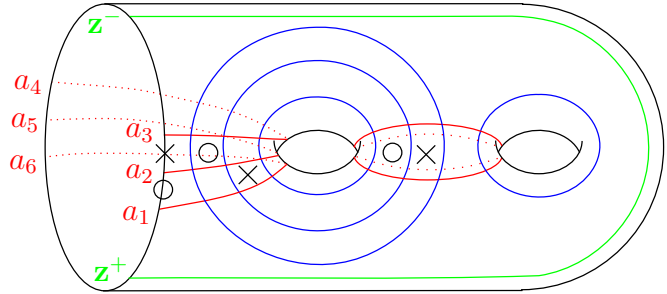} 
       \vskip .2 cm
       \caption[A type $2$ tangle Heegaard diagram.]{\textbf{A type $2$ tangle Heegaard diagram.}}\label{fig:hd2}
\end{figure}

A type $2$ tangle Heegaard diagram gives rise to a pair $(Y, \T)$ of a $3$-manifold $Y$ with $\bdy Y \cong S^2$ and a marked tangle $\T$. 
The topological construction is similar to the one for a type $1$ diagram.

We build the manifold $Y_2$ by following the type $1$ construction, except this time $\bdy \HH\cong \zz(-\mathcal S)^*$, where $\mathcal S$ is the marked sphere associated to $(Y,\T)$. The difference in the types of boundary components of $Y_2$ is that there are now $t-g+1$ spheres which meet $\Sigma\times \{1\}$ but do not meet $\{-\epsilon\}\times \mathcal S$, and there is one single sphere which meets both $\Sigma\times \{1\}$ and $\{-\epsilon\}\times \{\zzz^+, \zzz^-\}\subset \{-\epsilon\} \times \mathcal S$.  We again glue $3$-balls to all spheres except $\{0\}\times \mathcal S$ to obtain $Y$.

The tangle $\T\subset Y$ is again constructed by connecting the $X$s and $O$s. This time its marking comes from the order in which the $X$s and $O$s on $\bdy \Sigma$ appear along $-\bdy \Sigma$. The oriented arc $\zzz\times \{3/2\}$ is a $1$-component boundary-parallel tangle which is unlinked from $\T$.  
 
 We cannot use Morse theory directly to prove the statements that follow. One way to explain where the problem lies is that if we start with a Morse function for $\bdy Y$, then two index $0$ critical points on $\bdy Y$ that would correspond to $\zzz^+$ and $\zzz^-$ belong to the same $0$-handle in the handle decomposition for $Y$ specified by $\HH$. 

\begin{proposition}\label{prop:type1to2}
Every $(Y, \T)$ has a type $2$ tangle Heegaard diagram. 
\end{proposition}

\begin{proof}
Let $\HH$ be a type $1$ diagram for $(-Y, -\T)$. We perform the following series of moves near the boundary of the diagram, as in Figure \ref{fig:1to2}. Perform an index one/two stabilization near $z^+$ (Figure \ref{fig:1to2}(b)). Denote the new $\alpha$-circle by $\alpha'$, and the new $\beta$-circle by $\beta'$.  Slide all $\alpha$-arcs over $\alpha'$ so that now $\beta'$ crosses them once each, near $a_1, \ldots, a_{2n+1}$ (Figure \ref{fig:1to2} (c)).  Connect $z^-$ to $z^+$ by an arc $\zzz$ that goes once over the new handle parallel to $\beta'$ (Figure \ref{fig:1to2} (d)). Remove $\alpha'$ (Figure \ref{fig:1to2} (e)). Call the resulting diagram $\HH'$. Observe that $\zzz$ does not intersect any $\alpha$ or $\beta$ curves. The diagram  $-\HH'$ is a type $2$ tangle Heegaard diagram for $(Y, \T)$.  
\end{proof}

\begin{figure}[h]
 \centering
  \labellist
       \pinlabel $\rotatebox{90}{\textcolor{red}{\ldots}}$ at 36 105
       \pinlabel $\rotatebox{90}{\textcolor{red}{\ldots}}$ at 165 105
       \pinlabel $\rotatebox{90}{\textcolor{red}{\ldots}}$ at 294 105
       \pinlabel $\rotatebox{90}{\textcolor{red}{\ldots}}$ at 422 105
       \pinlabel $\rotatebox{90}{\textcolor{red}{\ldots}}$ at 550 105
       \pinlabel $(a)$ at 25 -20
       \pinlabel $(b)$ at 152 -20
       \pinlabel $(c)$ at 277 -20
       \pinlabel $(d)$ at 402 -20
       \pinlabel $(e)$ at 538 -20
     \endlabellist
       \includegraphics[scale=.6]{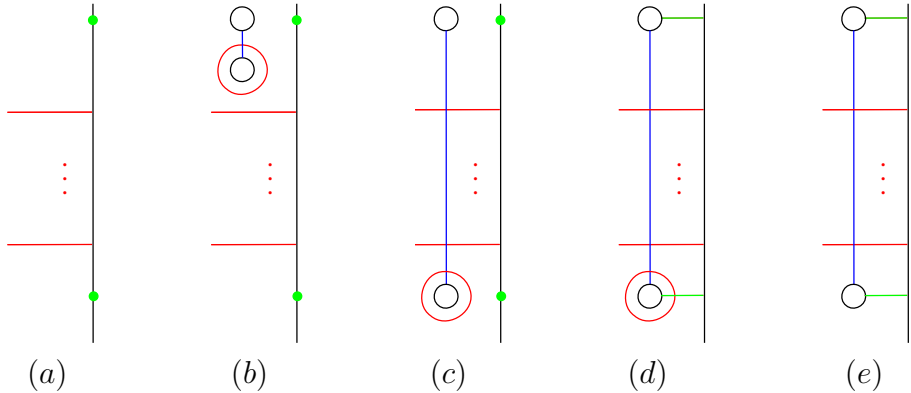} 
       \vskip .5 cm
       \caption[Transforming a type $1$ diagram to a type $2$ diagram.]{\textbf{Transforming a type $1$ diagram to a type $2$ diagram.}}\label{fig:1to2}
\end{figure}

We will say that a  type $2$ diagram like $-\HH'$, obtained from a type $1$ diagram as above, is in \emph{type $1$ position}.

\begin{proposition}\label{type2moves}
Any two type $2$ tangle Heegaard diagrams for a given tangle $(Y, \T)$ are related by a sequence of Heegaard moves:
\begin{itemize}
\item[-] isotopies of the $\alpha$-curves and $\beta$-curves, not crossing $\bdy \Sigma\cup \XX\cup \OO\cup \zzz$
\item[-] handle slides of $\alpha$-curves over $\alpha$-circles and $\beta$-circles over $\beta$-circles 
\item[-] index one/two stabilizations and destabilizations in the interior of $\Sigma$
\item[-] index zero/three stabilizations  and destabilizations in the interior of $\Sigma$
\end{itemize}
\end{proposition}

To prove this proposition, we make use of the following lemma.

\begin{lemma}\label{lemma2to1}
Any type $2$ diagram can be put in type $1$ position. 
\end{lemma}

\begin{proof}
Let $\HH = (\Sigma, \balpha, \bbeta, \XX, \OO, \zzz)$ be a type $2$ diagram for a pair $(Y, \T)$. The idea is to find a curve on $\Sigma$ which is disjoint from $\balpha$, bounds a disk in the $\alpha$-handlebody, and intersects $\zzz$ exactly once, and use it as a guide to modify the Heegaard diagram.  We exhibit one such curve below.

Let $\alpha'\subset \Sigma$  be an embedded circle which is a push-off of the union of $\alpha^a_{2n+1}$ and $(a_{2n+1}, a_{2n+2})\subset \bdy \Sigma$ into $\Sigma$ and does not intersect $\balpha$, see Figure \ref{fig:alpha_disk} and the more schematic first diagram in Figure \ref{fig:2special}. We will use $\alpha'$ as a guide while performing a series of Heegaard moves.

 Note that  $\alpha'$ bounds a disk in the $\alpha$-handlebody. This disk is a push-off of the disk $D = D_1\cup D_2$, where $D_1$ is the disk on $\bdy Y$ bounded by the interval $(a_{2n+1}, a_{2n+2})$  and the core of the $1$-handle of $\bdy Y$ attached at $a_{2n+1}$ and $a_{2n+2}$, and $D_2$ is the core of the $2$-handle for $\alpha^a_{2n+1}$ from the construction of $Y_2$. So $\HH' = (\Sigma, \balpha\cup \alpha', \bbeta, \XX, \OO, \zzz)$ still specifies the same pair $(Y, \T)$, or, to be more precise, $-(\Sigma, \balpha\cup \alpha', \bbeta, \XX, \OO, \bdy\zzz)$ is a type $1$ diagram for $(-Y, -\T)$. 

\begin{figure}[h]
 \centering
  \labellist
       \pinlabel $\textcolor{red}{\alpha'}$ at 100 150
     \endlabellist
       \includegraphics[scale=.6]{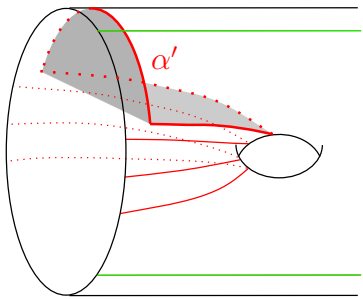} 
       \vskip .2 cm
       \caption[The circle $\alpha'$ and the disk it bounds.]{\textbf{The circle $\alpha'$ and the disk it bounds.}}\label{fig:alpha_disk}
\end{figure}

Also note that $\alpha'$ intersects $\zzz$ exactly once, near $\zzz^-$, so, since $\balpha\cap \zzz = \emptyset$,  no $\alpha$ circle in $\balpha$ is homologous to $\alpha'$ in $H_1(\Sigma; \Z)$. This means that, after sliding $\alpha$ curves over $\alpha'$ if necessary, we can draw $\HH$ in the following way.  Near the boundary we see  $\bdy\HH\times [0,\epsilon)$, where $\bdy\HH\times\{0\}$ is  the boundary of $\HH$. There is a $1$-handle for $\Sigma$ with feet attached at $(\zzz^+, \epsilon)$ and $(\zzz^-,\epsilon)$, $\alpha'$ is a meridian of that $1$-handle, $\zzz$ goes once over the handle. There may also be multiple $\beta$ curves going over the $1$-handle.
See Figure \ref{fig:2special} (b).  We continue the proof with such more schematic pictures drawn in a plane. 

We claim that, after an isotopy of $\bbeta$ if necessary, there is some $\beta'\subset \bbeta$ which intersects $\alpha'$ exactly once. Close $\zzz$ to a circle $\bar\zzz$ by connecting $\zzz^+$ to $\zzz^-$ along $\bdy\Sigma$, going through $a_1, \ldots, a_{2n+1}$. Since $\alpha'$ and $\bar\zzz$ are two circles on $\Sigma$ intersecting transversely in one point,  the  neighborhood of $\alpha'\cup \bar\zzz$ in $\Sigma$ is a punctured torus $T$, see Figure  \ref{fig:2special} (c). Since $\bbeta$ spans a $g$-dimensional subspace of $H_1(\Sigma)$, then $\Sigma\setminus\bbeta$ only contains genus $0$ pieces, so there is at least one $\beta$-circle, pick one and call it $\beta'$, cutting the punctured torus into a genus $0$ surface. Since no $\beta$ can intersect $\zzz$ or $\bdy\Sigma$, then $\beta'$ cannot intersect $\bar\zzz$. So $\beta'\cap T$ is homologous to $\bar\zzz$ in $H_1(T, \bdy T)$, so it can be isotoped to only intersect $\alpha'$ once. 
If any other $\beta$-curves intersect $\alpha'$, slide them over $\beta'$, so that $\beta'$ is the only curve intersecting $\alpha'$. Now the diagram near the boundary looks like what we described in the previous paragraph, except there is exactly one $\beta$ curve going over the $1$-handle. See Figure \ref{fig:2special} (d).

Since $\bbeta$ spans a $g$-dimensional subspace of $H_1(\Sigma)$, then all components of $\Sigma\setminus (\bbeta\cup \zzz)$ have genus zero. In particular, the region of $\Sigma\setminus (\bbeta\cup \zzz)$ that contains $\bar\zzz$ is planar, with boundary components $\bar\zzz$, $\beta'$, and possibly some other $\beta$ circles. See Figure  \ref{fig:2special} (e) ($\alpha$ curves omitted from the picture away from the boundary). Slide $\beta'$ over each $\beta$ circle in that region to move it close to the boundary of the diagram, i.e. so that it is a parallel push-off of $\bar\zzz$ into the interior of $\Sigma$, see Figure  \ref{fig:2special} (f). Remove $\alpha'$, which only served as a guide along the proof. See Figure \ref{fig:2special} (g). The resulting diagram is in type $1$ position.\qedhere
\begin{figure}[h]
 \centering
  \labellist
       \pinlabel $\rotatebox{90}{\textcolor{red}{\ldots}}$ at 30 105
       \pinlabel $\rotatebox{90}{\textcolor{red}{\ldots}}$ at 125 105
       \pinlabel $\rotatebox{90}{\textcolor{red}{\ldots}}$ at 220 105
       \pinlabel $\rotatebox{90}{\textcolor{red}{\ldots}}$ at 320 105
       \pinlabel $\rotatebox{90}{\textcolor{red}{\ldots}}$ at 415 105
       \pinlabel $\rotatebox{90}{\textcolor{blue}{\ldots}}$ at 475 105
       \pinlabel $\rotatebox{90}{\textcolor{red}{\ldots}}$ at 560 105
       \pinlabel $\rotatebox{90}{\textcolor{red}{\ldots}}$ at 660 105
              \pinlabel $\rotatebox{90}{\textcolor{blue}{\small{...}}}$ at 145 172
              \pinlabel $\rotatebox{90}{\textcolor{blue}{\small{...}}}$ at 145 29
              \pinlabel $\rotatebox{90}{\textcolor{blue}{\small{...}}}$ at 248 172
              \pinlabel $\rotatebox{90}{\textcolor{blue}{\small{...}}}$ at 248 29
                     \pinlabel $\rotatebox{90}{\textcolor{red}{\ldots}}$ at 30 105
       \pinlabel $(a)$ at 30 -15
       \pinlabel $(b)$ at 125 -15
       \pinlabel $(c)$ at 220 -15
       \pinlabel $(d)$ at 320 -15
       \pinlabel $(e)$ at 450 -15
       \pinlabel $(f)$ at 572 -15
       \pinlabel $(g)$ at 675 -15
     \endlabellist
       \includegraphics[scale=.6]{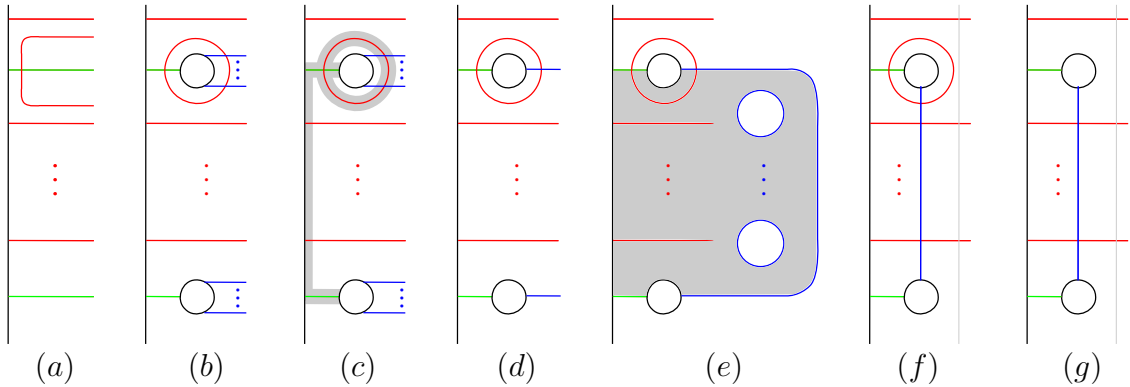} 
       \vskip .5 cm
       \caption[Putting a type $2$ diagram in type $1$ position.]{\textbf{Putting a type $2$ diagram in type $1$ position.} The last diagram is the mirror of the corresponding type $1$ diagram.}\label{fig:2special}
\end{figure}
\end{proof}

 \begin{proof}[Proof of Proposition \ref{type2moves}]
 Let $\HH_1$ and $\HH_2$ be two type $2$ diagrams for the same pair $(Y, \T)$. By Lemma \ref{lemma2to1}, both can be put in type $1$ positions $\HH_1'$ and $\HH_2'$ by a sequence of the moves described in Proposition \ref{type2moves}. Let  $\HH_1''$ and $\HH_2''$ be the corresponding type $1$ diagrams, so that $\HH_i' = -\HH_i''$ away from the boundary and the special $1$-handle from Proposition \ref{prop:type1to2}. 
 
Since $\HH_1''$ and $\HH_2''$ are related by a sequence of moves away from the boundaries, corresponding moves (the reflections of the original moves) can be performed between $\HH_1'$ and $\HH_2'$ away from the ``neighborhood" of the boundary containing $\zzz$ and the special $\beta$ circle from the proof of Lemma \ref{lemma2to1}, i.e. the $\beta$ circle shown in Figure \ref{fig:2special} (g). Thus, $\HH_1$ and $\HH_2$ are related by a sequence of Heegaard moves. 
 \end{proof}

\subsection{$3$-manifolds with two boundary components} \label{ssec:hd2}

For a tangle in a manifold $Y$ with $\bdy Y \cong S^2\coprod S^2$, we describe a Heegaard diagram with two boundary components. We will also want to keep track of a framed arc connecting the two boundary components of $Y$, by means of two arcs $\zzz_1$ and $\zzz_2$ that will connect the two boundary components of the Heegaard diagram.

\begin{definition}
A \emph{multipointed bordered Heegaard diagram with two  boundary components for a tangle} is a sextuple $\HH = (\Sigma, \balpha, \bbeta, \XX, \OO, \zzz)$ where
\begin{itemize}
\item[-] $\Sigma$ is a compact surface of genus $g$ with two boundary components
\item[-] $\balpha = \{\alpha_1^0,\ldots, \alpha_{m+1}^0,\alpha_1^1,\ldots, \alpha_{n+1}^1, \alpha_1^c, \ldots, \alpha_{t}^c\}$ is a set of pairwise disjoint, embedded curves: $m + n + 2$  arcs (where $m$ and $n$ have the same parity), each  with boundary on $\bdy \Sigma$, and $t$ closed curves in the interior of $\Sigma$
\item[-] $\bbeta$ is a set of $t+\frac{m+ n}{2}+1$ pairwise disjoint curves embedded in the interior of $\Sigma$
\item[-] $\XX$ and $\OO$ are two $(t+m + n-g+1)$-tuples of points in  $\Sigma\setminus(\balpha\cup\bbeta)$
\item[-] $\zzz = \{\zzz_1, \zzz_2\}$ is a set of two oriented arcs in $\Sigma\setminus (\balpha\cup \bbeta)$ with boundary on  $\bdy\Sigma \setminus \balpha$ 
\end{itemize}
subject to the conditions
\begin{itemize}
\item[-] $\bbeta$ span a $g$-dimensional subspace of $H_1(\Sigma; \Z)$
\item[-] $\{\alpha_1^c, \ldots, \alpha_t^c\}$ span a $(g-1)$-dimensional subspace of  $H_1(\Sigma; \Z)$, and along with the arcs, $\balpha$ span a $(g+1)$-dimensional subspace of $H_1(\Sigma, \bdy \Sigma; \Z)$
\item[-] $\{\alpha_1^0,\ldots, \alpha_{m+1}^0\}$ induce a concentric matching on one component of $\bdy\Sigma$, and they are labeled so that we can order their endpoints according to the orientation of $-\bdy \Sigma$ as  $a_1^0, \ldots, a_{2m+2}^0$ so that $\bdy\alpha_i^0 = \{a_i^0, a_{2m+3-i}^0\}$; $\{\alpha_1^1,\ldots, \alpha_{n+1}^1\}$ induce a concentric matching on the other component of $\bdy\Sigma$, and they are labeled so that we can order their endpoints according to the orientation of $\bdy \Sigma$ as  $a_1^1, \ldots, a_{2n+2}^1$ so that $\bdy\alpha_i^1 = \{a_i^1, a_{2n+3-i}^1\}$.
\item[-] $\zzz_1^+\coloneqq\bdy^+(\zzz_1)$ lies in the interior of the segment with boundary $a_{2m+2}^0$ and $a_1^0$ of $\bdy\Sigma\setminus \balpha$, and $\zzz_1^- \coloneqq\bdy^-(\zzz_1)$ lies on the segment with boundary $a_{2n+2}^1$ and $a_1^1$;  $\zzz_2^+\coloneqq\bdy^+(\zzz_2)$ lies in the interior of the segment with boundary  $a_{m+1}^0$ and $a_{m+2}^0$ of $\bdy\Sigma\setminus \balpha$, and $\zzz_2^- \coloneqq\bdy^-(\zzz_2)$ lies on the segment with boundary  $a_{n+1}^1$ and $a_{n+2}^1$.
\item[-] Each of the $t-g+1$ components of $\Sigma\setminus \balpha$ that do not meet $\bdy\Sigma$ contains one $X\in \XX$ and one $O\in \OO$, and each of the $m + n$ components of $\Sigma\setminus \balpha$ that meet $\bdy\Sigma$ but do not meet $\zzz$  contains either an $X$ in the interior and an $O$ on the segment of $\bdy\Sigma\setminus \balpha$ with the lower indexed endpoints, or an $O$ in the interior and an $X$ on the segment of $\bdy\Sigma\setminus \balpha$ with the lower indexed endpoints. 
\item[-] Each of the $t+\frac{m+n}{2}-g+1$ components of $\Sigma\setminus \bbeta$ that do not meet $\bdy\Sigma$ contains exactly one $X$ and one $O$. The unique component of $\Sigma\setminus\bbeta$ that meets  $\bdy\Sigma$ contains $\frac{m+n}{2}$ $X$s and $\frac{m+n}{2}$ $O$s on $\bdy \Sigma$.
\end{itemize}
We denote the component of $\bdy\HH$ containing $\alpha^0_i$ by $\bdy^0\HH$, and the component of $\bdy\HH$ containing $\alpha^1_i$ by $\bdy^1\HH$.
\end{definition} 

Figure \ref{fig:hd3} is an example of a tangle Heegaard diagram with two boundary components.

\begin{figure}[h]
 \centering
  \labellist
       \pinlabel $\textcolor{red}{a_1^0}$ at 18 88
       \pinlabel $\textcolor{red}{a_2^0}$ at 20 114
       \pinlabel $\textcolor{red}{a_3^0}$ at 20 140
       \pinlabel $\textcolor{red}{a_4^0}$ at -12 170
       \pinlabel $\textcolor{red}{a_5^0}$ at -14 146
       \pinlabel $\textcolor{red}{a_6^0}$ at -15 120  
       \pinlabel $\textcolor{red}{a_1^1}$ at 478 48
       \pinlabel $\textcolor{red}{a_2^1}$ at 476 80
       \pinlabel $\textcolor{red}{a_3^1}$ at 474 106
       \pinlabel $\textcolor{red}{a_4^1}$ at 474 135
       \pinlabel $\textcolor{red}{a_5^1}$ at 474 170  
       \pinlabel $\textcolor{red}{a_6^1}$ at 508 174
       \pinlabel $\textcolor{red}{a_7^1}$ at 510 150             
       \pinlabel $\textcolor{red}{a_8^1}$ at 512 124
       \pinlabel $\textcolor{red}{a_9^1}$ at 510 98
       \pinlabel $\textcolor{red}{a_{10}^1}$ at 508 70  
       \pinlabel $\textcolor{green}{\zzz_1^+}$ at -2 20
       \pinlabel $\textcolor{green}{\zzz_2^-}$ at 2 236
       \pinlabel $\textcolor{green}{\zzz_1^-}$ at 498 16
       \pinlabel $\textcolor{green}{\zzz_2^+}$ at 496 236
     \endlabellist
       \includegraphics[scale=.6]{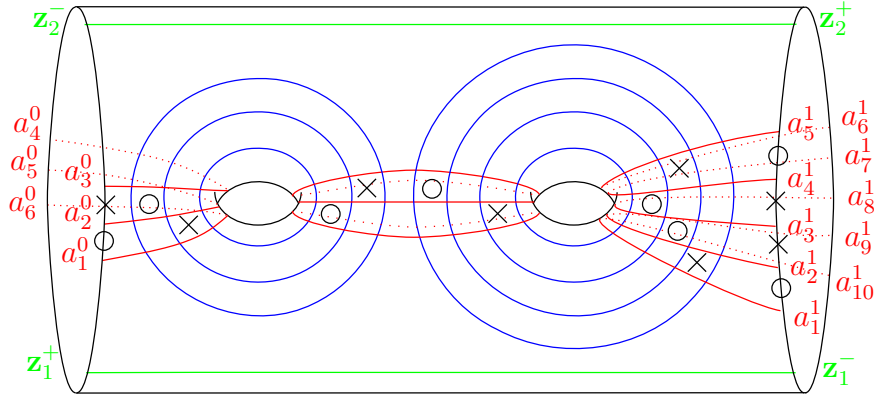} 
       \vskip .2 cm
       \caption[A tangle Heegaard diagram with two boundary components.]{\textbf{A tangle Heegaard diagram with two boundary components.}}\label{fig:hd3}
\end{figure}

A tangle Heegaard diagram with two boundary components  gives rise to a pair $(Y, \T)$ of a $3$-manifold $Y$ with $\bdy Y \cong S^2_0\coprod S^2_1$ and a marked $(m,n)$-tangle $\T$, with $\bdy^0\T\subset S^2_0$ and $\bdy^1\T\subset S^2_1$. We describe the topological construction below.

Let $\HH_{dr}$ be the Heegaard diagram obtained from $\HH$ by deleting a neighborhood of $\zzz_2$ (this process, called \emph{drilling}, was introduced in \cite{bimod}). The boundary of this deleted neighborhood  consists of the neighborhood $\zzz^0$ of $\zzz_2^-$ on $\bdy^0\HH$, the neighborhood $\zzz^1$ of $\zzz_2^+$ on $\bdy^1\HH$, and two disjoint push-offs of $\zzz_2$. Denote the push-off closer to $a^0_{m+1}$ by $\zzz_2^{front}$, and the other one by $\zzz_2^{back}$. The boundary of $\HH_{dr}$ is 
$$\bdy\HH_{dr} = (\bdy^0\HH\setminus\zzz^0)\cup  (\bdy^1\HH\setminus\zzz^1)\cup \zzz_2^{front}\cup \zzz_2^{back}.$$
It inherits the decorations of $(\bdy^0\HH\setminus\zzz^0)$ and $(\bdy^1\HH\setminus\zzz^1)$. We also place a basepoint $z^{front}$ on $\zzz_2^{front}$ and $z^{back}$ on $\zzz_2^{back}$. 

If we ignore $z^{front}$ and $z^{back}$,  $\HH_{dr}$ looks like a type $2$ diagram for an $(m+n)$-tangle, except that there is one extra $\alpha$-arc. 

We first build the pair $(Y_{dr}, \T_{dr})$ for $\HH_{dr}$ as we would for any type $2$ diagram. We obtain $(Y, \T)$ from  $(Y_{dr}, \T_{dr})$ by attaching a $3$-dimensional $2$-handle to the boundary sphere along the connected sum annulus arising from the decomposition $\bdy\HH_{dr} = \bdy^0\HH \# \bdy^1\HH$. More precisely, the attaching circle is the union of the two gradient flow lines from the index $2$ critical point passing through  $z^{front}$ and $z^{back}$.

\begin{proposition}
Every $(Y, \T)$ has a tangle Heegaard diagram with two boundary components. 
\end{proposition}

\begin{proof}
The idea of the  proof is the same as in the proof of \cite[Proposition 5.8]{bimod}. Choose an arc connecting $\bdy^0Y$ to $\bdy^1Y$ away from $\T$, and remove its neighborhood. Call the result $(Y_{dr}, \T_{dr})$, where the ordering on $\T_{dr}$ inherits the ordering on $\bdy^0\T$ concatenated with the reversed ordering on $\bdy^1\T$. Let $\HH'  = (\Sigma, \balpha, \bbeta, \XX, \OO, \zzz)$ be a type $2$ diagram for  $(Y_{dr}, \T_{dr})$. Add a parallel translate of $\alpha^a_{2m+1}$ by pushing it so that $a_{2m+1}$ is pushed in the negative direction along $\bdy \HH'$, and call this curve  $\alpha'$. Call the resulting diagram $\HH''$.  Add a $1$-handle to the two  intervals of $\bdy\Sigma\setminus \balpha$ between $\alpha^a_{2m+1}$ and $\alpha'$. The resulting surgery on $\bdy\Sigma$ splits it into two circles. Denote the circle containing $a_1$ by $\bdy^0\Sigma$, and the other circle by $\bdy^1\Sigma$. Let $\zzz_2$ be the co-core of the $1$-handle, oriented from $\bdy^0\Sigma$ to $\bdy^1\Sigma$. Relabel $\zzz$ to $\zzz_1$, $\alpha^a_i$ to $\alpha^0_{i}$ for $i\leq 2m+1$,   $\alpha^a_i$ to $\alpha^1_{2m+2n+2-i}$  for $i> 2m+1$, and $\alpha'$ to $\alpha^1_{2n+1}$. The resulting diagram $\HH$ is a diagram for $(Y, \T)$. Note that $\HH'' = \HH_{dr}$.
\end{proof}

In the case of two boundary components, it is no longer true that any two diagrams for a pair $(Y, \T)$ are related by Heegaard moves. However, if we keep better track of the parametrization of the boundary, we can still make this statement. 

\begin{definition}
A \emph{strongly marked $(m,n)$-tangle $(Y, \T, \gamma)$} is a marked $(m,n)$-tangle $(Y, \T)$, along with a framed arc $\gamma$  connecting $\bdy^0 Y$ to $\bdy^1 Y$ in the complement of $\T$ such that  $\gamma$ and its framing $\lambda_{\gamma}$  have ends on the equators of the two marked spheres,  and we see $-\bdy^0\T,  -\bdy^0\gamma, -\bdy^0\lambda_{\gamma}$ and $\bdy^1\T,  \bdy^1\gamma, \bdy^1\lambda_{\gamma}$ in this order along each equator.
\end{definition}

We say that a diagram $\HH$ is compatible with a strongly marked tangle  $(Y, \T, \gamma)$ if $\HH$ describes $(Y, \T)$, and after building $(Y, \T)$ from $\HH$,  the arc $\zzz_1$ with the framing that points into the $\beta$-handlebody yields  $\gamma$.

\begin{proposition}
If $\HH$ and $\HH'$ specify the same triple $(Y, \T, \gamma)$, then they are related by a sequence of  Heegaard moves like the ones described in Proposition \ref{type2moves}.
\end{proposition}

\begin{proof}
Let $\HH_{dr}$ and $\HH'_{dr}$ be the corresponding drilled diagrams. By Proposition \ref{type2moves}, they are related by a sequence of moves away from the boundary of the Heegaard surface, hence away from the drilling region. Performing the inverse of the drilling operation to each diagram along the way provides a sequence of moves between $\HH$ and $\HH'$. 
\end{proof}

\subsection{Generators}

Fix a tangle Heegaard diagram $\HH = (\Sigma, \balpha, \bbeta, \XX, \OO, \zzz)$ of some genus $g$ for some pair $(Y, \T)$.  Let $k \coloneqq|\bbeta|$.

\begin{definition}
A \emph{generator} of $\HH$ is a $k$-element subset $\xxx = \{x_1, \ldots, x_k\}$ of points in $\balpha\cap \bbeta$, such that there is exactly one point on each $\beta$-circle, exactly one point on each $\alpha$-circle, and at most one point on each $\alpha$-arc.
\end{definition}

We denote the set of generators of $\HH$ by $\mathfrak S(\HH)$, or simply by $\gens$ when $\HH$ is fixed. 

If $\HH$ is a diagram for a $2n$-tangle, then for a generator $\xxx\in \gens$, let $o(\xxx)\coloneqq \{i | \xxx\cap \alpha_i^a\neq \emptyset\}$, and let $\bar o(\xxx)\coloneqq[2n+1]\setminus o(\xxx)$. Even though $o(\xxx)$ and $\bar o(\xxx)$ are really index sets, we often refer to them as the set of $\alpha$-arcs occupied by $\xxx$, and the set of $\alpha$-arcs not occupied by $\xxx$. 

If $\HH$ is a diagram for an $(m,n)$-tangle, then for $\xx\in \gens$ we define $o^0(\xx)\coloneqq \{i | \xx\cap \alpha_i^0\neq \emptyset\}$, $\bar o^0(\xxx) \coloneqq[m+1]\setminus o^0(\xxx) $, $o^1(\xx)\coloneqq \{i | \xx\cap \alpha_i^1\neq \emptyset\}$, and $\bar o^1(\xxx) \coloneqq[n+1]\setminus o^1(\xxx)$.

\remark If $\HH$ is a type $1$ or a type $2$ diagram, then exactly $n$  or $n+1$ of the $\alpha$-arcs, respectively, are occupied by each generator. If $\HH$ is a diagram with two boundary components, the total number of occupied $\alpha$-arcs on the two sides is $\frac{m+ n}{2}+1$, but the number on each side may vary.

\subsection{Homology classes}
We will soon count pseudoholomorphic curves that connect generators. Each such curve carries a homology class, defined as follows. 

\begin{definition}
Fix generators $\xxx$ and $\yyy$, and let $I$ be the interval $[0,1]$. Let $\pi_2(\xxx, \yyy)$, the \emph{homology classes from $\xxx$ to $\yyy$}, denote the elements of 
$$H_2(\Sigma\times I\times I, ((\balpha\times \{1\}\cup \bbeta\times \{0\}\cup(\bdy\Sigma\setminus \zzz)\times I)\times I)\cup (\xxx\times I\times \{0\})\cup (\yyy\times I\times\{1\}))$$
which map to the relative fundamental class of $\xxx\times I\cup \yyy\times I$ under the composition of the boundary homomorphism and collapsing the rest of the boundary. 
\end{definition}

\begin{definition}
Given a homology class $B\in \pi_2(\xxx, \yyy)$, its \emph{domain} $[B]$ is the projection of $B$ to $H_2(\Sigma, \balpha\cup\bbeta\cup \bdy\Sigma)$. We can interpret the domain of $B$ as a linear combination of the components of $\Sigma\setminus (\balpha\cup \bbeta)$, which we call  \emph{regions}.
\end{definition}

Note that a homology class is uniquely determined by its domain.

\begin{definition}
 The coefficient of each region in a domain is called its \emph{multiplicity}. Given a point $p\in \Sigma\setminus(\balpha\cup \bbeta)$, we denote by $n_p(B)$ the multiplicity of $[B]$ at the region containing $p$. Alternatively, $n_p(B)$ is the intersection number of $B$ and $\{p\}\times I\times I$. 
 \end{definition}
 
By definition, the multiplicity of $[B]$ at  any region $D$ that contains a point in $\zzz$ is zero.
 
 \begin{definition}
 We define the set of \emph{empty homology classes} as
$$\hat\pi_2(\xxx,\yyy)\coloneqq \{B\in \pi_2(\xxx,\yyy)| n_X(B) = 0 \textrm{ and } n_O(B) = 0 \textrm{ for all } X\in \XX \textrm{ and } O\in \OO\}.$$
\end{definition}
 To define our Floer invariants, we will only be interested in this smaller set $\hat\pi_2(\xxx, \yyy)$. 
 
Concatenation at $\yyy\times I$, which corresponds to addition of domains, gives products $\ast:\pi_2(\xxx, \yyy)\times \pi_2(\yyy, \www)\to \pi_2(\xxx, \www)$ and   $\ast:\hat\pi_2(\xxx, \yyy)\times \hat\pi_2(\yyy, \www)\to \hat\pi_2(\xxx, \www)$. This operation turns $\pi_2(\xxx, \xxx)$ and $\hat\pi_2(\xxx, \xxx)$ into  groups called the group of \emph{periodic domains} and the group of \emph{empty periodic domains}, respectively.

We can split the boundary of a domain $[B]$ into three pieces,  $\bdy^{\bdy}B\subset \bdy\Sigma$, $\bdy^{\alpha}B\subset \balpha$, and $\bdy^{\beta}B\subset\bbeta$, oriented so that $\bdy^{\bdy}B+ \bdy^{\alpha}B+ \bdy^{\beta}B$ is the boundary of $[B]$. We can think of $\bdy^{\bdy}B$ as an element of $H_1(\bdy \Sigma, \bdy\balpha)$. For a Heegaard diagram $\HH$ with two boundary components, we can further split $\bdy^{\bdy}B$ into two pieces, $\bdy^i B\subset \bdy^i\HH$, so that $\bdy^{\bdy}B = \bdy^0B+ \bdy^1B$.

\begin{definition}
A homology class $B$ is called \emph{provincial} if $\bdy^{\bdy}B = 0$. 
For a diagram with two boundary components, a homology class $B$ is called \emph{left-provincial} if $\bdy^0B = 0$, and \emph{right-provincial} if $\bdy^1B = 0$. We denote the set of empty provincial homology classes from $\xxx$ to $\yyy$ by $\hat\pi_2^{\bdy}(\xxx, \yyy)$. 
\end{definition}
Observe that concatenation turns $\hat\pi_2^{\bdy}(\xxx, \xxx)$ into a group.

\subsection{Admissibility}

In order to get well-defined Heegaard-Floer invariants, we need to impose  some additional conditions on the tangle Heegaard diagrams. 

\begin{definition}
A tangle Heegaard diagram is called \emph{admissible} if every non-zero empty periodic domain has both positive and negative multiplicities. 

A tangle Heegaard diagram is called \emph{provincially admissible} if every non-zero empty provincial periodic domain has both positive and negative multiplicities. 

A tangle Heegaard diagram with two boundary components is called \emph{left} (respectively \emph{right}) \emph{admissible} if every non-zero empty right-provincial (respectively left-provincial) periodic domain has both positive and negative multiplicities.
\end{definition}

\begin{proposition}
Any tangle Heegaard diagram can be made admissible by performing isotopy on $\bbeta$. Further, any two admissible diagrams for a given $2n$-tangle or  a strongly marked $(m,n)$-tangle  are connected through a sequence of Heegaard moves, so that every intermediate diagram is admissible too. The same is true if we replace ``admissible" by ``provincially admissible".
\end{proposition}

\begin{proof}
This follows from a winding argument for the $\beta$-curves, just as in the case for closed manifolds. \cite[Section 5]{osz8}. Alternatively, see \cite[Proposition 4.11]{bs}
\end{proof}

\begin{corollary}
Every tangle $(Y, \T)$ has an admissible tangle Heegaard diagram. Similarly, Every tangle $(Y, \T)$ has a provincially admissible tangle Heegaard diagram. The same statements hold for every strongly marked tangle.
\end{corollary}

\subsection{Gluing}\label{ssec:gluing}

Any two multipointed bordered Heegaard diagrams can be glued along a matching boundary component: if $\HH_1$ and $\HH_2$ are diagrams, and $\zz_i$ are boundary components of $\HH_i$ with $\zz_1 = \zz_2^{\ast}$, one can glue $\HH_1$ to $\HH_2$ by identifying $\zz_1$ with $\zz_2^{\ast}$. In this way, one can glue a type $1$ diagram to the left i.e. $\bdy^0$, boundary of a diagram with two boundary components, a type $1$ diagram to a type $2$ diagram, a type $2$ diagram to the $\bdy^1$ boundary of a diagram with two boundary components, or the $\bdy^0$ boundary of a diagram with two boundary components to the  $\bdy^1$ boundary of another diagram with two boundary components. 

By gluing a type $1$ diagram, a sequence of diagrams with two boundary components, and a type $2$ diagram together, removing the union of the $\zzz$ markings, and placing an $X$ and an $O$ in the corresponding region, one obtains a closed Heegaard diagram for  the knot/link that is union of the corresponding tangles, together with an additional split unknot. See Figure \ref{fig:HD} for a schematic example. 

Below we describe in full detail how to glue Heegaard diagrams for tangles, and discuss the basic properties of the resulting diagram. 

For the rest of this section, we fix two Heegaard diagrams as follows. Let $\HH_1 =  (\Sigma_1, \balpha_1, \bbeta_1, \XX_1, \OO_1, \zzz_1)$ be a Heegaard diagram (of type $1$, or with two boundary components) for  some pair $(Y_1, \T_1)$, and if $\HH_1$ is of type $1$, denote its boundary by $\bdy^1\HH_1$. Let $\HH_2 =  (\Sigma_2, \balpha_2, \bbeta_2, \XX_2, \OO_2, \zzz_2)$ be a Heegaard diagram (of type $2$ or with two boundary components) for another pair $(Y_2, \T_2)$, and if $\HH_2$ is of type $2$, denote its boundary by $\bdy^0\HH_2$.  Suppose $\bdy^1 \HH_1=(\bdy^0 \HH_2)^{\ast}$, i.e. $\bdy^1 Y_1$ is identified with a marked sphere $\mathcal S$ and $\bdy^0 Y_2$ is identified with $-\mathcal S$.

\begin{definition} 
 The \emph{union} of $\HH_1$ and $\HH_2$, denoted $\HH_1\cup \HH_2$ is the Heegaard diagram $\HH$ obtained in the following way:
We remove all $\XX$ and $\OO$ markings on the boundaries of the two diagrams. We glue the two surfaces along their boundary, matching the $\alpha$ and $\zzz$ endpoints  and respecting the identification $\bdy^1 \HH_1=(\bdy^0 \HH_2)^{\ast}$, to obtain $\Sigma\coloneqq\Sigma_1\cup_{\bdy}\Sigma_2$.  We take $\balpha$ to be the set of circles $\balpha_1\cup_{\bdy}\balpha_2$, and we take $\bbeta$ to be $\bbeta_1\cup \bbeta_2$. If $\Sigma_1\cup \Sigma_2$ is a closed surface,  we remove $\zzz_1$ and $\zzz_2$, place two points marked $X'$ and $O'$ in the same region, and let $\XX = \XX_1\cup \XX_2\cup X'$ and $\OO = \OO_1\cup \OO_2\cup O'$. We get a closed Heegaard diagram  $(\Sigma, \balpha, \bbeta, \XX, \OO)$.  If $\Sigma_1\cup \Sigma_2$ has boundary, we let $\XX = \XX_1\cup \XX_2$ and $\OO = \OO_1\cup \OO_2$, and we take $\zzz$ to be the oriented arc(s) $\zzz_1\cup_{\bdy} \zzz_2$. We get a tangle Heegaard diagram  $\HH =  (\Sigma, \balpha, \bbeta, \XX, \OO, \zzz)$.
\end{definition}

Gluing Heegaard diagrams corresponds to gluing tangles. In the lemma below, all unions are formed by following the identifications with $\mathcal S$ given by the tangles. 

\begin{lemma}\label{lemma:gluehd}
When the union $\HH_1\cup \HH_2$ is a diagram with one  boundary component, it represents the pair $(Y_1\cup Y_2, \T_1\cup \T_2)$. 

When $\HH_1\cup \HH_2$ is a diagram with two  boundary components, it represents $(Y_1\cup Y_2, \T_1\cup \T_2, \gamma_1\cup \gamma_2)$. 

When $\HH_1\cup \HH_2$  is a  closed Heegaard diagram, it represents the link $(\T_1\cup \T_2)\cup U$ in $Y_1\cup Y_2$, where $U$ is an unknot unlinked from $\T_1\cup \T_2$.
\end{lemma}

\begin{proof}
This follows directly from the topological constructions of tangles from the 3 types of Heegaard diagrams described in Sections \ref{ssec:hd1} and \ref{ssec:hd2}, along with the fact that adding an $X$ and an $O$ in one and the same region introduces an unlinked unknot. 
\end{proof}

Generators and homology classes behave nicely under gluing. Let $\HH_1$ and $\HH_2$ be two tangle Heegaard diagrams  which agree along a boundary component. Note that given $\xx_1 \in\mathfrak S(\HH_1)$ and $\xx_2\in\mathfrak S(\HH_2)$ such that $\xx_1$ and $\xx_2$ occupy complementary sets of the new $\alpha$ circles obtained by gluing  $\alpha$ arcs, the union $\xx_1\cup\xx_2$ is a generator in $\mathfrak S(\HH_1\cup \HH_2)$. 

\begin{lemma}
Given  $\xx_1, \yy_1 \in\mathfrak S(\HH_1)$ and $\xx_2, \yy_2 \in\mathfrak S(\HH_2)$, there is a natural identification of $\pi_2(\xx_1\cup\xx_2, \yy_1\cup\yy_2)$ with the set of pairs $(B_1, B_2)$ in $\pi_2(\xx_1, \yy_1)\times \pi_2(\xx_2, \yy_2)$ such that $\bdy^1B_1 = -\bdy^0B_2$. The same statement holds if we replace $\pi_2$ with $\hat\pi_2$. 
\end{lemma}

\begin{proof}
The proof is straightforward. 
\end{proof}

Following notation from \cite{bfh2}, for $B_1$ and $B_2$ which agree along the boundary as above, we denote the corresponding homology class in  $\pi_2(\xx_1, \yy_1)\times \pi_2(\xx_2, \yy_2)$ by $B_1\natural B_2$.   Under this identification, the local multiplicity of $B_i$ at a point $p\in \Sigma_i\setminus(\balpha_i\cup \bbeta_i)$ agrees with the local multiplicity of $B_1\natural B_2$ at $p$ thought of as a point in $\Sigma_1\cup \Sigma_2$. 

\begin{lemma}
Suppose $\HH_1$ and $\HH_2$ are of type $1$ and type $2$, respectively. If one diagram is admissible, and the other one is provincially admissible, then $\HH_1\cup \HH_2$ is admissible.
\end{lemma}
 
\begin{proof}
The proof is identical to the proof of \cite[Lemma 4.33]{bfh2}, and we recall the argument here. Let $B_1\natural B_2$ be a positive periodic domain. If $\HH_1$ is admissible, then $B_1 = 0$, so $\bdy^{\bdy}B_1 = 0$, and since $\HH_2$ is provincially admissible, it follows that $B_2 =0$. Similarly, if $B_2$ is admissible, it follows that $B_1 =0$ and $B_2 = 0$.
\end{proof}

\begin{lemma}(Compare to \cite[Lemma 5.22]{bimod})
Suppose $\HH_1$ and $\HH_2$ are provincially admissible multipointed bordered Heegaard diagrams with two boundary components with $\bdy^1 \HH_1=(\bdy^0 \HH_2)^{\ast}$, and let $\HH = \HH_1\cup \HH_2$. If $\HH_1$ is right admissible, or $\HH_2$ is left admissible, then $\HH$ is provincially admissible. Furthermore: 
\begin{enumerate}
\item If $\HH_1$ and $\HH_2$ are both  left admissible (respectively right admissible), then  $\HH$ is left admissible (respectively right admissible).
\item If $\HH_1$ is admissible, then $\HH$ is left admissible. If $\HH_2$ is admissible, then $\HH$ is   right admissible.
\item If $\HH_1$ is admissible and $\HH_2$ is right admissible, or if $\HH_1$ is left admissible and $\HH_2$ is admissible, then $\HH$ is admissible.
\end{enumerate}
Analogous statements hold when one of the two Heegaard diagrams has one boundary component.
\end{lemma}

\begin{proof}
The proof is analogous to that of \cite[Lemma 5.22]{bimod}
\end{proof}

\section{Moduli spaces}
\label{sec:moduli}

In this section, we describe the holomorphic curves that will be considered in the definitions of the various invariants associated to tangle Heegaard diagrams.

Most of this discussion is a straightforward generalization of the one for  bordered Floer homology \cite{bfh2}. We count pseudoholomorphic curves in $\Sigma\times I\times \R$. In the bordered Floer setting, one counts curves that avoid a basepoint $z\in \bdy\Sigma$. Here, we avoid multiple basepoints, both in the interior, and on the boundary of $\Sigma$, as well as the arcs (or points) that we denote by $\zzz$.   

\subsection{Moduli spaces of holomorphic curves} \label{ssec:moduli}

Let $\HH = (\Sigma, \balpha, \bbeta, \XX, \OO, \zzz)$ be a tangle Heegaard diagram (with one boundary component and of type $1$ or type $2$, or with two boundary components).  We can think of the open surface $\Int(\Sigma)$  as a surface with a set of punctures $\bf{p}$ (one puncture for each boundary component of $\Sigma$).  
Choose a symplectic form $\omega_{\Sigma}$ such that the boundary $\bdy \Sigma$ is a cylindrical end, and let $j_{\Sigma}$ be a compatible almost complex structure. We will assume  that $\balpha^a$ is cylindrical near $\bdy\Sigma$, in the following sense. There is a neighborhood $U_{\bf p}$ of the punctures symplectomorphic to $\bdy \Sigma\times (0,\infty)\subset T^{\ast}(\bdy \Sigma)$, such that $j_{\Sigma}$ and $\balpha^a\cap U_{\bf p}$ are invariant with respect to the $\R$-action on $\bdy \Sigma\times (0,\infty)$. We write $\mathbb D = I\times \R$, and let $\omega_{\mathbb D}$ and $j_{\mathbb D}$ be the standard symplectic form and almost complex structure on $\mathbb D\subset \C$. Consider the projections
\begin{align*}
\pi_{\Sigma}:\Sigma\times \mathbb D &\to \Sigma\\
\pi_{\mathbb D}: \Sigma\times \mathbb D &\to \mathbb D\\
s: \Sigma\times \mathbb D &\to I\\
t: \Sigma\times \mathbb D &\to \R 
\end{align*}

\begin{definition}
We say that an almost complex structure $J$ on $\Sigma\times \mathbb D$ is \emph{admissible} if the following conditions hold:
\begin{itemize}
\item[-] $\pi_{\mathbb D}$ is $J$-holomorphic
\item[-] $J(\bdy_s) = \bdy_t$ for the vector fields tangent to the fibers of $\pi_{\Sigma}$
\item[-] The $\R$-action is $J$-holomorphic
\item[-] $J$ splits as $J = j_{\Sigma}\times j_{\mathbb D}$ near $\bf{p}\times \mathbb D$
\end{itemize}
\end{definition}

\begin{definition}
A \emph{decorated source} $S^{\triangleright}$ consists of
\begin{itemize}
\item[-] a topological type of smooth Riemann surface $S$ with boundary, and a finite number of punctures on the boundary
\item[-] a labeling of each puncture of $S$ by $+$, $-$, or $e$
\item[-] a labeling of each $e$ puncture by a Reeb chord $\rho$ in $(\bdy\Sigma, \bdy\balpha)$
\end{itemize}
\end{definition}

Given a decorated source $S^{\triangleright}$, we denote by $S_{\bar e}$ the result of filling in the $e$ punctures of $S$.

We consider maps
$$u:(S, \bdy S)\to (\Sigma\times \mathbb D, (\balpha\times \{1\}\times \R)\cup (\bbeta\times\{0\}\times \R))$$
such that
\begin{enumerate}
\item $u$ is $(j, J)$ holomorphic for some almost complex structure $j$ on $S$.\label{1}
\item $u:S\to \Sigma\times \mathbb D$ is proper.
\item $u$ extends to a proper map $u_{\bar e}:S_{\bar e}\to \Sigma_{\bar e}\times \mathbb D$
\item $u$ has finite energy in the sense of Bourgeois, Eliashberg, Hofer,
Wysocki and Zehnder \cite{behwz}.
\item $\pi_{\mathbb D}\circ u$ is a $g$-fold branched cover.
\item At each $+$ puncture $q$ of $S$, $\lim_{z\to q}t\circ u(z) = +\infty$.
\item At each $-$ puncture $q$ of $S$, $\lim_{z\to q}t\circ u(z) = -\infty$.
\item At each $e$ puncture $q$ of $S$, $\lim_{z\to q}\pi_{\Sigma}\circ u(z)$ is the Reeb chord $\rho$ labeling $q$.
\item $\pi_{\Sigma}\circ u:S\to \Int(\Sigma)$ does not cover any of the regions of $\Sigma\setminus (\balpha\cup\bbeta)$ that intersect $\XX\cup \OO\cup\zzz$.
\item For each $t\in \R$ and $\beta_i\in \bbeta$, $u^{-1}(\beta_i\times \{0\}\times \{t\})$ consists of exactly one point; for each $t\in \R$ and $\alpha^c_i\in \balpha$, $u^{-1}(\alpha^c_i\times \{1\}\times \{t\})$ consists of exactly one point; for each $t\in \R$ and $\alpha^a_i\in \balpha$, $u^{-1}(\alpha^a_i\times \{1\}\times \{t\})$ consists of at most one point. \label{10}
\item $u$ is embedded.\label{11}
\end{enumerate}

Under these conditions, at $-\infty$, $u$ is asymptotic to a $g$-tuple of arcs $x_i\times I\times \{-\infty\}$, and at $+\infty$, $u$ is asymptotic to a $g$-tuple of arcs $y_i\times I\times \{+\infty\}$, so that $\xxx \coloneqq\{x_1, \ldots, x_g\}$ and $\yyy \coloneqq\{y_1, \ldots, y_g\}$ are generators of $\HH$. We call $\xxx$ the \emph{incoming} generator, and $\yyy$ the \emph{outgoing} generator for $u$. Such a curve $u$ has an associated homology class $B=[u]\in \pi_2(\xxx, \yyy)$.

\begin{definition}
Given a map $u$ from a decorated source $S^{\triangleright}$, the \emph{height} of an $e$ puncture $q$ is the evaluation $\mathrm{ev}(q) = t\circ u_{\bar e}(q)\in \R$.
\end{definition}

\begin{definition}
Let $E(S^{\triangleright})$ be the set of $e$ punctures of $S$. Let $\vec P = (P_1, \ldots, P_m)$ be a partition of $E(S^{\triangleright})$ with $P_i$ nonempty. We say a map $u$ is \emph{$\vec P$-compatible} if for any $i$, all the punctures in $P_i$ have the same height, and $\mathrm{ev}(P_i)  <\mathrm{ev}(P_j)$ whenever $i<j$.
\end{definition}

To a partition $\vec P = (P_1, \ldots, P_m)$ we  associate a sequence of sets of Reeb chords $\vec{\brho}(\vec P) = (\brho_1, \ldots, \brho_m)$ where $\brho_i \coloneqq \{\rho | \rho \textrm{ labels } q, q\in P_i\}$. To such a sequence $\vec{\brho}$ we can associate a homology class
 $$[\vec{\brho}] = [\brho_1] + \cdots + [\brho_m]\in H_1(\bdy\Sigma,\bdy\balpha)$$
 and an algebra element
 $$a(\vec{\brho}) = a(\brho_1)\cdots a(\brho_m)\in \cala(\bdy\HH).$$
Note that $[a(\vec{\brho})] = [\vec{\brho}]$, and also if $u$ is a $\vec P$-compatible map satisfying (\ref{1})-(\ref{10}) with homology class $[u] = B$, then $[\vec{\brho}(\vec P)] = \bdy^{\bdy}B$. 

\begin{definition}
Given generators $\xxx$ and $\yyy$, a homology class $B\in \pi_2(\xxx, \yyy)$,  and a decorated source $S^{\triangleright}$, we let
$$\widetilde{\mathcal M}^B(\xxx,\yyy, S^{\triangleright})$$
denote the moduli space of curves $u$ with source $S^{\triangleright}$ satisfying (\ref{1})-(\ref{10}), asymptotic to $\xxx$ at $-\infty$ and to $\yyy$ at $+\infty$, and with homology class $[u]=B$.
Given also a partition $\vec P$ of $E( S^{\triangleright})$, we let
$$\widetilde{\mathcal M}^B(\xxx,\yyy, S^{\triangleright}, \vec P)$$
denote  the space of $\vec P$-compatible maps in $\widetilde{\mathcal M}^B(\xxx,\yyy; S^{\triangleright})$, and we let
$$\widetilde{\mathcal M}^B_{\textrm{emb}}(\xxx,\yyy, S^{\triangleright}, \vec P)$$ denote the space of maps in $\widetilde{\mathcal M}^B(\xxx,\yyy, S^{\triangleright}, \vec P)$ that also satisfy (\ref{11}).
\end{definition}

Many results carry over directly from the ones in \cite{bfh2} and \cite{bs}.

\begin{proposition}(Compare to \cite[Proposition 5.6]{bfh2})
There is a dense set of admissible $J$ for which the spaces $\widetilde{\mathcal M}^B(\xxx,\yyy, S^{\triangleright}, \vec P)$ are transversally cut out by the $\bar\bdy$ equations.
\end{proposition}

\begin{proposition}(Compare to \cite[Proposition 5.8]{bfh2})
The expected dimension $\mathrm{ind}(B, S^{\triangleright}, \vec P)$ of   $\widetilde{\mathcal M}^B(\xxx,\yyy, S^{\triangleright}, \vec P)$ is
$$\mathrm{ind}(B, S^{\triangleright}, \vec P) = g - \chi(S) + 2e(B)+ |\vec P|.$$
Here $e(B)$ is the Euler measure of the domain of $B$, and $|\vec P|$ is the number of parts in the partition $\vec P$.
\end{proposition}

Whether a curve in $\widetilde{\mathcal M}^B(\xxx,\yyy, S^{\triangleright}, \vec P)$ is embedded depends only on the topological data of $B$, $S^{\triangleright}$, and $\vec P$, i.e. there are entire components of embedded and of non-embedded curves. For embedded curves, there is another index formula that only depends on $B$ and $\vec P$. Before we state this formula, we make a couple of definitions regarding Reeb chords. Even though our matched circles are different, these definitions are identical to the ones in  \cite[Sections 3.3.1, 5.7.1]{bfh2}.

Let $m:H_1(\bdy\Sigma\setminus\zzz, \bdy\balpha; \Z)\times H_0(\bdy\balpha; \Z)\to \frac{1}{2}\Z$ be the map that counts local multiplicities. Specifically, for $a\in H_1(\bdy\Sigma\setminus\zzz, \bdy\balpha; \Z)$ and $p\in \bdy\balpha$, we define  the \emph{multiplicity} $m(a,p)$ of $p$ in $a$ as the average multiplicity with which $a$ covers the regions on either side of $p$, and extend bilinearly.

For $a,b\in H_1(\bdy\Sigma\setminus\zzz, \bdy\balpha; \Z)$, define
$$L(a,b) \coloneqq m(b, \bdy a),$$
where $\bdy$ is the connecting homomorphism from the homology long exact sequence. Note that $L(a,b) = -L(b,a)$ for any $a,b$.

For a set of Reeb chords $\brho$ in  $(\bdy\Sigma\setminus\zzz, \bdy\balpha)$,  define
$$\iota(\brho) \coloneqq -\sum_{\{\rho_i,\rho_j\}\subset \brho}|L([\rho_i],[\rho_j])| - \frac{|\brho|}{2}.$$

For a sequence of sets of Reeb chords $\vec{\brho} = (\brho_1, \ldots, \brho_m)$, define
$$\iota(\vec{\brho})\coloneqq \sum_{i}\iota(\brho_i)+ \sum_{i<j}L([\brho_i], [\brho_j]).$$

 Finally, we come to the index formula.

\begin{definition}
Let $B\in \pi_2(\xx, \yy)$ and $\vec{\brho}$ be a sequence of sets of Reeb chords. We define
\begin{align*}
\chi_{\textrm{emb}}(B, \vec{\brho}) &\coloneqq g + e(B) - n_{\xx}(B) - n_{\yy}(B) - \iota(\vec{\brho}),\\
\mathrm{ind}(B, \vec{\brho}) &\coloneqq e(B) + n_{\xx}(B) + n_{\yy}(B) + |\vec{\brho}| +  \iota(\vec{\brho}).
\end{align*}
\end{definition}

\begin{proposition}(Compare to \cite[Proposition 5.62]{bfh2} and \cite[Proposition 5.9]{bs})
For $u\in \widetilde{\mathcal M}^B(\xxx,\yyy, S^{\triangleright}, \vec P)$, either $u$ is embedded, and 
\begin{align*}
\chi(S^{\triangleright}) &=\chi_{\mathrm{emb}}(B, \vec{\brho}(\vec P)) ,\\
\mathrm{ind}(B, S^{\triangleright}, \vec P) &= \mathrm{ind}(B, \vec{\brho}(\vec P)),\\
\widetilde{\mathcal M}^B_{\textrm{emb}}(\xxx,\yyy, S^{\triangleright}, \vec P) &= \widetilde{\mathcal M}^B(\xxx,\yyy, S^{\triangleright}, \vec P),
\end{align*}
or $u$ is not embedded, and 
\begin{align*}
\chi(S^{\triangleright}) &>\chi_{\mathrm{emb}}(B, \vec{\brho}(\vec P)) ,\\
\mathrm{ind}(B, S^{\triangleright}, \vec P) &< \mathrm{ind}(B, \vec{\brho}(\vec P)),\\
\widetilde{\mathcal M}^B_{\textrm{emb}}(\xxx,\yyy, S^{\triangleright}, \vec P) &= \emptyset.
\end{align*}
\end{proposition}

Each of the moduli spaces has an $\R$-action by translation in the $t$ factor. For \emph{stable} curves, i.e. except when the moduli space consists of a single curve $u$ with $\pi_{\mathbb D}\circ u$ a trivial $g$-fold cover of $\mathbb D$ and $B=0$, this action is free. For moduli spaces of stable curves, we quotient by this  action.

\begin{definition} Given $\xxx$, $\yyy$, $S^{\triangleright}$, and $\vec P$, let
\begin{align*}
\mathcal M^B(\xxx,\yyy, S^{\triangleright}, \vec P) &\coloneqq \widetilde{\mathcal M}^B(\xxx,\yyy, S^{\triangleright}, \vec P)/\R\\
\mathcal M^B_{\textrm{emb}}(\xxx,\yyy, S^{\triangleright}, \vec P) &\coloneqq \widetilde{\mathcal M}^B_{\textrm{emb}}(\xxx,\yyy, S^{\triangleright}, \vec P)/\R
\end{align*}
\end{definition}

\subsection{Degenerations} The properties of moduli spaces that are needed in order to show that the invariants  are well-defined are the same as in \cite{bfh2}. To understand the compactifications of moduli spaces, one studies \emph{holomorphic combs}, i.e. 
trees of homomorphic curves in $\Sigma\times \mathbb D$ and in $\bdy \Sigma\times \R\times \mathbb D$. In the \emph{tilde} version (when one does not allow domains that cover $\XX\cup \OO$), most types of degenerations are the same as in \cite{bfh2}, and most results carry over. 

The only difference is in the homological assumptions on $\Sigma$, $\balpha$, $\bbeta$. Even though the $\alpha$ or the $\beta$ circles are not linearly independent, there are no new boundary degenerations, as every region of $\Sigma\setminus\balpha$, as well as every region of $\Sigma\setminus\bbeta$, contains an $X$ or an $O$,

\section{The modules associated to tangle Heegaard diagrams}
\label{sec:hd_modules}

In this section, we associate algebraic structures to tangle Heegaard diagrams. Before we proceed, recall that for any pointed matched circle $\zz$, the algebra $\az$ does not depend on the $\XX$ and $\OO$ markings on the circle. 

For the remainder of this paper, we let $V$ denote $\F_2\otimes \F_2$.

\subsection{The type $D$ structure}

We define type $D$ structures for type $2$ mutipointed  bordered Heegaard diagrams for tangles. The construction and results for  type $1$ diagrams are identical.

Suppose $\HH = (\Sigma, \balpha, \bbeta, \XX, \OO, \zzz)$ is a provincially admissible Heegaard diagram of type $2$ for some $2n$-tangle $(Y, \T)$. Let $J$ be an admissible almost complex structure.
We define a left type $D$ structure $\cftdt (\mathcal H, J)$ over $\mathcal A(-\bdy\mathcal H)$, as follows. 

Let $X(\mathcal H)$ be the $\F_2$ vector space spanned by $\mathfrak S(\mathcal H)$. Define $I_D(\xx)=I(\bar o(\xx))\in \mathcal I(-\bdy\HH)$. We define an action  on $X(\mathcal H)$ of $\mathcal I(-\bdy\mathcal H)$ by
$$I(\sss)\cdot \xx = \left\{  \begin{array}{ll}
\xx & \textrm{ if } I(\sss) = I_D(\xx)\\ 
0 & \textrm{ otherwise. } 
\end{array}\right.$$
Then $\cftdt(\mathcal H, J)$ is defined as an $\mathcal A(-\bdy\mathcal H)$-module by 
$$\cftdt (\mathcal H, J) = \mathcal A(-\bdy\mathcal H)\otimes_{\mathcal I(-\bdy\mathcal H)}X(\mathcal H).$$
Given $\xx, \yy\in \mathfrak S(\HH)$, we define
$$a_{\xx,\yy}\coloneqq \sum_{\substack{B\in \hat\pi_2(\xxx, \yyy)\\ \vec P  \textrm{ discrete}\\ \ind(B, \vec{\brho}(\vec P)) = 1}} \#\mathcal M^B_{\textrm{emb}}(\xxx,\yyy, S^{\triangleright}, \vec P)\cdot a(-P_1)\cdots a(-P_m).$$
Here all $P$ are discrete partitions, i.e. partitions $P=(P_1, \ldots, P_m)$ where $|P_i| = 1$.

The map $\delta:\cftdt(\mathcal H, J)\to\cala(-\bdy\HH)\otimes \cftdt(\mathcal H, J)$ is defined as
$$\delta(\xx) = \sum_{\yy\in \mathfrak S(\mathcal H)}a_{\xx,\yy}\otimes \yy.$$

\begin{theorem}
Let $\HH$ be a tangle Heegaard diagram of type $2$ for a marked tangle $\T$ in a $3$-manifold $Y$, equipped with an admissible almost complex structure $J$. If $\HH$ is provincially admissible, then $\cftdt(\HH, J)$ is a type $D$ structure over $\cala(-\bdy \HH)$. Moreover, if $\HH$ is admissible, then  $\cftdt(\HH, J)$ is bounded. 
\end{theorem}

\begin{proof}
The proof follows directly from the arguments for $\cfdhat$ in \cite[Chapter 6]{bfh2}. We outline the main steps. To show that the counts of holomorphic curves are finite, we observe that in a provincially admissible diagram there are only finitely many domains that contribute to the counts, and for any diagram there are only finitely many sequences $\vec P$ with nonzero $a(\vec{\brho}(\vec P))\in \cala(-\bdy\HH)$. To show that the compatibility condition for a type $D$ structure is satisfied, we count  possible degenerations of holomorphic curves. 
\end{proof}

\begin{theorem}\label{thm:dinv}
Up to homotopy equivalence and tensoring with $V$, $\cftdt(\HH, J)$ is independent of the choice of sufficiently generic admissible almost complex structure, and provincially admissible type $2$ tangle Heegaard diagram for $(Y, \T)$. Namely, if $\HH_1$ and $\HH_2$ are provincially admissible type $2$ diagrams for $(Y, \T)$ with almost complex structures $J_1$ and $J_2$, and $|\X_1| = |\X_2|+k$, then 
$$\cftdt(\HH_1, J_1)\simeq \cftdt(\HH_2, J_2)\otimes V^{\otimes k}.$$
\end{theorem}

\begin{proof}
To show invariance, we construct chain maps corresponding to a change of almost complex structure or the various Heegaard moves. We have two Heegaard moves that do not occur in \cite{bfh2} -- index zero/three stabilization and destabilization. Those always occur in the interior of the diagram, and result in the extra $V$, by the same argument as in the closed case (see, for  example, \cite{mos}), i.e. if $\HH'$ is obtained from $\HH$ by an index zero/three stabilization, then $\cftdt(\HH)\simeq \cftdt(\HH')\otimes V$.
\end{proof}

When we write $\cftdhat(Y, \T)$, we mean the type $D$ structure without the extra $V$s, i.e. what we get from a tangle Heegaard diagram with the minimum number of basepoints, which is  $|\X\cap \Int(\Sigma)| = |\T| = |\O\cap \Int(\Sigma)|$, or equivalently $|\X| = 2 |\T| = |\O|$.

\subsection{The type $A$ structure}

We define type $A$ structures for type $1$ mutipointed  bordered Heegaard diagrams for tangles. The construction and results for  type $2$ diagrams are identical.

Let $\HH = (\Sigma, \balpha, \bbeta, \XX, \OO, \zzz)$ be a provincially admissible type $1$ Heegaard diagram for a $2n$-tangle $(Y, \T)$, and let $J$ be an admissible almost complex structure. We define a type $A$ structure $\cftat(\HH, J)$ over $\cala(\HH)$. 

Define $I_A(\xx)=  I(o(\xx))$. The module $\cftat(\mathcal H, J)$ is generated over $\F_2$ by $X(\mathcal H)$, and the right action of $\mathcal I(\bdy \mathcal H)$ on $\cftat(\mathcal H, J)$ is defined on the generators by
$$\xx\cdot I(\sss) = \left\{  \begin{array}{ll}
\xx & \textrm{ if } I(\sss) = I_A(\xx)\\
0 & \textrm{ otherwise. } 
\end{array}\right.$$
For the $\ainf$ multiplication maps, we consider partitions $P = (P_1, \ldots, P_m)$ that are not necessarily discrete. When $I_A(\xx)\otimes a(\brho_1)\otimes \cdots\otimes a(\brho_n)\neq 0$, we define $m_{n+1}:\cftat(\HH, J)\otimes\az^{\otimes n}\to \cftat(\HH, J)$ by

$$m_{n+1}(\xx, a(\brho_1), \ldots, a(\brho_n)) \coloneqq\sum_{\yy\in \mathfrak S(\HH)} \sum_{\substack{B\in \hat\pi_2(\xxx, \yyy)\\ \{\vec P| \vec{\brho}(\vec P) =(\brho_1, \ldots, \brho_n)\}\\\ind(B, \vec{\brho}(\vec P)) = 1}}  \#\mathcal M^B_{\textrm{emb}}(\xxx,\yyy, S^{\triangleright}, \vec P)\cdot \yy.$$

\begin{theorem}
Let $\HH$ be a tangle Heegaard diagram of type $1$ for a marked tangle $\T$ in a $3$-manifold $Y$, equipped with an admissible almost complex structure $J$. If $\HH$ is provincially admissible, then $\cftat(\HH, J)$ is an $\ainf$-module over $\cala(\bdy \HH)$. Moreover, if $\HH$ is admissible,  then $\cftat(\HH, J)$ is bounded. 
\end{theorem}

\begin{theorem}\label{thm:ainv}
Up to $\ainf$ homotopy equivalence and tensoring with $V$, $\cftat(\HH, J)$ is independent of the choice of sufficiently generic admissible almost complex structure, and provincially admissible type $1$  tangle Heegaard diagram for $(Y, \T)$. Namely, if $\HH_1$ and $\HH_2$ are provincially admissible type $1$ diagrams for $(Y, \T)$ with almost complex structures $J_1$ and $J_2$, and $|\X_1| = |\X_2|+k$, then 
$$\cftat(\HH_1, J_1)\simeq \cftat(\HH_2, J_2)\otimes V^{\otimes k}.$$
\end{theorem}

\begin{proof}
The proofs of the two theorems are analogous to those for $\cftdhat$, except that we consider more degenerations, since we also consider sequences of sets of Reeb chords. 
\end{proof}

When we write $\cftahat(Y, \T)$, we mean the $\ainf$-module that we get from a diagram with $|\X\cap \Int(\Sigma)| = |\T|$.

\subsection{The type $\DA$ bimodule}

We define type $\DA$ structures for tangle Heegaard diagrams with two boundary components. One can similarly define type $\AA$, $\DD$, and $\AD$ structures. 

Suppose $\HH = (\Sigma, \balpha, \bbeta, \XX, \OO, \zzz)$ is a provincially admissible diagram with two boundary components $\bdy^0\HH$ and $\bdy^1\HH$ for a strongly marked $(m,n)$-tangle $(Y, \T, \gamma)$. Let $J$ be an admissible almost complex structure. We will define a  type $\DA$ bimodule $\cftdat(\HH, J)$  over $\Aa(-\bdy^0\HH)$ and $\Aa(\bdy^1\HH)$.

As a left-right $(\mathcal I(-\bdy^0 \HH),\mathcal I(\bdy^1 \HH))$-bimodule, $\cftdat(\HH, J)$ is freely generated over $\F_2$ by $\mathfrak S(\HH)$, with actions of $\mathcal I(-\bdy^0 \HH)$ and $\mathcal I(\bdy^1 \HH)$ defined on the generators by 

\[I(\sss_0)\cdot \xx\cdot I(\sss_1) =  \left\{  \begin{array}{ll}
\xx & \textrm{ if } \sss_0 = \bar o^0(\xx) \text{ and } \sss_1 = o^1(\xx)\\
0 & \textrm{ otherwise. } 
\end{array}\right.\]

To define the type $\DA$ structure maps, we need to study slightly different moduli spaces than before. Given a decorated source $S^{\triangleright}$, let $E_i$  be the set $e$ punctures labeled by Reeb chords in $\bdy^i\HH$. We need to forget the relative heights of the punctures in $E_0$ to those in $E_1$.

\begin{definition}
Define the moduli space
\[ \mathcal M^B_{\textrm{emb}}(\xxx,\yyy, S^{\triangleright}, \vec P_0, \vec P_1) = \bigcup_{\vec P|_{E_i} = \vec P_i}    \mathcal M^B_{\textrm{emb}}(\xxx,\yyy, S^{\triangleright}, \vec P),  \]
and define the index
\[\ind(B, \vec\brho_0, \vec\brho_1) = e(B) + n_{\xx}(B) + n_{\yy}(B) + |\vec\brho_0| + |\vec\brho_1| + \iota(\vec\brho_0)+ \iota(\vec\brho_1), \]
where $\vec\brho_i$ is a sequence of sets of Reeb chords in $\bdy^i\HH$.
\end{definition}

On the $\bdy^0\HH$ side we will only allow discrete partitions. If $\vec P_0$ is discrete, and labeled by the sequence of Reeb chords $\vec\brho(\vec P_0) = (\rho_1, \ldots, \rho_i)$, define
\[a_0(\xx,\yy, \vec P_0)\coloneqq I( \bar o^0(\xx)) \cdot a(-\rho_1)\cdot \cdots\cdot a(-\rho_i)\cdot I( \bar o^0(\yy))\in \Aa(-\bdy^0\HH)\]

On the $\bdy^1\HH$ side we  allow arbitrary partitions, and define
\[a_1(\xx, \yy, P_1) \coloneqq I( o^1(\xx)) \cdot a(\brho_1)\otimes \cdots\otimes a(\brho_j)\cdot I(o^1(\yy))\in \Aa(\bdy^1\HH)^{\otimes j},\]
where $\vec\brho(\vec P_1) = (\brho_1, \ldots, \brho_j)$.

Finally, the structure maps are defined by 

\[\delta^1_{k}(\xx, a_1, \ldots, a_{k-1}) \coloneqq\sum_{\yy\in \mathfrak S(\HH)} \sum_{\substack{B\in \hat\pi_2(\xxx, \yyy)\\\ind(B, \vec{\brho}(\vec P_0), \vec{\brho}(\vec P_1)) = 1\\ a_1(\xx, \yy, P_1)  = a_1\otimes \cdots\otimes a_{k-1}}}  \#\mathcal M^B_{\textrm{emb}}(\xxx,\yyy, S^{\triangleright}, \vec P_1, \vec P_2)\cdot a_0(\xx,\yy, \vec P_0)\otimes \yy.\]

\begin{theorem}
Let $\HH$ be a diagram with two boundary components for a strongly marked $(m,n)$-tangle $(Y, \T, \gamma)$, equipped with an admissible almost complex structure $J$. If $\HH$ is provincially admissible, then $\cftdat(\HH, J)$ is a type $\DA$ bimodule over $\Aa(-\bdy^0 \HH)$ and $\Aa(\bdy^1 \HH)$. Moreover, if $\HH$ is admissible,  then $\cftat(\HH, J)$ is bounded. 
\end{theorem}

\begin{theorem}\label{thm:dainv}
Up to homotopy equivalence and tensoring with $V\coloneqq\F_2\oplus \F_2$, the bimodule $\cftdat(\HH, J)$ is independent of the choice of sufficiently generic admissible almost complex structure, and provincially admissible tangle Heegaard diagram for $(Y, \T, \gamma)$. Namely, if $\HH_1$ and $\HH_2$ are provincially admissible diagrams for $(Y, \T, \gamma)$ with almost complex structures $J_1$ and $J_2$, and $|\X_1| = |\X_2|+k$, then 
$$\cftdat(\HH_1, J_1)\simeq \cftdat(\HH_2, J_2)\otimes V^{\otimes k}.$$
\end{theorem}

\begin{proof}
The proofs  are analogous to those for type $D$ and type $A$ structures. 
\end{proof}

When we write $\cftdahat(Y, \T, \gamma)$, we mean the bimodule that  we get from a diagram with $|\X\cap \Int(\Sigma)| = |\T|$. For a tangle $\T$ in $S^2\times I$, there is a canonical framed arc $\gamma$ determined by the product structure on the $3$-manifold, With this framed arc, $\T$ becomes a strongly marked tangle, and we simply write $\cftdat(\T)$.

For here on, we suppress the almost complex structure $J$ from the notation, and write $\cftat(\HH)$, $\cftdt(\HH)$, and $\cftdat(\HH)$.

\subsection{Other diagrams and modules.} \label{ssec:other} Similarly, one can associate a type $A$ structure to a type $2$ diagram, a type $D$ structure to a type $1$ diagram, or a type $\AA$, $\DD$, or $\AD$ structure to a diagram with two boundary components. 

One can also define $\beta$-bordered or $\alpha$-$\beta$-bordered mutipointed Heegaard diagrams for tangles, in the spirit of \cite{hfmor}, and associate modules or bimodules, respectively. The bordered grid diagrams of Section \ref{sec:borderedgrid} are examples of such diagrams. 

\section{Gradings} 
\label{ssec:morse_gradings}

For now, we only discuss gradings when the tangle lies in $B^3$ or $S^2\times I$. In those cases, one can define a homological grading by $\Z$, which we call the \emph{Maslov} grading, and a second (internal) grading by $\frac 1 2 \Z$ which  we call the \emph{Alexander} grading. In this section, all domains are assumed to avoid $\zzz$.

\subsection{Algebra} 

Fix a (right type) marked matched circle $\zz$ for an $n$-marked sphere, and let $\E =  (n+1, n+1, \id_{\SX},\id_{\TO})$ be the corresponding shadow, as in Section \ref{sec:reeb_alg}. Recall that the algebra $\az$ does not depend on the $\XX$ and $\OO$ markings on $\zz$, and equivalently, as an ungraded algebra, $\widehat\Aa(\E)=\Aa(\E)/(U_i= 0)$ does not depend on the sets $\SX$ and $\TO$. However, $\XX$ and $\OO$ markings play an important role in the bigrading on $\Aa(\E)$ defined in Section  \ref{ssec:grading}. In fact, to be consistent with  the algebra action for the bordered grid diagrams from Section \ref{sec:borderedgrid}, we need to reverse the roles of the markings, i.e. from now on we work with $\bar\E =  (n+1, n+1, \id_{\bar \SX},\id_{\bar \TO})$, where $\bar \SX = \{1\frac 1 2, \ldots, n\frac 1 2\}\setminus \SX$ and $\bar \TO = \{1\frac 1 2, \ldots, n\frac 1 2\}\setminus \TO$. 

There is an  obvious ungraded isomorphism $\widehat\Aa(\E)\cong \widehat\Aa(\bar \E)$ (preserve a strand diagram and just replace the types of horizontal strands in the underlying shadow), so instead of thinking about 
the isomorphism $\Aa(\zz)\cong\widehat\Aa(\E)$  from Section \ref{sec:reeb_alg}, we can equivalently think of $\az$ as isomorphic to $\widehat\Aa(\bar \E)$, using the same correspondence of strand diagrams as in Section \ref{sec:reeb_alg}. 
Namely,  a generator $\stphi$  for $\widehat\Aa(\bar\E)$ is mapped to the generator $I({\bf s})a(\brho)\in \Aa(\zz(\mathcal S))$ with  starting idempotent ${\bf s} = S$ and the following set of Reeb chords $\brho$: the Reeb chord from $i$ to $\phi(i)$ if $\phi(i)>i$, and the Reeb chord from $2n+3-i$ to $2n+3 - \phi(i)$  if $\phi(i)<i$. 

 The Maslov and Alexander gradings on $\Aa(\bar\E)$ defined in  Section \ref{ssec:grading} descend to gradings on $\widehat\Aa(\bar\E)$, and thus to gradings $M$ and $A$ on $\Aa(\zz)$ under the isomorphism discussed above. The Maslov grading turns $\Aa(\zz)$ into a differential graded algebra, and the Alexander grading is preserved by the differential and multiplication. We caution the reader that while the generators  $I(\sss)a(\brho)$ are homogeneous with respect to the gradings, $a(\brho)$ are not. 
 
 \subsection{Domains}
Let $\HH = (\Sigma, \balpha, \bbeta, \XX, \OO, \zzz)$ be a tangle diagram of any type.  To define the Maslov and Alexander grading of a domain $B\in \pi_2(\xx, \yy)$, we will make use of the  multiplicity map $m:H_1(\bdy\Sigma\setminus\zzz, \bdy\balpha; \Z)\times H_0(\bdy\balpha; \Z)\to \frac{1}{2}\Z$ from Section \ref{ssec:moduli} to record how $\bdy^{\bdy}B$ interacts with $\xx$, $\yy$, $\OO$, and $\XX$.

If $\HH$ is a diagram with one boundary component (of type $1$ or $2$) for a $2n$-tangle, define the following sets. 
For $1\leq i\leq 2n$, if there is an $X$ on $\bdy \Sigma$ between $a_i$ and $a_{i+1}$, then we place an $\OO$ marking in the interior of the segment $(a_i, a_{i+1})\subset \bdy \Sigma$ and another $\OO$ marking in the segment $(a_{4n+2-i}, a_{4n+3-i})$. In other words, there is a new $O$ on every component of $\bdy \Sigma\setminus \balpha$ that is on the boundary of a component of $\Sigma\setminus\balpha$ with an $O$ in the interior. Denote the set of all new $\OO$ markings by $\SO^{\bdy}$. Define a set $\SX^{\bdy}$ analogously. 
Given a generator $\xx\in \gens(\HH)$,  define $S_{\xx}$ to be the set of points in $\bdy\balpha$ that lie on $\alpha$-arcs occupied by $\xx$, and define  $\bar{S_{\xx}} \coloneqq \bdy\balpha\setminus S_{\xx}$. 

If $\HH$ is a diagram with two boundary components $\bdy^i\HH$ for $i= 0,1$, define sets $\SO^i$, $\SX^i$,$S_{\xx}^i$, and $\bar{S_{\xx}^i}$  analogously. 

When counting multiplicities below, we view a subset $S\subset \bdy \balpha$  as the element of $H_0(\bdy \balpha; \Z)$ for which each point of $S$ comes with coefficient $+1$, so that we can add and subtract sets. 

Note that even though $\SO^{\bdy}$ is not a subset of $\bdy \balpha$, defining $m([\bdy^{\bdy}B], \SO^{\bdy})$ makes sense as a generalization of the multiplicity function $m$. Precisely,  think of an interval $(a_i, a_{i+1})$ as a generator of  $H_1(\bdy\Sigma\setminus\zzz, \bdy\balpha; \Z)$, and define $m([\bdy^{\bdy}B], \SO^{\bdy})$  as the sum of the coefficients  in $[\bdy^{\bdy}B]$ of all intervals $(a_i, a_{i+1})$ that contain an $O\in \SO^{\bdy}$. Define  multiplicity counts for $\SX^{\bdy}$, $\SO^i$, and $\SX^i$ similarly. 

Let $B\in \pi_2(\xx,\yy)$ be a domain for a  diagram with two boundary components. Define
 \begin{align*}
M(B) &= -e(B) - n_{\xx}(B) - n_{\yy}(B) + \frac 1 2 m([\bdy^{\bdy}B], \bar{S^0_{\xx}} + \bar{S^0_{\yy}} + S^1_{\xx} + S^1_{\yy})  \\
 & - m([\bdy^{\bdy} B], \SX^0+\SO^1) + 2n_{\OO}(B)\\ 
A(B) &= \frac 1 2 m([\bdy^{\bdy} B], \SO^0 - \SX^0+ \SX^1 - \SO^1)+ n_{\OO}(B) - n_{\XX}(B).
\end{align*}

For a domain  $B\in \pi_2(\xx,\yy)$ on a type $1$ diagram, define
 \begin{align*}
M(B) &= -e(B) - n_{\xx}(B) - n_{\yy}(B) + \frac 1 2 m([\bdy^{\bdy}B], S_{\xx} + S_{\yy}) - m([\bdy^{\bdy}B], \SO^{\bdy})+ 2n_{\OO}(B)\\ 
A(B) &= \frac 1 2 m([\bdy^{\bdy}B], \SX^{\bdy}-\SO^{\bdy})+ n_{\OO}(B) - n_{\XX}(B).
\end{align*}

For a domain  $B\in \pi_2(\xx,\yy)$ on a type $2$ diagram, define
 \begin{align*}
M(B) &= -e(B) - n_{\xx}(B) - n_{\yy}(B) + \frac 1 2 m([\bdy^{\bdy}B], \bar{S_{\xx}} + \bar{S_{\yy}}) - m([\bdy^{\bdy}B], \SX^{\bdy})+ 2n_{\OO}(B)\\ 
A(B) &= \frac 1 2 m([\bdy^{\bdy}B], \SO^{\bdy}-\SX^{\bdy})+ n_{\OO}(B) - n_{\XX}(B).
\end{align*}

Note that the bigrading is additive under union (when we continue to view the sum of regions as a domain between the same two generators $\xx$ and $\yy$).

We will soon define the bigrading on the modules and bimodules. To show it is well-defined, we need to show that the bigrading on domains is additive under composition, and that it is zero on periodic domains.

\begin{proposition}\label{prop:gr_per}
For any periodic domain  $B\in \pi_2(\xx,\xx)$, we have $M(B) = 0$ and $A(B)=0$.
\end{proposition}
\begin{proof}
Since the bigrading is additive under union, we only need to show it is zero on provincial periodic domains, and on the regions of $\Sigma\setminus \balpha$ that intersect the boundary $\bdy \Sigma$. 
The proofs for all three types of diagrams are identical. The writeup below is for a type $1$ diagram.

For a periodic domain, the bigrading simplifies to 
 \begin{align*}
M(B) &= -e(B) - 2n_{\xx}(B) +  m([\bdy^{\bdy}B], S_{\xx}) - m([\bdy^{\bdy}B], \SO^{\bdy})+ 2n_{\OO}(B)\\ 
A(B) &= \frac 1 2 m([\bdy^{\bdy}B], \SX^{\bdy}-\SO^{\bdy})+ n_{\OO}(B) - n_{\XX}(B).
\end{align*}

Let $D_i$ be the region of  $\Sigma\setminus \balpha$ whose $\alpha$-arcs boundary consists of $\alpha_i^a$ and $\alpha^a_{i+1}$. Geometrically, $D_i$ is a rectangle with $t\geq 0$ disks removed from its interior, so it has Euler measure $e(D_i) = -t$. Each of the $t$ circle boundary components is an  $\alpha$-circle, hence it contains a point of $\xxx$ on it, and contributes $1$ to the count of $2n_{\xxx}(D_i)$. Each of the arcs $\alpha_i$ and $\alpha_{i+1}$ that is occupied by $\xxx$ contributes $1$ to $n_{\xxx}(D_i)$ and $1$ to $m([\bdy^{\bdy}B], S_{\xx})$. There are no other contributions to $m([\bdy^{\bdy}D_i], S_{\xx})$, so $ -e(D_i) - 2n_{\xx}(D_i) +  m([\bdy^{\bdy}D_i], S_{\xx})=0$. Last, $D_i$ contains exactly one $O$ or exactly one $X$. In  either case, $m([\bdy^{\bdy}D_i], \SO^{\bdy})  = 2n_{\OO}(D_i)$, and $m([\bdy^{\bdy}D_i], \SX^{\bdy})  = 2n_{\XX}(D_i)$.  It follows that $M(D_i) = 0$ and $A(D_i) = 0$.

For a provincial periodic domain $B$, the bigrading becomes 
 \begin{align*}
M(B) &= -e(B) - 2n_{\xx}(B) + 2n_{\OO}(B)\\ 
A(B) &= n_{\OO}(B) - n_{\XX}(B),
\end{align*}
which agrees with the bigrading for knot Floer homology, and has been shown to be zero in the case of knots and links in $S^3$ (note that a bordered diagram for a tangle in $B^3$ or $S^2\times I$ can be completed to a diagram for a knot or a link in $S^3$, so we can think of $B$ as a domain in the closed diagram). 
\end{proof}

The proof of additivity under composition  is a bit trickier, as there is linking information we need to consider.

\begin{proposition}\label{prop:b12gr}
If $B_1\in \pi_2(\xx, \yy)$ and $B_2\in \pi_2(\yy, \www)$, then $M(B_1\ast B_2) = M(B_1) + M(B_2)$ and $A(B_1\ast B_2) = A(B_1) + A(B_2)$.
\end{proposition}
\begin{proof}
Once again we write up the proof for a type $1$ diagram, as the notation in this case is the lightest. The other two cases are identical. 

First observe that $m([\bdy^{\bdy}B], \SO^{\bdy})$, $n_{\OO}(B)$, $m([\bdy^{\bdy}B], \SX^{\bdy})$, $n_{\XX}(B)$ are all clearly additive under composition, so the statement follows for the Alexander grading.  

Let $B = B_1\ast B_2$. Let $R_1$ be a union of  the regions $D_i$ as in Proposition \ref{prop:gr_per} with multiplicity, so that $B_1' = B_1 + R_1\in \pi_2(\xx, \yy)$  only covers $\bdy \Sigma$ inside the interval $[a_1, a_{2n+1}]$. Similarly, let $R_2$ be a union of regions $D_i$ so that $B_2' = B_2+ R_2\in \pi_2(\yy, \www)$ only covers $\bdy \Sigma$ inside the interval $[a_1, a_{2n+1}]$. Let $B' = B_1'\ast B_2'$, and note that  $B' = B + R_1+ R_2$.   Since the Maslov grading is additive under union, and by Proposition \ref{prop:gr_per}, we have that $M(B_i') = M(B_i)$ and $M(B') = M(B)$. So it suffices to show that $M(B') = M(B_1') + M(B_2')$.

To simplify notation, write $a = \bdy^{\bdy}B_1'$,  $b = \bdy^{\bdy}B_2'$, and note that $\bdy^{\bdy}B' = a+b$.  By \cite[Lemma 10.4]{bfh2} and since $m([\bdy^{\bdy}B], \SO^{\bdy})$ and $n_{\OO}(B)$ are additive under composition,
\begin{align*}
M(B') - M(B_1') + M(B_2') &= L(a, b) +  \frac 1 2\big( m(a+b, S_{\xx}) + m(a+b, S_{\www})   - m(a, S_{\xx})  \\
&\hspace{.5cm}-m(a, S_{\yy}) -    m(b, S_{\yy}) -m(b, S_{\www})\big) \\
 &=  L(a, b) +  \frac 1 2\big( m(b, S_{\xx}) + m(a, S_{\www}) -m(a, S_{\yy}) -    m(b, S_{\yy}) \big)
\end{align*}
Recall that $L(a,b) = m(b, \bdy a) = -m(a, \bdy b)$, so we can write $L(a, b) = \frac 1 2 (m(b, \bdy a) - m(a, \bdy b))$. Thus, showing that $M(B') = M(B_1') + M(B_2')$ is equivalent to showing that 
$$m(b, \bdy a) - m(a, \bdy b) + m(b, S_{\xx}) + m(a, S_{\www}) -m(a, S_{\yy}) -    m(b, S_{\yy}) =0.$$

Extend the matching $\mu:\{a_1, \ldots, a_{4n+2}\}\to [2n+1]$ linearly to a function $\mu_{\Z}:H_0(\bdy \balpha; \Z)\to \Z^{2n+1}$. For a generator $\xx$, think of $o(\xx)$ as an element of $\Z^{2n+1}$ where each occupied arc comes with coefficient $+1$. Since $B_1$ is a homology class in $\pi_2(\xx, \yy)$, we have $\bdy a = o(\yy) - o(\xx)$. Similarly, $\bdy b = o(\www) - o(\yy)$. 

Let $S_{\xx}^{bottom} = S_{\xx}\cap \{a_1, \ldots, a_{2n+1}\}$, and let $S_{\xx}^{top} = S_{\xx}\cap \{a_{2n+2}, \ldots, a_{4n+2}\}$. Recall that  we view any subset $S\subset \bdy \balpha$  as the element of $H_0(\bdy \balpha; \Z)$ where each point of $S$ comes with coefficient $+1$.

Since $[a_1, a_{2n+1}]\subset \bdy \Sigma$ only contains one endpoint of each $\alpha$ arc, and since $\mu_{\Z}(\bdy a) = o(\yy) - o(\xx)$, it follows that $\bdy a$ can only be the section  $S_{\yy}^{bottom} -  S_{\xx}^{bottom}$  of $o(\yy) - o(\xx)$. Then
$m(b, \bdy a) = m(b, S_{\yy}^{bottom} -  S_{\xx}^{bottom})$. Since $b$ only covers the ``bottom" of $\bdy \Sigma$, i.e. $[a_1, \ldots, a_{2n+1}]$, the multiplicity of $b$ at $a_i$ is zero whenever $i\geq 2n+2$, so $m(b, \bdy a) = m(b, S_{\yy} - S_{\xx})$. Similarly, $m(a, \bdy b) = m(a, S_{\www} - S_{\yy})$. This completes the proof.
\end{proof}

\subsection{Modules and bimodules} Let $\HH$ be a diagram  of type $1$ or type $2$ for some pair $(B^3, \T)$. 

 \begin{proposition}
 Given $\xx, \yy \in \gens(\HH)$, $\pi_2(\xx, \yy)$ is nonempty. 
 \end{proposition}
 \begin{proof}
 The proof is identical to that of \cite[Lemma 4.21]{bfh2}. Connect $\xx$ to $\yy$ by a union of paths $\gamma_{\alpha}\subset \balpha\cup (\bdy\Sigma\setminus\zzz)$ and $\gamma_{\beta}\subset \bbeta$. Then $\xx$ and $\yy$ are connected by a domain if and only if $\gamma_{\alpha}-\gamma_{\beta}$ can be made null-homologous in $\Sigma$ by adding or subtracting entire $\alpha$-curves and $\beta$-circles, if and only if the image of $\gamma_{\alpha}-\gamma_{\beta}$ in 
 $H_1(\Sigma\times I, \balpha\times \{1\}\cup \bbeta\times \{0\}\cup(\bdy\Sigma\setminus \zzz)\times I) \cong H_1(B^3, \bdy B^3)$ is zero. But $H_1(B^3, \bdy ^3)=0$, so this is always the case.
  \end{proof}

Since any two generators $\xx, \yy\in\gens(\HH)$ are connected by a domain, we can define relative gradings

\begin{align*}
M(\yy)-M(\xx) &= M(B)\\
A(\yy)-A(\xx) &= A(B),
\end{align*}
where $B\in \pi_2(\xx,\yy)$. We can assume $B$ does not cross $\zzz$: if any domain $B'$ intersects $\zzz$, we can add copies of the periodic domain(s) that are the region(s) of $\Sigma\setminus\balpha$ containing the points/arc $\zzz$, to obtain a domain $B\in\pi_2(\xx,\yy)$ that avoids $\zzz$.

When $\HH$ is a diagram for $(S^2\times I, \T)$, it is no longer true that any two generators are connected by a domain. However, the $\DA$ bimodule splits as
 \[\cftdat(\HH)\cong \bigoplus_{i=0}^{2m+1}\cftdat_i(\HH),\]
where $\cftdat_i(\HH)$ is is generated by $\gens_i\coloneqq\{\xx\in \gens| |o^0(\xx)| = i\}$. 

\begin{lemma}
For a fixed $i$, and for any  $\xx, \yy\in \cftdat_i(\HH)$, we have $\pi_2(\xx, \yy)\neq \emptyset$.
\end{lemma}
\begin{proof}
Let $\xx_{dr}$ and $\yy_{dr}$ be the generators  corresponding to $\xx$ and $\yy$ in $\HH_{dr}$. There is some domain $B_{dr}\in \pi_2(\xx_{dr}, \yy_{dr})$, since $\HH_{dr}$ is a diagram for $B^3$.  Add copies of  the two periodic regions of $\Sigma_{dr}\setminus\balpha$ containing $\zzz_1$ and $\{z^{front}, z^{back}\}$, to obtain a domain $B_{dr}'\in\pi_2(\xx, \yy)$ with zero multiplicity at $\zzz_1$ and $z^{back}$, resulting in some multiplicity $p$ at $z^{front}$. Write $\balpha^0$ for the set $\{\alpha_1^0, \ldots, \alpha_{m+1}^0\}$, and $\balpha^1$ for the set $\{\alpha_1^1, \ldots, \alpha_{n+1}^1\}$.   Let $S\in \Z\left<\bdy(\balpha^0\cup \balpha^1)\right>$ be the set of points (with sign and multiplicity) in the boundary of $\bdy^{\bdy}B_{dr}'$. The matching $\mu$  for the pointed matched circle $\bdy \HH_{dr}$ extends bilinearly to a map $\mu_{\ast}:\Z\left<\bdy\balpha\right>\to \Z\left<[m+n+2]\right>$. Since  $\mu_{\ast}(S) = o(\yy_{dr})-o(\xx_{dr})$,  and $o^0(\xx)$ and $o^0(\yy)$ both have  cardinality $i$, it follows that $p=0$, so after attaching a $1$-handle at $\{z^{front}, z^{back}\}$,  $B_{dr}'$ becomes a domain $B'$ on $\HH$.
\end{proof}
 Define relative gradings on  $\cftdat_i(\HH)$ by
 \begin{align*}
M(\yy)-M(\xx) &= M(B)\\
A(\yy)-A(\xx) &= A(B),
\end{align*}
where $B\in \pi_2(\xx,\yy)$ (again arrange for $B$ to have zero multiplicities at $\zzz_1$ and $\zzz_2$ by adding copies of the corresponding regions of $\Sigma\setminus\balpha$, if necessary).

By Proposition \ref{prop:gr_per}, the relative bigrading on the modules and bimodules is well-defined. 

\begin{proposition}\label{prop:bigr_str}
The various structures defined in this section are graded ($A$, $D$, or $\DA$) (bi)modules with respect to the grading $M$. Further, the internal grading $A$ is preserved by all structure maps. 
\end{proposition}

The proof is based on understanding the relation between the bigrading on a domain with a compatible sequence of sets of Reeb chords and the bigrading on the algebra elements associated to the Reeb chords. We start by relating the Maslov grading of algebra generators to $\iota$. 
\begin{lemma}\label{lemma:miota}
Let $a = I(\sss)a(\brho)I(\ttt)$ be a generator for $\az$. Then 
$$M(a) - \iota(\brho) = \frac 1 2 m([\brho], S+T) - m([\brho], \SO^{\bdy}),$$
where $S = \mu^{-1}(\sss)$ and $T = \mu^{-1}(\ttt)$. 
\end{lemma}
\begin{proof}
Let $a' = (\sss, \ttt, \phi)$ be the element in $\widehat\Aa(\bar\E)$ corresponding to $a$ under the isomorphism $\az\cong \widehat\Aa(\bar\E)$ discussed earlier. 
Recall that $M(a') = \inv(\phi)-\inv(\phi,\omega) + \inv(\omega)$. 
Decompose $\sss$ as $\sss^+\sqcup \sss^-\sqcup \sss^0$, so that $\phi^+\coloneqq\phi|_{\sss^+}$ is increasing,  $\phi^- \coloneqq\phi|_{\sss^-}$ is decreasing, and  $\phi^0\coloneqq \phi|_{\sss^0}$ is the identity. Then 
\begin{align*}
M(a') &= \inv(\phi)-\inv(\phi,\omega) \\
& = \inv(\phi^+) + \inv(\phi^-) + \inv(\phi^+, \phi^0)+ \inv(\phi^-, \phi^0)   + \inv(\phi^+, \phi^-)-\inv(\phi,\omega)\\
& = \inv(\brho)+ \inv(\phi^+, \phi^0)+ \inv(\phi^-, \phi^0)   + \inv(\phi^+, \phi^-)-\inv(\phi,\omega)
\end{align*}
By \cite[Lemma 5.57]{bfh2}, $\iota(\brho)$ can be written as
$$\iota(\brho) = \inv(\brho) - m([\brho], S(\brho)),$$
where $S(\brho)$ is the set of initial endpoints of $\brho$.

The upward-veering strands in $\phi$, i.e. the strands for $\phi^+$, correspond to the set of Reeb chords $\brho^+\subset\brho$ contained in $[a_{2n+2}, a_{4n+2}]$, and the downward-veering strands in $\phi$ correspond to the Reeb chords $\brho^-\subset\brho$ contained in $[a_1, a_{2n+1}]$. The horizontal strands of $\phi$ correspond to the projection under the matching $\mu$ of the dashed  horizontal strands in the strands diagram for $a$. Let $S(\brho^+)$ and $S(\brho^-)$ be the sets of initial endpoints of $\brho^+$ and $\brho^-$, respectively. Note that $S(\brho^+)$ is the section of $\sss^+$ contained in $[a_1, a_{2n+1}]$, and $S(\brho^-)$ is the section of $\sss^-$ contained in $[a_{2n+2}, a_{4n+2}]$. Equivalently, $S(\brho^+) = \mu^{-1}(\sss^+)\cap [a_1, a_{2n+1}]$ and $S(\brho^-) = \mu^{-1}(\sss^-)\cap [a_{2n+2}, a_{4n+2}]$.

Decompose $S$ as $S = S^+\sqcup S^-\sqcup S^0$, where $S^+ = \mu^{-1}(\sss^+)$,  $S^+ = \mu^{-1}(\sss^-)$, and $S^0$ is the set of initial points for the dashed horizontal strands. Decompose $T$ similarly  by the type of final endpoints as $T =  T^+\sqcup T^-\sqcup T^0$. Note that the multiplicity of a Reeb chord in $[a_1, a_{2n+1}]$ is zero at any point in $[a_{2n+2}, a_{4n+2}]$, and similarly the multiplicity of a Reeb chord in $[a_{2n+2}, a_{4n+2}]$ is zero at any point in $[a_1, a_{2n+1}]$, so $m([\brho^+], S(\brho)) = m([\brho^+], S(\brho^+)) = m([\brho^+], S^+)$, and similarly  $m([\brho^-], S(\brho)) = m([\brho^-], S^-)$. Since $[\brho] = [\brho^+]+ [\brho^-]$, 
$$\iota(\brho) = \inv(\brho) - m([\brho^+], S^+) - m([\brho^-], S^-).$$

Next,we express $M(a')$ in terms of $\brho$. Observe that 
$$\inv(\phi^+, \phi^0)+ \inv(\phi^-, \phi^0) = m([\brho^+], S^0) + m([\brho^-], S^0).$$
and 
$$\inv(\phi, \omega) = m([\brho], \SO^{\bdy}).$$
It remains to understand $ \inv(\phi^+, \phi^-)$. Let $s^-$ and $s^+$ be a downward-veering and an upward-veering strand, and let $\rho^-$ and $\rho^+$ be the corresponding Reeb chords on $\zz$. The strands $s^-$ and $s^+$ cross exactly when one of the following happens:
\begin{itemize}
\item[-] The initial endpoint of $s^+$ is between the initial and final endpoints of $s^-$. This happens exactly when $m([\rho^-], \mu^{-1}(\mu(S(\rho^+)))=1$ and $m([\rho^+], \mu^{-1}(\mu(S(\rho^-)))=0$. 

\item[-] The initial endpoint of $s^-$ is between the initial and final endpoints of $s^+$. Equivalently,  $m([\rho^+], \mu^{-1}(\mu(S(\rho^-)))=1$ and $m([\rho^-], \mu^{-1}(\mu(S(\rho^+)))=0$. 

\item[-] The initial endpoint of $s^+$ is the final endpoint of $s^-$, i.e. $m([\rho^+], \mu^{-1}(\mu(S(\rho^-)))=\frac 1 2$ and $m([\rho^-], \mu^{-1}(\mu(S(\rho^+)))=\frac 1 2$. 
\end{itemize}
The strands do not cross if and only if $m([\rho^+], \mu^{-1}(\mu(S(\rho^-)))=0=m([\rho^-], \mu^{-1}(\mu(S(\rho^+)))$.
By linearity then, 
\begin{align*}\inv(\phi^+, \phi^-) &= \sum_{\rho^-\in \brho^-, \rho^+\in \brho^+} (m([\rho^-], \mu^{-1}(\mu(S(\rho^+))) + m([\rho^+], \mu^{-1}(\mu(S(\rho^-))))\\
&= m([\brho^-], S^+) +  m([\brho^+], S^-).
\end{align*}
So,
\begin{align*}
M(a') & = \inv(\brho)+ \inv(\phi^+, \phi^0)+ \inv(\phi^-, \phi^0)   + \inv(\phi^+, \phi^-)-\inv(\phi,\omega)\\
& =  \inv(\brho)+ m([\brho^+], S^0) + m([\brho^-], S^0)   +m([\brho^-], S^+) +  m([\brho^+], S^-) -m ([\brho], \SO^{\bdy})\\
& =  \inv(\brho)+ m([\brho^+], S) + m([\brho^-], S)  - m([\brho^-], S^-) -  m([\brho^+], S^+) -m ([\brho], \SO^{\bdy})\\
& = \iota(\brho)+ m([\brho^+], S) + m([\brho^-], S)   -m ([\brho], \SO^{\bdy})\\
& = \iota(\brho)+ m([\brho], S)   -m ([\brho], \SO^{\bdy}).
\end{align*}
It is not hard to see that $m([\brho], S)  = m([\brho], T)$, so 
\begin{equation*}
M(a') - \iota(\brho) =   \frac 1 2 m([\brho], S+T)-m([\brho], \SO^{\bdy}).\qedhere
\end{equation*}
\end{proof}

Let $B\in \pi_2(\xx,\yy)$ be a domain for a diagram $\HH$ with two boundary components, let $\vec\brho_0 = (\rho_1, \ldots, \rho_i)$ be a sequence of Reeb chords on $\bdy^0\HH$, and let $\vec\brho_1 = (\brho_1, \ldots, \brho_j)$ be a sequence of sets of Reeb chords on $\bdy^1\HH$, both compatible with $B$. 
Recall that we write
\[a_0\coloneqq  a_0(\xx,\yy, \vec P_0)= I( \bar o^0(\xx)) \cdot a(-\rho_1)\cdot \cdots\cdot a(-\rho_i)\cdot I( \bar o^0(\yy))\in \Aa(-\bdy^0\HH)\]
and
\[a_1(\xx, \yy, P_1) = I( o^1(\xx)) \cdot a(\brho_1)\otimes \cdots\otimes a(\brho_j)\cdot I(o^1(\yy))\in \Aa(\bdy^1\HH)^{\otimes j},\]
and observe that we can equivalently write $a_1(\xx, \yy, P_1)$ as 
\[a_1(\xx, \yy, P_1) = I( o^1(\xx))  a(\brho_1)I_1\otimes I_1a(\brho_2)I_2\cdots\otimes I_{j-1} a(\brho_j)I(o^1(\yy)).\]
 Denote $I( o^1(\xx))  a(\brho_1)I_1, \ldots,  I_{j-1} a(\brho_j)I(o^1(\yy))$ by $a_1, \ldots, a_j$. 

\begin{proposition}\label{prop:domaingr}
For the triple $(B, \vec\brho_0, \vec\brho_1)$ we have 
\begin{align*}
 M(B, \vec\brho_0, \vec\brho_1) &= |\vec\brho_1|  -\ind(B, \vec\brho_0, \vec\brho_1) +\sum_{t=1}^j M(a_t) - M(a_0) + 2n_{\OO}(B)\\ 
A(B, \vec\brho_0, \vec\brho_1) &= \sum_{t=1}^j A(a_t) - A(a_0) + n_{\OO}(B) - n_{\XX}(B).
\end{align*}
\end{proposition}

\begin{proof}
The equality for the Alexander grading follows immediately from the definition. 

For the Maslov grading, denote the right hand side of the equation by $R$. 
 Note that while $a_1\otimes \cdots\otimes a_j\neq 0$, it may be that $a_1\cdots  a_j = 0$. Resolve crossings in each $a_t$ if necessary to get a nonzero product $a'=a_1'\cdots a_j'$. Note that $\iota(a_t') = \iota(a_t)-c_t$ and $M(a_t') = M(a_t)- c_t$, where $c_t$ is the number of resolved crossings to get from $a_t$ to $a_t'$, and  $L([a_s], [a_t]) = L([a_s'], [a_t'])$, since resolving crossings does not change the homology class. Then 
 \begin{align*}\sum_{t=1}^j M(a_j) -\iota(\vec\brho_1) &= \sum_{t=1}^j M(a_j)  - \sum_{t=1}^j\iota(\brho_t)- \sum_{s<t}L([\brho_s], [\brho_t])\\
 &= \sum_{t=1}^j M(a_j)  - \sum_{t=1}^j\iota(a_t)- \sum_{s<t}L([a_s], [a_t])\\
  &= \sum_{t=1}^j M(a_j')  - \sum_{t=1}^j\iota(a_t')- \sum_{s<t}L([a_s'], [a_t'])\\
  &= M(a') - \iota(a').
 \end{align*}
 By \cite[Lemma 18]{absgr}, $\iota(a_0) = -|\vec\brho_0| - \iota(\vec\brho_0)$.
 
Substituting the definition of $\ind$ in $R$, we get
\begin{align*}
R &=-e(B) - n_{\xx}(B) -n_{\yy}(B) - |\vec\brho_0|  - \iota(\vec\brho_0)-\iota(\vec\brho_1)+\sum_{t=1}^j M(a_t) - M(a_0) + 2n_{\OO}(B)\\
&= -e(B) - n_{\xx}(B) -n_{\yy}(B) - |\vec\brho_0|  - \iota(\vec\brho_0) +   M(a') - \iota(a')   - M(a_0) + 2n_{\OO}(B)\\
&= -e(B) - n_{\xx}(B) -n_{\yy}(B) +\iota(a_0) +   M(a') - \iota(a')   - M(a_0) + 2n_{\OO}(B).
\end{align*}
Applying Lemma \ref{lemma:miota} to $a'$ and $a_0$, and since $[a'] = [\vec\brho_1]$, we get  $R = M(B, \vec\brho_0, \vec\brho_1)$.
\end{proof}

The equalities below for a type $1$ or  a type $2$ diagram are a special case of Proposition \ref{prop:domaingr}, and follow immediately. 

\begin{proposition}\label{prop:domaingr2}
For a domain $B\in \pi_2(\xx,\yy)$ on a type $2$ diagram, and   a sequence of Reeb chords $\vec\brho$, 
\begin{align*}
M(B, \vec\brho) &= -\ind(B, \vec\brho) - M(-\vec\brho) + 2n_{\OO}(B)\\
A(B, \vec\brho) &=  - A(-\vec\brho) + n_{\OO}(B) - n_{\XX}(B).
\end{align*}
\end{proposition}
\begin{proposition}\label{prop:domaingr1}
For a domain $B\in \pi_2(\xx,\yy)$ on a type $1$ diagram, and  a sequence of sets of Reeb chords $\vec{\brho} = (\brho_1, \ldots, \brho_l)$, 
 \begin{align*}
M(B, \vec{\brho}) &= |\vec\brho|  -\ind(B, \vec{\brho}) +\sum_{i=1}^l M(\brho_i) + 2n_{\OO}(B)\\ 
A(B, \vec{\brho}) &= \sum_{i=1}^l A(\brho_i) + n_{\OO}(B) - n_{\XX}(B).
\end{align*}
 \end{proposition}
 
 Proposition  \ref{prop:bigr_str} follows.
 
 \begin{proof} [Proof of Proposition \ref{prop:bigr_str}]
 All algebraic structures here are defined by counting curves of index $1$. The claim follows directly by substituting $1$ for the index in the grading formulas from Propositions \ref{prop:domaingr}, \ref{prop:domaingr2}, and \ref{prop:domaingr1}.
 \end{proof}
 
 \subsection{Tensor products}
It is easy to see that the bigrading on domains is additive under gluing. 

\begin{proposition}
If $\HH_1$ and $\HH_2$ are diagrams with $\bdy^1\HH_1 = -\bdy^0\HH_2$, and $B$ is a domain on $\HH = \HH_1\cup \HH_2$ that decomposes as $B = B_1\times B_2$, with $B_i$ a domain on $\HH_i$, then $M(B) = M(B_1)+M(B_2)$ and $A(B) = A(B_1)+A(B_2)$.
\end{proposition}
\begin{proof}
This follows directly from the definitions of $M$ and $A$.
\end{proof}

Thus, for a generator $\xxx = \xxx_1\cup \xxx_2\in\mathfrak S(\HH)$, where $\xxx_i\in \mathfrak S(\HH_i)$,  the bigrading on  $\xxx$ agrees with the bigrading on $\xxx_1\boxtimes\xxx_2$.

\subsection{Absolute gradings} We finish this section by turning the relative grading into an absolute one. 

First, for any type of diagram, it is straightforward to verify that the homotopy equivalences from Theorems \ref{thm:dinv}, \ref{thm:ainv}, and \ref{thm:dainv} preserve the relative bigrading. 

Next, recall that under the correspondence between bordered grid diagrams and shadows, bordered grid diagrams inherit the bigrading defined  in Section \ref{ssec:grading}. A plumbing $G$ of a sequence of grid diagrams can be completed to a mutipointed bordered Heegaard diagram $\HH_G$ in a natural way, by embedding it on a smooth surface, as in Figure \ref{fig:HD}, and adding the appropriate $\zzz$ decoration in the region(s) outside the image of the embedding.  Under the natural correspondence of generators and maps, the resulting diagram $\HH_G$ has an associated type $A$, $D$, or $\DA$ structure, which we simply call $\cftt(\HH_G)$, identical to $\widetilde{\mathit{CT}}(G)$. The bigrading that $\cftt(\HH_G)$ inherits from  $\widetilde{\mathit{CT}}(G)$ agrees with the relative bigrading on $\cftt(\HH_G)$ defined in this section. 
We turn the bigrading from this section into an absolute one by requiring that it agrees with the one on $\widetilde{\mathit{CT}}$ for a chosen plumbing of grid diagrams.

\begin{definition}\label{def:absgr}
Given a tangle Heegaard diagram $\HH$ of any type for a tangle $\T$, let $\HH_G$ be a Heegaard diagram of the same type arising from a plumbing $G$ of grid diagrams representing  $\T$, and let $h:\cftt(\HH_G)\to \cftt(\HH)$ be the homotopy equivalence corresponding to a chosen sequence of Heegaard moves between $\HH_G$ and $\HH$. Define the absolute bigrading on $\HH$ by requiring that $h$ preserves gradings.
\end{definition}

We need to show that the absolute grading in Definition \ref{def:absgr} is independent of the choice of grid decomposition $G$, and also  independent of the choice of sequence of Heegaard moves, i.e. of $h$. 

\begin{lemma}\label{lem:indep_moves}
Fix $\HH$ and $\HH_G$ as in Definition \ref{def:absgr}, let $s$ and $s'$ be two sequences of Heegaard moves from $\HH_G$ to $\HH$, and let $h', h'':\cftt(\HH_G)\to \cftt(\HH)$ be the homotopy equivalences corresponding to $s'$ and $s''$. The two bigradings $\gr'$ and $\gr''$ induced by $h'$ and $h''$ according to Definition \ref{def:absgr} agree.
\end{lemma}

\begin{proof} We simplify notation and denote the bigrading $(M,A)$ from Section \ref{ssec:grading} by $\gr$.

We will complete $\HH_G$ to a closed Heegaard diagram $\bar\HH_G$ for a link, by gluing to it one (if $\HH$ is of type $1$ or $2$) or two (if $\HH$ had two boundary components) plumbings  of bordered grid diagrams. The proof in each case is analogous, so from here on we assume that $\HH$ is a type $1$ diagram. Let $H$ be some plumbing of grids so that $G\cup H$ represents a closed knot or link.  Let $\HH_H$ be the type $2$ Heegaard diagram corresponding to $H$, and let $\bar \HH_G = \HH_G\cup \HH_H$. 

Complete each diagram obtained along the sequences  $s'$ and $s''$  to a closed one by gluing to it $\HH_H$. The sequences of moves $s'$ and $s''$ extend to sequences of moves $\bar s'$ and $\bar s''$ connecting $\bar\HH_G$ to $\bar\HH\coloneqq \HH\cup \HH_H$, by fixing the $\HH_H$ area of each closed diagram and performing the moves specified by $s'$ and $s''$ outside the $\HH_H$ area. Observe that the resulting homotopy equivalences $\bar h',\bar h'':\cftt(\bar\HH_G)\to \cftt(\bar\HH)$ are exactly the maps $h'\boxtimes \id_{\HH_H}, h''\boxtimes \id_{\HH_H}:\cftt(\HH_G)\boxtimes \cftt(\HH_H)\to \cftt(\HH)\boxtimes \cftt(\HH_H)$. So the gradings induced by $\bar h'$ and $\bar h''$ are exactly the gradings $\gr'\boxtimes \gr$ and $\gr''\boxtimes \gr$.

By Theorem \ref{recover_hfk}, the grading on $\cftt(\bar\HH_G)\cong \widetilde{\mathit{CT}}(G\cup H)$  from Section \ref{ssec:grading}, which is given by $\gr(\xxx_G\cup \xxx_H) = \gr(\xxx_G) + \gr(\xxx_H)$,  agrees with the grading on $\cfkt(\bar \HH_G)$. Since $\bar h'$ and $\bar h''$ are homotopy equivalences arising from sequences of Heegaard moves, it follows that the gradings they induce on $\cftt(\bar \HH)$ agree with the grading on $\cfkt(\bar\HH)$ too. In particular, $\gr'\boxtimes \gr = \gr''\boxtimes \gr$, so $\gr' = \gr''$.
\end{proof}

\begin{lemma}\label{lem:indep_grid}
Let $\HH$ be a Heegaard diagram for a tangle $\T$.
Let   $\Pp=\{\P_1^\circ,\dots,\P_p^\circ\}$ and $\boldsymbol{\mathcal{Q}}=\{\mathcal{Q}_1^\circ,\dots,\mathcal{Q}_q^\circ\}$  be two sequences of shadows for  $\T$,  let $G'$ and $G''$ be the corresponding plumbings of bordered grid diagrams, and let $h':\cftt(\HH_{G'})\to \cftt(\HH)$ and $h'':\cftt(\HH_{G''})\to \cftt(\HH)$ be the homotopy equivalences corresponding to 
some two sequences of Heegaard moves $s'$ and $s''$ from $\HH_{G'}$ and $\HH_{G''}$, respectively, to $\HH$. The two bigradings $\gr'$ and $\gr''$ induced by $h'$ and $h''$ according to Definition \ref{def:absgr} agree.
\end{lemma}

\begin{proof}
Assume $\HH$ is a type $1$ diagram. The other cases are analogous. 

Again denote the bigrading $(M,A)$ from Section \ref{ssec:grading} by $\gr$.

Fix a plumbing $H$ of bordered grid diagrams, as  in the proof of Lemma \ref{lem:indep_moves}, so that $G'\cup H$ and $G''\cup H$ represent a closed knot or link. Let $\bar\HH_{G'} = \HH_{G'}\cup \HH_H$, $\bar\HH_{G''} = \HH_{G''}\cup \HH_H$, $\bar\HH= \HH\cup \HH_H$. 

We  now apply the same reasoning as in the proof of Lemma \ref{lem:indep_moves}. We get homotopy equivalences $\bar h':\cftt(\bar\HH_{G'})\to \cftt(\bar\HH)$ and $\bar h'':\cftt(\bar\HH_{G''})\to \cftt(\bar\HH)$. By Theorem \ref{recover_hfk}, the grading on $\cftt(\bar\HH_{G'})\cong \widetilde{\mathit{CT}}(G'\cup H)$  from Section \ref{ssec:grading}  agrees with the grading on $\cfkt(\bar \HH_{G'})$, so the grading $\gr'\boxtimes \gr$ induced by $\bar h'$ on $\cftt(\bar \HH)$ agrees with the grading on $\cfkt(\bar \HH)$ too. Similarly,  the grading $\gr''\boxtimes \gr$ induced by $\bar h''$ on $\cftt(\bar \HH)$ agrees with the grading on $\cfkt(\bar \HH)$. Thus, $\gr'\boxtimes \gr = \gr''\boxtimes \gr$, so $\gr' = \gr''$.
\end{proof}

\begin{proposition}\label{prop:absgr}
The bigrading  from Definition \ref{def:absgr} is well-defined. 
\end{proposition}

\begin{proof}
Lemmas \ref{lem:indep_moves} and \ref{lem:indep_grid} show that Definition \ref{def:absgr} is independent of the choices made. This completes the proof.
\end{proof}

We can now conclude that for  tangles in $B^3$ or $S^2\times I$, the homotopy equivalences from Theorems \ref{thm:dinv}, \ref{thm:ainv}, and \ref{thm:dainv} are graded. In other words, $\cftat$, $\cftdt$, and $\cftdat$ are graded tangle invariants. Below,  $V=\F_2\oplus \F_2$, with  one summand in grading $(0,0)$ and the other summand in grading $(-1, -1)$.

\begin{theorem}\label{thm:invariance}
Up to graded  homotopy equivalence and tensoring with $V$, the modules  defined in Section \ref{sec:hd_modules} are independent of the choices made in their definitions.  Namely:

If $\HH_1$ and $\HH_2$ are provincially admissible type $2$ diagrams for a $2n$-tangle $\T$ in $B^3$ with almost complex structures $J_1$ and $J_2$, and $|\X_1| = |\X_2|+k$, then there is a graded type $D$ homotopy equivalence
$$\cftdt(\HH_1, J_1)\simeq \cftdt(\HH_2, J_2)\otimes V^{\otimes k}.$$

If $\HH_1$ and $\HH_2$ are provincially admissible type $1$ diagrams for a $2n$-tangle $\T$ in $B^3$ with almost complex structures $J_1$ and $J_2$, and $|\X_1| = |\X_2|+k$, then there is a graded type $A$ homotopy equivalence
$$\cftat(\HH_1, J_1)\simeq \cftat(\HH_2, J_2)\otimes V^{\otimes k}.$$

If $\HH_1$ and $\HH_2$ are provincially admissible  diagrams for an $(m,n)$-tangle $\T$ in $S^2\times I$ with almost complex structures $J_1$ and $J_2$, and $|\X_1| = |\X_2|+k$, then there is a graded type $\DA$ homotopy equivalence
$$\cftdat(\HH_1, J_1)\simeq \cftdat(\HH_2, J_2)\otimes V^{\otimes k}.$$
\end{theorem}

Thus, given a marked  $2n$-tangle $\T$ in $B^3$, if $\HH$ is a type $1$ or a type $2$ diagram  for $\T$ with $|\X\cap \Int \Sigma| = |\T|$, we get an invariant of the tangle 
$$\cftahat(\T) \coloneqq \cftat(\HH)$$
up to type $A$ homotopy equivalence, or
$$\cftdhat(\T) \coloneqq \cftdt(\HH)$$
up to type $D$ homotopy equivalence, respectively. 

Similarly, given an $(m,n)$-tangle $\T$ in $S^2\times I$, if $\HH$ is a diagram with two boundary components for $\T$, we get an invariant of the tangle
$$\cftdahat(\T) \coloneqq \cftdat(\HH)$$
up to type $\DA$ homotopy equivalence.

Similar results hold for the various other modules from Subsection \ref{ssec:other}.

\section{Pairing (Nice diagrams)} 
\label{sec:hd_pairing}

In \cite{sarkarwang}, Sarkar and Wang introduced a class of Heegaard diagrams for three-manifolds called {\em nice}. These were used in \cite{LOT} to prove a pairing theorem in bordered Floer homology. In a similar vein, here we define nice Heegaard diagrams for tangles, and use them to prove a pairing theorem.

\begin{definition}\label{def:nice}
A tangle Heegaard diagram is called \emph{nice} if every region that does not contain an interior $X$ or $O$ and does not intersect $\zzz$  is a disk with at most $4$ corners. 
\end{definition}

\begin{proposition}
Any tangle Heegaard diagram can be turned into a nice diagram via a sequence of isotopies and handleslides of the $\beta$ curves in the interior of the Heegaard surface. 
\end{proposition}
\begin{proof}
The proof uses ``finger moves" and is analogous to that of \cite[Proposition 8.2]{bfh2}. 
\end{proof}

Nice diagrams are admissible:

\begin{lemma}
If $\HH$ is nice, then $\HH$ is admissible.
\end{lemma}
\begin{proof}
The proof is a straightforward generalization of the one for the closed case \cite[Corollary 3.2]{lmw}. Suppose $D$ is a nontrivial  domain in $\Sigma\setminus(\XX\cup \OO\cup\zzz)$ with only nonnegative multiplicities, and its  boundary is a linear combination of entire $\alpha$ and $\beta$ curves. Consider a curve that appears in $\bdy D$ with nonzero multiplicity, and orient it so that all regions directly to its left have positive multiplicity. If that curve is an $\alpha$ circle or a $\beta$ circle, then \cite[Lemma 3.1]{lmw} applies, i.e. one of these regions contains a basepoint, which gives a contradiction. So suppose that curve is an $\alpha$ arc, call it $\alpha_i$.  We verify that the argument in  \cite[Lemma 3.1]{lmw} can be used again to show that  one of these regions contains a basepoint. 

Suppose one of the regions directly to the left of $\alpha_i$ is a bigon. Then the other edge of that region is part of a $\beta$ circle, call it $\beta_j$. On the other side of $\beta_j$ there is a square (a bigon would imply $\alpha_i$ is a circle, not an arc) and the edge of that square across from $\beta_j$ is either part of a $\beta$ circle again, or part of $\bdy\Sigma$. In the first case, there is yet another square on the other side, and we look at that square. Eventually we reach a square with an edge on $\bdy\Sigma$. The union of all these regions forms a component of $\Sigma\setminus \balpha$ (with two corners), so we reach a contradiction, since every component of $\Sigma\setminus\balpha$ contains a point in $\XX\cup \OO\cup\zzz$. 

Now suppose there are no bigon regions directly to the left of $\alpha_i$. Then all those regions are squares, and they must form a chain that starts and ends at $\bdy\Sigma$. The edges across from $\alpha_i$ on those squares form a complete $\alpha$ arc, and the union of the squares is a component of $\Sigma\setminus \balpha$ (with four corners). This again is a contradiction.
\end{proof}

Since nice diagrams are admissible, there are only a few types of holomorphic curves, as one only counts domains that are squares or bigons. Specifically, for $\cftahat$, all multiplication maps $m_n$ for $n>2$ are zero, and for $\cftdahat$ all structure maps $\delta^1_{1+j}$ for $j>1$ are zero. 

We are now ready to state and prove a pairing theorem. By invariance (Theorem \ref{thm:invariance}),  assume that all diagrams below are nice. 

\begin{theorem}\label{thm:cft_pairing}
The following equivalences hold:

\begin{enumerate}
\item  If $\HH_1\cup \HH_2$ is the union of a type $1$ Heegaard diagram $\HH_1$ and a Heegaard diagram with two boundary components $\HH_2$ along $\bdy\HH_1$ and $(\bdy^0\HH_2)^*$, then
$$\cftat(\HH_1)\boxtimes \cftdat(\HH_2)\simeq \cftat(\HH_1\cup \HH_2).$$

\item If $\HH_1\cup \HH_2$ is the union of Heegaard diagrams  $\HH_1$ and $\HH_2$ with two boundary components  along $\bdy^1\HH_1$ and $(\bdy^0\HH_2)^*$, then
$$\cftdat(\HH_1)\boxtimes \cftdat(\HH_2)\simeq \cftdat(\HH_1\cup \HH_2).$$

\item  If $\HH_1\cup \HH_2$ is the union of a Heegaard diagram $\HH_1$ with two boundary components  and a Heegaard diagram  $\HH_2$ of type $2$  along $\bdy^1\HH_1$ and $(\bdy\HH_2)^*$, then
$$\cftdat(\HH_1)\boxtimes \cftdt(\HH_2)\simeq \cftdt(\HH_1\cup \HH_2).$$

\item  If $\HH_1\cup \HH_2$ is the union of a Heegaard diagram $\HH_1$ of type $1$  and a Heegaard diagram  $\HH_2$ of type $2$  along $\bdy\HH_1$ and $(\bdy\HH_2)^*$, then
$$\cftat(\HH_1)\boxtimes \cftdt(\HH_2)\simeq \cfkt(\HH_1\cup \HH_2).$$
\end{enumerate}

Moreover, when the underlying manifolds are $B^3$, $S^2\times I$, or $S^3$, the homotopy equivalences respect the bigrading. 
\end{theorem}

\begin{proof}
The proof  is analogous to that for bordered Heegaard Floer homology \cite[Theorem 11]{bimod}. We outline it for the first case. First note that $\HH_1\cup\HH_2$ is automatically a type 1 Heegaard diagram. Since $\HH_1$ and $\HH_2$ are nice diagrams, then both diagrams are admissible, so the corresponding type $A$ and type $\DA$ structures are bounded, and their box tensor product is well-defined. There is a correspondence of generators of $\cftat(\HH_1)\boxtimes \cftdat(\HH_2)$ and $\cftat(\HH_1\cup\HH_2)$. 

The differential on $\cftat(\HH_1)\boxtimes \cftdat(\HH_2)$ counts bigons and rectangles that are provincial in $\HH_1$ (corresponding to the differential $m_1$ on $\cftat(\HH_1)$), provincial in $\HH_2$ (corresponding to  the ``differential" on $\cftdat(\HH_2)$, i.e. the part of $\delta_1^1$ that outputs an idempotent algebra element), or provincial in $\HH_1\cup \HH_2$ but crossing the common boundary of $\HH_1$ and $\HH_2$ (for $(m_2\otimes \id)\circ(\id\otimes \delta_1^1)$ when $\delta_1^1$ outputs a nonidempotent algebra element).  The third kind can only be a rectangle. These are exactly all the provincial domains for $\cftat(\HH_1\cup \HH_2)$. So the differentials on $\cftat(\HH_1)\boxtimes \cftdat(\HH_2)$ and $\cftat(\HH_1\cup\HH_2)$ agree. 

Half-rectangles on $\HH_1\cup \HH_2$ that cross $\bdy^1\HH_2$ are entirely contained (left provincial) in $\HH_2$, and the same sets of these half-rectangles are counted for the right multiplications $m_2$ on $\cftat(\HH_1)\boxtimes \cftdat(\HH_2)$ and on $\cftat(\HH_1\cup\HH_2)$. 

Thus, the type $A$ structures $\cftat(\HH_1)\boxtimes \cftdat(\HH_2)$ and $\cftat(\HH_1\cup\HH_2)$ are isomorphic. 

The other cases are analogous. 
\end{proof}

In particular,  tangle Floer homology recovers knot Floer homology. For tangles in $B^3$ and $S^2\times I$,  this result is simply a restatement of Theorem \ref{recover_hfk}.

If $\HH_1$ or $\HH_2$ is not a nice diagram, the corresponding structure may not be bounded. In that case, the box tensor product is not defined, and we need to look at the $\cala_{\infty}$ tensor product  $\widetilde{\mathit{CFT}}(\HH_1) \widetilde{\otimes}( \cala((\bdy^0\HH_2)^*)\boxtimes \widetilde{\mathit{CFT}}(\HH_2))$. So by  \cite[Proposition 2.3.18]{bimod}, invariance,  and the above theorem,  $\widetilde{\mathit{CFT}}(\HH_1) \widetilde{\otimes}( \cala((\bdy^0\HH_2)^*)\boxtimes \widetilde{\mathit{CFT}}(\HH_2))\simeq \widetilde{\mathit{CFT}}(\HH_1\cup \HH_2)$, or using the shorter notation,  $\widetilde{\mathit{CFT}}(\HH_1) \widetilde{\otimes} \widetilde{\mathit{CFT}}(\HH_2)\simeq \widetilde{\mathit{CFT}}(\HH_1\cup \HH_2)$. Here $\widetilde{\mathit{CFT}}$ stands for any of the structures in Theorem \ref{thm:cft_pairing}.

\begin{corollary}\label{cor:hat_pairing}
The following equivalences hold:
\begin{enumerate}
\item  If $(Y_1, \T_1)$ is a $2m$-tangle with $\bdy Y_1$ identified with a marked sphere $\mathcal S$,   $(Y_2, \T_2,\gamma)$ is a strongly marked $(2m, 2n)$-tangle with $\bdy^0Y_2$ identified with $-\mathcal S$, and $(Y, \T)$ is their union along $\mathcal S$, then
$$\cftahat(Y_1, \T_1)\boxtimes \cftdahat(Y_2, \T_2,\gamma)  \simeq \cftahat(Y,\T)\otimes V^{\otimes(|\T_1|+|\T_2| - |\T|)}.$$

\item If $(Y_1, \T_1, \gamma_1)$ is a strongly marked $(m,n)$-tangle with $\bdy^1 Y_1$ identified with a marked sphere $\mathcal S$,   $(Y_2, \T_2, \gamma_2)$ is a strongly marked $(n, l)$-tangle with $\bdy^0 Y_2$ identified with $-\mathcal S$, and $(Y, \T, \gamma)$ is their union along $\mathcal S$, then
$$\cftdahat(Y_1, \T_1, \gamma_1)\boxtimes \cftdahat(Y_2, \T_2, \gamma_2)  \simeq \cftdahat(Y,\T, \gamma)\otimes V^{\otimes(|\T_1|+|\T_2| - |\T|)}.$$

\item If $(Y_1, \T_1, \gamma)$ is a strongly marked $(2m,2n)$-tangle with $\bdy^1 Y_1$ identified with a marked sphere $\mathcal S$,  $(Y_2, \T_2)$ is a $2n$-tangle with $\bdy Y_2$ identified with $-\mathcal S$, and  $(Y, \T)$ is their union along $\mathcal S$, then
$$\cftdahat(Y_1, \T_1, \gamma)\boxtimes \cftdhat(Y_2, \T_2)  \simeq \cftdhat(Y,\T)\otimes V^{\otimes(|\T_1|+|\T_2| - |\T|)}.$$

\item If $(Y_1, \T_1)$ is a $2n$-tangle with $\bdy Y_1$ identified with a marked sphere $\mathcal S$,  $(Y_2, \T_2)$ is a $2n$-tangle with $\bdy Y_2$ identified with $-\mathcal S$, and  $(Y, \T)$ is their union along $\mathcal S$, then
$$\cftahat(Y_1, \T_1)\boxtimes \cftdhat(Y_2, \T_2)  \simeq \cfkhat(Y,\T)\otimes V^{\otimes(|\T_1|+|\T_2| - |\T|)}\otimes W,$$
where $W =  \F_2\oplus \F_2$.
\end{enumerate}
Moreover, when the underlying manifolds are $B^3$, $S^2\times I$, or $S^3$, the homotopy equivalences respect the bigrading, where the two summands of $V$ are in $(M, A)$ bigradings $(0,0)$ and $(-1, -1)$, and the two summands of $W$ are in  bigradings $(0,0)$ and $(-1, 0)$. 
\end{corollary}
\begin{proof}
In each case, 
for a choice of nice Heegaard diagrams, we have an equivalence of ``tilde" modules as in the proof of Theorem \ref{thm:cft_pairing}.
To have precisely the ``hat" modules for $\T_1$ and $\T_2$, pick nice Heegaard diagrams $\HH_i$ with $|\X_i\cap \Int\Sigma_i| =|\T_i|$.  Note that on $\HH = \HH_1\cup \HH_2$ we have $|\X\cap \Int(\Sigma_1\cup \Sigma_2)| = |\T_1| + |\T_2|$, and we need a diagram such that $|\X\cap \Int\Sigma| = |\T|$ to obtain the ``hat" module for $\T$, so $\HH$ produces a module equivalent to the ``hat" module tensored with $|\T_1|+|\T_2| - |\T|$ copies of $V$.

Note that in the fourth case $\HH_1\cup \HH_2$ is a Heegaard diagram for the link $\T = \T_1\cup \T_2$ union a split unknot  $U$ in $Y$  (see Lemma \ref{lemma:gluehd}), so
$$\cfkt(\HH_1\cup \HH_2)\simeq \cfkhat(Y, \T \cup U)\otimes V^{\otimes(|\T_1|+|\T_2| - |\T|)} \simeq\cfkhat(Y, \T \cup U)\otimes V^{\otimes(|\T_1|+|\T_2| - |\T|)}\otimes W.$$
The second equivalence is a known fact in Heegaard Floer theory.
\end{proof}

Similar results hold for the various other modules from Subsection \ref{ssec:other}.

%%%%%%%%%%%%%%%%%%%%%%%%%%%%%%%%%%%%%%%%%%%%%%%%%%%%%%%
% section MorseTheory (end)
%%%%%%%%%%%%%%%%%%%%%%%%%%%%%%%%%%%%%%%%%%%%%%%%%%%%%%%

%%%%%%%%%%%%%%%%%%%%%%%%%%%%%%%%%%%%%%%%%%%%%%%%%%%%%%%

\bibliographystyle{hamsplain2}
%\nocite{*}
\bibliography{master}

\end{document}